\documentclass{amsart}
\usepackage{amssymb}
\usepackage{amsbsy}
\usepackage{slashed}
\usepackage[utf8]{inputenc}
\usepackage[margin=1in]{geometry}
\usepackage{enumerate}
\usepackage{tikz-cd}
\usepackage{upgreek}
\usepackage{enumitem}
\usepackage{hyperref}
\usepackage{setspace}
\makeatletter
\DeclareRobustCommand\widecheck[1]{{\mathpalette\@widecheck{#1}}}
\def\@widecheck#1#2{%
    \setbox\z@\hbox{\m@th$#1#2$}%
    \setbox\tw@\hbox{\m@th$#1%
       \widehat{%
          \vrule\@width\z@\@height\ht\z@
          \vrule\@height\z@\@width\wd\z@}$}%
    \dp\tw@-\ht\z@
    \@tempdima\ht\z@ \advance\@tempdima2\ht\tw@ \divide\@tempdima\thr@@
    \setbox\tw@\hbox{%
       \raise\@tempdima\hbox{\scalebox{1}[-1]{\lower\@tempdima\box
\tw@}}}%
    {\ooalign{\box\tw@ \cr \box\z@}}}
\makeatother

\theoremstyle{definition}
\newtheorem{thm}{Theorem}[section]
\newtheorem{defn}[thm]{Definition}
\newtheorem{rem}[thm]{Remark}

\newtheorem{lem}[thm]{Lemma}
\newtheorem{asm}[thm]{Assumption}
\newtheorem{exmp}[thm]{Example}

\newtheorem{cor}[thm]{Corollary}
\newtheorem{prop}[thm]{Proposition}
\newtheorem{defn-prop}[thm]{Definition-Proposition}
\newtheorem{defn-lem}[thm]{Definition-Lemma}
\newtheorem{prop-defn}[thm]{Proposition-Definition}
\newtheorem{thm*}{Theorem*}[section]
\newtheorem{prop*}[thm*]{Proposition}

\DeclareMathOperator{\pt}{pt}

\DeclareMathOperator{\CSD}{CSD}

\DeclareMathOperator{\Fib}{Fib}
\DeclareMathOperator{\Pin}{Pin}

\DeclareMathOperator{\U}{U}
\DeclareMathOperator{\Spec}{Spec}
\DeclareMathOperator{\grad}{grad}
\DeclareMathOperator{\red}{red}

\DeclareMathOperator{\Hom}{Hom}

\DeclareMathOperator{\HSW}{\mathit{HSW}}

\DeclareMathOperator{\Int}{Int}
\DeclareMathOperator{\reals}{\mathbb R}

\DeclareMathOperator{\C}{\mathbb{C}}
\DeclareMathOperator{\End}{End}
\DeclareMathOperator{\Img}{Im}

\DeclareMathOperator{\Hess}{Hess}
\DeclareMathOperator{\HMR}{\mathit{HMR}}
\DeclareMathOperator{\HM}{\mathit{HM}}
\DeclareMathOperator{\KR}{KR}

\DeclareMathOperator{\SF}{SF}
\DeclareMathOperator{\ev}{ev}

\DeclareMathOperator{\loc}{loc}
\DeclareMathOperator{\ind}{\mathsf{ind}}

\DeclareMathOperator{\Ftwo}{\mathbb F_2}

\renewcommand{\Re}{{\text{Re}}}

\newcommand{\whole}{\underline}
\newcommand{\del}{\ensuremath{\partial}}
\newcommand{\delbar}{\ensuremath{{\bar{\partial}}}}
\newcommand{\vol}{{\text{vol}}}
\newcommand{\T}{{\text{T}}}
\newcommand{\pertL}{{\slashed{\mathcal L}}}

\newcommand{\tr}{\ensuremath{{\sf{tr}}}}
\newcommand{\gr}{\ensuremath{{\sf{gr}}}}

\title{Monopole Floer Homology and Real Structures}
\date{\today}

\begin{document}
\author[J.~Li]{Jiakai Li}
                    
\address{Dept. of Math., 
Harvard Univ., 
Cambridge MA, 
United States 02138}

\email{jiakaili@math.harvard.edu}

\begin{abstract}
 	We define a ``real'' version of Kronheimer-Mrowka’s monopole Floer homology for a 3-manifold equipped with an involution. 
 	As a special case, we obtain invariants for links via their double branched covers. The new input is the notion of a real spin\textsuperscript{c} structure, which consists of a spin\textsuperscript{c} structure along with a compatible anti-linear involution on the spinor bundle.
\end{abstract}
\maketitle
\tableofcontents

\section{Introduction}
Andreas Floer introduced an infinite dimensional analogue of ``Morse homology (in the middle dimension)'',  known as \emph{Floer homology} \cite{FloerInstanton1988}. 
As topological invariants of 3-manifolds, there are several flavours of Floer homologies constructed from either pseudo-holomorphic curve theory or gauge theory.
In the case of Seiberg-Witten equations, Kronheimer and Mrowka \cite{KMbook2007} defined the \emph{monopole Floer homology groups}
\begin{equation*}
	\widehat{\HM}(Y), \quad \widecheck{\HM}(Y),\quad \overline{\HM}(Y)
\end{equation*}
(pronounced ``H-M-to'', ``H-M-from'', and ``H-M-bar'', respectively) of a 3-manifold $Y$.

The goal of this paper is to develop \emph{real} versions of the monopole Floer homology groups, as invariants of 3-manifolds with involutions (i.e. \emph{real} manifolds).
For a pair $(Y,\upiota)$ of a 3-manifold $Y$ and an orientation-preserving involution $\upiota:Y \to Y$,
we construct groups
\begin{equation*}
	\widehat{\HMR}(Y, \upiota), \quad \widecheck{\HMR}(Y, \upiota),\quad \overline{\HMR}(Y, \upiota).
\end{equation*}
Given a spin\textsuperscript{c} structure $\mathfrak s$,
the new input is a choice of an anti-linear involutive lift $\uptau: S \to S$ of  $\upiota$ on the spin\textsuperscript{c} bundle $S$, which in turn can be upgraded to a real structure on the Seiberg-Witten configuration space associated to $\mathfrak s$.

The real monopole Floer homology groups are direct sums over isomorphism classes of $(\mathfrak s, \uptau)$'s, which we will refer to as \emph{real spin\textsuperscript{c} structures} in Definition~\ref{defn:realspincstr}.
We write
\begin{equation*}
	\HMR^{\circ}_{\bullet} (Y,\upiota) =\bigoplus_{(\mathfrak s, \uptau)} \HMR^{\circ}_{\bullet}(Y,\upiota; \mathfrak s, \uptau),
\end{equation*} 
where the circle $\circ \in \{\widehat{\quad}, \widecheck{\quad}, \overline{\quad}\}$ is a placeholder for the flavours of the Floer homology, and the bullet point ``$\bullet$'' denotes the grading.
Each $\HMR^{\circ}_{\bullet}(Y,\upiota; \mathfrak s, \uptau)$ is a graded module over the ring
\begin{equation*}
	\mathcal R_n = \frac{\mathbb F_2[[\upsilon_1, \dots, \upsilon_n]]}{\upsilon_i^2 = \upsilon_j^2},
\end{equation*}
where each $\upsilon_i$ has degree $(-1)$, and $n$ is the number of components of the fixed-point set $Y^{\upiota}$.
It is a distinctive feature of $\HMR^{\circ}$ that the natural ``$U$ maps'' have degree $(-1)$ instead of $(-2)$.
The three groups fit into a long exact sequence of $\mathcal R_n$-modules
\begin{equation*}
	\begin{tikzcd}
		\ar[r] 
		& \widecheck{\HMR}_{k}(Y,\upiota;\mathfrak s,\uptau)	 \ar[r] 
		& \widehat{\HMR}_{k}(Y,\upiota;\mathfrak s, \uptau) \ar[r]
		& \overline{\HMR}_{k-1}(Y,\upiota;\mathfrak s, \uptau) \ar[r]
		& \widehat{\HMR}_{k-1}(Y, \upiota;\mathfrak s, \uptau) \ar[r]
		&{}
	\end{tikzcd}
\end{equation*}
We also establish the functoriality of $\HMR^{\circ}$. 
Given a cobordism $W:Y_- \to Y_+$ and an involution $\upiota_W:W \to W$ which induces involutions $\upiota_{\pm}: Y_{\pm} \to Y_{\pm}$, we have a cobordism map
\begin{equation*}
	\HMR^{\circ}(W,\upiota_W)=
	\sum_{(\mathfrak s_W,\uptau_W)}\HMR^{\circ}(W,\upiota_W;\mathfrak s_W,\uptau_W):
	\HMR^{\circ}_{\bullet}(Y_-,\upiota_-) \to \HMR^{\circ}_{\bullet}(Y_+,\upiota_+),
\end{equation*}
satisfying the composition law, where the infinite sum is taken over all real spin\textsuperscript{c} structures over $(W,\upiota_W)$.

While the definitions apply to general 3-manifolds with involutions, our construction stems from the study of double branched covers of links equipped with the covering involutions.
In particular, for a link $ K \subset S^3$, we denote the real monopole Floer homologies of its branched double cover by
\begin{equation*}
	\widehat{\HMR}(K), \quad \widecheck{\HMR}(K),\quad \overline{\HMR}(K).
\end{equation*}
Every spin\textsuperscript{c} structure on the double branched cover admits a unique real structure, so
\begin{equation*}
	\HMR^{\circ}_{\bullet} (K) =\bigoplus_{\mathfrak s} \HMR^{\circ}_{\bullet}(K; \mathfrak s),
\end{equation*}
as $\mathcal R_n$-modules.
Furthermore, $\HMR^{\circ}$ is functorial with respect to (not necessarily oriented) link cobordisms.
That is, given a properly embedded surface $S \subset  [0,1] \times S^3$ such that $\del S = \overline{K_-} \sqcup K_+$, we have a cobordism map
\begin{equation*}
	\HMR^{\circ}(S):
	\HMR^{\circ}_{\bullet}(K_-) \to 
	\HMR^{\circ}_{\bullet}(K_+),
\end{equation*}
induced from the double branched cover along $S$.

There are two sources of motivation for this work: knot Floer homologies from orbifold-like constructions and ``real'' gauge theoretic invariants.
An example of orbifold-like construction is the singular instanton Floer homologies defined by Kronheimer and Mrowka \cite{KMunknot2011}.
On the other hand, codimension-2 singularities also arise as loci of anti-holomorphic involutions on complex algebraic surfaces.
In the case where a real structure on a 4-manifold acts on gauge theoretic configuration spaces, one may define ``real'' versions of invariants.
For instance,  Gang Tian and Shuguang Wang
\cite{TianWang2009} defined \emph{real} Seiberg-Witten invariants for hermitian almost complex 4-manifolds.
Our construction can be viewed as the associated Floer homology of their theory.

While the orbifold story and the \emph{real} story both give rise to codimension-2 singularities in dimension three, the two directions are, in some sense, complementary.
For one, orbifold gauge fields are rather flexible, while for links in 3-manifolds that are not integral homology spheres, there are obstructions to the existence of double branched covers, and they are often not unique.
For another, if one looks one dimension lower, orbifold Riemann surfaces have point singularities, while \emph{real} Riemann surfaces have circle singularities.
From an Atiyah-Floer perspective, the two theories correspond to different flavours of Lagrangian intersection homology.
The former is closer to the Heegaard-Floer-theoretic knot Floer homology, but the latter should be a \emph{real} Lagrangian intersection homology under anti-symplectic involutions.

Despite being link invariants, our $\HMR$'s do not seem to be isomorphic to any of the existing link Floer homologies from Heegaard Floer theory, instanton theory, or their sutured variations.
For instance, the $\HMR$'s of torus knots contain no irreducible elements, as we shall see in Section~\ref{sec:examples}.
However, the branched cover approach has the advantage that many 3-manifold invariants based on Floer homology (such as the Fr\o yshov invariant \cite{Froyshov2010}) translate to knot invariants.

We expect similar versions of Heegaard Floer homology and embedded contact homology.
Apart from the potential Heegaard-Floer variant, there is evidence for a ``\emph{real} ECH'' for contact 3-manifolds, such as a real (anti-symplectic) version of Taubes' ``Gr=SW'' \cite{TaubesSWGR2000} in Gayet's preprint ~\cite{GayetSWR2004}.
See e.g. \cite{CengizOzturk2021},\cite{OzturkSalepci2015} for \emph{real} contact structures.
It would be interesting to explore contact-geometric applications, as the relevant  invariant submanifolds are naturally Lagrangians and Legendrians.

The idea of producing invariants of knots from Seiberg-Witten theory on double branched covers has been explored elsewhere in the literature.
For example, Baraglia and Hekmati (\cite{BaragliaHekmati2021}, \cite{BaragliaHekmati2022}) constructed equivariant Seiberg-Witten Floer cohomologies for rational homology spheres with finite group actions, using the Floer-spectral approach of Manolescu~\cite{Manolescu2003}.
The covering actions of branched covers act $\C$-linearly on the configurations spaces, and the fixed points descend to orbifold Seiberg-Witten solutions.
For $\C$-anti-linear involutions, Nobuhiro Nakamura \cite{NNakamura2015} studied $\Pin^-(2)$-equivariant homotopy refinements of the Seiberg-Witten invariants.
Furthermore, H. Konno, J. Miyazawa, and M. Taniguchi~\cite{KMT2021} constructed a $\mathbb Z/4$-equivariant Floer K-theory for knots using double branched covers with spin structures.
They recently also extended their construction to other spin\textsuperscript{c} structures \cite{KMT2022}, and proved several interesting applications.
We expect their groups to be isomorphic to ours, in the spirit of Manolescu-Lidman's isomorphism ``$\HSW \cong \HM$'' \cite{LidmanManolescu2018}.
Finally, Montague \cite{montague2022seibergwitten} defined equivariant Floer K-theory of spin 3-manifolds with cyclic group actions which applies to branched covers along knots.

Many aspects of $\HMR$ as a Floer homology for links such as the Fr\o yshov invariant,  unoriented skein exact triangles, along with more calculations of examples, will be developed in a subsequent paper.
We also expect a spectral sequence from a Khovanov-like homology to a suitable version of $\HMR$.

\subsection{Sectional Guide}
We begin in Section~\ref{sec:morse} with the finite dimensional model of Morse theory on real blow-ups of manifolds.
We explain the topological inputs of the theory in Section~\ref{sec:top_setup}.
We define real structures on configuration spaces in Section~\ref{sec:gauge_setup}, and blow-ups of real configuration spaces in Section~\ref{sec:config_slice}.
We set up the perturbation scheme and the associated transversality results for critical points in Section~\ref{sec:perturb} and Section~\ref{sec:Transversality}, respectively.
The moduli spaces of gradient flow lines are discussed in Section~\ref{sec:modspace_trajectories}.
We compactify the trajectory spaces in Section~\ref{sec:compact} and examine their boundary strata in Section~\ref{sec:Gluing}.
In Section~\ref{sec:floer}, we define the real monopole Floer homologies.
Sections~\ref{sec:mod_mfld_boundary} and~\ref{sec:cob_module} develop the functorial aspects of $\HMR^{\circ}$.
We conclude with discussions of examples in Section~\ref{sec:examples}.

\subsection{Acknowledgement}
I am grateful to my advisor Peter Kronheimer for his guidance and support, and for suggesting this problem.
I would also like to thank Xujia Chen, Aliakbar Daemi, Hokuto Konno, Jin Miyazawa, Ian Montague, and Masaki Taniguchi for many helpful conversations.

\section{Finite Dimensional Morse Theory Model}
\label{sec:morse}
The finite dimensional model for the \emph{real} Seiberg-Witten configuration space is a finite dimensional manifold $P$ with $\mathbb Z_2$-action and nonempty $\mathbb Z_2$-fixed point set $Q$.
Instead of applying the Borel model of $\mathbb Z_2$-equivariant homology, we will blow up $P$ along $Q$ to obtain a manifold $P^{\sigma}$ on which $\mathbb Z_2$ acts freely.
The quotient
\begin{equation*}
	 B^{\sigma} = P^{\sigma}/\mathbb Z_2.
\end{equation*}
is a manifold with boundary, and the boundary $\del  B^{\sigma}$ is a $\mathbb{RP}^n$-bundle over $Q$.
The three homology groups 
\begin{equation*}
	H_*(B^{\sigma}), \quad
	H_*(B^{\sigma},\del B^{\sigma}), \quad
	H_*(\del B^{\sigma})
\end{equation*}
associated to the pair $(B^{\sigma},\del B^{\sigma})$ provide finite dimensional models of the three flavours 
\begin{equation*}
	\widecheck{\HMR}, \quad \widehat{\HMR}, \quad \overline{\HMR}.
\end{equation*}
of real monopole Floer homology.
This section is the counterpart of \cite[Section~2]{KMbook2007}, and we use identical notations.

\subsection{Morse theory on real projective spaces}
\hfill \break
Equip $\mathbb{RP}^n$ with induced metric on unit sphere in $\reals^n$, and consider the following function on $\mathbb{RP}^{n}$ 
\begin{equation}
\label{eqref:lambda_fd}
	\Lambda^*(x) = \langle x, Lx\rangle/\|x\|^2,
\end{equation}
where $L$ is a symmetric matrix on $\reals^n$.
\begin{lem}
The gradient trajectories of $(\frac{1}{2}\Lambda^*)$ in $\mathbb{RP}^{n-1}$ are images under the quotient map $\pi:\reals^n \to \mathbb{RP}^{n-1}$, of the nonzero solutions to
\begin{equation*}
	dx/dt = -Lx,
\end{equation*} 	
in $\reals^n$. The equation above is the gradient flow equation for $f(x) = \langle x, Lx \rangle/2$.
\end{lem}
\begin{proof}
	Same as \cite[Lemma~2.3.1]{KMbook2007}.
\end{proof}
Recall that the downward gradient flow of $f$ restricted to the unit sphere $S^n$ is
\begin{equation*}
	dw/dt = - Lw + \Lambda(w)w.
\end{equation*}
From this equation we deduce the critical points.
\begin{lem}
\label{lem:RPn_crit}
	The critical points of the function $\Lambda^*$ on $\mathbb{RP}^n$ are the images of the eigenvectors in $\reals^n$ of $L$.
	Assume in addition that the spectrum of $L$ is simple, i.e. the eigenvalues $\lambda_1 < \dots < \lambda_n$ are distinct.
	Let $w_i$ be the unit eigenvector of $\lambda_i$.
	Then
	\begin{itemize}[leftmargin=*]
		\item The critical points $[w_i]$ are non-degenerate.
		\item The index of $[w_i]$ is $(i-1)$.
		\item The closures of the unstable and stable manifolds of $[w_i]$ in $\mathbb{RP}^{n-1}$ are projective subspaces spanned by
			\begin{equation*}
			[w_1],\dots,[w_i],
			\end{equation*}
			and
			\begin{equation*}
				[w_i],\dots,[w_n]
			\end{equation*}
			respectively. 
			Also, the unstable and stable manifolds themselves are the affine subspaces where the components of $w_i$ is non-zero.
		\item The space $M(i,j)$ of trajectories from $[w_i]$ to $[w_j]$ can be identified with the quotient of $\reals^*$ of the set of non-zero solutions $x(t)$ to the linear equations $dx/dt = -Lx$ which has the asymptotics
		\begin{align*}
			x(t) &\sim c_0 e^{-\lambda_i t}w_i, \quad \text{as }t \to -\infty, \text{ and}\\
			x(t) &\sim c_1 e^{-\lambda_j t}w_j, \quad \text{as }t \to \infty,
		\end{align*}
		where $c_0,c_1 \in \reals^*$.
\end{itemize}
\end{lem}
Computation of the Morse homology of a real projective space is slightly subtler than a complex projective space, for the index difference of adjacent critical points is one instead of two.
The differential is zero in $\mathbb F_2$-coefficients for $\mathbb{RP}^n$ because solutions to the gradient flow from $[w_{i}]$ to $[w_{i-1}]$ are given by
\begin{equation*}
	c_0 e^{-\lambda_i t}w_i + c_1 e^{-\lambda_{i-1}t}w_{i-1}.
\end{equation*}
Up to reparametrization, there are only two possibilities $(c_0 = 1, c_1 = \pm 1)$.

\subsection{Morse theory on manifolds with boundaries}
\hfill \break
Let $(B,\del B)$ be a Riemannian manifold with boundary.
Let $f$ be a Morse function such that the gradient vector field $V$ of $f$ is everywhere tangent to the boundary of $\del B$.
\begin{defn}
	Let $a$ be a critical point of $f$ on $\del B$.
	Then $a$ is \emph{boundary-stable} if the normal vector $N_a$ at $a$ is a positive eigenvector of the Hessian.
	Otherwise, $a$ is \emph{boundary-unstable}.	
\end{defn}
Let $\mathfrak c^s$ and $\mathfrak c^u$ be the set of boundary-stable and boundary-unstable critical points, and let $\mathfrak c^0$ be the set of interior critical points.
Unlike Morse theory in closed manifolds, the usual Morse-Smale condition can never be satisfied between a boundary-stable critical point $a$ and a boundary-unstable critical point $b$, as all trajectories from $a$ to $b$ lie entirely in $\del B$.
We make the following definition. 
\begin{defn}
	Let $a,b$ be two critical points and let $U_a$ be the unstable manifold of $a$, and $S_b$ be the stable manifold.
	The case when $a$ is boundary-stable, and $b$ is boundary-unstable is the \emph{boundary-obstructed} case.
	The Morse function $f$ is \emph{regular} if: 
	\begin{itemize}
		\item in the boundary-obstructed case, $U_a$ and $S_b$ are transverse as intersection in $\del B$, or
		\item in the remaining cases, $M(a,b) = U_a \cap S_b$ is transverse in $B$.
	\end{itemize}
\end{defn}
Let $M(a,b) = U_a \cap S_b$ be the space of trajectories.
Assume $f$ is regular, then $M(a,b)$ is either a manifold or a manifold with boundary. 
The latter happens if and only if $a$ is boundary-unstable and $b$ is boundary-stable; in particular, $\del M(a,b) = M(a,b) \cap \del B$, which we denote as $M^{\del}(a,b)$. 
If $M(a,b)$ or $M^{\del}(a,b)$ is nonempty, then we denote the quotient by reparametrization $\check{M}(a,b) = M(a,b)/\reals$ and $\check{M}^{\del}(a,b) = M^{\del}(a,b)/\reals$, respectively.

The dimension of $M(a,b)$ is given by the difference of Morse index in $B$, except in the boundary-obstructed case, where
\begin{equation*}
	\dim M(a,b) = \ind_M(a) - \ind_M(b) + 1.
\end{equation*}

To define Morse complexes, we first define
\begin{equation*}
	C^o_k = R \otimes \left(\bigoplus_{e_a \in \mathfrak c^o_k} \mathbb Z e_a\right),\quad
	C^s_k = R \otimes \left(\bigoplus_{e_a \in \mathfrak c^s_k} \mathbb Z e_a\right),\quad
	C^u_k = R \otimes \left(\bigoplus_{e_a \in \mathfrak c^u_k} \mathbb Z e_a\right),
\end{equation*}
where $R$ is a ring.
Through out this paper we will take $R = \mathbb F_2$ and ignore orientations of moduli spaces.
The three Morse complexes of interests are
\begin{equation*}
	\bar C_k = C^s_k \oplus C^u_{k+1},\quad
	\check C_k = C^o_k \oplus C^s_k \quad
	\hat C_k = C^o_k \oplus C^u_k.
\end{equation*}
To describe the differential in $\bar C_k$, we define
\begin{equation*}
	\bar\del: \bar C_k \to \bar C_{k-1}, \quad \bar\del e_a = \sum_{b \in \bar C_{k-1}} \# \check M^{\del}(a,b)\cdot e_b,
\end{equation*}
which decomposes as
\begin{equation*}
	\bar\del = \begin{pmatrix}
		\bar\del^s_s & \bar\del^u_s\\
		\bar\del^s_u & \bar\del^u_u
	\end{pmatrix}
\end{equation*}
according to $\bar C_k = C^s_k \oplus C^u_{k+1}$.
Moreover, counting of flow lines in the interior gives rise to the maps
\begin{align*}
	\del^o_o&: C^o_k \to C^o_{k-1}, \quad \del^o_o e_a = \sum_{b \in \mathfrak c^o_k} \# M(a,b)\cdot e_b,\\
	\del^o_s&: C^o_k \to C^s_{k-1}, \quad \del^o_s e_a = \sum_{b \in \mathfrak c^s_k} \# M(a,b)\cdot e_b,\\
	\del^u_o&: C^u_k \to C^o_{k-1}, \quad \del^u_o e_a = \sum_{b \in \mathfrak c^o_k} \# M(a,b)\cdot e_b,\\
	\del^u_s&: C^u_k \to C^s_{k-1}, \quad \del^u_s e_a = \sum_{b \in \mathfrak c^s_k} \# M(a,b)\cdot e_b.
\end{align*}
The Morse differentials $\check\del: \check C_k \to 
\check C_{k-1}$ and $\hat\del: \hat C_k \to \hat C_{k-1}$, can be written in the perspective direct sum decompositions, as
\begin{equation*}
	\check\del = \begin{pmatrix}
		\del^o_o & -\del^u_o\bar \del^s_u \\
		\del^o_s & \bar\del^s_s - \del^u_s\bar\del^s_u
	\end{pmatrix}
	\quad 
	\hat\del = \begin{pmatrix}
		\del^o_o & \del^u_o\\
		-\bar\del^s_u\bar\del^o_s & -\bar\del^u_u -\bar\del^s_u\del^u_s
	\end{pmatrix}
\end{equation*}
We summarize the main results in the following Proposition.
\begin{prop}	
	The squares $\delbar^2, \check\del^2, \hat\del^2$ are all zero.
	There are isomorphisms of between the homology of Morse complexes and singular homology groups:
	\begin{equation*}
		H_k(\bar C_*,\delbar) = H_k(\del B;\mathbb F_2),\quad
		H_k(\check C_*,\check\del) = H_k(B;\mathbb F_2),
			\quad
		H_k(\hat C_*,\hat\del) = H_k(B,\del B;\mathbb F_2).
	\end{equation*}
\end{prop}

\subsection{Real blow-ups}
\hfill \break
Let $P$ be a smooth, closed Riemannian manifold admitting an involution $\upiota: P \to P$.
Let $Q$ be the set of fixed points and $N = N(Q)$ be the normal bundle of $Q \subset P$.
Let $B$ be the quotient $P/\upiota$.

The real blow-up of $P$ is a manifold with boundary on which $\mathbb Z_2$ acts freely, constructed as follows.
Let $S(N) \subset N$ be the unit sphere bundle of $N$ over $Q$, and $p:S(N) \times [0,\epsilon) \to P$ be the map to a geodesic tubular neighbourhood $W$ of $Q$.
Let $p^{o}:S(N) \times (0,\epsilon) \to W\setminus Q$ be restriction on $(0,\epsilon)$ which is a diffeomorphism. 
Then the real blow-up $P^{\sigma}$ along $Q$ is given by:
\begin{equation*}
	P^{\sigma} = ((S(N) \times [0,\epsilon) \cup (P\setminus Q))/p^o.
\end{equation*}
The involution $\upiota$ lifts to a free involution on the blown-up manifold $P^{\sigma}$, which acts as $-1$ on $S(N) \to S(N)$.
Let $\pi:P^{\sigma} \to P$ be the induced projection, and restricting to the boundary $\del P^{\sigma}$, the map $\pi:S(N) \to Q$ is a sphere bundle.
We define the real blow-up of $B$ as the manifold with boundary
\begin{equation*}
	B^{\sigma} = P^{\sigma}/\upiota.
\end{equation*}
The induced projection $\pi|\del B^{\sigma}: \mathbb P(N) \to Q$ is a real projective bundle with fibre $S(N)/\upiota$.

Let $\tilde f: P \to \reals$ be an $\upiota$-invariant smooth function and $\tilde V$ be its gradient vector field.
Away from $Q$, the projection $\pi: P^{\sigma} \to P$ is diffeomorphism and we can pull back the vector field $\tilde V$ to $P^{\sigma} \setminus \del P^{\sigma}$.
By the same exact argument as in KM, we can extend $\tilde V$ all of $P$, such that $\tilde V$ is is tangent along the boundary $S(N)$.
\begin{exmp}
	Let us consider the model case when $P = \reals^n$, $\upiota (p) = -p$, and $\tilde f(p) = \frac{1}{2}\langle p, Lp \rangle$ where $L$ is a symmetric matrix. 
	The real blow-up $P^{\sigma}$ can be identified with
	\begin{equation*}
		P^{\sigma} = S^{n-1} \times [0,\infty),
	\end{equation*}
	which can be thought of as the extension of the ``polar coordinate" on $P$. 
	The downward gradient flow equation $dp/dt = -Lp$ of $\tilde f$ on $P$, in polar coordinate $P^{\sigma}\setminus S(N) \cong S^{n-1} \times (0,\infty)$ can be rewritten as
	\begin{align*}
		\dot\phi &= -L\phi + \Lambda(\phi)\phi,\\
		\dot s &= -\Lambda(\phi)s,
	\end{align*}
	which extends to a smooth flow on $P^{\sigma}$, preserving the boundary.
	The vector field $\tilde V^{\sigma}$ generated by the flow on $P^{\sigma}$ is everywhere tangent to the boundary.
	The vector field $V^{\sigma}$ is invariant under $\upiota$ and hence descends to the quotient $B^{\sigma} = \mathbb{RP}^{n-1} \times [0,\infty)$.
	Suppose the spectrum $\{\lambda_1 < \dots < \lambda_n\}$ of $L$ is simple and contains no zero. 
	Let $\{\phi_i\}$ be orthogonal unit eigenvectors corresponding to $\{\lambda\}$.
	Then the zeros of $V^{\sigma}$ on $B^{\sigma}$ are exactly $([\phi],0) \in \mathbb{RP}^{n-1} \times [0,\infty)$, where 
	\begin{equation*}
		\ind_{B^{\sigma}}([\phi],0)
		= 
		\begin{cases}
			i - 1 & (\lambda_i > 0)\\
			i	& (\lambda_i < 0).
		\end{cases}
	\end{equation*}
\end{exmp}
Back to the general case.
For any $q \in Q$, let $L_q$ be the restriction of the Hessian $\nabla \tilde V$ of $\tilde f$ to $N_q$, which is symmetric. 
We make the following assumption.
\begin{asm}\label{asm:simplespectrum}
	We suppose that the restriction of $\tilde f$ to $Q$ is Morse on $Q$.
	In addition, we assume for every $q \in Q$ the spectrum of $L_q$ is simple and contains no zero.
	We denote $\lambda_1(q) < \dots < \lambda_n(q)$ eigenvalues and $\phi_1(q),\dots,\phi_n(q)$ the corresponding eigenvectors.
\end{asm}
The zeros of the vector field $V^{\sigma}$ can be described in the similar way as the model case.
\begin{lem}\label{lem:classifyboundarycritical}
	Write $a \in \del B^{\sigma}$ as $(q,[\phi])$ where $q \in Q$ and $[\phi] \in \mathbb P(N_q)$.
	Then $a$ is a stationary point for the vector field $V^{\sigma}$ if and only if $q$ is a critical point for the restriction $\tilde f|_Q$ and $\phi$ is an eigenvector of $L_q$.	
	Under Assumption~\ref{asm:simplespectrum}, the zeros $(q,[\phi_i(q)])$ are nondegenerate, and 
	\begin{equation*}
		\ind_{B^{\sigma}}(q,\phi[q]) =
		\begin{cases}
			\ind_Q(q) + i - 1 & (\lambda_i(q) > 0),\\
			\ind_Q(q) + i & (\lambda_i(q) < 0).
		\end{cases}
	\end{equation*}
\end{lem}
\begin{proof}
	The proof again is the same as in [KM]. 
	We remark here that the Hessian of $V^{\sigma}$ takes the form, in the decomposition $T_qB^{\sigma}=T_qQ \oplus W \oplus \reals$,
	\begin{equation*}
		H = 
		\begin{pmatrix}
			h_q & 0 & 0\\
			x & L_q - \lambda_i & 0\\
			0 & 0 & \lambda_i
		\end{pmatrix}
	\end{equation*}
	where $x$ is the operator
\end{proof}
To sum up, the set $\mathfrak c$ of stationary points of $V^{\sigma}$ decomposes as $\mathfrak c^o \cup \mathfrak c^s \cup \mathfrak c^u$. 
The set $\mathfrak c^o$ consists of interior stationary critical points, which can be identified with critical points of the function $f$ on $B \setminus Q$. 
The set $\mathfrak c^s$ consists of boundary-stable critical points $(q,[\phi_i(q)])$, described in Lemma~\ref{lem:classifyboundarycritical}. Similarly, $\mathfrak c^u$ consists of boundary-unstable critical points.
Note that $(q,[\phi_i(q)])$ is boundary-stable if and only if $\lambda_i(q) > 0$, and boundary-unstable if and only if $\lambda_i(q) < 0$.
\begin{thm}
Suppose the flow generated by the vector field $V^{\sigma}$ is regular. 
Then $(\check C_*,\check\del)$, $(\hat C_*,\hat\del)$, and $(\bar C_*,\delbar)$ are complexes, i.e. the differentials square to zero.
In particular, the homology	of the three complexes are absolute and relative homology for the pair $(B^{\sigma},\del B^{\sigma})$, and homology of $\del B^{\sigma}$, respectively. 
\end{thm}

Let us rewrite the gradient flow of $-V^{\sigma}$ on $\del B^{\sigma} = \mathbb P(N)$, in two different ways.
First, we may view the flow in $\mathbb P(N)$ as a flow on $S(N)$.
We lift a trajectory $(q(t),[\phi(t)])$ of $-V^{\sigma}$ to a trajectory $(q(t),\phi(t))$ of $-\tilde V^{\sigma}$ on $S(N)$, where $|\phi(t)| = 1$ for all $t$.
Then
\begin{align*}
	\frac{d}{dt}q + (\grad 	\tilde f|_{Q}))_{q(t)} &= 0\\
	q^*(\nabla)\phi + ((L_{q(t)} - \Lambda_q(\phi))\phi)dt &= 0
\end{align*}
Here $\nabla$ is the connection on $N \to Q$ induced from the Levi-Civita connection.
Second, we may view the trajectory as a trajectory $(q(t),\phi(t))$ in $N$, and the corresponding equation is
\begin{align*}
	\frac{d}{dt}q + (\grad f)_{q(t)} &= 0\\
	q^*(\nabla)\phi + (L_{q(t)}\phi)dt &= 0.
\end{align*}
We have the analogues of Lemma and Lemma in Subsection.
\begin{lem}
	The trajectories of $-V^{\sigma}$ on the boundary $B^{\sigma}$ are the images under the quotient map \begin{equation*}
		\pi: N \setminus 0 \to \mathbb P(N) = \del B^{\sigma}
		\end{equation*}
		of the solutions $(q(t),\phi(t))$ to the equation which is unique up to the action of the scalar $\reals^*$ on $\phi$.
\end{lem}
\begin{prop}
	Let $a_0 = (q_0,[\phi_{i_0}(q_0)])$ and $a_1 = (q_1,[\phi_{i_1}(q_1)])$ be two zeros of $V^{\sigma}$ on $\del B^{\sigma}$, and $\lambda_0,\lambda_1$ be the corresponding eigenvalues. 
	Then the space of trajectories $M^{\del}(a_0,a_1)$ can be identified with quotients by $\reals^*$ of the set of solutions $(q(t),\phi(t))$ to with the asymptotics
	\begin{align*}
		\phi(t) &\sim c_0e^{-\lambda_0t}\phi_{i_0}(q_0), \quad \text{as }t \to -\infty,\\
		\phi(t) &\sim c_1e^{-\lambda_1t}\phi_{i_1}(q_1), \quad \text{as }t \to \infty,
	\end{align*}
	where $c_0,c_1$ are nonzero real constants.
\end{prop}
Finally, the trajectories from interior critical points to boundary-stable critical points can be characterized as follows.
\begin{prop}
The trajectories $x(t)$ of the vector field $-V^{\sigma}$ in the interior of $B^{\sigma}$ that approaches the boundary critical point $(q,[\phi_i(q)])$ as $t \to +\infty$ are in one-to-one correspondence with trajectories $y(t)$ to the gradient flow of $f$ on $B \setminus Q$ which approach $q$ and for which the distance from $y(t)$ to $Q$ has the asymptotics $ce^{-\lambda_i t}$.	
\end{prop}

\section{Topological Setup}
\label{sec:top_setup}
In this section we begin by reviewing real structures on vector bundles.
We then introduce spin\textsuperscript{c} structures compatible with real structures in dimension three and four, and give some examples of 3-manifolds with involutions.

\subsection{Real structures}
\hfill \break
In this subsection, we will define the notion of structure and give criteria for existence and uniqueness of real structures on line bundles.
A prototypical example of a space with real structure
is a complex algebraic variety $X$ equipped with an anti-holomorphic involution $\upiota$. 
And over $X$ there is a holomorphic vector bundle $\mathcal E \to X$, equipped with an anti-linear (conjugate-linear) involution $\uptau$, covering $\upiota$ on the base $X$:
	\[\begin{tikzcd}
	{\mathcal E} & {\mathcal E} \\
	X & X
	\arrow["\uptau", from=1-1, to=1-2]
	\arrow["\upiota", from=2-1, to=2-2]
	\arrow[from=1-1, to=2-1]
	\arrow[from=1-2, to=2-2]
	\end{tikzcd}.\]
In the topological category, the following definition was introduced by Atiyah~\cite{AtiyahK1966} in his construction of $\KR$-theory.
\begin{defn}
	A \emph{real space} $(X,\upiota)$ is a topological space $X$ equipped with an involution $\upiota:X \to X$. 
	A \emph{real vector bundle} $(E,\uptau) \to (X,\upiota)$ over a real space is a  complex vector bundle $E$ together with an anti-linear involution $\uptau: E \to E$ that covers $\upiota$.
	Such an anti-linear involution $\uptau$ is a \emph{real structure} on $E \to (X,\upiota)$.
\end{defn} 
	Here are some relevant examples for real spaces to help building intuitions.
	\begin{itemize}[leftmargin=*]
		\item $X$ is a Riemann surface, $\upiota$ is an anti-holomorphic (so orientation reversing) involution, and $X^{\upiota}$ is a collection of circles.
		\item $X$ is a 3-manifold, $\upiota$ is orientation preserving involution, and $X^{\upiota}$ is a link.
		\item $X$ is an almost complex 4-manifold, and $\upiota$ is anti-holomorphic involution (e.g., if $X$ is a complex algebraic surface and $X^{\upiota}$ is a real algebraic surface).
	\end{itemize}
\begin{rem}
The term ``real vector bundle'' can be easily be mistaken for vector bundles having $\reals^n$ fibres.
We will use the terminology \emph{vector bundles with real structures}.
Similarly, for \emph{real} spaces we will either italicize the adjective \emph{real}, or refer to them as $\mathbb Z_2$-spaces.
\end{rem}
\begin{rem}
	A related notion is a \emph{quarternionic} structure, which is an anti-linear lift of $\upiota$ that squares to $-1$.
	Whether an anti-linear involution admits an order-two (\emph{odd type}) or order-four (\emph{even type}) lift on spin\textsuperscript{c} bundle depends on the codimension of the fixed point set.
	This follows from a local calculation similar to~\cite[Section~8]{AtiyahBott1968II}.
\end{rem}
\begin{defn}
	Let $(E,\uptau) \to (X, \upiota)$ be a bundle with real structure, equipped with a hermitian inner product $\langle \cdot, \cdot \rangle$.
	Then $\uptau$ is \emph{compatible} with $\langle \cdot, \cdot \rangle$ if
	$
		\langle \uptau(\cdot), \uptau(\cdot) \rangle_{\upiota(x)} = \overline{\langle \cdot, \cdot \rangle_{x}}.
	$	
\end{defn}

Since the gauge groups of Seiberg-Witten theory are $\U(1)$-valued, we are primarily concerned with line bundles with real structures.
There are two relevant characteristic classes for $(L, \uptau)$: the first Chern class $c_1(L) \in H^2(X;\mathbb Z)$ and the first Stiefel-Whitney class $w_1(L^{\uptau}|X^{\upiota}) \in H^1(X^{\upiota};\Ftwo)$, where $L^{\uptau}$ is the $\uptau$-invariant sub-bundle of the restriction of $L$ to the fixed point $X^{\upiota}$.
A necessarily condition for a line bundle $L$ to admit a real structure is
\[
\upiota^*(c_1(L)) = c_1(\bar L) = -c_1(L).
\]

\subsubsection*{\textbf{Existence of real structures on line bundles}}
\hfill \break
Let $M$ be a connected manifold, often 3- or 4-dimensional.
Suppose $\upiota: M \to M$ is an involution with non-empty fixed point set.
Let $L \to M$ be a complex line bundle, equipped with a hermitian metric.
We equip $M$ with an $\upiota$-equivariant CW structure; that is, 
\begin{itemize}[leftmargin=*]
	\item for any $n$-cell $e:D^n \to M$, the image $\upiota \circ e$ is again a $n$-cell, and
	\item either $e(\Int D^n) \cap (\upiota \circ e)(\Int D^n) = \emptyset$, or $e(D)$ is fixed point-wise.
\end{itemize}
This cell structure descends to the quotient space $M/\upiota$, where we denote a cell $[e]$ as the coset $\{e,\upiota e\}$.
With the two cell structures on $M$ and $M/\upiota$, we define a chain-level map:
\begin{align*}
	\tilde\Theta:C^n(M;\mathbb Z) 
	&\to C^n(M/\upiota;\mathbb Z) \\
	\beta 
	&\mapsto
	\left([e] \mapsto \beta(e) + \beta(\upiota \circ e)\right).
\end{align*}
In particular, for $e$ invariant under $\upiota$, the image $\tilde\Theta(\beta)$ takes the value $2\beta(e)$.
The map commutes with the CW differential $\delta$'s and thus defines a map on the level of cohomology
\begin{equation}
\label{eq:theta_defn}
	\Theta: H^n(M;\mathbb Z) \to H^n(M/\upiota;\mathbb Z).
\end{equation}
This map fits into the following diagram:
\[\begin{tikzcd}
	{H^n(M;\mathbb Z)} & {H^n(M;\mathbb Z)^{\upiota^*}} \\
	{H^n(M/\upiota;\mathbb Z)}
	\arrow["\Theta", from=1-1, to=2-1]
	\arrow["{\pi^*}"', from=2-1, to=1-2]
	\arrow["{1+\upiota^*}", from=1-1, to=1-2]
\end{tikzcd}\]

\begin{lem}
\label{lem:cohom_char_real_str}
	Let $M$ be a manifold and assume $\upiota$  has nonempty fixed point locus.
	Let $L \to M$ be a complex line bundle.	
	Then $L$ admits a real structure if and only if 
	\begin{equation*}
		\Theta(c_1(L)) = 0.
	\end{equation*}
\end{lem}
\begin{proof}
	The proof is by obstruction theory, and one key observation is that a real structure over a trivialization on the $1$-skeleton is equivalent to an $\upiota$-invariant map to $S^1$.
	Denote $M^n$ as the $n$th skeleton.
	Over the 1-skeleton $M^1$, choose a trivialization $t_1$ of $L$:
	\begin{equation*}
		t_1: L|_{M^1} \to M^1 \times \C.
	\end{equation*}
	For each $2$-cell $e:(D^2,\del D^2) \to (M^2, M^1)$, choose a trivialization $t_e$ of $e^*L$:
	\begin{equation*}
		t_e: e^*L \to D^2 \times \C.
	\end{equation*}
	Over $\del D^2$, the difference between $t_1$ and $t_e$ is a map $c_e: \del D^2 \to S^1$, such that for any $x \in \del D^2$ and $v \in L_{e(x)}$
	\begin{equation}\label{eqn:real_lem_1}
		t_1(v) = c_e(x) t_e(v).
	\end{equation}
	We define a cochain $\gamma \in C^2(M;\mathbb Z)$, mapping a 2-cell $e$ to the degree of $c_e$ over $\del e$.
	Since $L$ is defined over the entire space, $\gamma$ is a cocycle and in represents the first chern class
	\begin{equation*}
		[\gamma] = e(L) = c_1(L).
	\end{equation*}
	The cochain $\tilde\Theta(\gamma)$ therefore takes a 2-cell $[e]$ to the sum of winding numbers of $c_e$ and $c_{\upiota e}$.

	First, assume $L$ admits a real structure $\uptau: L \to L$.
	From the trivialization $t_1$ on $L \to M^1$, we obtain a map $\beta: M^1 \to S^1$, defined by
	\begin{equation}\label{eqn:real_lem_2}
		t_1(\uptau(v)) = \beta(p) \overline{t_1(v)},
	\end{equation}
	for any $p \in M$ and $v \in L_p$.
	The relation $t_1(v) = t_1(\uptau^2(v))$ impies $\beta = \beta \circ \upiota$ so $\beta$ descends to the quotient, which we denote $\bar\beta: M/\upiota \to S^1$.
	
	Let us define an 1-cochain $\beta' \in C^1(M/\upiota;\mathbb Z)$ as follows.
	Choose a homotopy $h_{\beta}: I \times M^0/\upiota \to S^1$ from $\bar\beta$ to the constant map $1$.
	For an 1-cell $\ell$ connecting 0-cells $a$ to $b$, we set $\beta'(\ell)$ to be the degree of the map
	\begin{equation*}
		h_{\beta}(\cdot ,a)^{-1} * \overline\beta|_{\ell} * h_{\beta}(\cdot ,b): I \to S^1
	\end{equation*}
	where $*$ denotes concatenation of paths. We claim
	\begin{equation*}
		\delta \beta' = \tilde\Theta(\gamma).
	\end{equation*}
	For a 2-cell $e$ of $M$, combination of Equations~\eqref{eqn:real_lem_1} and ~\eqref{eqn:real_lem_2} yields the following:
	\begin{equation}
	\label{eqn:boundary_real_matching}
		t_{\upiota e}(\uptau(v)) =
		\overline{c_{\upiota e}(x)} \cdot \beta(e(x)) 
		\cdot \overline{c_e(x)} \cdot 
		\overline{t_e(v)},
	\end{equation}
	where $x \in D^2$ and $v \in L_{e(x)}$. 
	But $\uptau$ is defined over $e$, as a map $(\del e)^*\uptau: (\del e)^*L \to (\del \upiota e)^*L$ over $D^2$, so
	\begin{equation*}
		p \mapsto 
		\overline{c_{\upiota e}(p)} \cdot
		\overline{c_e(p)} 	\cdot 
		\beta(p) 
	\end{equation*}
	over $\del D^2$ must have degree zero.
	Thus
	\begin{equation*}
	\delta \beta'[e] = 
		\deg \beta = \deg c_{\upiota e} + \deg c_{e} = \tilde\Theta(\gamma)[e].
	\end{equation*} 
	
	Conversely, assume $\tilde \Theta(\gamma)$ is a coboundary $\delta \beta'$, for some $\beta' \in C^1(M/\upiota;\mathbb Z)$.
	The real structure over $M^1$ is equivalent to an $\upiota$-invariant map $\beta: M^1 \to S^1$ via  Equation~\eqref{eqn:real_lem_1}.
	And we define $\beta$ by mapping all 0-cells to $1$, and $1$-cells to $S^1$ with degree $\beta'$.
	In particular, the image of invariant $1$-cells under $M \to M/\upiota$ should takes constant value $1$.
	
	Next, we must extend the real structure to the $2$-skeleton.
	Let $e$ be a 2-cell of $M$ and we seek a function $f_e: D^2 \to S^1$ so as to define $\uptau$ by
	\begin{equation}
	\label{eqn:real_lem_3}
		t_{\upiota e} (\uptau(v)) = f_e(x) \overline{t_e(v)}.
	\end{equation}
	The boundary condition for $\uptau: e^*L \to (\upiota e)^*L$ is given by Equation~\eqref{eqn:boundary_real_matching}: 
	for any $x \in \del D^2$,
	\begin{equation*}
		f_e(x) = \overline{c_{\upiota e}(x)} \cdot
		\overline{c_e(x)} 	\cdot 
		\beta(e(x))
	\end{equation*}
	which indeed is extendable to an $f_e$ by $[\tilde\Theta(\gamma)]=0$.
	Moreover, since the boundary values for $f_e$ and $f_{\upiota e}$ are the same, we arrange so that $f_e = f_{\upiota e}.$
	It then follows that the resulting $\uptau$ is a real structure
	\begin{equation*}
		t_e(\uptau^2(v)) = f_{\upiota e}(x)\overline{f_e(x)} t_e(v) = t_e(v).
	\end{equation*}

	Finally, by induction consider $n$-cells $e: D^n \to (M, M^{n-1})$ and its image $\upiota \circ e$ with $n \ge 2$, we choose a trivialization $t_e: e^*L \to D^n \times \C$.
	The map $\uptau$ will again be defined by Equation~\eqref{eqn:real_lem_3}.
	There is no obstruction to extending from $\del D^n$ to $D^n$ because $\pi_{n-1}(S^1) = 0$, but we must choose $f_e = f_{\upiota e}$ to ensure $\uptau^2 = 1$.
	Indeed, the boundary value of $f_e$ is invariant under $\upiota$, because we assumed a real structure on $M^{n-1}$.
\end{proof}
\subsubsection*{\textbf{Uniqueness of real structures on line bundles}}
\begin{defn}
	Two real structures $\uptau_0$ and $\uptau_1$ on a hermitian line bundle $L$ are \emph{equivalent}   if there exists $g: M \to S^1$ such that
\begin{equation*}
	g \uptau_0 = \uptau_1 g.
\end{equation*}
\end{defn}
Suppose $\uptau_0$ and $\uptau_1$ are two real structures on $L$.
Then the map $u = \uptau_0^{-1}\uptau_1$ must be $\upiota$-invariant, as $(u \uptau_0)^2 = u \cdot \overline{\upiota^* u} \uptau_0^2 = 1$.
Let $\underline{\mathcal G}$ be $C^{\infty}(M,S^1)$ and $\mathcal G^{\upiota}$ be the subgroup of $\upiota$-invariant maps.
Let $\textsf{sym}: \underline{\mathcal G} \to \mathcal G^{\upiota}$ be the ``symmetrizing map'' sending each $g \in \underline{\mathcal G}$ to the invariant automorphism $g(\upiota^*g)$.
Therefore the space of equivalence classes of real structures on $L \to M$ is the quotient:
\begin{equation}
\label{eqn:realstr_quotient}
	\frac{\mathcal G^{\upiota}}{\Img \textsf{sym}}.
\end{equation}
To analyze this quotient, we use the decomposition of $\bar{\mathcal G}$ in Subsection~\ref{subsec:topofconfig}.
If an element $u \in \bar{\mathcal G}$ is homotopic to the identity, then $u = e^{if}$ for some $\upiota$-invariant function $f:M \to \reals$, in which case $u = g(\upiota^*g)$ for $g = \exp(if/2)$.
By considering harmonic maps to $S^1$, the quotient in \eqref{eqn:realstr_quotient} is a quotient of cohomolog groups
\begin{equation*}
	\frac{H^1(M;\mathbb Z)^{\upiota^*}}{\Img(1+\upiota^*:H^1(M;\mathbb Z) \to H^1(M;\mathbb Z)^{\upiota^*})}.
\end{equation*}
Since $H^1(M;\mathbb Z)$ is torsion-free and $H^1(M;\mathbb Z)^{\upiota^*} = H^1(M/\upiota;\mathbb Z)$, the above quotient is equivalent to
\begin{equation*}
	\frac{H^1(M/\upiota;\mathbb Z)}{\Img \Theta},
\end{equation*}
where $\Theta$ is the map defined in \eqref{eq:theta_defn}.

\subsection{Spin\textsuperscript{c} structures, involutions, and real structures on three-manifolds}
\hfill \break
Let $Y$ be a closed 3-manifold and $g$ be a Riemannian metric.
A \emph{spin\textsuperscript{c} structure} $\mathfrak s = (S,\rho)$ is a pair of a hermitian rank-two vector bundle $S \to Y$, and a Clifford multiplication $\rho: \T Y \to \End(S)$.
The Clifford multiplication provides an isometry between the tangent bundle and the subbundle $\mathfrak{su}(S)$ of traceless, skew-adjoint endomorphisms equipped with the inner product $\frac12\tr(a^*b)$.
Locally there is an orthonormal frame $\{e_1,e_2,e_3\}$ so that $\{\rho(e_i)\}$ are Pauli matrices:
\begin{equation*}
	\rho(e_1) = \begin{bmatrix}
		i & 0\\
		0 &-i
	\end{bmatrix},\quad 
	\rho(e_2) = \begin{bmatrix}
		0 & -1\\
		1 & 0
	\end{bmatrix}, \quad 
	\rho(e_3) = \begin{bmatrix}
		0 & i\\
		i & 0
	\end{bmatrix}.
\end{equation*}
Using the Riemannian metric, we extend $\rho$ to the cotangent bundle $\T^*Y$, the complexified cotangent bundle $\T^*_{\C}Y$, and any complex-valued forms $\bigwedge^*\T^*_{\C}(Y)$.
\hfill \break

Involutions on $Y$ interact with spin\textsuperscript{c} structures on $Y$ via real structures on spinor bundles.

\begin{defn}
\label{defn:realspincstr}
	Let $\upiota$ be an involution, $g$ be an $\upiota$-invariant Riemannian metric, and $\mathfrak s= (S,\rho)$ be a spin\textsuperscript{c} structure.
	A \emph{real structure compatible with} $\mathfrak s$
	 is a real structure $\uptau: S \to S$ on the spinor bundle that is compatible with the hermitian metric and compatible with Clifford multiplications in the following sense:
	 \begin{equation}
	\label{eqn:compat1}
		\rho(\upiota_* \xi)\uptau(\Phi_{y}) = \uptau(\rho(\xi)\Phi_y),
	\end{equation}
	for any $y \in Y$, any vector field $\xi$ on $Y$, and any spinor $\Phi \in \Gamma(S)$.
	A pair $(\mathfrak s, \uptau)$ of spin\textsuperscript{c} structure with a real structure $\uptau$ is a \emph{real spin\textsuperscript{c} structure}.
\end{defn}
\begin{rem}
	A real spin\textsuperscript{c} structure is \textbf{not} a $\reals^4$-bundle equipped with Clifford multiplication.
	While \emph{real bundles} can be confusing sometimes, there should not be any confusion in the present article.
\end{rem}

The following lemma gives a sufficient condition for the existence of real structures, in the case when the involution preserves a spin structure.
\begin{lem}
\label{lem:spin_lift}
	Let $(M, \mathfrak s)$ be a spin 3- or 4-manifold.
	Suppose $\upiota$ preserves the isomorphism class of $\mathfrak s$ and has a nonempty codimension-2 fixed point locus.
	Then the spin\textsuperscript{c} structure arising for the spin structure admits a real structure.	
\end{lem}
\begin{proof}
	By \cite[Proposition 8.46]{AtiyahBott1968II}, an involution having codimension-2 fixed set that preserves the isomorphism class of the spin structure can be lifted to an order-4 complex-linear map
	\begin{equation*}
		\hat{\upiota} :S \to S
	\end{equation*}
	satisfying $(\hat{\upiota})^2=-1$, which we refer to as a \emph{spin lift}.
	On the other hand, we have the anti-linear automorphism $\jmath: S \to S$ from right multiplication by $j \in \Pin(2)$.
(As a convention, the left multiplication by $j$ is complex linear.)
	The composition
	\begin{equation}
	\label{eqn:j_iota}
		\uptau = \jmath \circ \hat{\upiota}: s \mapsto \hat{\upiota}(s)\cdot j,
	\end{equation}
	is anti-linear and involutive.
\end{proof}
\begin{rem}
Involutions having order-4 lifts on the spin bundle are of \emph{odd type}.
The involution in~\eqref{eqn:j_iota} was used to define an involutions $I$ on the Seiberg-Witten configuration spaces in the work of Kato~\cite{Kato2022}, and Konno-Miyazawa-Taniguchi~\cite{KMT2021}.
In particular, involution $\uptau$ in \eqref{eqn:j_iota} commutes with $\jmath$.
If we demand the real structure to commute with $\jmath$, there are finitely many choices for $\uptau$.
\end{rem}
\begin{rem}
The lemma does not imply the existence of real spin\textsuperscript{c} structure structure on arbitrary $3$- or $4$- manifolds with involutions.
A useful sufficient condition for an involution $\upiota$ to preserve the spin structure is triviality of $\upiota^*$ over $H^1(Y;\mathbb Z_2)$.
\end{rem}

Suppose $(\mathfrak s_0, \uptau_0)= (S_0,\rho_0, \uptau_0)$ is a real spin\textsuperscript{c} structure.
Suppose $L$ is a complex line bundle with a real structure $\uptau^L: L \to L$.
Then
\begin{equation*}
		(S_0 \otimes L, \rho_0 \otimes 1_L).
\end{equation*}
is a real spin\textsuperscript{c} structure.
Conversely, any spin\textsuperscript{c} structure $(S, \rho) $ can be written as 
\begin{equation*}
	(S, \rho) = (S_0 \otimes L, \rho \otimes 1_L).
\end{equation*}
If $\uptau$ is a compatible real structure on $S$, then $\uptau_0$ uniquely determines a real structure $\uptau^L:L \to L$.
\begin{lem}
\label{lem:realspinset}
	Let $(Y, \upiota)$ be a 3-manifold with involution.
	If $(Y,\upiota)$ admits at least one real spin\textsuperscript{c} structure, then the space of real spin\textsuperscript{c} structures is a torsor over 
	\begin{equation*}
		\ker(\Theta) \times \frac{H^1(Y;\mathbb Z)^{\upiota^*}}{\Img(1+\upiota^*:H^1(Y;\mathbb Z) \to H^1(Y
		;\mathbb Z)^{\upiota^*})},
	\end{equation*}
	where $\Theta$ is the map defined in \eqref{eq:theta_defn}.
\end{lem}

\subsection{Spin\textsuperscript{c} structures with real structures on four-manifolds}
\hfill \break
Let $X$ be an oriented 4-manifold, possibly with boundary.
Let $\upiota: X \to X$ be a smooth involution, and $g$ be an $\upiota$-invariant Riemannian metric on $X$.
When $X$ has nonempty boundary, we assume $\upiota$ preserves the connected components of the boundary.
Let $\mathfrak s_X = (S_X,\rho_X)$ be a spin\textsuperscript{c} structure.
That is, $S_X$ is a hermitian rank-4 vector bundle and
$\rho: TX \to \Hom(S_X,S_X)$ is a Clifford multiplication.
Extend $\rho$ to complex-valued $k$-forms.
Since $\rho(\vol_X)$ has eigenvalue $(\pm 1)$, the rank-4 bundle
$S_X$ decomposes into a sum of two rank-2 bundles $S^+ \oplus S^-$.
For any vector field $\xi$, the Clifford multiplication $\rho(\xi)$ interchanges $S^{\pm}$. 
Locally, there exists an orthonormal frame $\{e_0,e_1,e_2,e_3\}$ such that
\begin{equation*}
	\rho(e_0) = \begin{bmatrix}
		0 & -I_2 \\
		I_2 & 0
	\end{bmatrix}, \quad 
	\rho(e_i) = \begin{bmatrix}
		0 & -\sigma_i^* \\
		\sigma_i  & 0
	\end{bmatrix}, \quad
	(i=1,2,3).
\end{equation*}
\begin{defn}
	A real structure \emph{compatible with} $\mathfrak s_X$
	 is a real structure $\uptau_X: S^{\pm}_X \to S^{\pm}_X$ on the spinor bundle that is compatible with the hermitian metric and satisfies 
	 \begin{equation}
	\label{eqn:compat4}
		\rho(\upiota_* \xi)\uptau(\Phi_{x}) = \uptau(\rho(\xi)\Phi_x),
	\end{equation}
	for any $x \in X$, any vector field $\xi$ on $X$, and any spinor $\Phi \in \Gamma(S^{\pm}_X)$.
	The subscript $x$ indicates value of section over $x$.
	A pair $(\mathfrak s_X, \uptau_X)$ of a spin\textsuperscript{c} structure with a real structure $\uptau_X$ is a \emph{real spin\textsuperscript{c} structure}.
\end{defn}
If there exists one real spin\textsuperscript{c} structure, then the set of real spin\textsuperscript{c} is a torsor over the group of line bundles with real structures.
In addition to Lemma~\ref{lem:realspinset} which provides real structures on spin structures preserved by $\upiota$, the following proposition
\cite[Proposition~2.4]{TianWang2009} applies possibly non-spin manifolds.
\begin{prop}
\label{prop:tianwang_aclift}
	Let $(X,J)$ be an almost complex manifold, and $\upiota:X \to X$ be an anti-holomorphic involution.
	Equip $X$ with a hermitian metric that is compatible with both $J$ and $\upiota$.
	Let $\mathfrak s_J = (S^{\pm}_J,\rho)$ be the canonical spin\textsuperscript{c} structure.
	Then there exists a canonical compatible real structure $\uptau:S^{\pm}_J \to S^{\pm}_J$.
\end{prop}
We illustrate Proposition~\ref{prop:tianwang_aclift} via an example.
Let $(X,J)$ be the non-spin complex manifold $\mathbb{CP}^2$.
Let the involution $\upiota$ be conjugation $\upiota:[z_0: z_1: z_2] \mapsto [\bar z_0: \bar z_1: \bar z_2]$, which fixes a real projective line $ \mathbb{RP}^2$.
The canonical spin\textsuperscript{c} bundles  can be identified with 
\begin{align*}
	S^+ &= \Lambda^{0,0} \oplus \Lambda^{0,2}, \\
	S^- &= \Lambda^{0,1}.
\end{align*}
The derivative $\upiota_*$ acts anti-linearly on $\T\mathbb{CP}^2$.
This action extends to the complexification $\T_{\C}\mathbb{CP}^2$ preserving the $\pm i$-eigenspaces, inducing
real structures on $\Lambda^{0,i}$.
The Clifford multiplication $\rho : \T\mathbb{CP}^n \to \Hom(S^+,S^-)$ involves dualizing tangent vectors via the hermitian metric.
The compatibility of $\rho$ and $\upiota$ follows from isometry assumption.
		
\subsection{Double branched covers of $\boldsymbol{S^3}$}
\hfill \break 
Let $K$ be an oriented link in $S^3$, with connected components $\{K_i: 1 \le i \le n\}$. 
Let $N_K$ be the the double cover of the  link exterior $(S^3 \setminus K)$ associated to the kernel of the homomorphism given by the sum of linking numbers
\begin{align*}
	H_1(S^3-K;\mathbb Z) 
	&\to \mathbb Z/2\\
	\alpha
	&\mapsto \sum_i \text{lk}(K_i, \alpha).
\end{align*}
Then $\del N_K$ is a union of $n$ tori.
We glue $n$ solid tori to $Y_K$, 
in a way that the meridian of each torus is glued onto the preimage of a meridian of $K_i$ under the 2-to-1 covering map $N_K \to (S^3-K)$.
The resulting manifold
\begin{equation*}
	Y_K = \Sigma_2(S^3, K)= N_K \bigcup  n(S^1 \times D^2)
\end{equation*}
is the double branched cover of $K$, and branch locus is exactly the union of cores of the solid tori.
Alternatively, since $K \subset S^3$ is null-homologous, let $F$ be an oriented connected Seifert surface of $K$.
Remove a regular neighbourhood of $K$ to obtain a manifold with torus-boundaries, and then cut along a regular neighbourhood of $F$.
The resulting manifold $Y'$ contains two copies of $F_{\pm}$, and $F_{\pm}$ are connected by $n$ annuli from the remaining boundary of knot complement.
Glue two copies of $Y'$ along $F_+ \cup F_-$, in a way that $F_{\pm}$ is glued onto $F_{\mp}$.
There are still torus boundaries left, and we glue n solid tori to $Y' \cup Y'$ to obtain $Y_K$.

Let $A$ be the Seifert form of $F$. 
If $\det(A+A^t)$ is nonzero, then
\begin{equation*}
	|H_1(Y_K;\mathbb Z)| = |\det(A+A^{t})| = |\Delta_K(-1)|,
\end{equation*}
where $\Delta_K$ is the Alexander polynomial.
In particular, if $K$ is connected, then $Y_K$ is a rational homology sphere.
In the case when $|\Delta_K(1)| = 0$, the rank of the first homology is equal to the nullity of the matrix $(A+A^t)$ \cite{KauffmanTaylor1976}:
\begin{equation*}
	b_1(Y_K) = \text{nullity}(A+A^t).
\end{equation*}

The covering involution of $Y_K$ acts on the first cohomology $H^1(Y_K;\mathbb Z_2)$ by $-1$, and hence trivially on $H^1(Y_K;\mathbb Z_2)$.
It follows that $\upiota$ preserves the isomorphism class of the spin structures on $Y_K$.
By Lemma~\ref{lem:cohom_char_real_str} and $H^1(Y_K/\upiota;\mathbb Z) = H^1(S^3; \mathbb Z)$, any spin\textsuperscript{c} structure admits a real structure.

Our construction of real monopole Floer homologies also carries over to links in integral homology spheres and their double branched covers
\begin{equation*}
	\Sigma_2(Y,K).
\end{equation*}

\begin{rem}[Double branched covers of general 3-manifolds]
	Let $Y$ be a connected oriented closed 3-manifold, and $K \subset Y$ be a 1-dimensional submanifold.
There exists a double branched cover of $Y$ along $K$ only if $[K] = 0 \in H_1(Y;\mathbb Z_2)$, and the choices of double branched covers form a torsor over the mod-$2$ homology classes $H_2(Y;\mathbb Z_2)$.
In particular, we may define real monopole Floer homolgies for null-homologous links in an arbitrary 3-manifold, by keeping track of all choices of double branched covers.
\end{rem}

\subsection{Double branched covers of $\boldsymbol{[0,1] \times S^3}$}
\hfill \break
Let $I = [0,1]$ and let $S \subset I \times S^3$ be a connected properly embedded (not necessarily orientable) surface such that $\del S = -K_0 \sqcup K_1$, and $K_i \subset \{i\} \times S^3$ is a link.
Since $(I \times S^3)$ is simply connected and $H_2(I \times S^3; \mathbb Z_2) = 0$, there exists a unique double branched cover $W$ along $S$, such that
\begin{equation*}
	\del W = -Y_{K_0} \sqcup Y_{K_1}.
\end{equation*}
Every second cohomology lies in the kernel of $\Theta$, and $H^1(W;\mathbb Z) = 0$, so
each spin\textsuperscript{c} structure supports a unique real structure.
It follows that 
the space of real spin\textsuperscript{c} structures over $W$ is isomorphic to
\begin{equation*}
	H^2(W;\mathbb Z).
\end{equation*}

\begin{rem}[Double branched cover of general four-manifolds]
Let $X$ be a compact orientable 4-manifold with boundary.
Let $S \subset X$ be a properly embedded (not necessarily orientable) surface such that
\begin{equation*}
	[S,\del S] = 0 \in H_2(X,\del X; \mathbb Z_2).
\end{equation*}
There exist $|H^1(X;\mathbb Z_2)|$ choices of double branched covers of $X$.
\end{rem}

\subsection{Doubling a 3-manifold}
\hfill \break
The following construction relates the real monopole Floer homology $\HMR^{\circ}$ to the ordinary monopole Floer homology $\HM^{\circ}$.
Let $Y$ be a connected oriented closed 3-manifold, and $K_0 \subset S^3$ be a link.
We form the connected sum
\begin{equation*}
	Y \# (S^3,K_0).
\end{equation*}
Let $\mathbf{Y}$ be the double branched cover along $K_0$, which is homeomorphic to the connected sum
\begin{equation*}
	Y \# Y_{K_0} \# \bar Y.
\end{equation*}
The involution $\upiota$ interchanges the two $Y$
 factor, and acts on twice-punctured $Y_{K_0}$ as covering involution $\upiota_{Y_{K_0}}$.
Identifying $H^*(\mathbf{Y};\mathbb Z)$ with $H^*(Y; \mathbb Z) \oplus H^*(Y_{K_0};\mathbb Z) \oplus H^*(Y;\mathbb Z)$, the involution-induced map is 
 \begin{equation*}
 	\upiota^* = 
 	\begin{pmatrix}
 		0 & 0 & 1\\
 		0 & \upiota^*_{Y_{K_0}} & 0 \\
 		-1 & 0 & 0
 	\end{pmatrix}.
 \end{equation*}
The kernel of $\Theta$ on $H^2(\mathbf Y;\mathbb Z)$ 
\begin{equation*}
	H^2(Y;\mathbb Z)(1, 0, 1) \oplus \ker (\Theta_{Y_{K_0}}) (0, 1, 0),
\end{equation*}
and the $\upiota^*$-invariant first cohomology is the subgroup
\begin{equation*}
	H^1(Y;\mathbb Z)(1, 0, -1).
\end{equation*}
Since $H^1(Y;\mathbb Z)(1, 0, -1)$ is contained in the image of $(1 + \upiota^*)$, there is a compatible real structure for each choice of spin\textsuperscript{c} structure in $\ker \Theta$.

\subsubsection*{\textbf{Adding atoms}}
\hfill \break
Let $K \subset Y$ and $K_0 \subset S^3$ be links.
The above construction can be phrased in terms of double branched covers of links, and is related to the ``adding atoms'' operation in \cite{KMunknot2011}.
We form the connected sum
\begin{equation*}
	(Y, K) \# (S^3, K_0)
\end{equation*}
and consider double branched covers of $Y$ along disjoint union $K \sqcup K_0$.
The affect of adding the link $K_0$ is forming the union
\begin{equation*}
	\left(\Sigma_2(Y,K) \setminus \text{two balls}\right) \cup_{S^2 \sqcup S^2} \left(\Sigma_2(S^3, K_0) \setminus \text{two balls}\right).
\end{equation*}
along the boundaries of the removed balls.

\subsection{Examples of manifolds with involutions}
\label{subsec:examples_involutions}
\subsection*{(a). $\boldsymbol{S^3}$}
\hfill \break
Let $S^3$ be the unit sphere in $\C^2$. 
The complex conjugation on $\C^2$ fixes a 2-plane that intersects $S^3$ in a great circle, inducing an involution on $S^3$.
Equivalently, by thinking of $S^3$ as $\reals^3 \cup \infty$ this is the involution that rotates around an axis by $180$ degrees, fixing the axis together with the point at infinity.
This involution is covering involution of the double branched cover of the unknot in $S^3$.

\subsection*{(a!). $\boldsymbol{S^3}$}
\hfill \break
The antipodal map on $S^3$ is a non-example as there is no fixed point.

\subsection*{(b). $\boldsymbol{\mathbb{RP}^3}$}
\hfill \break
The projective 3-space is the double branched cover of the Hopf link.
The involution $\upiota:\mathbb{RP}^3 \to \mathbb{RP}^3$ is given by the linear map
\begin{equation*}
	\begin{pmatrix}
		1 & & & \\
		& 1 & & \\
		& & -1 & \\
		& & & -1
	\end{pmatrix}
\end{equation*} 
on $\reals^4\setminus \{0\}$, which preserves exactly 2 planes (so 2 lines under the quotient $\mathbb{RP}^3$.)
Alternatively, $\mathbb{RP}^3$ is the quotient of the unit 3-ball $D^3 = \{x \in \reals^3; \|x\|\le 1\}$ with boundary identified via antipodal map.
The involution $\upiota$ can be realized as an 180-degree of $D^3$ rotation about an axis $\mathbf a$, which commutes with the antipodal map on $\del D^3$. 
The involution fixes $\mathbf a$ and a great circle spanning a plane orthogonal to $\mathbf a$.

\subsection*{(c). $\boldsymbol{S^1 \times S^2}$}
\hfill \break
The double branched cover of the 2-component unlink is $S^1 \times S^2$, where the involution $\upiota$ on the $S^1$-factor is given by
reflection of fixing 2 points, and on the $S^2$ is reflection along a great circle, swapping the upper and lower hemispheres.

\subsection*{(c1). $\boldsymbol{\#_{k-1}(S^1 \times S^2})$}
\hfill \break
The previous example generalizes to $\#_{k-1}(S^1 \times S^2)$ as the double branched cover of $k$-component unlinks.

\subsection*{(d). Lens spaces}
\hfill \break
Let $L(p,q)$ be the lens space for $p,q$ coprime.
Then $L(p,q)$ is the double branched cover of the two bridge knot $K(p,q)$, where the deck transformation can be described as follows.
Let $S^3$ again be the unit sphere in $\C^2$ and $L(p,q)$ be quotient of the $\mathbb Z/p$-action
\begin{equation*}
	(z_1,z_2) \mapsto (e^{2\pi i/p}z_1,e^{2\pi i/q}z_2).
\end{equation*}
The conjugation action $(z_1,z_2)\mapsto(\bar z_1,\bar z_2)$ descends to an involution $\upiota:L(p,q) \to L(p,q)$, such that
\begin{equation*}
	L(p,q)/\upiota = S^3,
\end{equation*}
branched along $K(p,q)$.

\subsection*{(e). Breskorn spheres}
\hfill \break
Let $\Sigma(p,q,r)$ be the Breskorn integral homology sphere, which is the link of a complex surface singularity with canonical orientation:
\begin{equation*}
	\Sigma(p,q,r) = 
	\left\{
		(z_1, z_2, z_3) \in \C^3 |
		z_1^p + z_2^q + z_3^r = 0
	\right\} \cap S^5.
\end{equation*}
There are two orientation-preserving involutions of interest to us.
\begin{itemize}
	\item $(z_1, z_2, z_3) \mapsto (-z_1,z_2,z_3)$.
	\item $(z_1,z_2,z_3) \mapsto (\bar{z}_1, \bar{z}_2, \bar{z}_3)$.
\end{itemize}
If $p = 2$ and $q,r$ are coprime, then the first action is the covering involution as the double branched cover of the right handed torus knot $T(q,r)$.
The second conjugation action is the covering involution as the double branched cover of $S^3$ along a Montesinos link $k(p,q,r)$.
See \cite{Saveliev1999} for more details.

\subsection*{(f). $\boldsymbol{S^1 \times \Sigma_g}$}
\hfill \break
Let $\upiota$ be an involution acting trivially on the $S^1$-factor and as a branched covering involution $\Sigma_g \to \Sigma_{g'}$, where $2g - 2 = 4 - 4g' - b$ and $b$ is the number of branch points.
The fixed point set is $b$ parallel strands of circles
\begin{equation*}
	S^1 \times \{p_1,\dots,p_b\}.
\end{equation*}

\subsection*{(f1). $\boldsymbol{S^1 \times S^2}$}
\hfill \break
Suppose $g=g'=0$, and $b=2$. 
The action of the $\Sigma_0$ factor is a 180-degree rotation.
Clearly, $\upiota$ preserves the both the spin structures on the $S^1$-factor and the unique spin structure on the $S^2$-factor.
Moreover,
since $H^2(S^1 \times S^2;\mathbb Z)^{-\upiota^*} = 0$ and $H^1(S^1 \times S^2;\mathbb Z)^{-\upiota^*} = \mathbb Z$, only the self-conjugate spin\textsuperscript{c} structure supports real structures, and up to isomorphism there are two real structures. 
The spin\textsuperscript{c} bundle $S$ has zero first Chern class and is topological trivial.

\subsection*{(g). $\boldsymbol{S^1 \times \Sigma_g}$}
\hfill \break
This is a generalization of the unlink double branched covers.
Let $\upiota$ be an involution that acts as a orientation reversing involution on $S^1$, and orientable reversing involution on $\Sigma_g$.
The fixed point sets are disjoint union of unlinked circles.

\section{Gauge Theory Setup}
\label{sec:gauge_setup}
In this section, we construct real structures on configuration spaces from real spin\textsuperscript{c} structures.
We will set up the Seiberg-Witten theory over 3-manifolds, 4-manifolds (possibly with boundaries), and cylinders.

\subsection{Configuration spaces in dimension three}
\hfill \break
Let $(Y,\upiota)$ be a 3-manifold equipped with an involution $\upiota$.
Let $g$ be an $\upiota$-invariant metric, and $(\mathfrak s,\uptau)$ be a real spin\textsuperscript{c} structure.
A connection $B$ on $S$ is \emph{spin\textsuperscript{c}} if it
satisfies
\begin{equation*}
	\nabla_B(\rho(\xi)\Phi) = \rho(\nabla \xi)\Phi + \rho(\xi)\nabla_B \Phi,
\end{equation*}
where $\xi$ is any vector field, $\nabla$ is the Levi-Civita connection of $g$, and $\Phi$ is any spinor. 
Any spin\textsuperscript{c} connection $B$ induces a connection $B^t$ on the determinant bundle $\Lambda^2S$.
Let $\whole{\mathcal A}(Y,\mathfrak s)$ be the space of all spin\textsuperscript{c} connections.
\begin{rem}
	From now on, we will decorate the notations for ordinary Seiberg-Witten theory by underlines to distinguish them from the \emph{real} counterparts.
\end{rem}
If $B_0 \in \whole{\mathcal A}(Y,\mathfrak s)$ is a reference connection, then any other spin\textsuperscript{c} connection $B$ can be written as $B_0 + b \otimes 1_S$, for some $b \in \Omega^1(Y;i\reals)$.
Correspondingly, the induced connection the determinant line bundle  is $B_0^t + 2b$.
Therefore $\whole{\mathcal A}(Y,\mathfrak s)$ is an affine space modelled on $\Omega^1(Y;i\reals)$.
The \emph{ordinary Seiberg-Witten configuration space} $\whole{\mathcal C}(Y, \mathfrak s)$ is the product
\begin{equation*}
	\whole{\mathcal C}(Y, \mathfrak s) =
	\whole{\mathcal A}(Y, \mathfrak s) 
	\times
	\Gamma(S).
\end{equation*}
The \emph{ordinary gauge group} $\whole{\mathcal G}(Y)$ consists of maps $u: Y \to S^1$.
Each gauge transformation $u$ acts on $\whole{\mathcal C}(Y, \mathfrak s)$ by
\[u \cdot (A,\Phi) = (A - u^{-1}du, u\Phi).\]
Denote the set of gauge equivalence classes of configurations by
\[
	\whole{\mathcal B}(Y,\mathfrak s) = \whole{\mathcal C}(Y,\mathfrak s)/\whole{\mathcal G}(Y).
\]
We define involutions on $\whole{\mathcal C}(Y, \mathfrak s)$ and $\whole{\mathcal G}(Y)$ as follows.
\begin{itemize}[leftmargin=*]
\item Let the real structure $\uptau$ act spin\textsuperscript{c} connections  by pulling back: $B \mapsto \uptau^*B$, defined as
\begin{equation*}
	\nabla_{\uptau^*B} \Phi = \uptau(\nabla_B(\uptau \Phi)).
\end{equation*}

\item Let $\uptau$ act on spinors $\Phi \in \Gamma(S)$ by $\Phi \mapsto \tau(\Phi)$, where
\begin{equation*}
	\uptau(\Phi)_y =\uptau(\Phi_{\upiota(y)}).
\end{equation*}

\item The map $u \mapsto \upiota(u) := \upiota^* u^{-1}$ defines an involution on $\underline{\mathcal G}(Y)$, which satisfies
\begin{equation*}
	\uptau(u\Phi) = \upiota(u)\uptau(\Phi).
\end{equation*}
\end{itemize}
We refer to the involution on $\whole{\mathcal C}(Y, \mathfrak s)$ as \emph{a real structure} on the configuration space.
Moreover, choose a $\uptau$-invariant reference connection $B_0$, so that $\whole{\mathcal A}(Y,\mathfrak s) = B_0 + 1_S \otimes \Omega(Y,i\mathbb R)$.
By conjugate-linearity, $\uptau$ acts on the imaginary 1-forms $\alpha$ by
\begin{equation*}
	\uptau(\alpha) = -\upiota^* \alpha.
\end{equation*}
We define an involution on all complex-valued forms  $\upiota: \Omega^*(Y;\C) \to \Omega^*(Y;\C)$ by $\alpha \mapsto -\upiota^* \alpha$ over imaginary-valued forms $\Omega^*(Y;i\reals)$, and $\alpha \mapsto \upiota^* \alpha$ over all real-valued forms $\Omega^*(Y;\reals)$. 
This involution is compatible with taking curvatures of spin\textsuperscript{c} connections:
\begin{equation*}
	F_{\uptau(A)} = \upiota(F_A).
\end{equation*}

\begin{defn}
A \emph{real spin\textsuperscript{c} connection} is a $\uptau$-invariant spin\textsuperscript{c} connection, and a \emph{$\uptau$-real spinor} is a $\uptau$-invariant section of $S$.
Denote by $\mathcal A(Y,\mathfrak s,\uptau)$ the space of real spin\textsuperscript{c} connections and denote by $\Gamma(S)^{\uptau}$  the space of $\uptau$-real spinors.
The \emph{(real Seiberg-Witten) configuration space}  $\mathcal C(Y,\mathfrak s,\uptau)$ is the product
\begin{equation*}
	\mathcal C(Y,\mathfrak s,\uptau) = \mathcal A(Y,\mathfrak s, \uptau) \times \Gamma(S)^{\uptau}
\end{equation*}
consisting of real spin\textsuperscript{c} connections and $\uptau$-real spinors.	
The \emph{real gauge group} $\mathcal G(Y,\upiota)$ is the subgroup of $\upiota$-invariant ordinary gauge transformations:
\begin{equation*}
	\mathcal G(Y,\upiota) = \{u: Y \to S^1 |
	\bar u(\upiota(y)) = u(y)\}.
\end{equation*}
The space of $\mathcal G$-equivalence classes of configurations is
\[\mathcal B(Y,\mathfrak s,\uptau)= \mathcal C(Y,\mathfrak s,\uptau)/\mathcal G(Y,\upiota).\]
\end{defn}

Assume $\upiota$ has nonempty fixed point locus on $Y$ so that the only constant gauge transformations in $\mathcal G$ are $\{\pm  1\} \cong \mathbb Z_2$.
The \emph{ordinary Dirac operator} $D_B:\Gamma(S) \to \Gamma(S)$ of a spin\textsuperscript{c} connection $B$ is the composition 
\begin{equation*}
	\begin{tikzcd}
		\Gamma(S) \ar[r,"\nabla_B"] 
		& \Gamma(\T^*Y\otimes S) \ar[r,"\rho"] 
		& \Gamma(S),
	\end{tikzcd}
\end{equation*}
where the second map $\rho$ is given by applying Clifford multiplication of the $1$-form component to the spinorial component.
The Dirac operator is self-adjoint and elliptic over $Y$.
Moreover, $D_B$ is equivariant with respect to $\uptau$, so it descends to an operator on the real spinors
\begin{equation*}
	D_B:\Gamma(S)^{\uptau} \to \Gamma(S)^{\uptau}.
\end{equation*}
The \emph{Chern-Simons-Dirac functional} (CSD) is a function $\mathcal L:\whole{\mathcal C}(\mathfrak s) \to \reals$, defined by
\begin{equation*}
	\mathcal L(B,\Psi) = -\frac18\int_Y (B^t-B_0^t) \wedge (F_{B^t}+F_{B_0^t}) + \frac12 \int_Y \langle D_B \Psi, \Psi\rangle d\vol_Y,
\end{equation*}
for a fixed reference spin\textsuperscript{c} connection $B_0$.
The CSD functional is not invariant under the entire group $\whole{\mathcal G}$, for
\begin{equation*}
	\mathcal L(u \cdot (B,\Psi))-\mathcal L(B,\Psi)= 2\pi^2([u] \cup c_1(S))[Y].
\end{equation*}
Moreover, the CSD functional is invariant under the action of $\uptau$, i.e.
\begin{equation*}
	\mathcal L(\uptau^*B,\uptau(\Psi)) = 
	\mathcal L(B,\Psi).
\end{equation*}
Indeed, both $(B^t-B^t_0)$ and $(F_{B^t}+F_{B_0^t})$ transform by $-\upiota^*$,  and the volume form is preserved by $\upiota^*$.
The second integral is invariant because it is  is compatible with the hermitian metric.
The formal $L^2$-gradient vector field of the Chern-Simons-Dirac function is 
\begin{equation*}
	\left(\frac{1}{2}*F_{B^t} + \rho^{-1}(\Psi\Psi^*)_0,
	D_B\Psi
	\right) \in i\Omega^1(Y) \oplus \Gamma(S),
\end{equation*}
where $(\Psi\Psi^*)_0 = \Psi\Psi^* - \frac12|\Psi|^2$ is a traceless,  self-adjoint endomorphism of the spinor bundle, and hence lies in the image of $\rho(i\Omega^1(Y))$.
The zeros of the above vector field are solutions to the $3$-dimensional Seiberg-Witten equations.
Furthermore, it follows from the compatibility with Clifford multiplication that $\grad \mathcal L$ is equivariant under $\uptau$.
The downward gradient flow lines of $\mathcal L$ in $\underline{\mathcal C}(Y,\mathfrak s)$ satisfying
\begin{equation}\label{eqn:SWflow1}
	\begin{cases}
		dB/dt = -\left(\frac{1}{2}*F_{B^t} + \rho^{-1}(\Psi\Psi^*)_0 \right) \otimes 1_S, \\
	d\Psi/dt = -D_B \Psi.
	\end{cases}	
\end{equation}
lies within the real subspace $\mathcal C(Y,\mathfrak s, \uptau)$, by $\uptau$-equivariance of $\grad \mathcal L$.

A $\uptau$-invariant configuration $(B,\Psi)$ is \emph{reducible} if $\Psi = 0$, and \emph{irreducible} otherwise.
We denote by $\mathcal C^*(Y,\mathfrak s,\uptau)$ the space of irreducible configurations.
In the case when the Spin\textsuperscript{c} structure $\mathfrak s$ is torsion, the set of reducible configurations corresponds to flat $U(1)$-connections on $\det(S)$, which up to gauge, is isomorphic to the torus 
\begin{equation*}
	H^1(Y;i\mathbb R)^{-\upiota^*}/H^1(Y;i\mathbb Z)^{-\upiota^*},
\end{equation*}
whose rank we will denote as 
\begin{equation*}
	b^1_{-\upiota^*}(Y) = \dim_{\reals} H^1(Y;\reals).
\end{equation*}
The adjective ``(ir)reducible" refers to whether $(B,\Psi)$ has a $\mathbb Z_2$-stabilizer.

\subsection{Configuration spaces in dimension four}
\hfill \break
Let $X$ be an oriented 4-manifold, $\upiota_X$ be an involution, and $g_X$ be an $\upiota_X$-invariant Riemannian metric, and $(\mathfrak s_X, \uptau_X)$ be a real spin\textsuperscript{c} structure. 
The ordinary configuration space $\underline{C}(X,\mathfrak s_X)$ is the product of the affine space $\underline{\mathcal A}(X,\mathfrak s_X)$ of spin\textsuperscript{c} connections and sections of the positive spinor bundle:
\begin{equation*}
	\underline{\mathcal C}(X,\mathfrak s_X) = \{(A,\Phi)\} = \underline{\mathcal A}(X,\mathfrak s_X) \times \Gamma(S^+),
\end{equation*}
on which the gauge group $\underline{\mathcal G}=\text{Map}(X,S^1)$ acts.
The real structure $\uptau_X$ defines an involution on  $\underline{\mathcal C}(X,\mathfrak s_X)$ in the same way as before, and we define the real subspace 
\begin{equation*}
	\mathcal C(X,\mathfrak s_X,\uptau_X) =
	\mathcal A(X,\mathfrak s_X,\uptau_X) \times \Gamma(S^+)^{\uptau_X},
\end{equation*}
where $\mathcal A(X,\mathfrak s_X,\uptau_X)$ is the space of $\uptau_X$-invariant spin\textsuperscript{c} connections.
The corresponding real gauge group is
\begin{equation*}
	\mathcal G(X,\upiota_X)=\{u:X \to S^1\big|\overline{u}(\upiota(x))= u(x)\}.
\end{equation*}
Let $D_A^{\pm}= \rho \circ \nabla_A:\Gamma(S^{\pm}) \to \Gamma(S^{\pm})$ be the Dirac operator of $A$ on $S^{\pm}$.
The Seiberg-Witten equations on $X$ are 
\begin{align*}
		\frac12 \rho_{X}(F^+_{A^t})-(\Phi\Phi^*)_0&=0,\\
		D^+_A\Phi &= 0.
\end{align*}
In other words, the Seiberg-Witten solutions are zeros of the Seiberg-Witten map $\mathfrak F$:
\begin{align*}
	\mathfrak F:\underline{C}(X,\mathfrak s_X), &\to C^{\infty}(X;i\mathfrak{su}(S^+)\oplus S^-)\\
	(A,\Phi) &\mapsto \left(\frac12 \rho_{X}(F^+_{A^t})-(\Phi\Phi^*)_0, D_A^+\Phi\right).
\end{align*}
The operator $\mathfrak F$ is both $\underline{\mathcal G}$-equivariant and $\uptau_X$-equivariant.
Thus $\mathfrak F$ descends to a map 
\[
	\mathfrak F: \mathcal C(X,\mathfrak s_X,\uptau_X) \to \Gamma(X;i\mathfrak{su}(S^+)\oplus S^-)^{\uptau_X}.
\]

Assume $X$ has a nonempty boundary $Y$. 
Let $\nu$ be the outward pointing normal vector field, which induces an isomorphism
\begin{equation*}
	\rho(\nu):S^+|_Y \to S^-|_Y,
\end{equation*}
and a spin\textsuperscript{c} structure $\mathfrak s = (S,\rho)$ on $Y$:
\begin{equation*}
	S := S^+|_Y \cong S^-|_Y,\quad
	\rho(\xi):=\rho_X(\nu)^{-1}\rho_X(\xi).
\end{equation*}
Moreover, $\uptau_X$ defines a compatible real structure $\uptau$ over $\mathfrak s$.

\subsection{An index formula for closed four-manifolds}
\hfill \break
The linearized ordinary Seiberg-Witten operator, coupled with a gauge fixing condition, arises from the sum of a (real-linear) elliptic complex on the form part
\[\begin{tikzcd}
	0 & {\Omega^0(X;i\reals)} & {\Omega^1(X;i\reals)} & {\Omega^+(X;i\reals)} & 0
	\arrow["d", from=1-2, to=1-3]
	\arrow[from=1-1, to=1-2]
	\arrow["{d^+}", from=1-3, to=1-4]
	\arrow[from=1-4, to=1-5]
\end{tikzcd},\]
and a complex-linear elliptic complex on the spinorial part
\[\begin{tikzcd}
	0 & {\Gamma(S^+)} & {\Gamma(S^-)} & 0
	\arrow["{D_A^+}", from=1-2, to=1-3]
	\arrow[from=1-1, to=1-2]
	\arrow[from=1-3, to=1-4]
\end{tikzcd}.\]
The index of the Seiberg-Witten operator is given by
\begin{equation*}
	d = 2\ind_{\C}(D^+_A) + \ind_{\reals}(d^* \oplus d^+) = \frac{1}{4}(c_1(S^+)^2[X] - \sigma(X)) + \frac{1}{2}(b^1(X) - b^+(X) - 1).
\end{equation*}
Since the real structure $\uptau_X$ commutes with the Dirac operator, the real analogues of above complexes are 
\[\begin{tikzcd}
	0 & {\Omega^0(X;i\reals)^{-\upiota^*}} & {\Omega^1(X;i\reals)^{-\upiota^*}} & {\Omega^+(X;i\reals)^{-\upiota^*}} & 0
	\arrow["d", from=1-2, to=1-3]
	\arrow[from=1-1, to=1-2]
	\arrow["{d^+}", from=1-3, to=1-4]
	\arrow[from=1-4, to=1-5]
\end{tikzcd},\]
and 
\[\begin{tikzcd}
	0 & {\Gamma(S^+)^{\uptau_X}} & {\Gamma(S^-)^{\uptau_X}} & 0
	\arrow["{D_A^+}", from=1-2, to=1-3]
	\arrow[from=1-1, to=1-2]
	\arrow[from=1-3, to=1-4]
\end{tikzcd}.\]
The index of the first complex is the alternating sum of the dimension of the  $(-\upiota^*)$-invariant cohomology groups, indicated by a subscript ``$(-\upiota^*)$'',
\begin{equation*}
	b^1_{-\upiota^*}(X) - b^+_{-\upiota^*}(X) - b^0_{-\upiota^*}(X).
\end{equation*}
As for the second complex, anti-linearity and $(\uptau_X)^2 = 1$ implies that the index is half of the ordinary one,
\begin{equation*}
	\ind_{\reals}^{\uptau_X}(D_A^+)  = \frac{1}{8}(c_1(S^+)^2[X] - \sigma(X)).
\end{equation*}
\begin{lem}
\label{lem:closed4index}
	The linearized Seiberg-Witten operator with Coulomb gauge condition is of index
\begin{equation*}
	d = \ind_{\reals}^{\uptau_X}(D_A^+)  = \frac{1}{8}(c_1(S^+)^2[X] - \sigma(X)) + (b^1_{-\upiota^*}(X) - b^+_{-\upiota^*}(X) - b^0_{-\upiota^*}(X)).
\end{equation*}
\end{lem}

\subsection{Configuration spaces on cylinder}
\hfill \break
Let $Z$ be the infinite cylinder $\reals_t \times Y$.
Assume the following 3-dimensional data are given: 
\begin{itemize}
	\item an involution $\upiota: Y \to Y$, 
	\item an $\upiota$-invariant Riemannian metric $g$, 
	\item a real spin\textsuperscript{c} structure $(\mathfrak s, \uptau) = (S,\rho, \uptau)$.
\end{itemize}
We define the following 4-dimensional data
\begin{itemize}
	\item an involution $\upiota_Z: Z \to Z$ which acts by $\upiota(t,y) = (t,\upiota(y))$,
	\item an $\upiota_Z$-invariant Riemannian metric $dt^2 + g$ on $Z$,
	\item a spin\textsuperscript{c} structure $\mathfrak s_Z = (S_Z,\rho_Z)$ where $S_Z = S \oplus S$, and
		\begin{equation*}
			\rho_Z(\del/\del t) = 
			\begin{pmatrix}
				0 & -I_S\\ I_S & 0,
			\end{pmatrix}
			\quad
			\rho_Z(v)
			\begin{pmatrix}
				0 & -\rho(v)^*\\ \rho(v) & 0
			\end{pmatrix},
		\end{equation*}
		where $v \in \Gamma(\T Y)$.
	\item a compatible real structure $\uptau_Z := (\uptau,\uptau):S_Z \to S_Z$.
\end{itemize}
We view a section of $S^+$ as a one-parameter family of sections of $S$:
\begin{equation*}
	\Phi = \Psi(t):\reals \to \Gamma(S).
\end{equation*}
If the section $\Phi \in \Gamma(S^+)$ is $\uptau_Z$-invariant, then the path $\Psi(t)$
lies in the $\uptau$-invariant spinors of $Y$.
On the other hand, a general spin\textsuperscript{c} connection $A$ on $Z$ takes the form
\[A =\frac{d}{dt}+(cdt)\otimes 1_S+B,\]
where $c=c(t)$ is a one-parameter family of imaginary-valued functions, and $B = B(t)$ is an one-parameter family of spin\textsuperscript{c} connections on $(Y,\mathfrak s)$.
If $A$ is $\uptau_X$-invariant, then $B(t)$ is $\uptau$-invariant for any $t$.
If $c \equiv 0$, then $A$ is in \emph{temporal gauge}.

In any case, we have the correspondence between $4$-dimensional configurations and paths of $3$-dimensional configurations
\begin{equation*}
	A \mapsto (\check A(t),\check\Phi(t)) = (B(t),\Psi(t)),
\end{equation*}
which forgets the $dt$-component of the connection $c(t)$.
The map is surjective \emph{up to gauge} as any spin\textsuperscript{c} connection on $Z$ can be put in temporal gauge via parallel transportation along the $t$-direction.
There is a map in reverse direction, given by
\begin{equation*}
	(B(t),\Psi(t)) \mapsto \left(A = \frac{d}{dt} + B(t),\Phi = \Psi(t)\right).
\end{equation*}

The four-dimensional Seiberg-Witten equations on $Z$ looks formally like a gradient flow equation:
\begin{equation}\label{eqn:SWflow2}
	\begin{cases}
		dB/dt - dc= -\left(\frac{1}{2}*F_{B^t} + \rho^{-1}(\Psi\Psi^*)_0 \right) \otimes 1_S, \\
	d\Psi/dt + c\Psi= -D_B \Psi,
	\end{cases}	
\end{equation}
which agrees with Equations~\eqref{eqn:SWflow1} if $A$ was in temporal gauge.
The Equations~\eqref{eqn:SWflow2} manifestly have 4-dimensional gauge symmetry, whereas Equations~\eqref{eqn:SWflow1} only have 3-dimensional gauge symmetry.

\subsection{Local model of singularity}
The term ``local model'' is somewhat misleading, as the anti-linearity assumption forces the action $\uptau$ to be global.
Regardless, let us describe the invariant configurations on a solid torus $S^1 \times D^2$ to draw contrasts between the anti-linear setup and the orbifold setup.
Let $Y = I \times D^2 \subset \reals^3$ be the product of $I = (-1,1)$ with the unit 2-disk $D^2 \subset \reals^2$.
Let the singular locus $K$ be $I \times \{0\}$ and the involution be 
\begin{equation*}
	(s, n_1, n_2) \mapsto (s, -n_1, -n_2).
\end{equation*}
Here $s \in I$ and $(n_1,n_2) \in D^2$ are coordinates, such that $\{\del_s, \del_{n_1},\del_{n_2}\}$ is an orthonormal frame.
Let $S$ be the trivial rank-2 bundle with standard hermitian inner product.
Define Clifford multiplication $\rho$ by
\[
	\rho(\del_s) = \sigma_1 = \begin{pmatrix}
		i & 0\\
		0 & -i
	\end{pmatrix}, \quad 
\rho(n_1) = \sigma_2 = \begin{pmatrix}
	0 & -1\\
	1 & 0
\end{pmatrix},	\quad
\rho(n_2) = \sigma_3 = \begin{pmatrix}
	0 & i\\
	i & 0
\end{pmatrix}.
\]
The compatibility of the anti-linear involution $\uptau$, together with
\[
	\upiota_*\left(\frac{\del}{\del s}\right) = \frac{\del}{\del s} \quad 
	\upiota_*\left(\frac{\del}{\del n_1}\right) = -\frac{\del}{\del n_1} \quad 
	\upiota_*\left(\frac{\del}{\del n_2}\right) = -\frac{\del}{\del n_2}
,\] 
implies
\begin{equation}
\label{eqn:dels_tau}
	\uptau \circ \rho (\del_s) = \rho(\del_s) \circ \uptau.
\end{equation}
Assume the lift $\uptau$ takes the form
\begin{align*}
	\uptau: S_y &\to S_{\upiota(y)},\\
	w &\mapsto T_y\bar w,
\end{align*}
where $T = T_y$ is a 2-by-2 complex matrix-valued function, and $\bar w$ is the conjugate of $w \in \C^2$.
Equation~\eqref{eqn:dels_tau} means
\[
	T_y\overline{\sigma_1} = - T_y \sigma_1 = \sigma_1 T_{\upiota(y)}.
\]
If $y \in K$, then $T_y = T_{\upiota(y)}$ must be off diagonal. 
On the other hand, the involution condition $\uptau^2 = 1$ implies 
\[
	T_{\upiota(y)}(\overline{T_y\overline{w}}) = T_{\upiota(y)}\overline T_y w = w,
\]
that is, $T_{\upiota(y)}\overline{T_y} = 1$. 
Therefore, for some $\lambda_1, \lambda_2 \in \C$,
\[
	T_y = \begin{pmatrix}
		0 & \lambda_2\\
		\lambda_1 & 0
	\end{pmatrix},
	\quad
	\lambda_1 \bar \lambda_2 = 1
\]
Since $\uptau$ preserves the hermitian metric, $\lambda_1 = \lambda_2 = \lambda(y) \in S^1$.
For this choice of $T$, the other two compatibility conditions are satisfied.
The invariant subspace $V_y \subset \C^2$ under $\uptau$ are spanned by real linear combinations of the two vectors
\[
	\begin{pmatrix}
		1\\ \lambda(y)
	\end{pmatrix}, 
	\begin{pmatrix}
		i \\ -\lambda(y) i
	\end{pmatrix}.
\]
Let $V \to K$ be the resulting rank-2 real subbundle.
Then any $\uptau$-invariant spinor $\Phi$ on $Y$ restricted to $K \subset I \times D^2$ are sections of the subbundle $V$.

We can also write down the real spin\textsuperscript{c} connections.
Consider the trivial connection $\nabla_0$ on $\underline{\C}^2$ as a reference Spin\textsuperscript{c} connection invariant under $\uptau$.
Any other Spin\textsuperscript{c} connection $\nabla$ is of the form $\nabla_0 + ia \otimes 1$, where  in polar coordinates
\begin{equation*}
	a = a_s ds + a_r dr + a_{\theta} d\theta,
\end{equation*}
such that $r \in (0,1)$ and $\theta \in (0,2\pi)$.
Then $\nabla$ is $\uptau$-invariant if and only if $\upiota^*a= -a$, or equivalently that $a_r$ and $a_{\theta}$ satisfy
\begin{equation*}
	a_r(r, \theta+\pi) = -a_r(r, \theta),
	\quad
	a_{\theta}(r,\theta+\pi) = -a_{\theta}(r, \theta).
\end{equation*}
In particular, $\nabla_0$ descends to the quotient  only if $a_r = a_{\theta} = 0$.

\section{Configuration Spaces and Slices}
\label{sec:config_slice}
We blow up the reducible locus of the configuration space $\mathcal C(M,\mathfrak s_M, \uptau_M)$, for 
$M$ either a 3- or 4- manifold.
The process is the infinite dimensional analogue of the real blowing up Section~\ref{sec:morse}.
\subsection{$\boldsymbol{\sigma}$-blow-up in dimension four}
\hfill \break
Let $X$ be a $4$-manifold, $\upiota_X$ be an involution, $g_X$ be an $\upiota_X$-invariant metric, and $(\mathfrak s_X,\uptau_X)$ be a real spin\textsuperscript{c} structure.
The reducible locus $Q$ of $\mathcal C(X,\mathfrak s_X,\uptau_X)$ is the subspace with zero spinorial parts $\mathcal A(X,\mathfrak s_X,\uptau_X) \times \{0\}$.
We replace $Q$ with
$\mathcal A(X,\mathfrak s_X,\uptau_X) \times \mathbb S(\Gamma(S^+)^{\uptau_X})$, where $\mathbb S$ denotes the unit sphere in the Hilbert space, with respect to the $L^2$-inner product.
Let the \emph{$\sigma$-blown up configuration space} be
\begin{equation*}
	\mathcal C^{\sigma}(X,\mathfrak s,\uptau)
	= 
	\{(A,s,\phi) \in \mathcal A(X,\mathfrak s_X,\uptau_X) \times \reals \times \Gamma(S^+)^{\uptau}\big|
	s \ge 0, \text{ and }\|\phi\|_{L^2(X)} = 1\}.
\end{equation*}
There is a blow-down map $\pi: \mathcal C^{\sigma}(Y,\mathfrak s,\uptau) \to \mathcal C(Y,\mathfrak s,\uptau)$ given by $(A,s,\phi) \mapsto (A,s\phi)$, which is bijective on the irreducible subset, and has fibre $\mathbb S(\Gamma(S^+)^{\uptau_X})$ over $Q$.

The space $\mathcal C^{\sigma}(X,\mathfrak s,\uptau)$ above arises as the real subspace of the blown-up configuration in the ordinary monopole theory
\begin{equation*}
	\underline{\mathcal C}^{\sigma}(X,\mathfrak s)
	= 
	\{(A,s,\phi) \in \underline{\mathcal A}(X,\mathfrak s_X) \times \reals \times \upgamma(S^+)\big|
	s \ge 0, \text{ and }\|\phi\|_{L^2(X)} = 1\},
\end{equation*}
where the $\uptau$-action on the first and third factors are the same as in Section~\ref{sec:gauge_setup}, and trivial on the $\reals$-factor.
The action descends to the non-blown-up configuration $\underline{\mathcal C}(X,\mathfrak s)$, under the blow-down map.
Over $\mathcal C^{\sigma}(X,\mathfrak s,\uptau)$, we extend the Seiberg-Witten map $\mathfrak F$ by setting
\begin{equation*}
	\mathfrak F^{\sigma}(A,s,\phi)= \left(\frac12 \rho(F^+_{A^t}) - s^2(\phi\phi^*)_0, D^+_A\phi\right),
\end{equation*}
which is $\mathcal G(\uptau_X)$-equivariant.

We view the Seiberg-Witten map as a section of $\mathcal V(X,\mathfrak s_X,\uptau_X) \to \mathcal C(X,\mathfrak s_X,\uptau_X)$, where $\mathcal V(X,\mathfrak s_X,\uptau_X)$ is the trivial vector bundle with fibre $C^{\infty}(X;i\mathfrak{su}(S^+)\oplus S^-)$.
Then the blown-up version $\mathfrak F^{\sigma}$ is a section $\mathcal V^{\sigma}(X,\mathfrak s_X,\uptau_X) \to \mathcal C(X,\mathfrak s_X,\uptau_X)$, where $\mathcal V^{\sigma}(X,\mathfrak s_X,\uptau_X)$ is the pullback of the trivial bundle $\mathcal V(X,\mathfrak s_X,\uptau_X)$ by $\pi$.
\subsection{$\boldsymbol{\sigma}$-blow-up in dimension three}
\hfill \break
Next, we define the blown-up configuration space over a 3-manifold $Y$ in the same way:
\begin{equation*}
	\mathcal C^{\sigma}(Y,\mathfrak s,\uptau)
	= 
	\{(B,r,\psi) \in \mathcal A(Y,\mathfrak s,\uptau) \times \reals \times \Gamma(S)^{\uptau}\big|
	r \ge 0, \text{ and }\|\psi\|_{L^2(Y)} = 1\}.
\end{equation*}
The correspondence between $4$-dimensional configurations and paths in the $3$-dimensional configurations continues in the blown-up setting, as follows.
\subsection{Configurations on Cylinders}
\hfill \break
Let $Z = I \times Y$ be the finite cylinder for $I = [-1,1]$.
A configuration in $\upgamma^{\sigma} = (A,s,\phi) \in \mathcal C^{\sigma}(Z,\mathfrak s_Z,\uptau_Z)$ defines a path of configuration
\begin{equation*}
	\check\upgamma(t) = (\check A(t), s\|\check\phi(t)\|_{L^2(Y)}, \check\phi(t)/\|\check\phi(t)\|_{L^2(Y)}),
\end{equation*}
if $\phi(t) \ne 0$ for any $t \in I$.
The correspondence from $\mathcal C^{\sigma}(Z,\mathfrak s_Z,\uptau_Z)$ to $C^{\infty}(I,\mathcal C^{\sigma}(Y,\mathfrak s, \uptau))$ is partially defined.
Assume in addition that $A$ is in temporal gauge. 
The Seiberg-Witten equation $\mathfrak F^{\sigma}(\upgamma^{\sigma}) = 0$ can be written as
\begin{align*}
		\frac{dB}{dt} &= -\left(\frac{1}{2}*F_{B^t} + r^2\rho^{-1}(\psi\psi^*)_0 \right) \otimes 1_S,\\
		\frac{dr}{dt} &= - \Lambda(B,r,\psi)r,\\
		\frac{d\psi}{dt} &= -D_B\psi+ \Lambda(B,r,\psi)\psi,
\end{align*}
where $\Lambda(B,r,\psi)$ is the analogue of the function in Equation~\eqref{eqref:lambda_fd}, defined by
\begin{equation*}
	\Lambda(B,r,\psi) = \langle \psi, D_B\psi\rangle_{L^2(Y)}.
\end{equation*}
The right hand side of the Seiberg-Witten equations 
\begin{equation*}
	(\grad \mathcal L)^{\sigma} = (\left(\frac{1}{2}*F_{B^t} + \rho^{-1}(\Psi\Psi^*)_0 \right), \Lambda(B,r,\psi)r, D_B\psi - \Lambda(B,r,\psi)\psi).
\end{equation*}
is the \emph{blown-up gradient of the Chern-Simons-Dirac funtctional}, denoted as $(\grad \mathcal L)^{\sigma}$.
The blown-up gradient agrees with $\pi^*\grad \mathcal L$ on the irreducible subset and is everywhere tangent to the boundary $\{r=0\}$.
The zeros of the blown-up gradient consist of 
\begin{itemize}
	\item $(B,r,\psi)$ for which $(B,r\psi)$ is a zero of $\grad \mathcal L$ and are preimages of zeros of $\grad \mathcal L$ under $\pi$,
	\item $(B,0,\psi)$ for which $(B,0)$ is a zero of $\grad \mathcal L$ and $\psi$ is an eigenvector of $D_B$.
\end{itemize}

\subsection{$\boldsymbol{\tau}$-blow-up for cylinders}
\hfill \break
Finally, we introduce the $\tau$-model of the blown-up configurations on $Z$, which is equivalent to the $\sigma$-model:
\begin{equation*}
	\mathcal C^{\tau}(Z,\mathfrak s_Z,\uptau_Z) = \{(A,s,\phi) \big| s(t) \ge 0, \text{ and }\|\phi(t)\|_{L^2(Y)} = 1\},
\end{equation*}
as a subspace of $\mathcal A(Z,\mathfrak s_Z,\uptau_Z) \times C^{\infty}(I;\reals) \times \upgamma(S^+)^{\uptau_Z}$.
A configuration $\upgamma = (A,s,\phi)$ determines an element in $\mathcal C^{\sigma}(Z,\mathfrak s_Z,\uptau_Z)$, $\phi(t) \ne 0$ for any $t$.
Moreover, $\upgamma$ determines a path $\check\upgamma = (\check A, s,\check\phi)$ in $\mathcal C^{\sigma}(Y,\mathfrak s,\uptau)$, by forgetting the $dt$-component of the connection.
Conversely, any path in $C^{\sigma}(Y,\mathfrak s,\uptau)$ determines an element in $C^{\tau}(Z,\mathfrak s_Z,\uptau_Z)$ in temporal gauge.
The flow equation for a general configuration $(A = d/dt + cdt + \check A, s(t),\check\phi)$ is
\begin{align*}
	\frac{d\check A}{dt} &= -\left(\frac{1}{2}*F_{\check A^t} + dc + s^2\rho^{-1}(\check\phi\check\phi^*)_0 \right) \otimes 1_S,\\
	\frac{ds}{dt} &= - \Lambda(B,r,\check\phi)s,\\
	\frac{d\phi}{dt} &= -D_B\check\phi - c\check\phi+ \Lambda(B,r,\check\phi)\check\phi.
\end{align*}
Rewriting the flow equation we see the four-dimensional gauge invariance:
\begin{align*}
	\frac{1}{2}\rho(F^+_{A^t}) - s^2(\phi\phi^*)_0 &= 0,\\
	\frac{d}{dt}s+ \Re \langle D_A^+ \phi,\rho(dt)^-1\phi\rangle_{L^2(Y)}s &= 0,\\
	D^+_A\phi - \Re \langle D_A^+ \phi, \rho(dt)^{-1}\phi \rangle_{L^2(Y)}\phi &= 0.
\end{align*}
To interpret the Seiberg-Witten map as a section of vector bundle, given any $(A,s,\phi) \in \mathcal C^{\tau}(Z,\mathfrak s_Z,\uptau_Z)$ we define the fibre
\begin{equation*}
	\mathcal V^{\tau}_{(A,s,\phi)} = 
	\left\{ (\eta,r,\psi)\big|
	\Re\langle \check\phi(t), \check \psi(t)\rangle = 0\right\},
\end{equation*}
which is a subspace of $C^{\infty}(Z;i\mathfrak{su}(S^+))\oplus C^{\infty}(I;\reals)\otimes C^{\infty}(Z;S^-)$.
Together, they form a bundle $\mathcal V^{\tau} \to \mathcal C^{\tau}$.
Hence we have
\begin{equation*}
	\mathfrak F^{\tau}(A,s,\phi)
	=
	\left(\frac{1}{2}*F_{\check A^t} + dc + s^2\rho^{-1}(\check\phi\check\phi^*)_0,
	\Lambda(B,r,\check\phi)s,
	D_B\check\phi + c\check\phi - \Lambda(B,r,\check\phi)\check\phi\right)
\end{equation*}
as a section of $\mathcal V^{\tau}$.

\subsection{Sobolev completions}
\hfill \break
Let $M$ be a smooth manifold of dimension three or four and possibly with boundary.
Let $E \to M$ be a vector bundle equipped with an inner product and a connection $\nabla$.
Let $k \ge 0$ be an integer and let $p > 0$.
The Sobolev space $L^p_k(M;E)$ is the completion of smooth sections of $E$ with respect to the Sobolev norm
\begin{equation*}
	\|f\|^p_{L^p_k} = \int_M |f|^p + |\nabla f|^p + \dots + |\nabla^k f|^p d\vol.
\end{equation*}
If $M$ is closed, the fractional $L^2$-Sobolev space can be defined as
\begin{equation*}
	\|f\|_{L^2_k}=\|(1+\Delta)^{k/2}f\|_{L^2}.
\end{equation*}

The vector bundles of interests will be either be cotangent bundles or spinor bundles.
For simplicity, let $W$ be $S$ if $M$ is a 3-manifold, or be $S^+$ if $M$ is a 4-manifold.
Fix a smooth, $\uptau$-invariant reference spin\textsuperscript{c} connection $A_0$ of $W$ on $M$, and a lift of the involution $\uptau_M: W \to W$.
\begin{defn}
	We define
\begin{itemize}[leftmargin=*]
\item the \emph{ordinary $L^2_k$-configuration space} by
\begin{equation*}
	\underline{\mathcal C}_k(M,\mathfrak s) 
	= 
	(A_0,0) + L^2_k(M;i\T^*M\oplus W)
	= \underline{\mathcal A}_k(M,\mathfrak s) \times L^2_k(M;W), 
\end{equation*}
(where $\underline{A}_k(M,\mathfrak s)$ is the space of all $L^2_k$ -spin\textsuperscript{c} connections,)
\item  the \emph{real $L^2_k$-configuration space} by
\begin{equation*}
	\mathcal C_k(M,\mathfrak s,\uptau) 
	= 
	(A_0,0) + L^2_k(M;i\T^*M\oplus W)^{\uptau}
	= \mathcal A_k(M,\mathfrak s,\uptau) \times L^2_k(M;W)^{\uptau},
\end{equation*}
(where $\mathcal A_k(M,\mathfrak s,\uptau)$ is the space of $\uptau$-invariant spin\textsuperscript{c} connections),

\item and the \emph{(real) $L^2_k$-blown-up configuration space} by
\begin{equation*}
	\mathcal C^{\sigma}_k(M,\mathfrak s,\uptau)  =
	\left \{(A,s,\phi) \in \mathcal A_k(M,\mathfrak s,\uptau) \times \reals_{\ge 0} \times \mathbb S(L^2_k(M;W)^{\uptau}) \bigg| s \ge 0, \|\phi\|_{L^2(M)} = 1
	\right\}.
\end{equation*}
(Note that we are taking the unit sphere of $L^2_k$-spinors with respect to the $L^2$-norm.
This is a smooth Hilbert manifold with boundary $\{s = 0\}$)
\end{itemize}	
\end{defn}
The gauge groups will be completed to a Hilbert manifold as follows.
Assume $2(k+1) > \dim M$ so that $L^2_k \hookrightarrow C^0$.
Let
\begin{itemize}[leftmargin=*]
	\item $\underline{\mathcal G}_{k+1}(M)$ be the subspace of $L^2_k(M;\C)$ of functions with pointwise norm $1$, and
	\item $\mathcal G_{k+1}(M,\upiota)$ be the subset of $\underline{\mathcal G}_{k+1}(M)$ that are in addition invariant under $\upiota$,
\end{itemize}
both equipped with the subspace topology.
Thus $\underline{\mathcal G}_{k+1}(M)$ and $\mathcal G_{k+1}(M,\upiota)$ are Hilbert Lie groups.
The actions of the real gauge groups have $\mathbb Z_2$-stabilizers over the configuration spaces if the spinors are zero, and are free over the blown-up configuration spaces.
In any case, let us define the spaces of gauge equivalence classes of configurations:
\begin{itemize}[leftmargin=*]
	\item $\mathcal B_k(M,\mathfrak s,\uptau) 
	:= \mathcal C_k(M,\mathfrak s,\uptau)/\mathcal G_{k+1}(M,\upiota)$,
	\item $\mathcal B_k^{\sigma}(M;\mathfrak s,\uptau) :
	= \mathcal C_k^{\sigma}(M;\mathfrak s,\uptau)/\mathcal G_{k+1}(M;\upiota)$, 
	\item $\mathcal B_k^{\tau}(Z;\mathfrak s,\uptau) 
	:= \mathcal C_k^{\tau}(Z,\mathfrak s,\uptau)/\mathcal G_{k+1}(Z,\upiota)$,
	\item $\tilde{\mathcal B}_k^{\tau}(Z;\mathfrak s,\uptau): 
	= \tilde{\mathcal C}_k^{\tau}(Z;\mathfrak s,\uptau)/\mathcal G_{k+1}(Z;\upiota)$.
\end{itemize}
In the last two cases, $Z = I \times Y$ is a cylinder.
\begin{rem}
	The spaces $\mathcal B^{\sigma}(M;\mathfrak s,\uptau)$ and $\tilde{\mathcal B}(Z;\mathfrak s,\uptau)$ are Hausdorff by \cite[Proposition~9.3.1]{KMbook2007}.	
\end{rem}
\subsection{Tangent bundles and slices }
\hfill \break
For $j \le k$ and over the non-blown-up configuration space $\mathcal C_k(M,\mathfrak s, \uptau)$, we define the ``generalized'' tangent bundle $\mathcal T_j \to \mathcal C_k(M;\mathfrak s,\uptau)$ by the product bundle
\begin{equation*}
	T_j = L^2_j(M;i\T^*M \oplus W)^{\uptau} \times \mathcal C_k(M;\mathfrak s, \uptau).
\end{equation*}
When $j = k$, this is the ordinary tangent bundle.
Over the blown-up configuration space, we define the $L^2_j$-tangent  tangent bundle $\mathcal T_j^{\sigma} \to \mathcal C^{\sigma}_k(M;\mathfrak s,\uptau)$ by setting the fibre at $\upgamma = (A_0,s_0,\phi_0)$ to be
\begin{equation*}
	\{(a,s,\phi) \big| \Re\langle \phi_0,\phi\rangle_{L^2} = 0\}
	\subset L^2_j(M;i\T^*M)^{\uptau} \times \reals \times L^2_j(M;W)^{\uptau}.
\end{equation*}
When $j = k$, the tangent bundle is exactly $\mathcal T_j^{\sigma}$.
Similarly, we define the ``tangent  bundle'' $\mathcal T_j \to \mathcal C^{\sigma}_k(M;\mathfrak s, \uptau)$ over the blown-up configuration space, which decomposes into a direct sum along the gauge group actions and the orthogonal direction.

Indeed, at $\upgamma = (A_0,\Phi_0)$ the derivative of the gauge group action is
\begin{align*}
	\mathbf{d}_{\upgamma}:\T_e\mathcal G_{k+1}(M)
	= L^2_{k+1}(M;i\reals)^{-\upiota^*} &\to \T_{\upgamma}\mathcal C_k(M)\\
	\xi &\mapsto (-d\xi, \xi\Phi_0).
\end{align*}
Let $\mathcal J_{\upgamma,k}$ be the \emph{$L^2_j$-completion} of the image of $\mathbf{d}_{\upgamma}$ and $\mathcal K_{\upgamma,k}$ be the orthogonal complement, with respect to the \emph{$L^2$-inner product}, that is,
\begin{equation*}
	\mathcal K_{\upgamma,k}
	=
	\{(a,\phi) \big| -d^*a + i\Re\langle i\Phi_0, \phi \rangle = 0 \text{ and } \langle a, \nu \rangle = 0 \text{ at } \del M\},
\end{equation*}
as a subspace of $\mathcal T_j$.
As for the blow-up, let $\mathcal J^{\sigma}_{\upgamma,j}$
be the closure of the image of the map
\begin{align*}
	\mathbf{d}^{\sigma}_{\upgamma}:\T_e\mathcal G_{k+1}(M) = L^2_{k+1}(M;i\reals)^{-\upiota^*}
	&\to \T_{\upgamma}\mathcal C^{\sigma}_k(M)\\
	\xi &\mapsto (-d\xi, 0, \xi\phi_0)
\end{align*}	
for any $\upgamma = (A_0, s_0, \phi_0)$.
The complementary subspace $\mathcal K_{\upgamma,j} \subset \mathcal T_{\upgamma,j}$ is not orthogonal with respect to any natural metric.
Rather, it is defined as the set of triples $(a,s,\phi)$ satisfying
\begin{align*}
	\langle a, \nu \rangle 
	&= 0 \text{ at }\del M,\\
	-d^*a + is_0^2 \Re\langle i\phi_0,\phi \rangle 
	&= 0,\\
	\Re\langle i\phi_0,\phi \rangle_{L^2(M)} &= 0.
\end{align*}
So we have bundle decomposition over $\mathcal C^{\sigma}_k$ (this is the $\uptau$-invariant analogue of \cite[Proposition~9.3.5-6]{KMbook2007}):
\begin{equation}
\label{eqn:tangent_decomp}
	\mathcal T^{\sigma}_{j} = \mathcal J^{\sigma}_{j} \oplus \mathcal K^{\sigma}_{j},
\end{equation}
and a decomposition 
\begin{equation}
\label{eqn:tangent_decomp_irred}
	\mathcal T_{j} = \mathcal J_{j} \oplus \mathcal K_{j} 
\end{equation}
over the irreducible locus of the configuration space $\mathcal C_k^* \subset \mathcal C_k$, compatible with the blow-down map.

Finally, $\upgamma = (A_0, \Phi_0) \in \mathcal C_k(M,\uptau)$, we define the slice 
\begin{equation*}
	\mathcal S_{k,\upgamma} = \{(A_0 + a, \Phi) \big|
	-d^*a + i\Re\langle i\Phi_0,\Phi\rangle = 0 \text{ and } \langle a,\nu \rangle = 0 \text{ at } \del M\}.
\end{equation*}
Then $\mathcal S_{k}$ is a closed Hilbert submanifold of $\mathcal C_k$, whose tangent space at $\upgamma$ is $\mathcal K_{\upgamma,k}$.
Moreover, an open neighbourhood $U \subset \mathcal S_{k,\upgamma}$ of $\upgamma$ provides a local chart for the Hilbert manifold $\mathcal B_k(M,\mathfrak s)$ around $[\upgamma]$, via the composition
\begin{equation*}
	S_k(M,\upgamma) \hookrightarrow \mathcal C_k(M;\mathfrak s) \to \mathcal B(M;\mathfrak s).
\end{equation*}
For a blown-up configuration $\upgamma \in \mathcal C^{\sigma}_k(M,\mathfrak s)$ we define the slice $\mathcal S^{\sigma}_{k,\upgamma}$ to be the closed Hilbert submanifold of $\mathcal C^{\sigma}(M,\mathfrak s)$ consisting of triples $(A_0+a,s,\phi)$ satisfying
\begin{align*}
	\langle a, \nu \rangle &= 0 \text{ at }\del M,\\
	-d^*a + iss_0\Re\langle i\phi_0,\phi \rangle &=0,\\
	\Re\langle i\phi_0,\phi\rangle_{L^2(M)} &=0.
\end{align*}
If $s \ne 0$, then $\mathcal S^{\sigma}_{k,\upgamma}$ projects onto $\mathcal S_{k,\upgamma}$ and can be considered as the proper transform of $\mathcal S_{k,\upgamma}$.

\subsection{Topology of configuration space}
\label{subsec:topofconfig}
\hfill \break
Let $\upgamma_0 = (A_0,0)$ be a reducible configuration, where $A_0$ is $\uptau$-invariant.
The slice $\mathcal S_{k,\upgamma_0}$ defined as
\begin{align*}
	d^* a &= 0,\\
	\langle a,\nu \rangle &= 0 \text{ at }\del M.
\end{align*}
is a global slice.
Indeed, any configuration $(A,\Phi)$ can be put into $\mathcal S_{k,\upgamma}$ via a gauge transformation $u = e^{\xi}$, by solving the Neumann problem
\begin{align*}
	\langle d\xi, \nu\rangle &= \langle a,\nu \rangle \text{ at }\del M,\\
	\Delta \xi &= d^*a,
\end{align*}
which has a unique solution $\xi \in L_{k+1}^{-\upiota^*}$ such that $\int_M \xi = 0$.
Define the following topologically contractible group
\begin{equation*}
	\mathcal G_{k+1}^{\perp} = \left\{e^{\xi} \bigg| \upiota^*(\xi) = -\xi, \text{ and }\int_M \xi = 0\right\}.
\end{equation*}
The gauge group action is a diffeomorphism
\begin{align*}
	\mathcal G_{k+1}^{\perp} \times \mathcal K_{k,\upgamma_0} &\to \mathcal C_k\\
	(e^{\xi},(a,\phi)) &\mapsto (A_0 + (a-d\xi)\otimes 1, e^{\xi}\phi).
\end{align*}
The gauge group $\mathcal G_{k+1}$ decomposes as $\mathcal G^h \times \mathcal G_{k+1}^{\perp}$, where $\mathcal G^h$ consists of the harmonic maps $u: M \to S^1$, satisfying $u(\upiota(x)) = \bar u(x)$.
If the fixed point locus is nonempty (so that the only constant gauge transformations are $\mathbb Z_2 = \{\pm 1\}$), there is a exact sequence
\begin{equation*}
	1 \to \mathbb Z_2 \to \mathcal G^h \to H^1(M;\mathbb Z)^{-\upiota^*} \to 1.
\end{equation*}
For the group of component $\pi_0$, we have the exact sequence
\begin{equation*}
	0 \to 
	\mathbb Z_2 \to \pi_0(\mathcal G^h) \to H^1(M;\mathbb Z)^{-\upiota^*} 
	\to 0
\end{equation*}
Since $\mathcal G^h$ is abelian and $H^1(M;\mathbb Z)^{-\upiota^*}$ is free, 
\begin{equation*}
	\pi_0(\mathcal G^h) \cong \mathbb Z_2 \times H^1(M;\mathbb Z)^{-\upiota^*} 
\end{equation*}
where the isomorphism depends on a section $s:  H^1(M;\mathbb Z)^{-\upiota^*} 
 \to \mathcal G^h$.

Let $\mathcal K^*_{k,\upgamma_0}$ be the subset of $\mathcal K_{k,\upgamma_0}$ for which $\Phi \ne 0$.
We have the following homotopy equivalence
\begin{equation*}
	\mathcal B^{\sigma} \simeq \mathcal K^*_{k,\upgamma_0}/\mathcal G^h,
\end{equation*}
coming from the fibration
\begin{equation*}
	\frac{L^2(M;W)^{\uptau}- 0}{\{\pm 1\}}
	\hookrightarrow
	\frac{\mathcal K^*_{k,\upgamma_0}}{\mathcal G^h}
	\to
	\frac{H^1(M;i\reals)^{-\upiota^*}}{H^1(M;i\mathbb Z)^{-\upiota^*}},
\end{equation*}
which is trivializable by Kuiper's theorem~\cite{KUIPER196519}.
Denote the $(-\upiota^*)$-invariant subtorus of the Picard torus by $\mathbb T$:
\begin{equation*}
	\mathbb T := \frac{H^1(M;i\reals)^{-\upiota^*}}{H^1(M;i\mathbb Z)^{-\upiota^*}}.
\end{equation*}
In other words, we have a fibration
\begin{equation*}
	\mathbb{RP}^{\infty} \to \mathcal B^{\sigma}
	\to \mathbb T.
\end{equation*}
The cohomology $H^*(\mathcal B^{\sigma};\mathbb Z_2)$ is a module of $H^*(\mathbb T;\mathbb Z_2)$, and 
there is a non-canonical isomorphism 
\begin{equation*}
	H^*(\mathcal B^{\sigma};\mathbb Z_2) \cong  \mathbb F_2[\upsilon]\otimes \Lambda^*(H_1(M;\mathbb Z)^{-\upiota^*}/\text{torsion}),
\end{equation*}
where $\upsilon$ has degree $1$.
The choice of $\upsilon$ depends on a trivialization $\mathcal B^{\sigma} \simeq \reals^{\infty} \times \mathbb T$.
We shall return to the module structure in the cobordism section.

\subsection{Seiberg-Witten maps on completions}
\subsubsection*{\textbf{(4-manifolds)}}
Let $X$ be a compact four-manifold, possibly with boundary.
For $j \le k$, let $\mathcal V_j(X,\mathfrak s) \to \mathcal C_k(X,\mathfrak s)$ be the trivial bundle with fibre $L^2_j(X;i\mathfrak{su}(S^+)\oplus S^-)$, and $\mathcal V_j^{\sigma} \to \mathcal C^{\sigma}_k$ be the pullback bundle under blow-down.
The Seiberg-Witten map $\mathfrak F^{\sigma} $ can be extended to be the section of $\mathcal V_j^{\sigma} \to \mathcal C^{\sigma}_k$ be the extension of 
\begin{equation*}
	\mathfrak F^{\sigma}:\mathcal C^{\sigma}_k(X,\mathfrak s) 
	\to \mathcal V^{\sigma}_{k-1} \subset \mathcal V^{\sigma}_j,
\end{equation*}
which is $\mathcal G_{k+1}(X)$-equivariant if $j=k-1$.
\subsubsection*{\textbf{(3-manifolds)}}
In the case of 3-manifolds, we consider the (smooth) gradient vector field $\grad \mathcal L$ on the completions
\begin{align*}
	\grad \mathcal L &: \mathcal C_k(Y,\mathfrak s) \to \mathcal T_{k-1},\\
	(\grad \mathcal L)^{\sigma} &:
	\mathcal C^{\sigma}_k(Y,\mathfrak s) \to 
	\mathcal T_{k-1}^{\sigma}.
\end{align*}
\subsubsection*{\textbf{(cylinders)}}
Let $Z = I \times Y$ and $k > 0$ be integer.
We can define the configuration 
\begin{equation*}
	\mathcal C^{\tau}_k(Z,\mathfrak s)
	\subset
	\mathcal A_k(Z,\mathfrak s) \times L^2_k(I;\reals) \times L^2_k(Z;S^+)
\end{equation*}
as the space of triples $(A,s,\phi)$ with $s(t) \ge 0$ and $\|\phi(t)\|_{L^2(Y)} = 1$ for all $t \in I$.
This is not a Hilbert manifold with boundary because of the condition $s(t) \ge 0$, but it is a closed subspace of the Hilbert manifold
\begin{equation*}
	\tilde{\mathcal C}_k^{\tau}(Z,\mathfrak s) \subset 
	\mathcal A_k(Z,\mathfrak s) \times L^2_k(I;\reals) \times L^2_k(Z;S^+),
\end{equation*}
consisting of triples that satisfy only $\|\phi(t)\|_{L^2(Y)} = 1$ for all $t \in I$, i.e. we drop the $s(t) \ge 0$ assumption.

\subsection{Revisiting equivalence of real structures}
\hfill \break
Let $(M, \upiota)$ be a \emph{real} 3 or 4-manifold.
Let $\mathfrak F$ be the (unperturbed) Seiberg-Witten map, which is either the non-blown-up version, the $\sigma$-version, or the $\tau$-version.
The moduli spaces of Seiberg-Witten solutions are of the form
\begin{equation*}
	\frac{\mathfrak F^{-1}(0)^{\uptau}}{\mathcal G(\upiota)}.
\end{equation*}
Suppose we are given two real structures $\uptau_0$ and $\uptau_1$.
Recall that $\uptau_0$ and $\uptau_1$ are equivalent if there exists an ordinary gauge transformation $g$ such that 
\begin{equation*}
	g\uptau_0 = \uptau_1 g.
\end{equation*}
Hence $g$ induces an isomorphism of the moduli spaces
\begin{equation*}
	\frac{\mathfrak F^{-1}(0)^{\uptau_0}}{\mathcal G(\upiota)}
	\to 
	\frac{\mathfrak F^{-1}(0)^{\uptau_1}}{\mathcal G(\upiota)}.
\end{equation*}
This continues to hold if $\mathfrak F$ is a perturbed Seiberg-Witten map that is equivariant with respect to the ordinary gauge transformation.

\section{Perturbations}
\label{sec:perturb}
In this section, we shall perturb the CSD functional to achieve transversality of moduli spaces of trajectories.
Our perturbation scheme is based on \cite[Section~10-11]{KMbook2007}, where all analytical estimates carry over to our setting.
For notational convenience, we fix once and for all a real spin\textsuperscript{c} structure $(\mathfrak s,\uptau)$, and may suppress ``$\mathfrak s$'' and ``$\uptau$'' in notations.
(All objects in this section are invariant, unless they are underlined.)
\subsection{Abstract perturbations}
\hfill \break
The perturbed CSD functional takes the form of $\pertL = \mathcal L + f$, where $f: \mathcal C(Y,\uptau) \to \reals$ is a $\mathcal G(Y,\upiota)$-invariant function.
The formal gradient field is given by
\[
	\grad \pertL = \grad \mathcal L + \mathfrak q,
\]
where $\mathfrak q$ is a section of $\mathcal T_0 \to \mathcal C(Y,\uptau)$, such that for any path $\upgamma: [0,1] \to \mathcal C(Y,\uptau)$
\[
	f \circ \upgamma(1) - f \circ \upgamma(0) = \int_0^1 \langle \dot\upgamma, \mathfrak q \rangle_{L^2(Y)} dt.
\]
For a path $\upgamma$, we can pull back the perturbation $\mathfrak q$ to $Z = I \times Y$ to get a section
\begin{equation*}
	\hat{\mathfrak q}:\mathcal C(Z, \uptau_Z) \to \mathcal V_0(Z, , \uptau_Z).
\end{equation*}
Indeed, any $(A,\Phi) \in \mathcal C(Z, \uptau_Z)$ restricts to a continuous path $(\check A(t),\check \Phi(t)) \in \mathcal C(Y, \uptau)$ and continuous path $\mathfrak q(\check A(t),\check\Phi(t)) \in L^2(Y;i\T^*Y\oplus S)^{\uptau}$, where $i\T^*Y \oplus S$ can be identified with $i\mathfrak{su}(S^+)\oplus S^-$ via Clifford multiplications.
We pack the desired analytic properties of a perturbation for our applications in the following definition.
\begin{defn}
\label{defn:tame}
	Let $k \ge 2$ be an integer.
	A perturbation $\mathfrak q: \mathcal C(Y,\uptau) \to \mathcal T_0$ is \emph{$k$-tame} it is the formal gradient of a $\mathcal G(Y,\upiota)$-equivariant function on $\mathcal C(Y,\uptau)$, satisfying the following properties:
	\begin{enumerate}[leftmargin=*]
		\item[(i)] the associated four-dimensional perturbation $\hat{\mathfrak q}$ defines a smooth section
			\[
				\hat{\mathfrak q}:\mathcal V_k(Z,\uptau_Z) \to \mathcal C_k(Z,\uptau_Z);
			\]
		\item[(ii)] for every $j \in [1,k]$ the section $\hat{\mathfrak q}$ extends to a continuous section
			\[
				\hat{\mathfrak q}:\mathcal V_j(Z,\uptau_Z) \to \mathcal C_j(Z,\uptau_Z);
			\]
		\item[(iii)] for every $j \in [-k,k]$ the first derivative
			\[
				\mathcal D\hat{\mathfrak q} \in C^{\infty}(\mathcal C_k(Z, \uptau_Z), \Hom(\T\mathcal C_k(Z, \uptau_Z),\mathcal V_k(Z, \uptau_Z))) 
			\]
			extends to a map
			\[
				\mathcal D\hat{\mathfrak q} \in C^{\infty}(\mathcal C_k(Z, \uptau_Z),\Hom(\T\mathcal C_j(Z),\mathcal V_j(Z, \uptau_Z)));
			\]
		\item[(iv)] there is a constant $m_2$ such that
			\[
				\|\mathfrak q(B,\Psi)\|_{L^2} \le m_2(\|\Psi\|_{L^2} + 1)
			\]
			for every $(B,\Psi) \in \mathcal C_k(Y, \uptau)$;
		\item[(v)] for any fixed smooth connection $A_0$, there is a real function $\mu_1$ such that the inequality
			\[
				\|\hat{\mathfrak q}(A,\Phi)\|_{L^2_{1,A}} \le \mu_1(\|(A-A_0,\Phi)\|_{L^2_{1,A_0}})
			\]
			holds for all configurations $(A,\Phi) \in \mathcal C_k(Z, \uptau_Z)$;
		\item[(vi)] the three-dimensional perturbation $\mathfrak q$ defines a $C^1$-section
			\[
				\mathfrak q: C_1(Y, \uptau) \to \mathcal T_0.
			\]
	\end{enumerate}
	Moreover, $\mathfrak q$ is \emph{tame} if it is $k$-tame for every $k \ge 2$.
\end{defn}
\subsection{Perturbed Seiberg-Witten maps and gradient}
\subsubsection*{\textbf{Seiberg-Witten map on $\mathcal C_k(Z, \uptau_Z)$}}
Let $Z = I \times Y$ and $\mathfrak q$  be a perturbation on $Y$.
The perturbed Seiberg-Witten map is a section of $\mathcal V_{k-1} \to \mathcal C_k$ given by
\begin{equation*}
	\mathfrak{F}_{\mathfrak q} = \mathfrak F + \hat{\mathfrak q}.
\end{equation*}
If we decompose the three-dimensional perturbation as $
\mathfrak q = (\mathfrak q^0,\mathfrak q^1)$ where
\begin{equation*}
	\mathfrak q^0 \in L^2(Y;i\T^*Y)^{-\upiota^*},\quad
	\mathfrak q^1 \in L^2(Y;S)^{\uptau},
\end{equation*}
then the the four-dimensional perturbation $\hat{\mathfrak q}$ can be written as $(\hat{\mathfrak q}^0,\hat{\mathfrak q}^1)$ such that
\begin{equation*}
	\hat{\mathfrak q}^0 \in L^2(Z;i\mathfrak{su}(S^+))^{-\upiota_Z^*},\quad
	\hat{\mathfrak q}^1 \in L^2(Z;S^-)^{\uptau}.
\end{equation*}
The perturbed Seiberg-Witten equations are of the form
\begin{align*}
	\rho_Z(F^+_{A^t}) - 2(\Phi\Phi^*)_0 &= -2\hat{\mathfrak q}^0(A,\Phi),\\
	D^+_A\Phi &= -\hat{\mathfrak q}^1(A,\Phi),
\end{align*}
which can be written as a flow 
\begin{align*}
	\frac{d}{dt}B^t &= - * F_{B^t} - 2\rho^{-1}(\Psi\Psi^*)_0 - 2 \mathfrak q^0(B,\Psi),\\
	\frac{d}{dt}\Psi &= -D_B \Psi - \mathfrak q^1(B,\Psi).
\end{align*}

\subsubsection*{\textbf{Seiberg-Witten maps on $\mathcal C_k^{\sigma}(Z, \uptau_Z)$}}
The perturbation $\mathfrak q$ induces a perturbation on the cylindrical blown-up configuration:
\begin{equation*}
	\hat{\mathfrak q}^{\sigma}:\mathcal C^{\sigma}_k(Z, \uptau_Z) \to \mathcal V^{\sigma}_k,
\end{equation*}
by defining $\hat{\mathfrak q}^{\sigma} = (\hat{\mathfrak q}^0,\hat{\mathfrak q}^{1,\sigma})$ and the second factor 
\begin{align*}
	\hat{\mathfrak q}^1:\mathcal C_k^{\sigma}(Z, \uptau_Z) 
	&\to L^2_k(Z;S^-)^{\uptau_Z},\\
	(A,s,\phi) &\mapsto
	\int_0^1 (\mathcal D_{(A,rs\phi)}\hat{\mathfrak q}^1)(\phi)dr.
\end{align*}
Notice that $\hat{\mathfrak q}^1(A,0) = 0$ by $\mathbb Z_2$-equivariance.
There is a corresponding perturbed Seiberg-Witten equations on the blow-up
\begin{equation*}
	\mathfrak F_{\mathfrak q}^{\sigma}=
	\mathcal F^{\sigma} + \hat{\mathfrak q}^{\sigma}:
	\mathcal C_k^{\sigma} \to \mathcal V_{k-1}^{\sigma}
\end{equation*}
as a smooth section of the bundle.
Again, we write the perturbed Seiberg-Witten equation on the blow-up as a gradient flow equation:
\begin{align*}
	\frac{d}{dt}B^t &= -*F_{B^t} - 2r^2\rho^{-1}(\psi\psi^*)_0 - 2\mathfrak q^0(B,r\psi), \\
	\frac{d}{dt}r &= -\Lambda_{\mathfrak q}(B,r,\psi)r,\\
	\frac{d}{dt}\psi &= - D_B \Psi  - \tilde{\mathfrak q}^1(B,\Psi) + \Lambda_{\mathfrak q}(B,r,\psi)r.
\end{align*}
where
similar to $\hat{\mathfrak q}^{1,\sigma}$, we define
\begin{equation*}
	\tilde{\mathfrak q}^1(B,r,\psi) =
	\int_0^1 \mathcal D_{(B,sr\psi)}\mathfrak q^1(0,\psi)ds,
\end{equation*}
and similar to $\Lambda$ in, we define
\begin{equation}
\label{eqn:Lambda_q}
	\Lambda_{\mathfrak q}(B,r,\psi) = \Re \langle \psi,
	D_B\psi + \tilde{\mathfrak q}^1(B,r,\psi)\rangle_{L^2}.
\end{equation}

\subsubsection*{\textbf{Perturbed gradient on $\mathcal C_k^{\sigma}(Y, \uptau)$}}
The perturbed gradient 
\[(\grad \pertL)^{\sigma} = (\grad L)^{\sigma} + \mathfrak q^{\sigma}\] 
on the blown-up configuration is a smooth section of the vector bundle $\mathcal T^{\sigma}_{k-1} \to \mathcal C^{\sigma}_k$, which in coordinates looks like
\begin{equation*}
	\mathfrak q^{\sigma}(B,r,\psi) =
	\left(\mathfrak q^0(B,r\psi), \quad
	\langle \tilde{\mathfrak q}^1(B,r,\psi),\psi\rangle_{L^2(Y)}r, \quad
	\tilde{\mathfrak q}(B,r,\psi)^{\perp}\right).
\end{equation*}
Here $\perp$ denotes the orthogonal projection to the real orthogonal complement of $\psi$.
We define the \emph{perturbed Dirac operator}
\begin{equation*}
	D_{B,\mathfrak q}\psi = D_B\psi + \mathcal D_{(B,0)}\mathfrak q^1(0,\psi).
\end{equation*}
It follows that a point $(B,r,\psi)$ of the gradient is either
\begin{itemize}
	\item $r \ne 0$, and $(B,r\psi)$ is a critical point of $\grad \pertL$, or
	\item $r = 0$, and $(B,0)$ is a critical point of $\grad \pertL$ and $\psi$ is an eigenvector of $D_{B,\mathfrak q}$.
\end{itemize}

\subsection{Cylinder functions}
\hfill \break
We first define two classes of functions, where the first depends only on the form component and the second depends on both form and spinorial components.
\subsubsection*{(o)}
Recall the exact sequence of groups
\begin{equation*}
	\mathbb Z_2 \to \mathcal G^h \to H^1(Y;\mathbb Z)^{-\upiota^*}.
\end{equation*}
A section $v:H^1(Y;\mathbb Z)^{-\upiota^*} \to \mathcal G^h$ defines a subgroup $\mathcal G^{h,o} = v(H^1(Y;\mathbb Z)^{-\upiota^*})$, and
\begin{equation*}
	\mathcal G^o = \mathcal G^{h,o} \times \mathcal G^{\perp} \subset \mathcal G.
\end{equation*}
Fix a point $y_0 \in Y$, we can choose $\mathcal G^{h,o}$ so that
\begin{equation*}
	\mathcal G^{h,o} = \{u | u(y_0) = 1\} \subset \mathcal G^h.
\end{equation*}
The \emph{based configuration space} is the intermediate quotient
\begin{equation*}
	\mathcal B^o_k(Y,\uptau) = \mathcal C_k(Y,\uptau)/\mathcal G_{k+1}^o(Y,\upiota).
\end{equation*}
In particular, $\mathcal B^o_k(Y,\uptau)$ is a branched double cover of $\mathcal B_k(Y,\uptau)$, along the reducible locus.

\subsubsection*{(i)}
Let $c \in \Omega^1(Y;i\reals)$ be a coclosed, $(-\upiota^*)$-invariant one-form.
Define the following function
\begin{equation*}
	r_c:\mathcal C(Y,\uptau) \to \reals,\quad
	r_c(B_0 + b\otimes 1,\Psi)
	= \int_Y b \wedge *\bar c = \langle b, c\rangle_Y,
\end{equation*}
which is invariant under the identity component $\mathcal G^e$ in general, but invariant under $\mathcal G$ if $c$ is harmonic.
Pick an integral basis $\{\omega_1, \dots, \omega_t\}$ of $H^1(Y;i\reals)^{-\upiota^*}$ and let 
\[
	\mathbb T = H^1(Y;i\reals)^{-\upiota^*}/2\pi H^1(Y;i\mathbb Z)^{-\upiota^*}.
\]
We define a ``period'' map
\begin{equation*}
	(B,\Psi) \mapsto (r_{\omega_1}(B,\Psi),\dots, r_{\omega_t}(B,\Psi)) \mod 2\pi\mathbb Z^t,
\end{equation*}
which agrees with the projection map onto the harmonic forms
\begin{equation*}
	(B_0 + b\otimes 1,\Psi) \mapsto
	[b_{\text{harm}}] \in \mathbb T. 
\end{equation*}

\subsubsection*{(ii)}
Consider the smooth rank-two vector bundle 
$$\mathbb S \to \mathbb T \times Y$$
given by the quotient $H^1(Y;i\reals)^{-\upiota^*} \times S \to H^1(Y;i\reals)^{-\upiota^*} \times Y$ by the gauge group $\mathcal G^{h,0}$.
There is an involution $\uptau$ on $\mathbb S$ induced from the involution on $S$.
Any $\uptau$-invariant section $\Upsilon$ of $\mathbb S$ can be lifted to a section
\begin{equation*}
	\tilde{\Upsilon}:H^1(Y;i\reals)^{-\upiota^*} \times Y \to 
	H^1(Y;i\reals)^{-\upiota^*} \times S,
\end{equation*}
which is quasi-periodic in the sense that
\begin{equation*}
	\tilde{\Upsilon}_{\alpha + \kappa}(y) := \Upsilon(\alpha + \kappa, y) = 
	u(y)\Upsilon(y).
\end{equation*}
Any invariant section $\Upsilon$ defines a $\mathcal G^o(Y,\upiota)$-equivariant map
\begin{align*}
	\Upsilon^{\dag}:\mathcal C(Y, \uptau) &\to C^{\infty}(S)^{\uptau}\\
	(B_0 + b \otimes 1,\Psi) &\mapsto
	e^{-Gd^*b}\tilde{\Upsilon}_{b_{\text{harm}}},
\end{align*}
where $G:L^2_{k-1}(Y) \to L^2_{k+1}(Y)$ is Green's operator for $\Delta = d^*d$, and $C^{\infty}(S)^{\uptau}$ is the space of $\uptau$-invariant smooth sections of $S \to Y$.
(Note that $\Upsilon^{\dag}$ does not depend on the second factor.)

Define the $\mathcal G^o(Y,\upiota)$-invariant map
\begin{align*}
	q_{\Upsilon} &:\mathcal C(Y, \uptau) \to \reals \\
	(B,\Psi) &\mapsto
	\int_Y \langle \Psi, \Upsilon^{\dag}(B,\Psi)\rangle
	= \langle \Psi, \tilde{\Upsilon}^{\dag}\rangle_Y,
\end{align*}
which is equivariant under $\{\pm 1\}$ on the spinor.
While the inner product is a priori $\C$-valued, the $\uptau$-invariance implies $q_{\Upsilon}$ is real-valued.
\subsubsection*{(i+ii)}
Choose $(n+t)$ coclosed 1-forms $c_1,\dots,c_{n+t}$ so that the first $n$ 1-forms are coexact and the last $t$ 1-forms are harmonic and represent an integral basis of $H^1(Y;\reals)^{-\upiota^*}$.
Choose $m$ $\uptau$-invariant sections $\Upsilon_1,\dots,\Upsilon_m$ of $\mathbb S$.
Combining the two classes of $\mathcal G^o$-invariant functions, we consider
\begin{equation*}
	p:\mathcal C(Y,\uptau) \to \reals^n \times \mathbb T \times \mathbb \reals^m
\end{equation*}
given by
\begin{equation}\label{eqn:pcylinder}
	p(B,\Psi) = (r_{c_1}(B,\Psi),\dots, r_{c_{n+t}}(B,\Psi),q_{\Upsilon_1}(B,\Psi),\dots, q_{\Upsilon_m}(B,\Psi)),
\end{equation}
which is $\mathcal G^o$-invariant and equivariant under $\{\pm 1\}$.
We use $p$ to define cylinder functions:
\begin{defn}
	A gauge invariant function $f:\mathcal C(Y, \uptau) \to \reals$ is a \emph{cylinder function} if it is of the form $f = g \circ p$ such that
	\begin{itemize}
		\item the map $p:\mathcal C(Y, \uptau) \to \reals^n \times \mathbb T \times \mathbb R^m$ is define as in Equation~\eqref{eqn:pcylinder} for some $n,m \ge 0$.
		\item the function $g$ is a $\mathbb Z_2$-invariant function on $\reals^n \times \mathbb T \times \mathbb R^m$ with compact support.
	\end{itemize}
\end{defn}
\begin{rem}
\label{rem:different_perturb}
	In fact, we do not need to define $p$ using pairing with $\uptau$-invariant objects, as long as we allow $g$ to be only $\mathbb Z_2$-invariant. 
\end{rem}
Indeed, gradients of cylinder functions are tame.
\begin{thm}
	If $f$ is a cylinder function, then its gradient
	\begin{equation*}
		\grad f: \mathcal C(Y,\uptau) \to \mathcal T_0
	\end{equation*}
	is a tame perturbation, in the sense of Definition~\ref{defn:tame}
\end{thm}

\begin{proof}
	This theorem can be proved exactly the same way as the ordinary version in \cite[Section~11.3-5]{KMbook2007}, as there is no analytical difference.
\end{proof}

\subsection{Embedding critical sets}
\hfill \break
The following theorem, analogous to \cite[Proposition~11.2.1]{KMbook2007}, essentially asserts that any compact subset of the based configuration can be embedded into finite dimensional manifolds.
\begin{prop}\label{thm:emb}
Let $M$ be a finite dimensional $C^1$ submanifold $M $ of $\mathcal B^o_k(Y,\uptau)$ and let $K$ be a compact subset of $M$:
\begin{equation*}
	K \subset M \subset \mathcal B^o_k(Y,\uptau).
\end{equation*}
Assume both $K$ and $M$ are invariant under the action of $\mathbb Z_2$.
Then there exist a collection of co-closed invariant form $c_{\nu}$, a collection of invariant sections $\Upsilon_{\mu}$ of $\mathbb S$, and a neighbourhood $U$ of $K$ in $M$, such that the corresponding map
\begin{equation*}
	p:\mathcal B^o_k \to \reals^n \times \mathbb T \times {\mathbb R}^m
\end{equation*}
is an embedding of $U$. 
\end{prop}
\begin{proof}
	The proof in \cite[Proposition~11.2.1]{KMbook2007} applies to our $\uptau$-invariant situations.
	The two separation properties (the following two bullet points) in the ordinary case follows from pairing connections (resp. spinors) with connections (resp. spinors) with inner product.
	It suffices to show:
	\begin{itemize}
		\item given two points $x \ne y \in K$, there exists collection $c_i$ and $\Upsilon_{\mu}$ such that $p$ separates $x$ and y;
		\item given $x \in K$ and $v \in T_xM$, there exists $p$ whose differential at $x$ does not annihilate $v$.
	\end{itemize}
	For instance, to prove point separation, let $x = [B_x,\Psi_x]$ and $y = [B_y,\Psi_y]$ be two points, consisting of $\uptau$-invariant connections and spinors.
	After a $\mathcal G^{h,o}$-gauge transformation, we assume that $B_x$ and $ B_y$ have the same harmonic parts; otherwise they project onto distinct image in $\mathbb T$, and are detected by $r_{\omega}$ for harmonic $\omega$'s.
	Moreover, up to a $\mathcal G^{\perp}$-gauge transformation, we can assume $B_x = B_y = B$ where $B= B_0 + b\otimes 1$ and $d^*b = 0$; otherwise, can find a coexact $(-\upiota^*)$-invariant $c$ such that $r_c$ separates $B_x$ and $B_y$.
	Then it must be case that $\Psi_x \ne \Psi_y$, but there exists an $\uptau$-invariant section $\Upsilon$ of $\mathbb S$ such that $\Upsilon_{b_{\text{harm}}}$ has nonzero inner product with $(\Psi_x-\Psi_y)$.
	The second bullet-point is essentially the same.
\end{proof}
We record the following corollary for later use.
\begin{cor}
\label{cor:separatetangent}
	Given any $[B,\Psi]$ in $\mathcal B^*_k(Y,\uptau)$ and any non-zero tangent vector $v$ to $\mathcal B^*_k(Y,\uptau)$ at $[B,\Psi]$, there exists a cylinder function $f$ whose differential $\mathcal D_{[B,\Psi]}f(v)$ is non-zero.
\end{cor}
	 From Theorem~\ref{thm:emb} we see that it suffices to use $\uptau$-invariant connections and spinor sections to define function $p$ that embeds compact subset of submanifolds inside $\mathcal B$.

\subsection{Large Banach spaces of tame perturbations}
\hfill \break
Our choice for a Banach space of perturbations consists of the following:
\begin{itemize}
	\item two positive integers $n,m$;
	\item $(-\upiota^*)$-invariant coexact forms $c_1,\dots, c_n$ and $\uptau$-invariant sections $\Upsilon_1,\dots, \Upsilon_m$ of the bundle $\mathbb S$;
	\item a compact subset $K$ of $\reals^n \times \mathbb T \times \reals^m$;
	\item a smooth $\mathbb Z_2$-invariant function $g$ on $\reals^n\times\mathbb Z\times \reals^m$ supported in $K$.
\end{itemize}
First, for any pair $(n,m)$ we choose a countable collections of $(n+m)$-tuples 
\[(c_1,\dots,c_n,\Upsilon_1,\dots,\Upsilon_m)\] which are dense in the $C^{\infty}$-topology in the space of all such $(n+m)$-tuples.
Next, we choose a countable collection of compact subset $K$ of $\reals^n\times \mathbb T \times \reals^m$ that is dense in the Hausdorff topology.
Finally, for each $K$ we choose a collection of functions $g_{\alpha} = g(n,m,K)_{\alpha}$ with the following properties
\begin{itemize}[leftmargin=*]
	\item each $g_{\alpha}$ is $\mathbb Z_2$-invariant and supported in $K$;
	\item the collection $\{g_{\alpha}\}$ is dense in the $C^{\infty}$-topology of smooth, $\mathbb Z_2$-invariant functions, supported in $K$;
	\item the subset of $\{g_{\alpha}\}$ which vanish on the set
		\[K_0 = K \cap (\reals^n \times \mathbb T \times \{0\})\] are dense in the $C^{\infty}$-topology of smooth, $\mathbb Z_2$-invariant functions supported in $K$ and vanishing on $K_0$.
\end{itemize}
Rename the countable collection of perturbations we constructed above by $\{\mathfrak q_i\}$, indexed by $i \in \mathbb N$.
\begin{thm}\label{thm:largebanachspace}
Let $\mathfrak q_i, i \in \mathbb N$ be any countable collection of tame perturbations arising as gradients of cylinder functions on $\mathcal C(Y,\uptau)$.
Then there exists a separable Banach space $\mathcal P$ and a linear map
	\begin{align*}
		\mathfrak D:\mathcal P 
		&\to C^0(\mathcal C(Y,\uptau),\mathcal T_0)\\
		\lambda &\mapsto \mathfrak q_{\lambda},
	\end{align*}
	such that every $\mathfrak q_{\lambda}$ is a tame perturbation and the image contains the family $\{\mathfrak q_i\}_{i \in \mathbb N}$.
	Furthermore, we have
	\begin{itemize}[leftmargin=*]
		\item for a cylinder $Z = I \times Y$ and all $k \ge 2$, the map
			\begin{equation*}
				\mathcal P \times \mathcal C_k(Z,\uptau_Z) \to \mathcal V_k(Z,\uptau_Z)
			\end{equation*}
			given by $(\lambda,\upgamma) \mapsto \hat{\mathfrak q}_{\lambda}(\upgamma)$ is smooth;
		\item the map $\mathcal P \times \mathcal C_1(Y,\uptau) \to \mathcal T_1(Y,\uptau)$ given by $(\lambda,\beta) \mapsto \mathfrak q_{\lambda}\beta$ is continuous and satisfies bounds
			\begin{align*}
				\|\mathfrak q_{\lambda}(B,\Psi)\|_{L^2} &\le \|\lambda\|m_2(\|\Psi\|_{L^2} + 1,)\\
				\|\mathfrak q_{\lambda}(B,\Psi)\|_{L^2_{1,A_0}} &\le \|\lambda\|_{\mu_1}\left(\|B-B_0,\Psi\|_{L^2_{1,A_0}}\right).
			\end{align*}
	\end{itemize}
\end{thm}

\begin{defn}\label{defn:largebanachspaceperturb}
	A \emph{large Banach space of tame perturbation} is a separable Banach space $\mathcal P$ and a linear map $\mathfrak D:\mathcal P \to C^0(\mathcal C(Y,\uptau),\mathcal T_0)$, satisfying the condition of Theorem~\ref{thm:largebanachspace} and contain a countable collection of tame pertubations $\{\mathfrak q_i\}$ obtained by making choices described above.
\end{defn}

\subsection{Some analytical results about perturbed Seiberg-Witten equations}
\hfill \break
In this subsection, we collect some analytic results for later use.
We define the perturbed analytic energy
\begin{equation*}
	\mathcal E^{\text{an}}_{\mathfrak q}(A,\Phi)
	=
	\int^{t_2}_{t_1} \left\| \frac{d}{dt}\check A - dc\right\|^tdt
	+
	\int^{t_2}_{t_1} \left\| \frac{d}{dt}\Phi + c\check\Phi \right\|^tdt
	+
	\int^{t_2}_{t_1} \left\| \grad \pertL\right\|^2dt,
\end{equation*}
which agrees with twice the drop in $\pertL$
\begin{equation*}
	\mathcal E^{\text{an}}_{\mathfrak q}(\upgamma)
	= 2\left(\pertL(\check\upgamma(t_1)) - \pertL(\check\upgamma(t_2))\right)
\end{equation*}
for a solution to the perturbed Seiberg-Witten equation.

\begin{lem}
	There is a continuous function $\zeta$ on $C^{\sigma}_k(Y,\uptau)$ such that for any solution $\upgamma^{\tau}$ to the solution $\mathcal F^{\tau}_{\mathfrak q}(\upgamma^{\tau})$ on a cylinder $Z = [t_1,t_2] \times Y$, we have the inequality:
	\begin{equation*}
		\frac{d}{dt} \Lambda_{\mathfrak q}(\check{\upgamma}^{\tau}(t))
		\le 
		\zeta(\check{\upgamma}(t)) \|\grad \pertL(\check\upgamma(t))\|_{L^2_k(Y)}.
	\end{equation*}
\end{lem}
\begin{proof}
	See \cite[Lemma~10.9.1]{KMbook2007}.
\end{proof}
\begin{thm}[Compactness of perturbed trajectories]
\label{thm:perturbed_local_compactness_blowup}
	Let $\mathfrak q$ be $k$-tame perturbation.
	Let $Z = [t_1,t_2] \times Y$ be a cylinder, and let $Z_{\epsilon} = [t_1+\epsilon, t_2 - \epsilon]$ be a smaller cylinder.
	Let $\upgamma^{\tau}_n \in C^{\tau}_k(Z,\uptau_Z)$ be a sequence of solutions to the perturbed equations $\mathfrak F_{\mathfrak q}^{\tau}(\upgamma^{\tau}) = 0$ on $Z$, and let $\check{\upgamma}^{\tau}$ be the corresponding paths in $\mathcal C^{\sigma}_k(Y,\uptau)$.
	Suppose that the drop of the Chern-Simons-Dirac functional is uniformly bounded on the larger cylinder $Z$,
	\begin{equation*}
		\pertL(\check{\upgamma}^{\tau}_n(t_1))-\pertL(\check{\upgamma}^{\tau}_n(t_2)) \le C_1,
	\end{equation*}
	and suppose that there are one-sided bounds on the value at the endpoints of the smaller cylinder $Z_{\epsilon}$:
	\begin{align*}
		\Lambda_{\mathfrak q}(\check{\upgamma}^{\tau}_n(t_1 + \epsilon)) &\le C_2\\
		\Lambda_{\mathfrak q}(\check{\upgamma}^{\tau}_n(t_2 - \epsilon)) &\ge C_2.
	\end{align*}
	Then there is a sequence of gauge transformations, $u_n \in \mathcal G_{k+1}(Z,\upiota_Z)$ such that, up to subsequence, the transformed solutions $u_n(\check{\upgamma}^{\tau}_n)$ has the following property:
	for every interior domain $Z' \Subset Z_{\epsilon}$, the transformed solutions belong to $\mathcal C^{\tau}_{k+1}(Z',\uptau_{Z'})$ and converge in the topology of $\mathcal C^{\tau}_{k+1}(Z',\uptau_{Z'})$ to a solution $\upgamma^{\tau} \in \mathcal C^{\tau}_{k+1}(Z',\uptau_{Z'})$ of the equation $\mathfrak F_{\mathfrak q}^{\tau}(\upgamma^{\tau})$.
\end{thm}
\begin{proof}
	See \cite[Theorem~10.9.2]{KMbook2007}.
\end{proof}

\section{Hessians and Transversality} 
\label{sec:Transversality}

Recall the bundle decomposition from equation~\eqref{eqn:tangent_decomp}:
\[
	\mathcal T_j^{\sigma} = \mathcal J^{\sigma}_j \oplus \mathcal K^{\sigma}_j,
\]
where the first factor is tangent to the gauge orbits, and the second factor is orthogonal to the gauge action and tangent to the Coulomb slice $S^{\sigma}_j$.
\begin{defn}
	A critical point $\mathfrak s \in \mathcal C^{\sigma}_k(Y,\uptau)$ of the vector field $(\grad \pertL)^{\sigma}$ is nondegenerate if the smooth section $(\grad \pertL)^{\sigma}$ of $\mathcal T^{\sigma}_{k-1}$ is transverse to the subbundle $\mathcal J_{k-1}^{\sigma}$.
\end{defn}
We state the main transversality theorem for perturbed critical points.
\begin{thm}
	Let $\mathcal P$ be a large Banach space of tame perturbations, as in Definition~\ref{defn:largebanachspaceperturb}.
	Then there is a residual (and therefore non-empty) subset of $\mathcal P$ such that for every $\mathfrak q$ in this subset, all the zeros of the subsection $(\grad \pertL)^{\sigma}$ of $\mathcal T_{k-1}^{\sigma} \to \mathcal C^{\sigma}_k(Y)$ are nondegenerate. 
	For such a perturbation, the image of zeros in $\mathcal C_k(Y)$ comprises a finite set of gauge orbits.
\end{thm}
The rest of the section will be devoted to the proof the theorem, and along the way, we define Hessians and give an alternative criterion of nondegeneracy in terms of surjectivity of the Hessian operator.
\subsection{An intermediate characterization}
\hfill \break
Recall $\mathcal A_k$ is the affine space of invariant Spin\textsuperscript{c} connections, acted on by the invariant gauge group $\mathcal G_{k+1}$.
We write
\begin{equation*}
	\mathcal T_j^{\red} = \mathcal A_k \times L^2_j(Y;i\T^* Y)^{\uptau},
\end{equation*}
which decomposes along tangent space of gauge orbit and its complement as
\begin{equation*}
	 \mathcal T_j^{\red} =\mathcal T^{\red}_j \oplus \mathcal K_j^{\red}
\end{equation*}
and where the fibres consist of exact and coclosed imaginary-valued $L^2_j$ 1-forms, respectively.
Let $(\grad \pertL)^{\red}$ be the restriction of the gradient to $\mathcal A_k \times \{0\} \subset \mathcal C_k(Y)$, as section
\begin{equation*}
	(\grad \pertL)^{\sigma} : \mathcal A_k \to \mathcal T^{\red}_{k-1}.
\end{equation*}
Given $B \in \mathcal A_k$, we define the perturbed Dirac operator
\begin{align*}
	D_{\mathfrak q, B}: L^2_k(Y;S)^{\uptau} 
	&\to L^2_{k-1}(Y;S)^{\uptau},\\
	\phi &\mapsto
	D_B\phi + \mathcal D_{(B,0)} \mathfrak q^1(0,\phi).
\end{align*}
The perturbed Dirac operator acts on the invariant sections as the perturbations are constructed $\uptau$-invariantly.
By the anti-linearity of $\uptau$, the Hilbert space $L^2_k(Y;S)^{\uptau}$ does not have a natural complex structure, and the Dirac operator is only a real operator.
We are ready to state the characterization.
\begin{prop}
\label{prop:interchar}
	Let $\mathfrak a = (B, r, \psi)$ be a critical point of the vector field $(\grad \pertL)^{\sigma}$. Then the non-degeneracy of $\mathfrak a$ an be characterized by one of the following conditions, according to whether $r$ is zero or non-zero.
	\begin{enumerate}[leftmargin=*]
		\item[(i)] If $r \ne 0$, then $\mathfrak a$ is nondegenerate if and only if the corresponding point $(B,r\psi) \in \mathcal C^*_k(Y)$ is a non-degenerate zero of $\grad \pertL$, in the sense that $\grad \pertL$ is transverse to the subbundle $\mathcal J_{k-1}$ of $\mathcal T_{k-1}$.
		\item[(ii)] If $r=0$, then $\psi$ is an eigenvector of $D_{\mathfrak q,B}$, with eigenvalue $\lambda$, say, and $\mathfrak a$ is non-degenerate if and only if the following three conditions hold. 
		\begin{enumerate}
			\item[(a)] $B \in \mathcal A_k$ is a non-degenerate zero of $(\grad \pertL)^{\red}$, in the sense that $(\grad \pertL)^{\red}$ is transverse to the subbundle $\mathcal J_{k-1}^{\red}$ of $\mathcal T^{\red}_{k-1}$;
			\item[(b)] $\lambda$ is a simple eigenvalue of the real operator $D_{\mathfrak q,B}$;
			\item[(c)] $\lambda$ is not zero.
		\end{enumerate}
	\end{enumerate}
\end{prop}
\begin{proof}
	We sketch the proof, based on \cite[Proposition~12.2.5]{KMbook2007}. 
	The $r \ne 0$ case follows from the blow-down map being diffeomorphism on the irreducible set.
	Suppose $r=0$.
	The derivative at $(B,0,\psi) \in \mathcal C^{\sigma}_k$  has the shape
	\begin{equation*}
		\mathcal D_{(B,0,\psi)} :
		\begin{bmatrix}
			b \\ t \\ \phi 
		\end{bmatrix}
		\mapsto
		\begin{bmatrix}
			\mathcal D_B(\grad \pertL)^{\red} & 0 & 0\\
			0 & \lambda & 0\\
			x & 0 & D_{\mathfrak q,B} - \lambda
		\end{bmatrix}
		\begin{bmatrix}
			b \\ t \\ \phi
		\end{bmatrix},
	\end{equation*}
	acting on triples $(b,t,\phi)$ with $\Re \langle \psi, \phi \rangle = 0$.
	The entry $x$ vanishes when $\mathfrak a$ is critical point (cf. Proposition~\ref{prop:Hess_op}.)
	If we write $\mathbf d^{\sigma}$ for the derivative of the gauge group action on $\mathcal C^{\sigma}(Y)$, then the non-degeneracy at $\mathfrak a = (B,0,\phi)$ is equivalent to the surjectivity of $\mathbf{d}_{\alpha}^{\sigma} \oplus \mathcal D_{(B,0,\psi)}$ which is the matrix
	\begin{equation*}
		\begin{bmatrix}
			-d & \mathcal D_B(\grad \pertL)^{\red} & 0 & 0\\
			0 & 0 & \lambda & 0\\
			\psi \cdot & x & 0 & D_{\mathfrak q,B} - \lambda
		\end{bmatrix}.
	\end{equation*}
	\begin{itemize}[leftmargin=*]
		\item Condition~(a) is equivalent to the surjectivity of the first row $(-d,\mathcal D_B(\grad \pertL)^{\red})$.
		\item Condition~(c) is equivalent to the surjectivity of the second row, multiplication by $\lambda$.
		\item Condition~(b) means that the operator $D_{\mathfrak q,B} -\lambda$ has cokernel the real span of $i\psi$, but $(0,0,i\psi)$ is the image of $(i,0,0,0)$.\qedhere 
	\end{itemize}
\end{proof}

\subsection{Almost self-adjoint first-order elliptic operators}
\hfill \break
We begin with a definition modelled on the linearization of $(\grad \pertL)^{\sigma}$, which roughly is a sum of self-adjoint elliptic operator and a compact non-symmetric term.
The setup is the following.
Let $E \to Y$ be a vector bundle that decompose as a direct sum of real and complex vector bundles.
Let $\upiota: Y \to Y$ be an involution and $\uptau: E \to E$ be a conjugate-linear involution lifting $\upiota$, in the sense that it acts conjugate-linearly on the complex summand and linearly on the real summand.
Then $\uptau$ acts on sections of $E$ by
\begin{equation*}
	(\uptau s)(y) = \uptau(s(\upiota(y))).
\end{equation*}
We use superscript $\uptau$ to denote $\uptau$-invariant subspaces.
\begin{defn}
\label{defn:kASAFOE}
	An operator $L$
	is \emph{$k$-almost self-adjoint first order elliptic} ($k$-\textsc{asafoe}) if it can be written as
	\begin{equation*}
		L = L_0 + h,
	\end{equation*}
	where
	\begin{itemize}
		\item $L_0$ is a first order, self-adjoint, elliptic differential operator with smooth coefficient, acting on \emph{$\uptau$-invariant sections} of a vector bundle $E \to Y$, and
		\item $h$ is an operator on \emph{$\uptau$-invariant sections} of $E$ which we suppose to be a map
		\begin{equation*}
			h: C^{\infty}(Y;E)^{\uptau} \to L^2(Y;E)^{\uptau}
		\end{equation*}
		which extends to a bounded map on $L^2_j(Y;E)^{\uptau}$ for all $j$ in the range $|j| \le k$.
	\end{itemize}
\end{defn}
\begin{exmp}
An example of a $k$-\textsc{asafoe} operator is the perturbed Dirac operator:
\begin{equation*}
	D_{\mathfrak q, B}:L^2_k(Y;S)^{\uptau} \to L^2_{k-1}(Y;S)^{\uptau}
\end{equation*}
where $B= B_0 + b \otimes 1$ and
	\begin{itemize}
	\item $L_0 = D_{B_0}$,
	\item $h\phi = \rho(b)\phi + \mathcal D_{(B,0)}\mathfrak q^1(0,\phi)$.
\end{itemize}
The verification for $D_{\mathfrak q, B}$ being $k$-\textsc{asafoe} is uses Sobolev multiplication and the tameness of perturbations.
In fact, $h$ is also symmetric because it arises from a formal gradient.
\end{exmp}

The following lemma is concerned with the spectrum of $L$.
\begin{lem}
\label{lem:spectra_kASAFOE}
	Let $L = L_0 + h$ be $k$-\textsc{asafoe}. 
	\begin{enumerate}[leftmargin=*]
		\item[(i)] There are finitely many eigenvalues of the complexfication $L \otimes 1_{\C}$ in any compact subset of the complex plane $\C$, and the generalized eigenspaces of the complexification are finite-dimensional. All the generalized eigenvectors belong to $L^{2,\uptau}_{k+1}$.
		\item[(ii)] If $h$ is symmetric too, then the eigenvalues are real, and there is a complete orthonormal system of eigenvectors $\{e_n\}$ in $L^2(E)^{\uptau}$. Then span of eigenvectors is dense in $L^2_{k+1}$.
		\item[(iii)] In the non-symmetric case, the imaginary parts of the eigenvalues $\lambda$ of $L \otimes 1_{\C}$ are bounded by the $L^2$-operator norm of $(h-h^*)/2$.
	\end{enumerate}
\end{lem}

\subsection{Hessians}
\label{subsec:Hessians}
\subsubsection*{(Hessians on non-blown-up configuration spaces)}
The tangent bundle $\mathcal T_j = L^2_j(Y;\uptau)$ is the product bundle
\begin{equation*}
	\mathcal T_j = L^2_j(Y;i\T^*M \oplus S)^{\uptau} \times \mathcal C_k(Y;\uptau).
\end{equation*}
Consider the composition on the irreducible part $\mathcal C_k^*(Y)$ in \eqref{eqn:tangent_decomp_irred}:
\begin{equation*}
	\mathcal T_j|_{\mathcal C^*_k} = \mathcal T_j \oplus \mathcal K_j,
\end{equation*}
and let $\Pi_{\mathcal K_{k-1}}: \mathcal T_{k-1} \to \mathcal K_{k-1}$ be the $L^2$-orthogonal projection.
Since the decomposition was defined using standard $L^2$-inner product, $\grad \pertL$ is a section of $\mathcal K_{k-1} \to \mathcal T_{k-1}$.
\begin{defn}
	Let $\alpha \in \mathcal C^*(Y)$.
	We define the Hessian operator
	\begin{equation*}
		\Hess_{\mathfrak q,\alpha}:\mathcal K_{k,\alpha} \to \mathcal K_{k-1,\alpha},
	\end{equation*}
	as the composition of two linear maps:
	\begin{equation*}
		\Pi_{\mathcal K_{k-1}} \circ \mathcal D_{\alpha} \grad \pertL: 
		T_{\alpha}\mathcal C^*_{k}(Y) = \mathcal T_{k,\alpha} \to \mathcal K_{k-1,\alpha}.
	\end{equation*}
	All in all, the Hessian defines a $\mathcal G_{k+1}(Y,\upiota)$-equivariant smooth bundle map
	\begin{equation*}
		\Hess_{\mathfrak q}:\mathcal K_k \to \mathcal K_{k-1}.
	\end{equation*}
\end{defn}
\begin{prop}
\label{prop:Hess_op}
	The operator $\Hess_{\mathfrak q,\alpha}: \mathcal K_{k,\alpha} \to \mathcal K_{k-1,\alpha}$ is symmetric.
	There is a complete orthonormal system $\{e_n\}$ in $\mathcal K_{0,\alpha}$, with the property that each $e_n$ is smooth, and
	\begin{equation*}
		\Hess_{\mathfrak q,\alpha}e_n = \lambda_n e_n,
	\end{equation*}	
	for some $\lambda_n \in \reals$. The span of the eigenvectors is dense in $\mathcal K_{k,\alpha}$ for all $k$. The number of eigenvalues $\lambda_n$ in any bounded interval is finite.
	In particular, $\Hess_{\mathfrak q,\alpha}$ is Fredholm of index zero, and therefore surjective if and only if it is injective.
\end{prop}
\begin{proof}
	This is essentially \cite[Proposition~12.3.1]{KMbook2007}, but we repeat the proof to introduce some notations.
	The symmetry of the Hessian follows from the observation that $\Hess_{\mathfrak q, \alpha}$ is the pullback of a operator that is the covariant Hessian of the circle valued function $\pertL$ on $\mathcal B^*_k(Y)$, with respect to the $L^2$-inner product.
	To prove the rest of the analytic properties, we introduce the extended Hessian operator:
	\begin{equation*}
		\widehat{\Hess}_{\mathfrak q,\alpha}:
		\mathcal T_{k,\alpha} \oplus L^2_k(Y,i\reals)^{-\upiota^*} \to 
		\mathcal T_{k-1,\alpha} \oplus L^2_{k-1}(Y,i\reals)^{-\upiota^*}
	\end{equation*}
	given by
	\begin{equation*}
		\widehat{\Hess}_{\mathfrak q,\alpha} =
		\begin{bmatrix}
			\mathcal D_{\alpha} \grad \pertL & \mathbf d_{\alpha}\\
			\mathbf{d}_{\alpha}^* & 0
		\end{bmatrix}.
	\end{equation*}
	
	In the decomposition
	\begin{equation*}
		\mathcal T_{j,\alpha} = L^2_j(Y;S)^{\uptau} \oplus
			L^2_j(Y;i\T^*Y)^{-\upiota^*},
	\end{equation*}
	the extended Hessian operator takes the form
	\begin{equation*}
		\widehat{\Hess}_{\mathfrak q,\alpha} 
	= 
	\begin{bmatrix}
		D_{A_0} & 0 & 0\\
		0 &*d & -d\\
		0 & d^* & 0 
	\end{bmatrix}+h,
	\end{equation*}
	where $h$ is the sum of a zeroth-order operator and terms arising from the perturbation.
	The extended Hessian is k-\textsc{asafoe}: indeed, the first term is a self-adjoint elliptic operator on $Y$ acting on $\uptau$-invariant sections of $S \oplus i\T^*Y \oplus i\underline{\reals}$, and the second term satisfies the second bullet point of Definition~\ref{defn:kASAFOE}.
	
	In the decomposition
	\begin{equation*}
		\mathcal T_j = \mathcal J_j \oplus \mathcal K_j,
	\end{equation*}
	the Hessian has the shape
	\begin{equation*}
		\widehat{\Hess}_{\mathfrak q, \alpha} =
		\begin{bmatrix}
			0 & x & \mathbf{d}_{\alpha} \\
			x^* & \Hess_{\mathfrak q,\alpha} &0\\
			\mathbf{d}^*_{\alpha} & 0 & 0	
		\end{bmatrix},
	\end{equation*}
	where recall $\mathbf d_{\alpha}$ is the the linearization of the gauge group action and
	\begin{align*}
		x 
		&= \Pi_{\mathcal J_{k-1}} \circ \mathcal D_{\alpha} \left(\grad \pertL |_{\mathcal K_{k,\alpha}}\right),\\
		&= \mathbf d_{\alpha}(\mathbf d_{\alpha}^*\mathbf d_{\alpha})^{-1} \mathbf d_{\alpha}^* \mathcal D_{\alpha} \grad \pertL|_{\mathcal K_{k,\alpha}}.
	\end{align*}
	If $x$ vanishes at critical points, then we can apply Lemma~\ref{lem:spectra_kASAFOE} to conclude the proof.
	Indeed, $x$ is the derivative of a section of a subbundle that takes value in the subbundle, so it vanishes at points where the section vanishes.
\end{proof}

\subsubsection*{(Hessians on the blown-up configuration space)}
Recall that the fibre of the $L^2_j$-tangent bundle of the blown-up configuration space is defined as
\begin{equation*}
	\{(a,s,\phi) \big| \Re\langle \phi_0,\phi\rangle_{L^2} = 0\} \subset L^2_j(Y;i\T^*Y)^{\uptau} \times \reals \times L^2_j(Y;S)^{\uptau}
\end{equation*}
and we have a decomposition
\begin{equation*}
	\mathcal T^{\sigma}_j = \mathcal J^{\sigma}_j \oplus \mathcal K^{\sigma}_{j}
\end{equation*}
that descends to $\mathcal J_j \oplus \mathcal K_j$ over the irreducible subset $\mathcal C^*_k$.
Since $\grad \pertL$ is a section of the subbundle $\mathcal K_{k-1}$ of $\mathcal T_{k-1}$ over $\mathcal C_k$, by continuity the vector field $(\grad \pertL)^{\sigma}$ defines a section of $\mathcal K^{\sigma}_{k-1}$.
To differentiate $(\grad \pertL)^{\sigma}$, we consider $\mathcal C^{\sigma}_k$ as a subspace of the affine space
\begin{equation*}
	\mathcal A_k \times \reals \times L^2_k(Y;S)^{\uptau}.
\end{equation*}
We differentiate $(\grad \pertL)^{\sigma}$ along the submanifold and project it back to $\mathcal T^{\sigma}_{k-1}$, and then to the summand $\mathcal K^{\sigma}_{k-1}$ using the projection $\mathcal J^{\sigma}_{k-1} \to \mathcal K^{\sigma}_{k-1}$.
The resulting map $\mathcal T^{\sigma}_k \to \mathcal K^{\sigma}_{k-1}$ is given by
\begin{equation*}
	x \mapsto \Pi_{\mathcal K^{\sigma}_{k-1}}\mathcal D(\grad \pertL)^{\sigma}(x).
\end{equation*}
Finally, we define the Hessian as the restriction of this operator on $\mathcal K^{\sigma}_k \subset \mathcal T^{\sigma}_k$:
\begin{equation*}
	\Hess_{\mathfrak q}^{\sigma} : \mathcal K^{\sigma}_k \to \mathcal K^{\sigma}_{k-1}.
\end{equation*}

While $\mathcal K^{\sigma}_k$ is the pull-back of the tangent bundle of the quotient configuration space and $\Hess_{\mathfrak q}^{\sigma}$ can be regarded as a covariant derivative of a vector field on $\mathcal B^{\sigma}_{\mathfrak q}$, 
it is not a Levi-Civit\`{a} derivative because $\mathcal J^{\sigma}_j \oplus \mathcal K^{\sigma}_{j}$ is not orthogonal. 
Hence $\Hess_{\mathfrak q}^{\sigma}$ is no longer a symmetric operator.
In terms of the decomposition
\begin{equation*}
	\mathcal T_{j,\mathfrak a}^{\sigma} =
	\mathcal J_{j,\mathfrak a}^{\sigma} \oplus
	\mathcal K_{j,\mathfrak a}^{\sigma},
\end{equation*}
the derivative of the vector field has the shape
\begin{equation*}
	\mathcal D_{\mathfrak a}(\grad \pertL)^{\sigma} =
	\begin{bmatrix}
		0 & x\\
		y & \Hess^{\sigma}_{\mathfrak q,\mathfrak a}
	\end{bmatrix},
\end{equation*}
and $x,y$ both vanish at critical points.
It follows that
\begin{lem}[criterion of non-degeneracy]
	A critical point $\mathfrak a$ in the blown-up configuration space is non-degenerate if and only if $\Hess^{\sigma}_{\mathfrak q,\mathfrak a}$ is surjective.	
\end{lem}
Finally, we introduce the blown-up version of extended Hessian, which will appear in the later section on moduli spaces of trajectories:
\begin{equation}
\label{eqn:extendedHessianTL}
	\widehat{\Hess}_{\mathfrak q, \mathfrak a}^{\sigma}: \mathcal T^{\sigma}_{k,\alpha} \oplus L^2_k(Y;i\reals)^{-\upiota^*}
	\to 
	\mathcal T^{\sigma}_{k-1,\alpha} \oplus L^2_{k-1}(Y; i\reals)^{-\upiota^*}.
\end{equation}
To this end, we must first introduce the operator
\begin{equation*}
	\boldsymbol{d}^{\sigma,\dag}_{\mathfrak a}: \mathcal T^{\sigma}_{k,\alpha} \to L^2_{k-1}(Y;i\reals)^{-\upiota^*}
\end{equation*}
given by $\mathfrak a = (B_0, s_0, \psi_0)$ by
\begin{equation*}
	(b,s,\psi) \mapsto -d^*b + is_0^2 \Re \langle i\psi_0, \psi \rangle + i|\psi_0|^2 \Re (\mu_Y \langle i\psi_0,\psi \rangle).
\end{equation*}
whose kernel is $\mathcal K^{\sigma}_{\mathfrak a}$.
In the decomposition of \eqref{eqn:extendedHessianTL}, the Hessian has the matrix form
\begin{equation*}
	\boldsymbol{d}^{\sigma,\dag}_{\mathfrak a} =
	\begin{bmatrix}
		\mathcal D_{\mathfrak a}(\grad \pertL)^{\sigma} & \boldsymbol{d}^{\sigma}_{\mathfrak a}\\
		\boldsymbol{d}^{\sigma,\dag}_{\mathfrak a}& 0
	\end{bmatrix}.
\end{equation*}
Decompose the tangent bundle $\mathcal T^{\sigma}_{j,\alpha} = \mathcal J^{\sigma}_{j,\alpha} \oplus \mathcal K^{\sigma}_{j,\alpha} $, we have

\begin{equation*}
	\widehat{\Hess}_{\mathfrak q, \alpha}^{\sigma} =
	\begin{bmatrix}
		0 & x & \mathbf{d}^{\sigma}_{\alpha} \\
		y & \Hess^{\sigma}_{\mathfrak q,\alpha} &0\\
		\mathbf{d}^{\sigma,\dag}_{\alpha} & 0 & 0	
	\end{bmatrix},
\end{equation*}
where $\Hess^{\sigma}_{\mathfrak q,\mathfrak a}$ is the Hessian on $\mathcal K^{\sigma}_{j,\mathfrak a}$.
In particular, the entries $x,y$ both vanish at critical points.
The extended Hessian is not symmetric operator unlike its blown-down counterpart; however, it still has real spectrum as the following lemma suggests.
\begin{lem}
	\label{lem:ext_Hess_real_eigen}
	If $\mathfrak b$ is non-degenerate critical point in the blown-up configuration space, then the extended Hessian $\widehat{\Hess}_{\mathfrak q, \mathfrak b}^{\sigma}$ is invertible and has real spectrum.
	 In particular, it is hyperbolic.	
\end{lem}
\begin{proof}
	The proof of \cite[Lemma~12.4.2]{KMbook2007} applies. 
	We give the main idea of this technical lemma, as a similar trick will appear in the 4-dimensional operator over cylinders.
	The conclusion will eventually follow from Lemma~\ref{lem:spectra_kASAFOE}, but the technical complication arises from the fact that $\mathcal T^{\sigma}_{j,\sigma}$ is not a trivial vector bundle.
	Let $(b,\psi, r)$ be the fibre of $\mathcal T^{\sigma}_{j}$ over a configuration $\mathfrak a = (B_0, r_0, \psi_0)$.
	While neither $\psi$ (as it is required to be orthogonal to $\psi_0$) nor $r$ is a unconstrained section of a vector bundle, the combination
	\begin{equation*}
		\boldsymbol{\psi} = \psi + r\psi_0
	\end{equation*}
	is an unconstrained $\uptau$-invariant section of the spinor bundle $S$.
	Hence we rewrite the formulae (see \cite[Section~12.4]{KMbook2007}) in the new coordinate
	\begin{equation*}
		(b,\mathbf \psi, c) \in l^2_j(Y;i\T^*Y \oplus S \oplus i\reals)^{\uptau}.
	\end{equation*}
	At the end, the extended Hessian takes the form
	\begin{equation*}
		\widehat{\Hess}_{\mathfrak q, \mathfrak a}^{\sigma} (b,\mathbf{\psi},c) = L_0(b,\mathbf{\psi},c) + h_{\mathfrak a}(b, \mathbf{\psi}, c),
	\end{equation*}
	where $L_0$
	is elliptic and self-adjoint, and $h_{\mathfrak a}$ satisfies the definition of $k$-\textsc{asafoe}.
	
	In the irreducible case, the real spectrum follows from the extended Hessian being conjugate to the Hessian in $\mathcal K_{j,\mathfrak a}$ via the projection map.
	In the reducible case, closer examination of the formula shows that the operator has a lower triangular block whose diagonal blocks are symmetric.
	As a side note, at some point in the argument \cite{KMbook2007} use $S^1$-invariance to deduce that some vector is orthogonal to $i\psi_0$; while we do not have $S^1$-invariance, the vector $i\psi_0$ is not an invariant spinor if $\psi_0$ is $\uptau$-invariant.
\end{proof}

\subsection{Proof of transversality}
\hfill \break
In this subsection we prove the transversality for irreducible and reducible critical points. 
Both cases will be based on the following Lemma \cite[Lemma~12.5.1]{KMbook2007}.
\begin{lem}
\label{lem:fredMorseSmale}
	Let $\mathcal E$ and $\mathcal F$ be separable Banach manifolds, and let $\mathcal S \subset \mathcal F$ be a closed submanifold.
	Let 
	\begin{equation*}
		F:\mathcal E \times \mathcal P \to \mathcal F
	\end{equation*}	
	be a smooth map, and write $F_p = F(-,p)$.
	Suppose that $F$ is transverse to $\mathcal S$, and that for all $(e,p)$ in $F^{-1}_p(\mathcal S)$, the composite 
	\begin{equation*}
		\T_e\mathcal E \xrightarrow{\mathcal D_e F_p} \T_f \mathcal F \xrightarrow{\pi} \T_f\mathcal F/\T_f \mathcal S
	\end{equation*}
	is Fredholm.
	Then there is a residual set of $p$ in $\mathcal P$ for which the map $F_p:\mathcal E \to \mathcal F$ is transverse to $\mathcal S$. 
	\hfill \qedsymbol
\end{lem}
\subsubsection*{(The irreducible case)}
Before applying the lemma, note that by first part of Proposition~\ref{prop:interchar}, it suffices to consider the vector field $\grad \pertL$ on $\mathcal C^*_k$.
Set $\mathcal P$ to be large Banach space of perturbations.
Let the map
\begin{equation*}
	\mathfrak g: \mathcal C^*_k \times \mathcal P 
	\to \mathcal K_{k-1}
\end{equation*}
be defined by
\begin{align*}
	\mathfrak g(\alpha,\mathcal P) 
	&= \grad \pertL (\alpha)\\
	&= \grad \mathcal L + \mathfrak q(\alpha).
\end{align*}
This is a smooth map between Banach manifolds, which plays the role of $F$ in Lemma~\ref{lem:fredMorseSmale}.
The zero set of $\mathfrak g$ are critical points parametrized by perturbations, we have the following lemma.
\begin{lem}
\label{lem:g_transverse}
	The map $\mathfrak g$ is transverse to zero section of $\mathcal K_{k-1}$.
\end{lem}
\begin{proof}
	The proof is identical to the proof of \cite[Lemma~12.5.2]{KMbook2007}.
The main idea is to show surjectivity of the map
	\begin{align*}
		\mathcal K_{k,\alpha} \times \T_{\mathfrak q}\mathcal P 
		&\to \mathcal K_{k-1,\alpha},\\
		((b,\psi),\delta\mathfrak q)
		&\mapsto \Hess_{\mathfrak q,\alpha}(b,\psi) + \delta \mathfrak q(\alpha).
	\end{align*}
This can be reduced to the question whether any non-zero $v$ in the kernel of $\Hess_{\mathfrak q,\alpha}$ there exists a $\delta \mathfrak q \in \mathcal P$ such that the $L^2$-inner product of $\delta\mathfrak q(\alpha)$ and $v$ is non-zero.
But the existence of such $\delta \mathfrak q$ is guaranteed by Corollary~\ref{cor:separatetangent}, together with the density condition of Definition~\ref{defn:largebanachspaceperturb}.
\end{proof}
From Lemma~\ref{lem:g_transverse} and implicit function theorem, we learn that the quotient
\begin{equation*}
	\mathcal Z := \mathfrak g^{-1}(0)/\mathcal G_{k+1}
	\subset
	\mathcal B^*_k(Y) \times \mathcal P
\end{equation*}
is a Banach manifold.
To conclude the proof, we apply Lemma~\ref{lem:fredMorseSmale} to
\def\arraystretch{1.5}%
\begin{center}
\begin{tabular}{ |c|c|c|c|c|} 
 \hline
 $F$ & $\mathcal E$ & $\mathcal P$ & $\mathcal F$   & $\mathcal S$ 
 \\  \hline
 $\mathfrak g$ & $\mathcal B^*_k$  & $\mathcal P$ & $\mathcal K_{k-1}$ &  zero-section of $\mathcal K_{k-1}$
 \\ \hline
\end{tabular}
\end{center}
where we observe that transversality of $\mathfrak g_p$ is equivalent to transversality of the critical point.
\subsubsection*{(The reducible case)}
Arguments in the irreducible case show that the map
\begin{equation*}
	\mathfrak g^{\red}: \mathcal A_k \times \mathcal P \to \mathcal K^{\red}_{k-1}
\end{equation*}
given by
\begin{equation*}
	\mathfrak g^{\red}(B,\mathfrak q) = (\grad \mathcal L^{\red})(B) + \mathfrak q(B)
\end{equation*}
is transverse and the quotient
\begin{equation*}
	\mathcal Z^{\red} = (\mathfrak g^{\red})^{-1}(0)/\mathcal G_{k+1}
	\subset (\mathcal A_k/\mathcal G_{k+1}) \times \mathcal P
\end{equation*}
is a smooth Banach manifold.
Moreover, by the earlier arguments we can
achieves Condition~(ii)(a) of Proposition~\ref{prop:interchar}.
To proceed, we must alter the Hessian of the perturbed gradient in the directions normal to the reducibles.

Denote as
\begin{equation*}
	\mathcal P^{\perp} \subset \mathcal P
\end{equation*}
the perturbations $\mathfrak q$ that vanish at the reducible locus in $\mathcal C_k(Y)$.
Let $\text{Op}^{\text{sa}}$ be the the space of self-adjoint Fredholm maps $L^2_k(Y;S)^{\uptau} \to L^2_{k-1}(Y;S)^{\uptau}$ having the form $D_{B_0}+h$ where $B_0$ is a smooth spin\textsuperscript{c} connection and $h$ is a self-adjoint operator that extends to a bounded operator $h:L^{2,\uptau}_j \to L^{2,\uptau}_j$.
This Banach space is stratified according to the dimension of the kernel, and the set of operators whose spectrum is not simple is a countable union of Fredholm maps $F_n$ of negative index. 
For definition of $F_n$, see the discussion in \cite[Section~12.6]{KMbook2007}.

Once $\mathcal P^{\perp}$ and  $\text{Op}^{\text{sa}}$ are defined, we consider the map 
\begin{align*}
	M: \mathcal A_k \times \mathcal P^{\perp}
	&\to \text{Op}^{\text{sa}}\\
	(B,\mathfrak q^{\perp}) &\mapsto
	D_{\mathfrak q^{\perp},B}.
\end{align*}
\begin{lem}
	The map $M$ is transverse to the stratification of self-adjoint operators according to the dimension of the kernel, and also to the Fredholm maps $F_n$.	
\end{lem}
\begin{proof}
	The proof is the same as \cite[Lemma~12.6.2]{KMbook2007}, upon replacing usual configuration space by $\uptau$-configuration space, and replacing $\mathbb R^n \times \mathbb T \times \mathbb C^m$ by $\mathbb R^n \times \mathbb T \times \mathbb R^m$. And instead of \cite[Proposition~11.2.1]{KMbook2007}, we use Proposition~\ref{thm:emb}.
\end{proof}
We define the ``bad set''
\begin{equation*}
	\mathcal W \subset \mathcal P^{\perp} \times (\mathcal A_k/\mathcal G_{k+1})
\end{equation*}
to be the set of pairs $(\mathfrak q^{\perp},B)$, where the spectrum of $D_{\mathfrak q^{\perp},B}$ either is non-simple of contains zero.
Then $\mathcal W$ is a countable union of Banach submanifolds $\mathcal W_n$, each of which has finite positive codimension.
The arguments in \cite[Lemma~12.6.2]{KMbook2007} imply the submanifolds $\mathcal P^{\perp} \times \mathcal Z^{\red}$ and $\mathcal W_k \times \mathcal P$ of $\mathcal P^{\perp} \times (\mathcal A_k/\mathcal G_{k+1}) \times \mathcal P$ intersect transversely.
Therefore the intersection
\begin{equation*}
	(\mathcal P^{\perp} \times \mathcal Z^{\red}) \cap (\mathcal W_k \times \mathcal P)
\end{equation*}
is a locally finite union of Banach submanifolds of $\mathcal P^{\perp} \times \mathcal Z^{\red}$, from each of which the projection to $\mathcal P^{\perp} \times \mathcal P$ is Fredholm of negative index.
The Sard-Smale theorem guarantees a residual set of pairs $(\mathfrak q^{\perp},\mathfrak q)$ which are regular values of the projection $\mathcal P^{\perp} \times \mathcal Z^{\red} \to \mathcal P^{\perp} \times \mathcal P$ are not in the image of the projection
\begin{equation*}
	(\mathcal P^{\perp} \times \mathcal Z^{\red}) \cap (\mathcal W_k \times \mathcal P) \to \mathcal P^{\perp} \times \mathcal P.
\end{equation*} 
Such a pair $(\mathfrak q^{\perp},\mathfrak q)$ satisfies (ii)(a)-(ii)(b) of Proposition~\ref{prop:interchar}.
The subset $\mathcal P$ arising as $\mathfrak q^{\perp} + \mathfrak q$ is then also residual.

\section{Moduli Space of Trajectories}
\label{sec:modspace_trajectories}
We choose the perturbation $\mathfrak q \in \mathcal P$  once and for all, such that all critical points of $(\grad \pertL)^{\sigma}$ are non-degenerate.
We also fix a real spin\textsuperscript{c} structure $(\mathfrak s,\uptau)$.
\subsection{Two notions of configurations on infinite cylinders}
\hfill \break
In this section, we turn to the infinite cylinder
$Z = \reals \times Y$.
There are two definitions of configurations on $Z$: one modelled on $L^2_{k,\text{loc}}$ and the other on $L^2_k$.
The discussion follows \cite[Section~13]{KMbook2007} closely and we list the necessary definitions for later use.
\subsubsection*{The $L^2_{k,\text{loc}}$-definition}
Let $I \subset \reals$ be an interval, possibly equal to $\reals$, and choose a $\uptau$-invariant spin\textsuperscript{c} connection $A_0$.
The locally $L^2_k$ configuration in $\tau$-model
\begin{equation*}
	\tilde{\mathcal C}^{\tau}_{k,\text{loc}}(I\times Y)
	\subset \mathcal A_{k,\loc}(I \times Y) \times L^2_{k,\loc}(I,\reals) \times L^2_{k,\loc}(I \times Y;S^+)^{\uptau},
\end{equation*}
consists of $L^2_{k,\loc}$ triples $(A,s,\phi)$ where $A$ is a $L^2_{k,\loc}$-connection (i.e $A - A_0$ is $L^2_{k,\loc}$) and $\phi(t)$ has $L^2$-norm 1 on $\{t\}\times Y$ for any $t$.
Define the closed subspace $\mathcal C^{\tau}_{k,\loc} \subset \tilde{\mathcal C}^{\tau}_{k,\loc} $ by the additional constraint $s \ge 0$.
Let $\mathcal G_{k+1,\loc}(I\times Y; \upiota)$ be the gauge group of $L^2_{k+1,\loc}$-maps valued in $S^1 \subset \C$.
The two quotient spaces of gauge equivalence classes of configurations will be denoted by
\begin{align*}
	\mathcal B^{\tau}_{k,\loc}(I\times Y) &= 
	\mathcal C^{\tau}_{k,\loc}(I\times Y) /
	\mathcal G^{\tau}_{k,\loc}(I\times Y), \\
	\tilde{\mathcal B}^{\tau}_{k,\loc}(I\times Y) &=
	\tilde{\mathcal C}^{\tau}_{k,\loc}(I\times Y) /
	\mathcal G^{\tau}_{k,\loc}(I\times Y) .
\end{align*}

The $L^2_{k,\loc}$ version of perturbed Seiberg-Witten map is
\begin{equation*}
	\mathfrak F^{\tau}_{\mathfrak q}:\tilde{\mathcal C}^{\tau}_{k,\loc}(I \times Y) \to
	\mathcal V^{\tau}_{k-1,\loc}(I \times Y),
\end{equation*}
where the fibre $V^{\tau}_{j,\loc}$ at $\upgamma = (A_0,s_0,\phi_0)$ is the subspace
\begin{equation*}
	\mathcal V^{\tau}_{j,\loc} \subset
	L^2_{j,\loc}(I \times Y; i \mathfrak{su}(S^+))^{\uptau}
	\oplus L^2_{j,\loc}(I;\reals) \oplus
	L^2_{j,\loc}(I\times Y; S^-)^{\uptau}
\end{equation*}
consisting of triples $(a,s,\phi)$ with $\Re\langle \check \phi_0(t),\check \phi(t)\rangle_{L^2(Y)} = 0$ for all $t$.

Let $\mathfrak a \in \mathcal C^{\sigma}_k(Y)$ be a critical point of $(\grad \pertL)^{\sigma}$, and write $[\mathfrak a]$ for the image in $\mathcal B^{\sigma}_k(Y)$.
There is a translation-invariant element $\upgamma_{\mathfrak a} \in \mathcal C^{\tau}_{k,\loc}(Z)$ such that $\mathfrak F^{\tau}_{\mathfrak q}(\upgamma_{\mathfrak a}) = 0$.
\begin{defn}	
A configuration $[\upgamma] \in \tilde{\mathcal B}^{\tau}_{k,\loc}(Z)$ is \emph{asymptotic} to $[\mathfrak a]$ as $t \to \pm \infty$ if
\begin{equation*}
	[\tau_t^* \upgamma] \to [\upgamma_{\mathfrak a}] \quad
	\text{in }\tilde{\mathcal B}^{\tau}_{k,\loc}(Z),
\end{equation*}
as $t \to \pm \infty$ and $\tau_t: Z \to Z$ is the map $(s,y) \mapsto (s+t,y)$.
In this case, we write
\begin{equation*}
	\lim_{\leftarrow} [\upgamma] = [\mathfrak a],\quad
	\lim_{\rightarrow} [\upgamma] = [\mathfrak a],
\end{equation*}
if $[\upgamma]$ is asymptotic to $[\mathfrak a]$ as $t \to -\infty$ and $t \to +\infty$, respectively.
\end{defn}
\begin{defn}
	Let $M([\mathfrak a],[\mathfrak b])$ be the space of all configurations $[\upgamma]$ in $\mathcal B^{\tau}_{k,\loc}(Z)$ that are asymptotic to $[\mathfrak a]$ as $t \to -\infty$ and to $[\mathfrak b]$ as $t \to \infty$, and solves the perturbed Seiberg-Witten equations: i.e.
	\begin{equation*}
		M([\mathfrak a],[\mathfrak b]) =
		\{[\upgamma] \in \mathcal B^{\tau}_{k,\loc}(Z)|
		\mathfrak F^{\tau}_{\mathfrak q}(\upgamma) = 0,
		\lim_{\leftarrow} [\upgamma] = [\mathfrak a],
		\lim_{\rightarrow} [\upgamma] = [\mathfrak b]\}.
	\end{equation*}
	$M([\mathfrak a],[\mathfrak b])$ is a \emph{moduli space} of trajectories on $Z = \reals \times Y$ and $\check M([\mathfrak a],[\mathfrak b])$ is defined similarly for $\tilde{\mathcal B}^{\tau}_{k,\loc}(I\times Y)$.
\end{defn}
\begin{rem}
	By regularity of the perturbed Seiberg-Witten equations, any element in $M([\mathfrak a],[\mathfrak b])$ has a smooth representative and there are natural homeomorphic bijections between moduli spaces of different regularities.
	Therefore there is no mentioning of regularity in the notation $M([\mathfrak a],[\mathfrak b])$.
\end{rem}
We decompose $M([\mathfrak a],[\mathfrak b])$ as
\begin{equation*}
	M([\mathfrak a],[\mathfrak b]) = 
	\bigcup_z M_z([\mathfrak a],[\mathfrak b])
\end{equation*} 
according to the relative homotopy classes of paths
\begin{equation*}
	z \in \pi_1(\mathcal B^{\sigma}_k(Y),[\mathfrak a],[\mathfrak b]),
\end{equation*}
which is an affine space on $H^1(Y;i\mathbb Z)^{-\upiota^*}$.
\subsubsection*{The $L^2_{k}$-definition}
Given two critical points $[\mathfrak a], [\mathfrak b] \in \mathcal B^{\sigma}_k$, choose smooth representatives $\mathfrak a,\mathfrak b \in \mathcal C^{\sigma}_k(Y)$, and let $\upgamma_{\mathfrak a}, \upgamma_{\mathfrak b}$ be two translation-invariant configurations $Z$.
Choose a smooth configuration $\upgamma_0 \in \mathcal C^{\tau}_{k,\loc}(Z)$ that agrees with $[\mathfrak a],[\mathfrak b]$ near $\pm \infty$, respectively.
The $L^2_k$-configuration space is defined as
\begin{equation*}
	\mathcal C^{\tau}_k(\mathfrak a,\mathfrak b) =
	\{\upgamma \in \mathcal C^{\tau}_{k,\loc}(Z)|
	\upgamma-\upgamma_0 \in L^2_k(Z;i\T^*Z)^{-\upiota^*}\times
	L^2_k(\reals;\reals)\times L^2_{k,A_0}(Z;S^+)^{\uptau}\}.
\end{equation*}
And similarly we define the $\tilde{\mathcal C}^{\tau}_k(\mathfrak a,\mathfrak b)$ as subset of $\tilde{\mathcal C}^{\tau}_{k,\loc}$.
\begin{defn-lem}
	The \emph{gauge group $\mathcal G_{k+1}(Z)$} is the subgroup of $\mathcal G_{k+1,\loc}$ that preserves $\tilde{\mathcal C}^{\tau}_k$, which can also be defined as
\begin{equation*}
	\mathcal G_{k+1}(Z) = \{u: Z \to S^1 | 1 - u \in L^2_{k+1}(Z;\C)\},
\end{equation*}
	which is independent of $[\mathfrak a]$ and $[\mathfrak b]$.
\end{defn-lem}
\begin{proof}
	The key input is an elliptic estimate for a gauge transformation that preserves a configuration. The exact same proof in \cite[Lemma~13.3.1]{KMbook2007} works for invariant trajectories.
\end{proof}
\begin{defn}\label{defn:L2kB}
	Suppose $[\mathfrak a]$ and $[\mathfrak b]$ are two gauge-equivalence classes of critical points in $\mathcal B^{\sigma}_l(Y)$.
	Let $z \in \pi_1(\mathcal B^{\sigma}_k(Y),[\mathfrak a], [\mathfrak b])$ be a relatively homotopy class of paths.
	Choose lifts $\mathfrak a_z$ and $\mathfrak b_z$ such that that a path joining these lifts projects to a path in the class $z$.
	We define the quotient spaces,
	\begin{align*}
		\mathcal B^{\tau}_{k,z}([\mathfrak a],[\mathfrak  b]) 
		&= \mathcal C^{\tau}_{k,z}([\mathfrak a],[\mathfrak  b]) / \mathcal G_{k+1}^{\tau}(Z),\\
		\tilde{\mathcal B}^{\tau}_{k,z}([\mathfrak a],[\mathfrak  b]) 
		&= \tilde{\mathcal C}^{\tau}_{k,z}([\mathfrak a],[\mathfrak  b]) / \mathcal G_{k+1}^{\tau}(Z),
	\end{align*}
	and the union over all homotopy classes:
	\begin{align*}
		\mathcal B^{\tau}_{k}([\mathfrak a],[\mathfrak  b])  &= \bigcup_z \mathcal B^{\tau}_{k,z}([\mathfrak a],[\mathfrak  b]),\\
		\tilde{\mathcal B}^{\tau}_{k}([\mathfrak a],[\mathfrak  b])
		&= \bigcup_z \tilde{\mathcal B}^{\tau}_{k,z}([\mathfrak a],[\mathfrak  b]).
	\end{align*}
	The space $\mathcal B^{\tau}_{k,z}([\mathfrak a],[\mathfrak  b])$ is independent of the choices of lift $\mathfrak a_z$ and $\mathfrak b_z$.
	\hfill \qedsymbol
\end{defn}
\subsubsection*{Equivalence of two notions}
We end the subsection with the following theorem.
\begin{thm}
	Let $\upgamma \in \mathcal C^{\tau}_{k,\loc}$ represent an element $[\upgamma] \in M_z([\mathfrak a],[\mathfrak b])$.
	Let $\mathfrak a = \mathfrak a_z$ and $\mathfrak b = \mathfrak b_z$ be two lifts as in Definition~\ref{defn:L2kB}.
	Then there exists a gauge transformation $u \in \mathcal G_{k+1,\loc}$ such that $u(\upgamma)$ belongs to $\mathcal C^{\tau}_k(\mathfrak a,\mathfrak b)$.
	Moreover, if $u$ and $u'$ are two such gauge transformations, then $u^{-1}u'$ belongs to $\mathcal G_{k+1}(Z)$.
	The resulting bijection is a homeomorphism:
	\begin{equation*}
		M_z([\mathfrak a],[\mathfrak b]) \to
		\{[\upgamma] \in \mathcal B^{\tau}_{k,z}([\mathfrak a],[\mathfrak b]) | \mathfrak F^{\tau}_{\mathfrak q}(\upgamma) = 0\}.
	\end{equation*}	
	A similar statement holds for the large moduli space $\tilde M_z([\mathfrak a],[\mathfrak b])$.
\end{thm}

\begin{proof}
	The proof uses exponential decay of trajectories asymptotic to nondegenerate critical points, which in turn is based on the analytic properties of perturbations.
	The arguments can be found in \cite[Section~13.4-13.6]{KMbook2007}.
\end{proof}
\subsection{Spectral flow}
\hfill \break
The first goal is to set up the Fredholm theory for the space of trajectories via spectral flow, introduced by Atiyah-Patodi-Singer.
To begin, let $L = L_0 + h$ be a $k$-\textsc{asafoe} operator on $E \to Y$.
The operators $L, L_0, h$ act on $\uptau$-invariant sections of $E$.
Consider the translation-invariant operator
\begin{equation*}
	D = \frac{d}{dt} + L_0 + h,
\end{equation*}
over the pullback of $E$ on $Z = \reals \times Y$, where $\upiota$ and $\uptau$ extend to the pull-back sections by acting trivially on the $\reals$-factor.
In particular, we view $D$ as an operator between real Hilbert spaces
\begin{equation*}
	D: L^2_j(Z;E)^{\uptau} \to L^2_{j-1}(Z;E)^{\uptau},
\end{equation*}
whose \emph{spectrum} is understood to be the spectrum of its complexification.
\begin{defn}
	A $k$-\textsc{asafoe} operator $L$ is \emph{hyperbolic} if the spectrum of $L$ is disjoint from the imaginary axis.
\end{defn}
For a hyperbolic $k$-\textsc{asafoe} operator, a Fourier decomposition argument shows that
\begin{prop}
	If $L_0 + h$ is hyperbolic, then for all $j$ in the range $|j| \le k$, then operator
	\begin{equation*}
		D = \frac{d}{dt}+L_0+h:L^2_{j+1}(Z;E)^{\uptau} \to L^2_j(Z;E)^{\uptau}
	\end{equation*}	
	is invertible, and in particular Fredholm.
\end{prop}
\begin{proof}
	For more details, see \cite[Proposition~14.1.2.]{KMbook2007} and proof therein.
\end{proof}
Next, we drop the translation invariant hypothesis, and assume $L$ is time-dependent in the sense that $h = h_t$ is a time-dependent operator on $Y$.
In contrast with the time-independent case, the operator $D$ on the cylinder is not necessarily invertible, and the index is given by the spectral flow which we shall define now.
Notice that the time-dependent case gives rise to a family of operators $L_0 + h_t$ over $Y$.

To define spectral flow in general, suppose we are given a $[0,1]$-family of operators:
\begin{equation*}
	L_0 + h_t, \quad t \in [0,1]
\end{equation*}
such that $L_0$ is first order, self-adjoint elliptic, and $h_t$ is a continuous path in the space of bounded operators in $L^{2,\uptau}$.
Associated to this family is the following subset of $(0,1) \times \C$
\begin{equation*}
	S = \{(t, \lambda) \big| \lambda \in \Spec(L_0 + h_t)\}.
\end{equation*}
Roughly speaking, the spectral flow is the signed count of $S$ with $(0,1) \times i\reals$, i.e. the number of times $L_0 + h_t$ picks up a purely imaginary eigenvalue.
To make this count well-defined, we deform the path so that
\begin{itemize}[leftmargin=*]
	\item $h_t$ is smooth over $(0,1)$, and
	\item every $(t,\lambda) \in S$ is \emph{simple}, i.e. the generalized $\lambda$-eigenspace of $L_0 + h_t$ is 1-dimensional.
\end{itemize}
The two bulletpoints can always be achieved because the ``bad subset'' in the space of bounded operators on $L^2$ is at least co-dimension two (see \cite[Section 14.2]{KMbook2007}).
They ensure $S$ is an oriented smooth $1$-manifold that intersects $(0,1) \times i\reals$ transversely, for which the oriented count of intersection points is well-defined.
\begin{defn-prop}[Spectral flow]
	Let $L_0 + h_{0}$ and $L_0 + h_{1}$ be two hyperbolic operators, connected by a path of operators $L_0 + h_t$.
	The \emph{spectral flow} $\SF(L_0 + h_t)$ is the intersection number of $S$ for the deformed path described above.
	The spectral flow depends only on the endpoints.
\end{defn-prop}
\begin{prop}
\label{prop:spectralflowfredholm}
	Let $L_0$ be a first-order, self-adjoint elliptic operator acting on sections of a vector bundle $E \to Y$, and let $h_t$ be a time-dependent bounded operator on $L^2(Y;E)^{\uptau}$, varying continuously in the operator norm topology and equal to a constant $h_{\pm}$ on the ends.
	Suppose $L_0 + h_{\pm}$ are hyperbolic.
	Then the operator
	\begin{equation*}
		Q = \frac{d}{dt} + L_0 + h_t:
		L^2_1(Z;E)^{\uptau} \to L^2(Z;E)^{\uptau}
	\end{equation*}	
	is Fredholm and has index equal to the spectral flow of the operators $L_0 + h_t$.
\end{prop}
\begin{proof}
Since we have Fredholmness on the finite cylinder and the Fredholmness on the two cylindrical ends, the Fredholmness of $Q$ follows from the  then by gluing the right-inverses, in an $\upiota$-invariant fashion.
The computation of the Fredholm index involves reducing the problem to operators over $\reals$. 
The reduction argument applies equally well to the $\uptau$-invariant sections.
For details, see \cite[Proposition~14.2.1]{KMbook2007}.
\end{proof}

\subsection{Local structure of moduli spaces}
\label{subsec:loc_str}
\subsubsection*{\textbf{(Slices for gauge group action)}}
The first goal is to define the slice $\mathcal S^{\tau}_{k,\upgamma}$ of the gauge group $\mathcal G_{k+1}(Z;\upiota)$, 
acting on the moduli space $\tilde C^{\tau}_k(\mathfrak a, \mathfrak b)$ over the infinite cylinder $Z$.
Like the finite cylinder setting, the tangent bundle $\mathcal T^{\tau}_j \to \tilde{\mathcal C}$ is defined by setting the fibre over $\upgamma = (A_0, s_0, \phi_0)$ to be
\begin{align*}
	\mathcal T^{\tau}_{j,\upgamma} 
	&= \{(a,s,\phi) | \Re \langle \phi_0|_t, \phi_t\rangle_{L^2(Y)} = 0\}\\
	&\subset L^2_j(Z;i\T^*Z)^{\uptau} \oplus L^2_j(\reals;\reals) \oplus L^2_{j,A_0}(Z;S^+)^{\uptau}.
\end{align*}
Then the linearization of the gauge group action is
\begin{align*}
	\mathbf{d}^{\tau}: \text{Lie}(\mathcal G_{j+1}(Z;\upiota)) \times \tilde{\mathcal C}^{\tau}_k(\mathfrak a, \mathfrak b)
	&\to
	\mathcal T^{\tau}_j\\
	\xi 
	&\mapsto
	(-d\xi, 0, \xi \phi_0).
\end{align*}
\begin{defn}
	The slice $\mathcal S_{k,\upgamma}^{\tau}$ at $(A_0, s_0, \phi_0)$ is the set of triples $(A,s,\phi)$ satisfying
	\begin{equation*}
	-d^*a + iss_0\Re\langle i\phi_0,\phi\rangle + i|\phi_0|^2\Re \mu_Y(\langle i\phi_0,\phi\rangle) = 0,
	\end{equation*}
	where $A = A_0 + a \otimes 1$. The left side of the above equation can be defined as an operator
	\begin{equation*}
		\text{Coul}_{\upgamma}^{\tau}: \tilde{\mathcal C}^{\tau}_k(\mathfrak a, \mathfrak b) \to L^2_{k-1}(Z;i\reals)^{-\upiota^*}.
	\end{equation*}
\end{defn}
The linearization of $\text{Coul}_{\upgamma}^{\tau}$ extends to an operator:
\begin{align*}
	\mathbf{d}^{\tau,\dag}_{\upgamma}: \mathcal T^{\tau}_j 
	&\to L^2_{j-1}(Z;i\reals),\\
	(a,s,\phi) 
	&\mapsto
	-d^*a + is_0^2\Re\langle i\phi_0, \phi \rangle + i|\phi_0|^2 \Re  \mu_Y\langle i\phi_0,\phi \rangle.
\end{align*}
Denote by
\begin{equation*}
	\mathcal K^{\tau}_{j,\upgamma} \subset \mathcal T^{\tau}_{j,\upgamma}
\end{equation*}
the kernel of $\mathbf{d}^{\tau,\dag}_{\upgamma}$ and by
\begin{equation*}
	\mathcal J^{\tau}_{j,\upgamma} \subset \mathcal T^{\tau}_{j,\upgamma}
\end{equation*}
the image of $\mathbf{d}^{\tau,\dag}_{\upgamma}$.
Using the arguments in \cite[Proposition~14.3.2]{KMbook2007}, one can show that $\mathcal K^{\tau}_{j,\upgamma}$ and $\mathcal J^{\tau}_{j,\upgamma}$ define complementary closed sub-bundle of $\mathcal T^{\tau}_j$.
(However these two bundles are not orthogonal with respect any natural metric.)
As a consequence, the quotient space $\tilde{\mathcal B}_k([\mathfrak a],[\mathfrak b])$ is a Hilbert manifold, locally modelled on the slice $\mathcal S^{\tau}_{k,\upgamma}$.

\subsubsection*{\textbf{(Linearization and relative grading)}}
Next, we show the perturbed Seiberg-Witten operator restricted to slice has Fredholm linearization.
Assume the perturbation is chosen so that all critical points are non-degenerate.
Define
\begin{align*}
	\mathcal V^{\tau}_{j,\upgamma} 
	&=
	\left\{
	(\eta, r, \psi) \big|
	\Re \langle \check\phi_0(t), \check\psi(t)\rangle_{L^2(Y)} = 0\right\}\\
	&\subset
	L^2_j(Z;i\mathfrak{su}(S^+))^{-\upiota^*} \oplus L^2_j(\reals;\reals) \oplus L^2_{j,A_0}(Z;S^-)^{\uptau},
\end{align*}
so that the perturbed Seiberg-Witten equations defines a section
\begin{equation*}
	\mathfrak F^{\tau}_{\mathfrak q} = \mathfrak F^{\tau} + \check{\mathfrak q}^{\tau}: \tilde{\mathcal C}^{\tau}_k(\mathfrak a, \mathfrak b) \to \mathcal V^{\tau}_{k-1}.
\end{equation*}
The smoothness of the section can be deduced the same way as \cite[Lemma~14.4.1]{KMbook2007}, from the tameness of perturbations.
\begin{lem}
\label{lem:pSWsmooth_taumodel}
	The perturbed Seiberg-Witten equations define a smooth section $\tilde{\mathcal C}^{\tau}_k(\mathfrak a, \mathfrak b) \to \mathcal V^{\tau}_{k-1}$.	\hfill \qedsymbol
\end{lem}
Let $\upgamma \in \tilde{\mathcal C}^{\tau}([\mathfrak a],[\mathfrak b])$ be a configuration.
To define $\mathcal D\mathfrak F^{\tau}_{\mathfrak q}$ at $\upgamma$, we ``differentiate and then project''. 
More precisely, we compose the derivative of $\mathcal F^{\tau}_{\mathfrak q}$ in the ambient vector bundle, with the composition
\begin{equation*}
	\Pi^{\tau}_{\upgamma}: 
	L^2_j(Z;i\mathfrak{su}(S^+))^{\uptau} \oplus L^2_j(\reals;\reals) \oplus L^2_{j,A_0}(Z;S^-)^{\uptau}
	\to \mathcal V^{\tau}_{j,\upgamma}
\end{equation*}
defined by $L^2$-projection slicewise.
That is,
\begin{align*}
	\Pi^{\tau}_{\upgamma}: 
	(\eta, r, \psi) 
	&\mapsto (\eta, r, \Pi^{\perp}_{\phi_0(t)}\psi),\\
	\Pi^{\perp}_{\phi_0(t)} \psi
	&=
	\psi - \Re \langle \check\phi_0(t), \psi(t) \rangle_{L^2(Y)} \phi_0.
\end{align*}
The resulting operator 
\begin{equation}
\label{eqn:lin_SW_op}
	\mathcal D \mathfrak F^{\tau}_{\mathfrak q,\upgamma}: \mathcal T^{\tau}_{j,\upgamma} \to \mathcal V^{\tau}_{j,\upgamma}
\end{equation}
can be interpreted as follows.
Let $\check\upgamma = (\check B_0(t), r_0(t), \phi(t))$ be the path $\reals \to \mathcal C^{\sigma}(Y)$ defined by $\upgamma$.
Ignoring regularity, 
\begin{itemize}[leftmargin=*]
	\item the codomain can be regarded as sections along $\check\upgamma$ of the tangent bundle $\mathcal T^{\sigma}(Y) \to \mathcal C^{\sigma}(Y)$, via Clifford multiplications, and
	\item the domain can be regarded as sections along $\check\upgamma$ of $\mathcal T^{\sigma}(Y) \oplus L^2(Y;i\reals)^{-\upiota^*}$, where for any $(a, r, \psi)$ in the domain, we express $a = b + cdt$ such that $b$ is temporal gauge and $c$ correspond to the summand $L^2(Y;i\reals)^{-\upiota^*}$.
\end{itemize}
We write sections in the domain as $(V,c)$, where $V(t) = (b(t), r(t), \psi(t))$ is an element of $\mathcal T^{\sigma}_{\upgamma_0(t)}(Y)$ and $c(t) \in L^2(Y;i\reals)^{-\upiota^*}$.
The derivative of $V$ with respect to $t$ is defined as a covariant derivative
\begin{equation*}
	\frac{D}{dt}V = \left( \frac{db}{dt}, \frac{dr}{dt}, \Pi^{\perp}_{\phi_0(t)} \frac{d\psi}{dt}\right),
\end{equation*}
for the tangent bundle $\mathcal T^{\sigma}$ is not a trivial bundle along the path.
If $\upgamma_0$ is in temporal gauge, then the linearization of perturbed Seiberg-Witten operator ~\eqref{eqn:lin_SW_op} is given by
\begin{equation*}
	(V,c) \mapsto \frac{D}{dt} V + \mathcal D(\grad \pertL)^{\sigma}(V) + \mathbf{d}^{\sigma}_{\upgamma_0(t)}c,
\end{equation*}
where $\mathbf{d}^{\sigma}$ is the derivative of gauge group action on $C^{\sigma}(Y)$.
We couple this operator with a Coulomb gauge fixing condition 
\[
	\mathbf{d}^{\tau, \dag}_{\upgamma_0}(V,c) = 0.
\]
When $(V,c)$'s are viewed as sections of Hilbert vector bundles along $\upgamma_0$, the above equation takes the form
\begin{equation*}
	\frac{d}{dt} c + \mathbf{d}^{\sigma,\dag}_{\upgamma_0(t)}(V) = 0.
\end{equation*}
Combining the two operators we obtain
\begin{align*}
	Q_{\upgamma_0} 
	&=
	\mathcal D_{\upgamma_0}(\mathfrak F^{\tau}_{\mathfrak q}) \oplus \mathbf{d}^{\tau, \dag}_{\upgamma_0},\\
	Q_{\upgamma_0} : 
	\mathcal T^{\tau}_{j,\upgamma_0} 
	&\to \mathcal V^{\tau}_{j-1,\upgamma_0} \oplus L^2_{j-1}(Z;i\reals)^{-\upiota^*}.
\end{align*}
In path notation, we have
\begin{equation*}
	(V,c) \mapsto \frac{D}{dt}(V,c) + L_{\upgamma_0(t)}(V,c),
\end{equation*}
where for any $\mathfrak a \in \mathcal C^{\sigma}_{k}(Y)$, the operator $L$ is
\begin{equation*}
	L_{\mathfrak a} = 
	\begin{bmatrix}
		\mathcal D_{\mathfrak a}(\grad \pertL)^{\sigma} & \mathbf{d}_{\mathfrak a}^{\sigma} \\
		\mathbf{d}_{\mathfrak a}^{\sigma,\dag} & 0
	\end{bmatrix}
	: \mathcal T^{\sigma}_{j,\mathfrak a} \oplus L^2_j(Y;i\reals)^{-\upiota^*} \to 
	\mathcal T^{\sigma}_{j-1,\mathfrak a} \oplus L^2_{j-1}(Y;i\reals)^{-\upiota^*},
\end{equation*}
which is exactly the extended Hessian $\widehat{\Hess}^{\sigma}_{\mathfrak q, \mathfrak a}$ on the blown-up configuration space.
\begin{prop}
\label{prop:Qfredholm}
	For each pair of (non-degenerate) critical points $\mathfrak a,\mathfrak b$ and each $\upgamma_0$ in $\mathcal C^{\tau}_k(\mathfrak a,\mathfrak b)$, the linear operator
	\begin{equation*}
		Q_{\upgamma_0} = \mathcal D_{\upgamma_0}\mathfrak F^{\tau}_{\mathfrak q} \oplus \mathbf{d}_{\upgamma_0}^{\tau,\dag}:
		\mathcal T^{\tau}_{j,\upgamma_0} \to
		\mathcal V^{\tau}_{j-1,\upgamma_0} \oplus
		L^2_j(Z;i\reals)^{-\upiota^*}
	\end{equation*}	
	is Fredholm for all $j$ in the range $1 \le j \le k$, and satisfies a G\r{a}rding inequality,
	\begin{equation*}
		\|u\|_{L^2_j} \le C_1\|Q_{\upgamma_0}u\|_{L^2_{j-1}} + C_2\|u\|_{L^2_{j-1}}.
	\end{equation*}
	The index $Q_{\upgamma_0}$ is independent of $j$.
\end{prop}
\begin{proof}
	
To prove the Fredholmness, we need to put our problem in the framework of Proposition~\ref{prop:spectralflowfredholm} where $L$ is the linearization of $\mathfrak F^{\tau}_{\mathfrak q}$.
This is not straightforward, however, as $ V^{\tau}_j$ is not a trivial bundle, but a subbundle of a natural trivial bundle.
More concretely, a section $V = (b, r, \psi)$ along $(\check B_0(t), r_0(t), \phi_0(t))$ is constrained by $\langle \psi|_t, \phi_0|_t \rangle_{L^2(Y)} = 0$.
To this end, we use the trick in Lemma~\ref{lem:ext_Hess_real_eigen}, by recasting $V$ as a pair $(b,\boldsymbol{\psi})$ and by setting
\begin{equation*}
	\boldsymbol{\psi}(t) = \psi(t) + r(t)\phi_0(t)
\end{equation*}
which is unconstrained. The operator can be written as
\begin{equation*}
	(b,\boldsymbol{\psi}, c) \mapsto
	\frac{d}{dt}(b,\boldsymbol{\psi},c) + 
	L_0 (b,\boldsymbol{\psi}, c) +
	\tilde h_{\upgamma(t)}(b,\boldsymbol{\psi},c),
\end{equation*}
where $L_0$ is a self-adjoint elliptic differential operator acting on $\uptau$-invariant sections of $i\T^* Y \oplus S \oplus i\reals$. 
In verifying the hypothesis of Proposition~\ref{prop:spectralflowfredholm}, we use Lemma~\ref{lem:ext_Hess_real_eigen} to deduce that the  operators on the two ends are hyperbolic when both $\mathfrak a$ and $\mathfrak b$ are non-degenerate.
For details and proof of the G\r{a}rding inequality, see \cite[Theorem~14.4.2]{KMbook2007}.
\end{proof}
The following corollary motivates the definition of the relative grading.
\begin{cor}
	The restriction of the bundle map
	\begin{equation*}
		\mathcal D \mathfrak F^{\tau}_{\mathfrak q}: \mathcal \mathcal K^{\tau}_{j,\upgamma} \to \mathcal V^{\tau}_{j-1,\upgamma}
	\end{equation*}	
	is Fredholm and has the same index as $\mathcal Q_{\gamma}$.
	Furthermore, the index of $Q_{\upgamma}$ depends only on the endpoints $\mathfrak a$ and $\mathfrak b$, and is equal to the spectral flow of the extended Hessian $\widehat{\Hess}^{\sigma}_{\mathfrak q, \check\upgamma}$.
\end{cor}

\begin{defn}
\label{defn:rel_grading}
	For critical points $\mathfrak a, \mathfrak b \in \mathcal C^{\sigma}_{k}(Y)$ and a path $\upgamma \in \mathcal C^{\tau}_k(\mathfrak a,\mathfrak b)$ we define the \emph{relative grading}
	\begin{equation*}
		\gr(\mathfrak a,\mathfrak b) := \ind Q_{\upgamma}.
	\end{equation*}
	For two gauge orbits $[\mathfrak a], [\mathfrak b]$ and $z$ a relative homotopy class of the path $\mathbf{\pi} \circ \check\upgamma$, we define the relative grading 
	\begin{equation*}
		\gr_z([\mathfrak a],[\mathfrak b]) := 
		\gr(\mathfrak a, \mathfrak b).
	\end{equation*}
\end{defn}
By additivity of spectral flows, we have additivity of the relative grading 
\[\gr(\mathfrak a,\mathfrak c) = \gr(\mathfrak a,\mathfrak b) + \gr(\mathfrak b, \mathfrak c).
\]
To compute the change of $\gr_z$ under change of homotopy class of paths between $[\mathfrak a]$ and $[\mathfrak b]$, it suffices to consider the case $[\mathfrak a] = [\mathfrak b]$ and loops based at $[\mathfrak a]$.
Let $\mathfrak a, \mathfrak a'$ be the corresponding lifts in $\mathcal C^{\sigma}_k(Y)$ of a loop $z_u$, where $u$ is a gauge transform such that $\mathfrak a' = u(\mathfrak a)$.

\begin{lem}\label{lem:loopindex1}
Identifying the homotopy class of an $\upiota$-invariant $u:Y \to S^1$ with an element $H^1(Y;\mathbb Z)$, the relative grading is 
	\begin{equation*}
	\gr_{z_u}([\mathfrak a],[\mathfrak a]) = \frac{1}{2}([u] \cup c_1(S))[Y].
\end{equation*}
\end{lem}
\begin{proof}
	This is the $\upiota$-invariant version of \cite[Lemma~14.4.6]{KMbook2007}.
	The index of $Q_{\upgamma}$ can be interpreted as the index of the linearized Seiberg-Witten operator over $S^1 \times Y$, where the real spin\textsuperscript{c} structure $(\mathfrak s_{u},\uptau_u)$ is obtained by gluing the two ends of the pullback real spin\textsuperscript{c} structure of $(\mathfrak s, \uptau)$ over $I \times Y$ by the automorphism $u$.
	The spin\textsuperscript{c} structure has first chern class
	\begin{equation*}
		c_1(\mathfrak s_u) = c_1(\mathfrak s) + 2\eta \cup [u],
	\end{equation*}
	where $\eta$ is the generator of $H^1(S^1;\mathbb Z)$.
	Since $\upiota_X$ acts trivially on the $S^1$-factor, we have 
	\[
	b^1_{-\upiota_X^*}(S^1 \times Y) = b^1_{-\upiota^*}(Y), \quad 
	b^+_{-\upiota_X^*}(S^1 \times Y) = b^1_{-\upiota^*}(Y), \quad 
	b^0_{-\upiota_X^*}(S^1 \times Y) = 0. 
	\]
	By Lemma~\ref{lem:closed4index}, the index of the real Seiberg-Witten operator is 
	\[
		d = \frac{1}{2}\left(\eta \cup [u] \cup c_1(S)\right)[S^1 \times Y]
		= \frac{1}{2} ([u] \cup c_1(S))[Y]. \qedhere
	\]
\end{proof}

\subsubsection*{\textbf{(Boundary-unosbtructed trajectories)}}
We return to the discussion on the moduli spaces
\begin{equation*}
	M_z([\mathfrak a],[\mathfrak b]) \subset  \tilde M_z([\mathfrak a],[\mathfrak b]) \subset \tilde{\mathcal B}^{\tau}_{k,z}([\mathfrak a],[\mathfrak b]).
\end{equation*}
The first space is the subspace subject to $s \ge 0$, which can be identified by the quotient under the involution 
\begin{equation*}
	\mathbf{i}: [A,s,\phi] \mapsto [A, -s, -\phi].
\end{equation*}
in the blown-up coordinates, by the unique continuation property of trajectories.
A neighbourhood $\mathcal U_{\upgamma}$ of $[\upgamma]$ in $\tilde M_z([\mathfrak a],[\mathfrak b])$ is the zero set of the map
\begin{equation*}
	\mathfrak F^{\tau}_{\mathfrak q}|_{\mathcal U_{\upgamma}} : 
	\mathcal U_{\upgamma} \to \mathcal V^{\tau}_{k-1},
\end{equation*}
where the linearization is the Fredholm operator
\begin{equation*}
	\mathcal D_{\upgamma} \mathfrak F^{\tau}_{\mathfrak q}: \mathcal K^{\tau}_{k,\upgamma} \to \mathcal V^{\tau}_{k-1,\upgamma},
\end{equation*}
of index $\ind \mathcal Q_{\gamma}$.
In particular, if $\mathcal D_{\upgamma} \mathfrak F^{\tau}_{\mathfrak q}$ is surjective, then the moduli space $\tilde M_z([\mathfrak a],[\mathfrak b])$ is a smooth manifold near $[\upgamma]$ whose dimension is the index, which is by definition equal to $\gr_z([\mathfrak a],[\mathfrak b])$.

\subsubsection*{\textbf{(Boundary-obstructed trajectories)}}
The regularity in the boundary-obstructed case requires modification, as the moduli spaces are not regular as $\mathcal Q_{\upgamma}$ is never surjective. 
\begin{defn}
	If $\mathfrak a \in \mathcal C^{\sigma}_k(Y)$ belongs to the reducible locus, then $\mathfrak a$ is \emph{boundary-stable} if $\Lambda_{\mathfrak q}(\mathfrak a) > 0$ and \emph{boundary-unstable} if $\Lambda_{\mathfrak q}(\mathfrak a)< 0$.	
\end{defn}
Since $s$ satisfies the equation $ds/dt = -\Lambda_{\mathfrak q}s$, if $M_z([\mathfrak a],[\mathfrak b])$ contains an irreducible trajectory, then $[\mathfrak a]$ is either irreducible or boundary-unstable, and $[\mathfrak b]$ is either irreducible or boundary-stable.
Let $\upgamma$ be a reducible trajectory, then the operator $\mathcal Q_{\upgamma}$ has the following decomposition
\begin{equation*}
	\mathcal Q_{\upgamma} = \mathcal Q^{\del}_{\upgamma} \oplus \mathcal Q^{\nu}_{\upgamma}
\end{equation*}
of the ``boundary'' and ``normal'' parts,
depending on whether the involution $\mathbf{i}$ is trivial or non-trivial.
In particular, the first operator
\begin{equation*}
	\mathcal Q_{\upgamma}^{\del} = (\mathcal D_{\upgamma}\mathfrak F^{\tau}_{\mathfrak q})^{\del} \oplus \mathbf{d}_{\upgamma}^{\tau,\dag},
\end{equation*}
where $(\mathcal D_{\upgamma}\mathfrak F^{\tau}_{\mathfrak q})^{\del}: (\mathcal T^{\tau}_{k,\upgamma})^{\del} \to (\mathcal V^{\tau}_{k,\upgamma})^{\del}$ is the $\mathbf{i}$-invariant part, and the normal part is an operator over the real line
\begin{align*}
	\mathcal Q_{\upgamma}^{\nu}:
	L^2_k(\reals;i\reals) 
	&\to L^2_{k-1}(\reals;i\reals)\\
	\mathcal Q^{\nu}_{\upgamma} &=
	(ds/dt) + \Lambda_{\mathfrak q}(\check\upgamma)s.
\end{align*}
We can easily compute the dimensions of kernels and cokernels of $\mathcal Q^{\nu}_{\upgamma}$, and summarize as follows.
\begin{lem}
\label{lem:dim_cases}
	The dimensions of the kernel and cokernel of $\mathcal Q^{\nu}_{\upgamma}$ are:
	\begin{itemize}
		\item $(1,0)$ if $\mathfrak a, \mathfrak b$ are boundary-unstable and boundary-stable respectively;
		\item $(0,1)$ if $\mathfrak a, \mathfrak b$ are boundary-stable and boundary-unstable respectively;
		\item $(0,0)$ if $\mathfrak a, \mathfrak b$ are either both boundary-stable or boundary-unstable.
	\end{itemize}
	In particular, the cokernels of $\mathcal Q_{\upgamma}$ and $\mathcal Q^{\del}_{\upgamma}$ are the same for reducible trajectories except in the second case.
\end{lem}
The case when $\mathcal Q_{\upgamma}^{\nu}$ is not surjective is \emph{boundary-obstructed}, and by Lemma~\ref{lem:dim_cases} the best one can hope is that the only cokernel arises from $\mathcal Q_{\upgamma}^{\nu}$.
\begin{defn}
	A moduli space $M([\mathfrak a],[\mathfrak b])$ is \emph{boundary-obstructed} if $[\mathfrak a]$ and $[\mathfrak b]$ are both reducible, $\mathfrak a$ is boundary-stable, and $\mathfrak b$ is boundary-unstable.	
\end{defn}
\begin{defn}
	Let $[\upgamma]$ be a solution in $M_z([\mathfrak a],\mathfrak b])$.
	Suppose the moduli space is not boundary-obstructed. 
	Then $\upgamma$ is \emph{regular} if $\mathcal Q_{\upgamma}$ is surjective.
	Suppose the moduli space is boundary-obstructed. 
	Then $\upgamma$ is \emph{regular} if $\mathcal Q^{\del}_{\upgamma}$ is surjective.
	The moduli space $M_z([\mathfrak a],\mathfrak b])$ is \emph{regular} if all its elements are regular.
\end{defn}
There are four possibilities of regular $M_z([\mathfrak a],[\mathfrak b])$ according to the type of endpoints.
\begin{prop}
\label{prop:class_regular_mod_space}
	Suppose the moduli space $M_z([\mathfrak a],[\mathfrak b])$ is regular, and $d = \gr_z([\mathfrak a],[\mathfrak b])$. 
	Then $M_z([\mathfrak a],[\mathfrak b])$ is
	\begin{itemize}[leftmargin=*]
		\item a smooth $d$-manifold consisting entirely of irreducible solutions if either $\mathfrak a$ or $\mathfrak b$ is irreducible;
		\item a smooth $d$-manifold with boundary if $\mathfrak a$, $\mathfrak b$ are boundary-unstable and -stable, respectively; in which case the boundary of the moduli space consists of the reducible elements;
		\item a smooth $d$-manifold consisting entirely of reducibles if $\mathfrak a,\mathfrak b$ are either both boundary-stable or both boundary-unstable;
		\item a smooth $(d+1)$-manifold consisting entirely of reducibles in the boundary-obstructed case. 
	\end{itemize}

\end{prop}
\begin{proof}
	Same as \cite[Proposition~14.5.7]{KMbook2007}.
	The first three cases correspond to the involution having no fixed points, a fixed submanifold of codimension one, or the involution trivial.
\end{proof}

\subsection{Simplest moduli spaces}
\hfill \break
	Let us consider the simplest situation when $\mathfrak a_1$ and $\mathfrak a_2$ are two reducible points in $\mathcal C^{\sigma}_k(Y)$ on the same fibre of the blow up $\mathbf{\pi}^{-1}(\alpha) \subset S^{\infty}$ where $\alpha  \in \mathcal C_k(Y)$.
	Each $\mathfrak a_i$ corresponds to an eigenvalue of the perturbed Dirac operator $D_{\mathfrak q, B}$.
	The eigenvalues can be recovered as the value of $\Lambda_{\mathfrak q}$ at the critical points
	\begin{equation*}
		\lambda_1 = \Lambda_{\mathfrak q}(\mathfrak a_1), \quad
		\lambda_2 = \Lambda_{\mathfrak q}(\mathfrak a_2).
	\end{equation*}
	Let $z_0$ be the homotopy class of paths joining $[\mathfrak a_1]$ to $[\mathfrak a_2]$ in $\mathcal C^{\sigma}_k(Y)$.
	Let $i$ be the ``signed count'' of the number of eigenvalues of $D_{\mathfrak q,B}$ between $\lambda_1$ and $\lambda_2$, including one endpoint:
	\begin{equation*}
		i = \begin{cases}
			|\{\mu \in \Spec(D_{\mathfrak q, B})|: \lambda_2 < \mu \le \lambda_1, &
			\text{if }\lambda_1 \ge \lambda_2\\
			|\{\mu \in \Spec(D_{\mathfrak q, B})|: \lambda_2 > \mu \ge \lambda_1, &
			\text{if }\lambda_1 \le \lambda_2\
		\end{cases}
	\end{equation*}
	\begin{prop}\label{prop:modspaceRPi}
		Let $i$ be as above.
		If $\lambda_1 \ge \lambda_2$, then $M_{z_0}([\mathfrak a_1],[\mathfrak a_2])$ is diffeomorphic to an open subset of a projective space $\mathbb{RP}^i$, obtained as the complement of two hyperplanes.
		If $\lambda_1 < \lambda_2$, then the moduli space is empty.
	\end{prop}
	\begin{proof}
	See the proof \cite[Proposition~14.6.1]{KMbook2007} and apply Lemma~\ref{lem:RPn_crit}.
	\end{proof}
	\begin{cor}
		The relative grading is given by
		\begin{equation*}
			\gr(\mathfrak a_1,\mathfrak a_2) =
			\begin{cases}
				i & \text{if $\lambda_1$ and $\lambda_2$ have the same sign},\\
				i-1 & \text{if $\lambda_1$ is positive and $\lambda_2$ is negative},\\
				i+1 & \text{if $\lambda_1$ is negative and $\lambda_2$ is positive}.
			\end{cases}
		\end{equation*}
	\end{cor}

\subsection{Transversality of moduli spaces of trajectories}
\hfill \break
We follow \cite[Section~15.1]{KMbook2007}.
\begin{thm}
	There is $\mathfrak q \in \mathcal P$ such that:
	\begin{itemize}
		\item all critical points $\mathfrak a \in \mathcal C^{\sigma}_k(Y,\upiota)$ are nondegenerate;
		\item for each pair of critical points $\mathfrak a, \mathfrak b$ and each relative homotopy class $z$, the moduli space $M_z([\mathfrak a],[\mathfrak b])$ is regular. 
	\end{itemize}
\end{thm}
\begin{proof}
	We start with a perturbation $\mathfrak q_0$ satisfying the fist condition, guaranteed by Theorem~\ref{thm:largebanachspace}.
	The critical set for $\mathfrak q_0$ is a finite collection of points, where the irreducible points come in pairs.
	Note that in the ordinary setting, the orbits of irreducible points are circles.
	From Proposition~\ref{thm:emb}, we choose a map that embeds the critical set
	\begin{equation*}
		p_0 : \mathcal B_k^{o} \to \reals^n \times \mathbb T \times \reals^m,
	\end{equation*}
	and separates the orbits of the different critical points.
	For each $[\mathfrak a] \in \mathcal B_k(Y)$, let $\mathcal O_{[\mathfrak a]} \subset \mathcal B^o_k(Y)$ be an open, $\mathbb Z_2$-invariant neighbourhood of the corresponding $\mathbb Z_2$-orbit, and their images under $p_0$ has disjoint closures.
	Denote the union
	\begin{equation*}
		\mathcal O = \bigcup_{[\mathfrak a]} \mathcal O_{[\mathfrak a]} \subset \mathcal B^o(Y).
	\end{equation*}
	We require $\mathcal O$ to be small enough so that there is no essential loop based at any $p_0([\mathfrak a])$ is contained in $\overline{p_0(\mathcal O)}$.
	Let $\mathcal P_{\mathcal O}$ be the closed linear space of perturbations 
	\begin{equation*}
		\mathcal P_{\mathcal O} = \left\{ \mathfrak q : \mathfrak q|_{\mathcal O} = \mathfrak q_0|_{\mathcal O}\right\}.
	\end{equation*}
	The definition ensures there is an open neighbourhood of $\mathfrak q_0$ in $\mathcal P_{\mathcal O}$, for which the perturbed the Seiberg-Witten gradient has no critical points outside $\mathcal O$.
	In other words, perturbation of this form will not change the critical set.
	
	The next goal is to show the set of perturbations $\mathfrak q \in \mathcal P_{\mathcal O}$ satisfying Condition~(ii) for all $\mathfrak a, \mathfrak b$ whose images belong to $\mathcal O \subset \mathcal B_k(Y)$ is a residual subset of $\mathcal P_{\mathcal O}$.
	To this end, we appeal to the general scheme in Lemma~\ref{lem:fredMorseSmale}, by considering the map
	\begin{align*}
		\mathfrak W: \mathcal C^{\tau}_k(\mathfrak a, \mathfrak b) \times \mathcal P_{\mathcal O} 
		&\to \mathcal V^{\tau}_{k-1}(Z) \\
		(\upgamma, \mathfrak q) 
		&\mapsto \mathfrak F^{\tau}_{\mathfrak q}(\upgamma).
	\end{align*}
	One needs to show $\mathcal D\mathfrak W$ is surjective at all irreducibles and a summand $(\mathcal D\mathfrak W)^{\del}$ is surjective at all reducibles.
	Let us mention the ingredients of the proof (and give the reference for details below) in the irreducible case where $[\mathfrak a],[\mathfrak b]$ have distinct image in the blown-down, for instance.
	
	Given an irreducible element $(\upgamma, \mathfrak q)$.
	We appeal to the unique continuation result \cite[Proposition~10.8.2]{KMbook2007} to find an interval $J \subset \reals$ for which $\upgamma|_J$ has no constant image in $\mathcal B(Y)$, and $p_0(\upgamma(J))$ lies outside $\overline{\mathcal O}$.
	Let $V$ be a hypothetical element of $\mathcal V^{\tau}_{k-1,\upgamma}$ orthogonal to the image of $\mathcal D \mathfrak F^{\tau}_{\mathfrak q}$, which under standard isomorphism gives rise to an $L^2_1$-section of $\mathcal T^{\sigma}_0(Y)$ along $\check\upgamma$.
	A linear unique continuation result \cite[Lemma~7.13.]{KMbook2007} implies that $V|_J$ is nonzero, and an integration-by-parts argument shows that $\check V(t)$ is everywhere orthogonal to the tangent space of gauge group action.
	
	On the other hand, we seek a cylinder function $f$ such that
	\begin{equation*}
		\left\langle \delta {\mathfrak q}^{\sigma}(t), \check V(t) \right\rangle_{\mathcal T^{\sigma}_{0,\check\upgamma(t)}} \ge 0.
	\end{equation*}
	with strict inequality at a point.
	The construction of $f$ involves using Propositon~\ref{thm:emb} to embed $\underline{\upgamma}(S)$ for certain compact subset $S \subset \reals$, and the Corollary~\ref{cor:separatetangent} to ensure $\mathcal D f$ nonnegative along certain vector field along $J$.
	For the proof of surjectivity of $\mathcal D\mathfrak W$ and $\mathcal (D\mathfrak W)^{\del}$, and the case when $\mathfrak a= \mathfrak b$ and the homotopy class $z$ is trivial, see \cite[Proposition~15.1.3]{KMbook2007}.
	The aforementioned results in \cite{KMbook2007} are general enough to adapt in the invariant setup.
\end{proof}

\subsection{Atiyah-Patodi-Singer boundary value problem}
\hfill \break
The setup of moduli spaces of finite cylinders involves the Atiyah-Patodi-Singer boundary value problem \cite{APS1973}.
We start with an abstract boundary value problem.
Let $X$ be a compact Riemannian 4-manifold that is cylindrical near the boundary, i.e. containing an isometric copy of $I \times Y$.
Let $\upiota_X$ be an involution on $X$ which is cylindrical near the boundary, in the sense that it is the pullback of some involution $\upiota: Y \to Y$.
Let $E,F$ be two vector bundles over $X$ with inner product, and the restriction near $I \times Y$ are equipped with the pull-back of a bundle $E_0 \to Y$.
Let $\upiota^E_X: E \to E$ and $\upiota^F_X: F \to F$  involutive lifts of $\upiota_X$, and are pullbacks of involutive lifts $\upiota^{E_0}: E_0 \to E_0$, over $I \times Y$.

Suppose we have an operator between the invariant sections 
\begin{equation*}
	D: C^{\infty}(X;E)^{\upiota^E_X} \to L^2(X;E)^{\upiota^F_X}
\end{equation*}
of the form
\begin{equation*}
	D = D_0 + K
\end{equation*}
where $D_0$ is an elliptic first order differential operator and $K$ is an operator which extends to a bounded operator
\begin{equation*}
	K: L^2_j(X;E)^{\upiota^E_X} \to L^2_j(X;F)^{\upiota^F_X}.
\end{equation*}
for $-k+1 \le j \le k-1$.
Near the boundary, we assume in addition that $D_0$ takes the form
\begin{equation*}
	D_0 u|_{I \times Y} = \frac{du}{dt} + L_0u,
\end{equation*}
where 
\begin{equation*}
	L_0: C^{\infty}(Y;E_0)^{\upiota^{E_0}} \to C^{\infty}(Y;E_0)^{\upiota^{E_0}}.
\end{equation*}
Under the above conditions, $D$ extends to a bounded operator $L^{2,\upiota^E}_{j+1} \to L_j^{2,\upiota^F}$ for $j \le k-1$, with closed range, finite-dimensional cokernel, but infinite-dimensional kernel.

For a natural Fredholm boundary value problem, we need to control the negative eigenspaces of the boundary values.
To this end, we define $H^+_0$ and $H^-_0$ as the closures in $L^2_{1/2}(Y;E_0)^{\upiota^{E_0}}$ of the spans of the eigenvectors belonging to positive and non-positive eigenvalues of $L_0$, respectively.
Moreover, define the projection
\begin{equation*}
	\Pi_0: L^2_{1/2}(Y;E_0)^{\upiota^{E_0}} \to L^2_{1/2}(Y;E_0)^{\upiota^{E_0}}
\end{equation*}
with images $H^-_0$ and kernel $H^+_0$.
Notice that we are taking the eigenvalue and eigenvectors of $L_0$ as an operator on the space of invariant sections of $E_0$.
\begin{thm}
\label{thm:cylinder_spectral_fred}
	Suppose that $X$ is a compact manifold with boundary, and let $D$ and $\Pi_0$ be as above.
	Then:
	\begin{enumerate}[leftmargin=*]
		\item[(i)] For $1 \le j \le k$, the operator
		\begin{equation*}
			D \oplus (\Pi_0 \circ r) :
			L^2_j(X;E)^{\iota^E} \to 
			L^2_{j-1}(X;F)^{\iota^F} \oplus 
			(H^-_0 \cap L^2_{j-1/2}(Y;E_0)^{\upiota^{E_0}}
		\end{equation*}
		is Fredholm. In particular, the restriction of $\Pi \circ r$ to $\ker(D)$ is Fredholm.
		\item[(ii)] If $u_i$ is a bounded sequence in $L^2_j(X;E)^{\upiota^E}$ and $Du_i$ is Cauchy in $L^2_{j-1}$, then $(1-\Pi_0) r(u_i)$ has a convergent subsequence in $H^+_0 \cap L^2_{j-1/2}$. In particular, the restriction of $(1-\Pi_0) \circ r$ to $\ker(D)$ is compact.
		\item[(iii)] If $u$ is in $L^2_j(X;E)^{\upiota_X^E}$ for $j \le k$, and the image of $u$ under the operator is in $L^2_{k-1}(X;F)^{\upiota_X^F} \oplus (H^-_0 \cap L^2_{k-1/2}(Y;E_0)^{\upiota^{E_0}}$, then $u$ lies in $L^2_k(X;E)^{\upiota^E_X}$. In particular, the kernel consists of $L^2_k$-sections.
	\end{enumerate}	
\end{thm}
\begin{proof}
	The Fredholmness is a parametrix patching argument combined with the special case when $Z = (-\infty, 0]$ and $L_0$ is invertible on the invariant $L^2_s$ sections, both generalize to the $\upiota^E$-invariant situation.
	The proof rest of the assertions also follows from the ordinary case in \cite[Theorem~17.1.3]{KMbook2007}.
\end{proof}
Once the general theorem is in place, let us prove a gauge-theoretic version.
Let $I \times Y$ be a finite cylinder.
We are interested in the gauge-equivalence classes of solutions to the perturbed Seiberg-Witten equations
\begin{equation*}
	M(Z) \subset \tilde M(Z) = 
	\left\{
	[\upgamma] \in \tilde{\mathcal B}^{\tau}_k(Z) 
	\big| 
	\mathfrak F^{\tau}_{\mathfrak q}(\upgamma)=0
	\right\}
\end{equation*}
over the finite cylinder.
Without boundary conditions, this is an infinite-dimensional space.
In fact, unlike the infinite cylinder case, the space actually depends on $k$ due to the regularity near boundary.
Still, we have the following result.
\begin{thm}
	The subspace $\tilde M(Z) \subset \tilde{\mathcal B}^{\tau}_k(Z)$  is a closed Hilbert submanifold. The subset $M(Z)$ is a Hilbert submanifold with boundary, and can be identified as the quotient of $\tilde M(Z)$ by the involution $\mathbf{i}$.
\end{thm}
\begin{proof}
	See \cite[Theorem 17.3.1]{KMbook2007}.
\end{proof}
\subsection{Moduli spaces over finite cylinders}
\hfill \break
Write the boundary of $Z$ as $\bar Y \sqcup Y$, where bar over $Y$ indicates orientation reversal.
We have restriction maps
\begin{align*}
	R_Y: \tilde M(Z) 
	&\to \mathcal B^{\sigma}_{k-1/2}(Y),\\
	R_{\bar Y}: \tilde M(Z) 
	&\to \mathcal B^{\sigma}_{k-1/2}(\bar{Y}),
\end{align*}
whose derivatives can be considered as 
\begin{align*}
	\mathcal D R_Y: \tilde M(Z) 
	&\to \mathcal K^{\sigma}_{k-1/2}(Y),\\
	\mathcal D R_{\bar Y}: \tilde M(Z) 
	&\to \mathcal K^{\sigma}_{k-1/2}(\bar{Y}),
\end{align*}
The analogue of the three-dimensional operator $D$ above is the Hessian
\begin{equation*}
	\Hess_{\mathfrak q,\mathfrak a}^{\sigma}: \mathcal C^{\sigma}(Y) \to \mathcal K^{\sigma}_j,
\end{equation*}
which has discrete spectrum and finite-dimensional generalized eigenspaces. 
This follows from the consideration of the extended Hessian, see the discussion preceding \cite[Theorem~17.3.2]{KMbook2007}.
If $\Hess_{\mathfrak q,\mathfrak a}$ is hyperbolic, then there is a spectral decomposition
\begin{equation*}
	\mathcal K^{\sigma}_{k-1/2,\mathfrak a} = \mathcal K^+_{\mathfrak a} \oplus \mathcal K^-_{\mathfrak a}.
\end{equation*}
In the non-hyperbolic case, we pick an $\epsilon$ sufficiently small that there are no eigenvalues in $(0,\epsilon)$, and we define $\mathcal K^{\pm}_{\mathfrak a}$ using the spectral decomposition of $\Hess_{\mathfrak q,\mathfrak a}-\epsilon$.
This defines a family of decompositions, which is not continuous at points where $\Hess_{\mathfrak a}$ has kernel.
Denote
\begin{align*}
	\mathcal K^{\sigma}_{k-1/2, (\mathfrak a, \bar{\mathfrak a})} (\bar Y \sqcup Y) 
	&\cong
	 \mathcal K^{\sigma}_{k-1/2,\bar{\mathfrak a}}(Y) \oplus 
	 \mathcal K^{\sigma}_{k-1/2,\mathfrak a}(Y) \\
 \mathcal K^{-}(\bar Y \sqcup Y) 
	&\cong
	 \mathcal K^{+}_{\bar{\mathfrak a}}(Y) \oplus 
	 \mathcal K^{-}_{\mathfrak a}(Y).
\end{align*}
\begin{thm}
\label{prop:cyinder_finite_fredholm}
	Let $\upgamma, \mathfrak a$ and $\bar{\mathfrak a}$ be as above, and let $\pi$ be the projection
	\begin{equation*}
		\pi: \mathcal K^{\sigma}_{k-1/2, (\mathfrak a, \bar{\mathfrak a})} (\bar Y \sqcup Y) \to \mathcal K^{\sigma}_{k-1/2, (\mathfrak a, \bar{\mathfrak a})} (\bar Y \sqcup Y)
	\end{equation*}
	with kernel $\mathcal K^+_{(\mathfrak a, \bar{\mathfrak a})}(\bar Y, Y)$. Then the two composite maps 
	\begin{align*}
		\pi \circ \mathcal D(R_{\bar Y}, R_Y):
		T_{[\upgamma]}\tilde M(Z) 
		&\to \mathcal K^-_{(\mathfrak a, \bar{\mathfrak a})},\\
		(1-\pi) \circ \mathcal D(R_{\bar Y}, R_Y):
		T_{[\upgamma]}\tilde M(Z) 
		&\to \mathcal K^+_{(\mathfrak a, \bar{\mathfrak a})},
	\end{align*}
	are respectively Fredholm and compact.
\end{thm}
\begin{proof}
	We consider the linearization of Seiberg-Witten map coupled with Coulomb gauge condition
	\begin{equation*}
		Q_{\upgamma} = \frac{D}{dt} + L_{\upgamma(t)}
	\end{equation*}
	in place of the operator $D$ in Theorem~\ref{thm:cylinder_spectral_fred}, where $L$ is essentially the extended Hessian operator.
	The corresponding minus spectral summand on $Y$, for example, is
	\begin{equation*}
		\tilde H^-_Y \subset
		\mathcal T^{\sigma}_{j-1/2,\mathfrak a} \oplus
		L^2_{j-1/2}(Y;i\reals)^{-\upiota^*}.
	\end{equation*}
	However, the projection operator relevant to our theorem is not the spectral projection above.
	In fact, the theorem follows from a generalization of Theorem~\ref{thm:cylinder_spectral_fred}, involving \emph{commensurate projections}, 
	This generalization is discussed in \cite[Section~17.2]{KMbook2007}.
	The details the proof can be found in \cite[Theorem~17.3.2]{KMbook2007}, which applies verbatim in our case. 
\end{proof}

\section{Compactness}
\label{sec:compact}
We will compactify moduli spaces by adding broken trajectories, following \cite[Section~16]{KMbook2007}.
\subsection{Compactification by broken trajectories}
\begin{defn}[unparametrized trajectory]
	An \emph{unparametrized trajectory} is an equivalence class of non-trivial trajectories in $M_z([\mathfrak a],[\mathfrak b])$ under the action of translations.
	Denote space of unparametrized trajectories as $\check M_z([\mathfrak a],[\mathfrak b])$.
\end{defn}
\begin{defn}[unparametrized broken trajectory]
	An \emph{unparametrzied broken trajectory} joining $[\mathfrak a]$ to $[\mathfrak b]$ consists of the following data:
	\begin{itemize}[leftmargin=*]
		\item an integer $n \ge 0$, the number of components;
		\item an $(n+1)$-tuple of critical points $[\mathfrak a_0],\dots,[\mathfrak a_n]$ with $[\mathfrak a_0] = \mathfrak a$ and $[\mathfrak a_n] = [\mathfrak b]$ the restpoints;
		\item for each $i$ with $1 \le i \le n$, an unparametrized trajectory $[\check{\upgamma}_i]$ in $\check M_{z_i}([\mathfrak a_{i-1},\mathfrak a_i]$, the $i$th component of the broken trajectory.
	\end{itemize}
	The \emph{homotopy class} of the broken trajectory is the class of the path obtained by concatenating representatives of the classes $z_i$, or the constant path at $[\mathfrak a]$ if $n=0$.
	Write $\check M^+_z([\mathfrak a],[\mathfrak b])$ for the space of unparametrized broken trajectories in the homotopy class $z$, and denote a typical element by $[\check{\mathbf{\upgamma}}] = ([\check{\upgamma}_1],\dots,[\check{\upgamma}_n])$. \qed 
\end{defn}

	We topologize the space of unparametrized broken trajectories as follows.
	Let 
	$$[\check{\boldsymbol \upgamma}] = ([\check\upgamma_1], \dots, [\check\upgamma_n])$$ 
	belong to $\check M^+_z([\mathfrak a],[\mathfrak b])$, with $[\check\upgamma_i] \in \check M_{z_i}([\mathfrak a_{i-1}],[\mathfrak a_i])$ being presented by a (parametrized) trajectory
	\begin{equation*}
		[\upgamma_i] \in M_{z_i}([\mathfrak a_{i-1}],[\mathfrak a_i]).
	\end{equation*}
	Let $U_i \subset \mathcal B^{\tau}_{k,\loc}(Z)$ be any open neighbourhood of $[\upgamma_i]$, and $T \in \reals^+$.
	We define 
	$$\Omega = \Omega(U_1,\dots,U_n,T)$$ 
	to be the subset of $\check M^+_z([\mathfrak a],[\mathfrak b])$ consisting of  unparametrized broken trajectories $[\check{\boldsymbol{\delta}}] = ([\check\delta_1],\dots, [\check\delta_n])$ satsifying the following condition:
	there exists a map $(\jmath,s):\{1,\dots,n\} \to \{1,\dots,m\} \times \reals$ such that
	\begin{itemize}[leftmargin=*]
		\item $[\tau_{s(i)}\delta_{\jmath(i)}] \times U_i$, and
		\item if $1 \le i_1 \le i_2 \le n$, then either $\jmath(i_1) < \jmath(i_2)$, or $\jmath(i_1) = \jmath(i_2)$ and $s(i_1) + T \le s(i_2)$.
	\end{itemize}
	Here $\tau_s \delta$ denotes the translate $\tau_s\delta(t) = \delta(s+t)$.
	We declare the sets $\{\Omega(U_1,\dots, U_n, T)\}$ to be a neighbourhood base for $[\boldsymbol{\check \upgamma}$ in $\check M^+_z([\mathfrak a],[\mathfrak b])$.
With this topology, we have the following compactness theorem.
\begin{thm}
	The space of unparametrized broken trajectories $\check M_z^+([\mathfrak a],[\mathfrak b])$ is compact.
\end{thm}
This theorem follows from the stronger proposition below.
\begin{prop}
	For any $C > 0$ and any $[\mathfrak a]$, there are only finitely many $z$ with energy $\mathcal E_{\mathfrak q}(z) \le C$ for which $\check M_z^+([\mathfrak a],[\mathfrak b])$ is non-empty.
	Furthermore, each $\check M_z^+([\mathfrak a],[\mathfrak b])$ is compact.
	In other words, the space of broken trajectories $[\mathbf{\check \upgamma}] \in \check M^+([\mathfrak a],[\mathfrak b])$ with energy $\mathcal E_{\mathfrak q}(\mathbf{\upgamma}) \le C$ is compact.
\end{prop}
\begin{proof}
	Let us give a brief sketch of the proof in \cite[Proposition~16.1.4]{KMbook2007}.
	The local compactness result Proposition~\ref{thm:perturbed_local_compactness_blowup} over finite cylinder can be strengthened to $L^2_{k,\loc}$-compactness over the infinite cylinder $Z = \reals \times Y$, assuming uniform bounds on both $\mathcal E_{\mathfrak q}$  and the function $\Lambda_{\mathfrak q}$.
	Although we do not assume control of $\Lambda_{\mathfrak q}$ in the statement of the proposition, the uniform bound on the the variation of $d\Lambda_{\mathfrak q}/dt$ follows from the $\upiota$-invariant analogue of \cite[Corollary~13.5.3]{KMbook2007}, which relies crucially on the exponential decay near critical points.
	Given a sequence of trajectories $\{\check{\upgamma}_n^{\tau}\}$, we
	extract subsequence and possibly translate the trajectories, and then decompose the real line into intervals where trajectories either have high energy or low energy. 
	Finally we appeal to local compactness.
\end{proof}
Without assumptions on the energy upper bound but with assumptions on regularity of the moduli spaces, we have following.
\begin{prop}
	Suppose that all moduli spaces $M_z([\mathfrak a],[\mathfrak b])$ are regular. Then for given $[\mathfrak a]$ and $[\mathfrak b]$, there are only finitely many homotopy classes $z$ for which the moduli space $\check M^+_z([\mathfrak a],[\mathfrak b])$ is non-empty.
\end{prop}

We also consider a smaller compactification.
\begin{defn}
	The space $\bar M_z(X^*,[\mathfrak b])$ is the image of $M^+_z(X^*,[\mathfrak b])$ under the map
	\begin{align*}
		r:M^+_z(X^*,[\mathfrak b]) 
		&\to \mathcal B^{\sigma}_{k,\loc}(X^*)\\
		([\upgamma_0],[\check\upgamma]) &\mapsto
		[\upgamma_0].
	\end{align*}
	For the family version, if $P$ is a manifold parametrizing a family of metrics and perturbations on $X$, then $\bar M_z(X^*,[\mathfrak b])_P$ is defined as the image
	\begin{align*}
		M^+_z(X^*,[\mathfrak b])_P
		&\to P \times \mathcal B^{\sigma}_{k,\loc}(X^*)\\
		(p,[\upgamma_0],[\check\upgamma]) &\mapsto
		(p,[\upgamma_0]).
	\end{align*}
\end{defn}

\subsection{The compactification as a stratified space}
\begin{defn}
	A space $N^d$ is a \emph{$d$-dimensional space stratified by manifolds} if there are closed subsets
	\begin{equation*}
		N^d \supset N^{d-1} \supset \cdots \supset N^0 \supset N^{-1} = \emptyset
	\end{equation*}	
	such that $N^d \ne N^{d-1}$ and each space $N^e \setminus N^{e-1}$ (for $0 \le e \le d)$) is either empty or homeomorphic to manifold of dimension $e$.
	The difference $N^e \setminus N^{e-1}$ is the \emph{$e$-dimensional stratum}.
	we will also use the term \emph{stratum} to refer to any union of path components of $N^e \setminus N^{e-1}$.
\end{defn}
\begin{exmp}
\label{exmp:bad_stratified}
	Here is a pathological example.
	Let $N^1$ be the union of all circles $C_n$ of centre $(-1/n,0)$ and radius $1/n$, indexed by $n \in \mathbb N$, and the segment joining $(0,0)$ and $(1,0)$.
	Then $N^0 = \{(0,0), (1,0)\}$ and $N^1 \setminus N^0$ is a manifold with countably many path-components.
\end{exmp}
The moduli spaces of Seiberg-Witten solutions are naturally stratified.
\begin{prop}
	Suppose that $M_z([\mathfrak a],[\mathfrak b])$ is nonempty and of dimension $d$.
	Then the compactification $\check M^+_z([\mathfrak a],[\mathfrak b])$ is a $(d-1)$-dimensional space stratified by manifolds.
	If $M_z([\mathfrak a],[\mathfrak b])$ contains irreducible trajectories, then the $(d-1)$-dimensional stratum of $\check M^+_z([\mathfrak a],[\mathfrak b])$ consists of the irreducible part of $\check M^+_z([\mathfrak a],[\mathfrak b])$.
\end{prop}
The stratification structure of $M_z([\mathfrak a],[\mathfrak b])$ can be complicated due to the presence of boundary-obstructed solutions.
We will only discuss the codimension-1 strata in the following proposition.
For more details, we refer the readers to \cite[Section~16.5]{KMbook2007}.
To state the proposition,
consider the space
\begin{equation}
\label{eqn:productmodspace}
	\check M_{z_1}([\mathfrak a_0],[\mathfrak a_1]) \times \cdots \times
	\check M^{z_{\ell}}([\mathfrak a_{n-1}],[\mathfrak a_n]).
\end{equation}
Denote $d_i$ the dimension of $M_{z_i}([\mathfrak a_{i-1}],[\mathfrak a_i])$, and $\epsilon_i$ such that $d_i - \epsilon_i = \gr_{z_i}([\mathfrak a_{i-1}],[\mathfrak a_i])$.
The vector $(d_i-\epsilon_i)$ is the \emph{grading vector}, and $(\epsilon_i)$ is the \emph{obstruction vector}.
\begin{prop}
	Let $M_z([\mathfrak a],[\mathfrak b])$ be a $d$-dimensional moduli space containing irreducibles, so that $\check M_z^+([\mathfrak a],[\mathfrak b])$ is a compact $(d-1)$-dimensional space stratified by manifolds, with top stratum the irreducible part of $\check M_z([\mathfrak a],[\mathfrak b])$.
	Then the $(d-2)$-dimensional stratum of $\check M_z^+$ is the union of pieces of three types:
	\begin{itemize}[leftmargin=*]
		\item the top stratum is of the form~\eqref{eqn:productmodspace} with grading vector $(d_1,d_2)$ and obstruction vector $(0,0)$;
		\item the top startum of form~\eqref{eqn:productmodspace} with grading vector $(d_1,d_2-1,d_3)$ and obstruction vector $(0,1,0)$;
		\item the intersection of $\check M_z([\mathfrak a],[\mathfrak b])$ with the reducibles, if $M_z([\mathfrak a],[\mathfrak b])$ contains both reducibles and irreducibles.
	\end{itemize}
	The third case occurs only when $[\mathfrak a]$ is boundary-unstable and $[\mathfrak b]$ is boundary-stable.
	In the first case, $d_1 + d_2 = d$.
	In the second case, $d_1 + d_2 + d_3 = d + 1$.
	In all cases, the $d_i$ are positive.
 \end{prop}

We end the subsection with the description of the reducible moduli spaces.
\begin{prop}
	Suppose $M^{\red}_z([\mathfrak a],[\mathfrak b])$ is non-empty and of dimension $d$.
	Then the space of unparametrized, broken reducible trajectories, $\check M^{\red +}_z([\mathfrak a],[\mathfrak b])$ is a compact $(d-1)$-dimensional space stratified by manifolds.
	The top stratum consists of $\check M^{\red}_z([\mathfrak a],[\mathfrak b])$ alone.
	The $(d-\ell)$-dimensional stratum consists of the spaces of unparametrized broken trajectories with $\ell$-factors:
	\begin{equation*}
		\check M_z^{\red}([\mathfrak a_0],[\mathfrak a_1]) \times
		\cdots \times 
		\check M_z^{\red}([\mathfrak a_{\ell-1}],[\mathfrak a_{\ell}])
	\end{equation*}
\end{prop}

\subsection{\v{C}ech cohomology}
This subsection provides a \v{C}ech model for evaluating cohomology classes over moduli spaces of Seiberg-Witten equations, which will be applied in defining cobordism maps and module structure.

An \emph{open cover} $\mathcal U$ of a space $B$ has \emph{covering order $\le d + 1$} if every $(d+2)$-fold intersection
\begin{equation*}
	U_0 \cap U_1 \cap \cdots \cap U_{d+1}
\end{equation*}
is empty, for $U_i \in \mathcal U$ distinct.
To an open covering $\mathcal U$, one can associate the \emph{nerve} $K(\mathcal U)$ of $\mathcal U$, which is a simplicial complex. 
The \v{C}ech cohomology of $B$ is the limit of simplicial homology of the nerves
\begin{equation*}
	\check{H}^n(B;\mathbb Z_2) = \lim_{\rightarrow} H^n_{\text{simp}}(K(\mathcal U); \mathbb Z_2)
\end{equation*}
as $\mathcal U$ runs through the open coverings of $B$.

If $\mathcal U' \subset \mathcal U$ and the open subsets $\{U \cap B': U \in \mathcal U'\}$ covers $B'$, then $K(\mathcal U'|B')$ is a subcomplex of $K(\mathcal U')$.
We define the relative \v{C}ech cohomology of the pair $(B,B')$ as
\begin{equation*}
	\check H^n(B,B';\mathbb Z_2) =
	\lim_{\rightarrow} H^n_{\text{simp}}(K(\mathcal U),K(\mathcal U'|B');\mathbb Z_2).
\end{equation*}

Suppose $N = N^d$ is a space stratified by manifolds $:
	N^d \supset N^{d-1} \supset \cdots \supset N^0$.
Assume $N$ embeds in a metric space $B$.
Then an open cover $\mathcal U$ of $B$ is \emph{transverse to the strata} if $\mathcal U|N^e$ has covering order no larger than $(e+1)$ for all $e \le d$.
A transverse refinement always exists \cite[Lemma~21.2.1]{KMbook2007}:
\begin{lem}
\label{lem:refinetransversestrata}
	Let $N^{d_k}_k$ be a countable locally finite collection of spaces stratified by manifolds.
	Then every open cover $\mathcal U$ of $B$ has a refinement $\mathcal U'$ that is transverse to the strata in every $N^{d_k}_k$.
	\hfill \qedsymbol
\end{lem}
By the above lemma, we compute the \v{C}ech cohomology of $B$ as
\begin{equation*}
	\check H^n(B;\mathbb Z_2) = \lim_{\rightarrow} H^n(K(\mathcal U);\mathbb Z_2)
\end{equation*}
over open coverings of $B$ that are transverse to the stratification of a stratified space $N^d$.

\subsection{The Stokes theorem and boundary multiplicities}
\hfill \break
Suppose $N^d$ is a space stratified by manifolds $
	N^d \supset N^{d-1} \supset \cdots \supset N^0$, and denote $M^e = N^e \setminus N^{e-1}$.
Suppose  each $N^d$ admits a covering of dimension at most $d$ (this is always achievable by the result of Nagami \cite{KNagami1970}).
Then
\begin{equation*}
	\check H^d(N^d, N^{d-1};\mathbb Z_2) = H^d_c(M^d;\mathbb Z_2)
\end{equation*}
is a free abelian group with generators $\mu_{\alpha}^d$, corresponding to the components $M^d_{\alpha}$ of $M^d$.
Let
\begin{equation*}
	I_{\alpha}: \check H^d(N^d, N^{d-1};\mathbb Z_2) \to \mathbb Z_2,
\end{equation*}
be the map which is $1$ on the generator and zero on the other.
From the long exact sequence of the triple $(N^d, N^{d-1}, N^{d-2})$, there is a coboundary map
\begin{equation*}
	\delta_*:H^{d-1}_c(M^{d-1};\Ftwo)
	= \bigoplus_{\beta} H^{d-1}_c(M^{d-1}_{\beta};\Ftwo) \to 
	\bigoplus_{\alpha} H^{d-1}_c(M^{d}_{\alpha};\Ftwo) =
	H^d_c(M^d;\Ftwo).
\end{equation*}
The entry 
\begin{equation*}
	\delta_{\alpha\beta} = I_{\alpha} \delta_* \mu_{\beta}^{d-1}.
\end{equation*}
is the 
\emph{multiplicity} of 
of $M^{d-1}_{\beta}$ appearing in boundary of $M^d_{\alpha}$.
For a component $M^{d-1}_{\beta}$, the \emph{boundary multiplicity} is the finite sum
\begin{equation*}
	\delta_{\beta} = \sum_{\alpha} \delta_{\beta\alpha}.
\end{equation*}
Suppose $N^d$ is embedded in a metric space $B$ and choose $\mathcal U$ to be an open covering of $B$ that is transverse to all of the manifold strata.
Every \v{C}ech cochain
\begin{equation*}
	u \in \check C^d(\mathcal U|N^d;\mathbb Z_2)
	= C^d_{\text{simp}}(K(\mathcal U|N^d);\mathbb Z_2)
\end{equation*}
is automatically coclosed, which vanishes on $K(\mathcal U|N)$, as there is no $d$-simplex, and hence defines a class $[u]$ in $\check H^d(N^d, N^{d-1};\mathbb Z_2)$.
We therefore have the ``integration'' maps
\begin{align*}
	\langle -,[M_{\alpha}^d]\rangle : \check C^d(\mathcal U;\mathbb Z_2) &\to \mathbb Z_2\\
	u &\mapsto I_{\alpha}[u|_{N^d}]
\end{align*}
In particular, for $x \in \check H^{d-1}(N^{d-1}, N^{d-2};\Ftwo) = H^{d-1}_c(M^{d-1};\Ftwo)$, 
\begin{equation*}
	I_{\alpha}(\delta_*x) = \sum_{\alpha} \delta_{\alpha \beta}I_{\beta}x.
\end{equation*}
The Stokes theorem holds by definition:
	for $v \in \check{C}^{d-1}(\mathcal U;\Ftwo)$, we have
	\begin{equation}
	\label{eqn:stokes}
		\sum_{\beta} \delta_{\alpha\beta}\langle v, [M^{d-1}_{\beta}]\rangle
		= \langle \delta v, [M^d_{\alpha}]\rangle,
	\end{equation}
	where $\delta$ is the \v{C}ech coboundary map $\check C^{d-1}(\mathcal U;\Ftwo) \to \check C^{d}(\mathcal U;\Ftwo)$.
\hfill \break

The Stokes theorem has the following application to 1-dimensional spaces admitting \emph{$\delta$-structures}. 
(We will define $\delta$-structures in Section~\ref{sec:Gluing}).
\begin{lem}
\label{lem:bound_multiplicities_0}
	Let $N^1$ be a compact 1-dimensional space stratified by manifolds, so that $N^0$ is a finite number of points.
	Suppose that $N^1$ has a codimension-1 $\delta$-structure along $N^0$. 
	Then the number of points of $N^0$ is zero modulo two.
\end{lem}
\begin{proof}
	This is \cite[Corollary~21.3.2]{KMbook2007}.
\end{proof}
As a demonstration of Lemma~\ref{lem:bound_multiplicities_0}, recall example~\ref{exmp:bad_stratified}.
The stratified space $N$ is a union of infinitely many circles, where the $0$-stratum
	$N^0$ has two points, even though $N^1 \setminus N^0$ has infinitely many components.

\section{Gluing}
\label{sec:Gluing}
This section is concerned with the structure of the neighbourhoods of the boundary strata as gluing of broken trajectories.
\subsection{An abstract gluing theorem}
\hfill \break
The abstract gluing theorem is an variant of the one in \cite[Section~18.2]{KMbook2007}.
Let us set up some notations.
Let $Z^T$ be a finite cylinder $[-T,T] \times Y$, for $T > 0$, and $Z^{\infty}$ be the broken infinite cylinder
\begin{equation*}
	Z^{\infty} = (\reals^{\ge} \times Y) \sqcup
	 (\reals^{\le} \times Y).
\end{equation*}
Let $E \to Z$ be the pullback of a bundle $E_0$ on $Y$, where $Z$ denotes either $Z^T$, $Z^{\infty}$, or $\reals \times Y$ that projects onto the factor $Y$.
Let $\upiota$ be an involution on $Y$.
Let $\uptau: E_0 \to E_0$ be an involutive lift of $\upiota:Y \to Y$ which induces an involution on the sections $L^2_k(Y;E_0)$ by pullback, i.e. 
\begin{equation*}
	\uptau s(y) := \uptau(s(\upiota(y)))
\end{equation*} 
and thus an involution on $L^2_k(Z;E)$.
Let $D$ be an operator acting on sections of $E$ that commutes with $\uptau$, having the form
\begin{equation*}
	Du = \frac{du}{dt} + Lu,
\end{equation*}
where $L$ is an operator on $Y$, also commutating with $\uptau$:
\begin{equation*}
	L:L^2_k(Y;E_0) \to L^2_{k-1}(Y;E_0),
\end{equation*}
and hence defines an operator on the invariant sections:
\begin{equation*}
	L:L^2_k(Y;E_0)^{\uptau} \to L^2_{k-1}(Y;E_0)^{\uptau}.
\end{equation*}
The domain of $D$ will either be $L^2_k(Z^T;E)^{\uptau}$ or $L^2_k(Z^{\infty};E)^{\uptau}$.
In the case of the infinite cylinder $Z^{\infty}$, the operator acts on the weighted Sobolev space $L^2_{k,\delta}(Z^{\infty},E)$ where $\delta \in \reals$ and by definition
\begin{equation*}
	s \in L^2_{k,\delta}(Z^{\infty};E) \iff e^{\delta|t|}s \in L^2_{k-1,\delta}(Z^{\infty},E).
\end{equation*}
Suppose we are given a linear map
\begin{equation*}
	\Pi: L^2_{k-1,1/2}(Y \sqcup \bar Y;E_0)^{\uptau} \to H,
\end{equation*}
to some Hilbert space $H$. Omitting the restriction from our notations, we write
\begin{equation*}
	\Pi: L^2_k(Z;E)^{\uptau} \to H
\end{equation*}
for the composition of $\Pi$ with the restriction map, over $Z = Z^{T}$ or $Z^{\infty}$.
We also simplify notations if no confusion arises, and write
\begin{align*}
	\mathcal E^T &= L^2_k(Z^T;E)^{\uptau},\\
	\mathcal F^T &= L^2_{k-1}(Z^T;E)^{\uptau},
\end{align*}
and $\mathcal E^{\infty}_{\delta}$ and $\mathcal F^{\infty}_{\delta}$ for the weighted versions.
We make the following assumption.
\begin{asm}
\label{asm:inftyinvertible}
	We suppose the linear operator
	\begin{equation*}
		(D,\Pi):\mathcal E^{\infty} \to \mathcal F^{\infty} \oplus H
	\end{equation*}
	is invertible.
\end{asm}
Let $\delta > 0$ be sufficiently small so that the operator $(D,\Pi)$ is invertible on the weighted Sobolev space.
Let $C_0$ be a constant at least as large as the operator norm of the inverse of $(D,\Pi)$ in both the unweighted and weighted spaces. That is, for all $u$,
\begin{equation}
	\label{eqn:C0}
\begin{aligned}
	\|u\| &\le C_0(\|Du\| + \|\Pi u\|), \\
	\|u\|_{\delta} &\le C_0(\|Du\|_{\delta} + \|\Pi u\|),
\end{aligned}
\end{equation}

We shall consider a nonlinear operator $D+\alpha$ having $D$ as its linearization at $u = 0$.
Suppose there is a continuous map
\begin{equation*}
	\alpha_0: C^{\infty}(Y;E_0) \to L^2(Y;E_0),
\end{equation*}
and that
\begin{equation*}
	\alpha: C^{\infty}(\reals \times Y; E) \to
	L^2_{\loc}(\reals \times Y;E)
\end{equation*}
is defined by restriction to slices $\{t\} \times Y$, both of which commute with $\uptau$ and define operators on the respective invariant sections.
We also consider $\alpha$ operators over compact intervals.
\begin{asm}
\label{asm:smoothandvanishsecondorder}
	We suppose that $\alpha$ defines a smooth map
	\begin{equation*}
		\alpha: L^2_{k}([-1,1] \times Y;E)^{\upiota} 
		\to
		L^2_{k-1}([-1,1] \times Y;E)^{\upiota}.
	\end{equation*}
	We assume also that $\alpha(0) = 0$ and $\mathcal D_0 \alpha = 0$.
\end{asm}

By Assumption~\ref{asm:smoothandvanishsecondorder} and proof of Lemma~\ref{lem:pSWsmooth_taumodel} 
(corresponding to \cite[Lemma~14.4.1]{KMbook2007}), $\alpha$ defines a smooth map $\mathcal E^T \to \mathcal F^T$ on the finite cylinders, and smooth map over the infinite cylinders.
Moreover, since $\alpha$ is $C^1$ in both $\mathcal E^{\infty}$ and $\mathcal E^{\infty}_{\delta}$, and has vanishing derivatives at the origin, it is uniformly Lipshitz with small Lipshitz constants on small balls about $0$.
That is, for any $\epsilon > 0$, we can find an $\eta > 0$, such that for all $u, u'$ in $\mathcal E^{\infty}$, 
\begin{equation*}
	\|u\|,\|u'\| \le \eta \implies
	\|\alpha(u) - \alpha(u')\| \le \epsilon \|u - u'\|,
\end{equation*}
and for all $u, u' \in \mathcal E^{\infty}_{\delta}$,
\begin{equation*}
	\|u\|_{\delta},\|u'\|_{\delta} \implies
	\|\alpha(u) - \alpha(u')\|_{\delta} \le \epsilon
	\|u - u'\|_{\delta}.
\end{equation*}
\begin{asm}
	We will suppose $\eta_1 > 0$ is chosen so that the above Lipshitz property holds with $\epsilon = 1/(2C_0)$, where $C_0$ is the constant in ~\eqref{eqn:C0}.
\end{asm}
Finally, we consider the zeros of the maps
\begin{align*}
	F^T &= D+ \alpha : \mathcal E^T \to \mathcal F^T,\\
	F^{\infty} &= D + \alpha : \mathcal E^{\infty} \to \mathcal F^{\infty},
\end{align*}
and write
\begin{align*}
	M(T) &= (F^T)^{-1}(0) \subset \mathcal E^T,\\
	M(\infty) &= (F^{\infty})^{-1}(0) \subset \mathcal E^{\infty}.
\end{align*}
We are ready to state the following abstract gluing theorem, which is the $\uptau$-invariant analogue of \cite[Theorem~18.3.5]{KMbook2007}.
\begin{thm}
	For $T \ge T_0$, the solution sets $M(T)$ and $M(\infty)$ are Hilbert submanifolds of $\mathcal E^T$ and $\mathcal E^{\infty}$ in a neighbourhood of zero. There is an $\eta > 0$ and smooth maps from the $\eta$-ball in the Hilbert space $H$ to the solution sets,
	\begin{align*}
		u(T,-) &: B_{\eta}(H) \to M(T) \\
		u(\infty,-) &: B_{\eta}(H) \to M(\infty),	
	\end{align*}
	which are diffeomorphisms onto their images, and which satisfy
	\begin{equation*}
		\Pi u (T,h) =\Pi u(\infty,h) = h.
	\end{equation*}
	Furthermore, for $T \in [T_0,\infty]$, the map
	\begin{equation*}
		\mu_T: B_{\eta}(H) \to L^2_{k-1/2}(Y \sqcup \bar Y;E_0)^{\uptau}
	\end{equation*}
	defined by composing $u(T,-)$ with the restriction maps to the boundary,
	\begin{equation*}
		\mu_T(h) = r u(T,h)
	\end{equation*}
	is a smooth embedding of $B_{\eta}(H)$.
	As a function on $[T_0,\infty) \times B_{\eta}(H)$, the map $(T,h) \mapsto \mu_T(h)$ is smooth for finite $T$; and $\mu_T$ converges to $\mu_{\infty}$ in the $C^{\infty}_{\loc}$-topology as $T \to \infty$.
	Finally, there is an $\eta' > 0$, independent of $T$, such that the images of the maps $u(T,-)$ contain all solutions $u \in M(T)$ with $\|u\| \le \eta'$.
\end{thm}
\begin{proof}
	The proof in \cite[Section~18.3]{KMbook2007} applies, under the new assumptions.
\end{proof}
\subsection{Gluing gauge-theoretic trajectories near critical points}
\hfill \break
Let $\mathfrak a \in \tilde{\mathcal C}^{\sigma}_k(Y)$ be a critical point of $(\grad \pertL)^{\sigma}$, and $\upgamma_{\mathfrak a}$ be the translation invariant solution on the cylinder, in temporal gauge.
Let $T > 0$ and $Z^T = [-T, T] \times Y$.
We may treat $\upgamma_{\mathfrak a}$ as an element in $\tilde M(Z^T)$.
Let $Z^{\infty}$ be the disjoint union of two cylinders:
\begin{equation*}
	Z^{\infty} = (\reals^{\le} \times Y) \sqcup (\reals^{\ge} \times Y),
\end{equation*}
where we heuristically regard the two boundary components with $\{-T\} \times Y$ and $\{T\} \times Y$, as $T \to \infty$ and the cylinder becomes broken in the middle.
Let $\tilde M(Z^{\infty},[\mathfrak a])$ be the space of solutions that are asymptotic to $[\mathfrak a]$ on both ends, which can either be regarded as in $L^2_k$ or $L^2_{k,\loc}$.
In particular,
\begin{equation*}
	\tilde M(Z^{\infty},[\mathfrak a]) = \tilde M(\reals^{\ge} \times Y,[\mathfrak a]) \times \tilde M(\reals^{\le} \times Y,[\mathfrak a]).
\end{equation*}
The boundary of $Z^T$ and $Z^{\infty}$ are both $Y \sqcup \bar Y$, and there are restriction maps
\begin{align*}
	R: \tilde M(Z^{T},[\mathfrak a]) &\to \tilde B^{\sigma}_{k-1/2}(Y \sqcup \bar Y),\\
	R: \tilde M(Z^{\infty},[\mathfrak a]) &\to \tilde B^{\sigma}_{k-1/2}(Y \sqcup \bar Y),	
\end{align*}
defined as the $L^2_{k-1/2}$-completion of the complement $K^{\sigma}_{k,\mathfrak a}$ to the gauge group orbit.

\begin{thm}
\label{thm:gauge_glue_crit}
	There exists $T_0$ such that for all $T \ge T_0$, we can find smooth maps
	\begin{align*}
		u(T,-) &: B(\mathcal K) \to \tilde M(Z^T),\\
		u(\infty,-) &: B(\mathcal K) \to \tilde M(Z^{\infty}),
	\end{align*}
	which are diffeomorphisms from a ball $B(\mathcal K) \subset \mathcal K$ onto neighbourhoods of the constant solution $[\upgamma_{\mathfrak a}]$.
	These can be chosen so that the map
	\begin{equation*}
		\mu_T:B(\mathcal K) \to \tilde B^{\sigma}_{k-1/2}(Y \sqcup 
		\bar Y)^{\uptau}
	\end{equation*}
	defined by composing $u(T,-)$ with the restriction maps to the boundary
	\begin{equation*}
		\mu_T(h) = Ru(T,h)
	\end{equation*}
	is a smooth embedding of $B(\mathcal K)$ for $T \in [T_0,\infty]$, with the following properties:
	\begin{itemize}[leftmargin=*]
		\item as a function on $[T_0,\infty) \times B(\mathcal K)$, the map $(T,h) \to \mu_T(h)$ is smooth for finite $T$; 
		\item $\mu_T$ converges to $\mu_{\infty}$ in the $C^{\infty}_{\loc}$ topology as $T \to \infty$;
		\item there is an $\eta > 0$, independent of $T$, such that the images of the maps $u(T,-)$ can be taken to contain all solutions $[\upgamma] \in M(Z^T)$ with $\|\upgamma_{\mathfrak a} - \upgamma\|_{L^2_{Z^T}} \le \eta$;
	\end{itemize}
	Finally, in the case that $\mathfrak a$ is reducible, the parametrization $u$ are equivariant for the $\mathbb Z/2$-action arising from the standard $\mathbb Z/2$-action $\mathbf i$ on $\tilde C^{\sigma}_{k-1/2}(Y)$.
\end{thm}

\begin{proof}
	The theorem can be proved by following the arguments in \cite[Section~18.4]{KMbook2007} and replacing objects by the associated $\upiota$-invariant analogues.
	For completeness, let us sketch how the abstract version applies.
	Roughly, we have the following correspondence.
\bgroup
\def\arraystretch{1.5}%
\begin{center}
\begin{tabular}{ |c|c|} 
 \hline
 Abstract version & Gauge theory version
 \\  \hline
 $\mathcal E^T$ & $\tilde C^{\tau}_k(Z^T)$ 
 \\ \hline
 $\mathcal E^{\infty}$ & $\tilde C^{\tau}_k(Z^{\infty};[\mathfrak a])$ 
 \\ \hline
 $\mathcal F^T$ & $V^{\tau}_{k-1}(Z^T)$
 \\ \hline
 $\mathcal F^{\infty}$ & $V^{\tau}_{k-1}(Z^{\infty})$ 
 \\ \hline
 $H$ & $H^-_{\bar Y} \oplus H^-_{Y}$ 
 \\ \hline
 $\Pi$ & $\Pi^-_{\bar Y} \oplus \Pi^-_{Y}$
 \\ \hline
 $D+\alpha$ & $\mathfrak F^{\tau}_{\mathfrak q} + \text{Coul}^{\tau}_{\mathfrak a}$
 \\ \hline
\end{tabular}
\end{center}
\egroup
The Hilbert spaces $H^-_{\bar Y}$ and $H^-_{Y}$  are defined by
\begin{align*}
	H^-_Y &= \{0\} \oplus \mathcal K^- \oplus L^2_{k-1/2}(Y;i\reals)^{-\upiota^*}\\
	H^-_{\bar Y} &= \{0\} \oplus \mathcal K^+ \oplus L^2_{k-1/2}(Y;i\reals)^{-\upiota^*},
\end{align*}
and $\Pi^-_{\bar Y} \oplus \Pi^-_{Y}$ are the corresponding projection maps.
There are two major caveats.
\begin{itemize}
	\item $\tilde C^{\tau}_k(Z^T)$ is not a linear space because of the slice-wise norm constraint, so we must instead work in a local chart;
	\item neither the domain nor the range are sections of finite-dimensional vector bundles.
\end{itemize}
Finally, we refer the readers to \cite[Section~18.4]{KMbook2007} for the verification of Assumptions \ref{asm:inftyinvertible} and \ref{asm:smoothandvanishsecondorder}, and derivation of the rest of Theorem~\ref{thm:gauge_glue_crit} from abstract gluing.
Assumption~\ref{asm:inftyinvertible} is essentially a consequence of the $\uptau$-invariant version of  \cite[Proposition~17.2.6]{KMbook2007}.
\end{proof}

\subsection{Neighbourhoods of the boundary strata}
\begin{defn}
	Let $(Q,q_0)$ be a topological space, let $\pi: S \to Q$ be a continuous map, and let $S_0 \subset \pi^{-1}(q_0)$.
	Then $\pi$ is a \emph{topological submersion} along $S_0$, if for all $s_0 \in S_0$ we can find a neighbourhood $U$ of $s_0 \in S$, a neighbourhood $Q'$ of $q_0$ in $Q$, and a homeomorphism $(U \cap S_0) \times Q' \to U$ such that the diagram
\[\begin{tikzcd}
	{(U \cap S_0) \times Q'} && U \\
	{Q'} && {Q'}
	\arrow[Rightarrow, no head, from=2-1, to=2-3]
	\arrow[from=1-1, to=2-1]
	\arrow[from=1-3, to=2-3]
	\arrow[from=1-1, to=1-3]
\end{tikzcd}\]
	commutes.
\end{defn}
Consider the stratum
\begin{equation}
\label{eqn:unobs_stratum}
	\prod_{i=1}^n \check M([\mathfrak a_{i-1}],[\mathfrak a_i]) \subset \check M^+([\mathfrak a_0],[\mathfrak a_n]),
\end{equation}
and assume none of the moduli spaces involved are boundary-obstructed.
Then the following theorem says that the neighbhood of the stratum~\eqref{eqn:unobs_stratum} looks like a product of itself  with an $(n-1)$-cube.
\begin{thm}
[Boundary-unobstructed]
	Suppose the moduli spaces $M([\mathfrak a_{i-1}],[\mathfrak a_i])$ are boundary-unobstructed for $i = 1,\dots, n$.
	Then there is a neighbourhood $\check W$ of the subset~\eqref{eqn:unobs_stratum} in $\check M^+([\mathfrak a_0],\dots,[\mathfrak a_n])$ and a map
	\begin{equation*}
		\mathbf S: \check W \to (0,\infty]^{n-1}
	\end{equation*}
	such that $\mathbf S^{-1}(\infty,\dots,\infty)$ is the subset~\eqref{eqn:unobs_stratum}, and such that $\mathbf S$ is a toplogical sumbersion along~\eqref{eqn:unobs_stratum}.
\end{thm}
The proof of the boundary-unobstruted case in \cite[Section~19.3]{KMbook2007}, applies word-by-word, with $\uptau$-invariant subspaces.
The boundary-obstructed case below, based on \cite[Theorem~19.4.1]{KMbook2007},  is more elaborate.
\begin{thm}
[Boundary-obstructed]
	Suppose the moduli spaces $M([\mathfrak a_{i-1}],[\mathfrak a_i])$ are boundary-obstructed for $i \in O$ and boundary-unobstructed for $i \in O' = \{1,\dots,n\}\setminus O$.
	Then there is an open set $\check W$ such that 
	\[ 
		\prod_{i=1}^n \check M([\mathfrak a_{i-1},[\mathfrak a_i]) 
		\subset \check W \subset \check M^+([\mathfrak a_0],[\mathfrak a_n])
	\]
	admitting an map
	\begin{equation*}
		\mathbf S: \check W \to (0, \infty]
	\end{equation*}
	which satisfies the following properties.
	\begin{enumerate}[leftmargin=*]
		\item[(i)] There is a topological embedding of $\check W$ in a space $E\check W$ with a map $\mathbf S$ to $(0,\infty]^{n-1}$ such that the following diagram commutes:
	\[\begin{tikzcd}
	{\check W} && {E\check W} \\
	{(0,\infty]^{n-1}} && {(0,\infty]^{n-1}}
	\arrow[Rightarrow, no head, from=2-1, to=2-3]
	\arrow["{\mathbf S}", from=1-1, to=2-1]
	\arrow["{\mathbf S}"', from=1-3, to=2-3]
	\arrow["j"', from=1-1, to=1-3]
	\end{tikzcd}\]
		\item[(ii)] The map $\mathbf S: E\check W \to (0,\infty]^{n-1}$ is a topological submersion along the fibre over $\boldsymbol{\infty}$.
		\item[(iii)] The image $j(\check W) \subset E\check W$ is the zero set of a continuous map
		\begin{equation*}
			\delta:E\check W \to \reals^O
		\end{equation*}
		which vanishes at the fibre over $\boldsymbol{\infty}$.
		In particular therefore, the fibre over $\boldsymbol{\infty}$ in both $\check W$ and in $E\check W$ is identified with the stratum $\prod \check M([\mathfrak a_{i-1}],[\mathfrak a_i])$.
		\item[(iv)] If $\check W^o \subset W$ and $E\check W^o \subset E\check W$ are the subsets where none of the $S_i$ is infinite, then the restriction of $j$ is an embedding of smooth manifolds, and $\delta|_{E\check W^o}$ is transverse to zero,
		\item[(v)] Let $i_0 \in O$, and let $\delta_{i_0}$ be the corresponding component of $\delta$. Then for all $z \in E\check W$, we have
		\begin{itemize}
			\item If $i_0 \ge 2$ and $S_{i_0-1}(z) = \infty$, then $\delta_{i_0}(z) \ge 0$;
			\item If $i_0 \le n-1$ and $S_{i_0} = \infty$, then $\delta_{i_0}(z) \le 0$.
		\end{itemize}
	\end{enumerate}
\end{thm}

\subsection{Codimension-one strata}
\hfill \break
The following definition is modelled on codimension-one strata of Seiberg-Witten trajectory spaces.
\begin{defn}
\label{defn:delta1str}
	Let $N^d$ be a $d$-dimensional space stratified by manifolds and $M^{d-1} \subset N$ be a union of components of the $(d-1)$-dimensonal stratum.
	Then $N$ has a \emph{codimension-1 $\delta$-structure} along $M^{d-1}$ if $M^{d-1}$ is smooth and we have the following additional data:
	\begin{itemize}
		\item an open set $W \subset N$ containing $M^{d-1}$,
		\item an embedding $j:W \to EW$, and
		\item a map
		\begin{equation*}
			\mathbf{S} = (S_1,S_2): EW \to (0,\infty]^2
		\end{equation*}
	\end{itemize}
	satisfying the following:
	\begin{enumerate}[leftmargin=*]
		\item[(i)] the map $\mathbf{S}$ is a topological submersion along the fibre over $(\infty,\infty)$;
		\item[(ii)] the fibre of $\mathbf S$ over $(\infty,\infty)$ is $j(M^{d-1})$;
		\item[(iii)] the subset $j(W) \subset EW$ is the zero set of a map $\delta:EW \to \reals$;
		\item[(iv)] the function $\delta$ is strictly positive where $S_1 = \infty$ and $S_2$ is finite, and strictly negative where $S_2 = \infty$ and $S_1$ is finite;
		\item[(v)] on the subset of $EW$ where $S_1$ and $S_2$ are both finite, $\delta$ is smooth and transverse to zero.
	\end{enumerate}
\end{defn}
An example of $\delta$-structure to keep in mind is the following.
\begin{exmp}
	Let $EW$ be $\{(x_0,x_1) \in \reals^2 \big| x_0,x_1 \ge 0\}$ and
\begin{equation*}
	W = \{(x_0,x_1) \in EW \big| x_0 = x_1 \} \subset EW,
\end{equation*}
$S_i = 1/x_i$, and
\begin{equation*}
	\delta = x_1 - x_2.
\end{equation*}
The fibre of $(\infty,\infty)$ is a point, which is the embedded image of $M^0$, and $W$.
\end{exmp} 
We state the following counterpart of \cite[Theorem~19.5.4]{KMbook2007} about structure of codmension-1 strata in terms of $\delta$-structures.
\begin{thm}
	Suppose the moduli space $M_z([\mathfrak a],[\mathfrak b])$ is $d$-dimensional and contains irreducible trajectories, so that the moduli space $\check M^+_z([\mathfrak a],[\mathfrak b])$ is a $(d-1)$-dimensional space stratified by manifolds.
	Let $M' \subset \check M^+([\mathfrak a],[\mathfrak b])$ be any component of the codimension-1 stratum.
	Then along $M'$, the moduli space $\check M^+([\mathfrak a],[\mathfrak b])$ either is a $C^0$-manifold with boundary, or has a codimension-1 $\delta$-structure.
	The latter only occurs when $M'$ consists of 3-component broken trajectories, with the middle component boundary-obstructed.	
\end{thm}

\subsection{Gluing of reducible solutions}
\hfill \break
The structure theorem for codimension-1 strata in the reducible moduli spaces is easier to state and is based on \cite[Theorem~19.6.1]{KMbook2007}.
\begin{thm}
	Suppose that the moduli space $M^{\red}_z([\mathfrak a],[\mathfrak b])$ is $d$-dimensional and non-empty, so that the compactified moduli space of broken reducible trajectories $\check M^{\red +}([\mathfrak a], [\mathfrak b])$ is a $(d-1)$-dimensional space stratified by manifolds.
	Let $M' \subset \check M^{\red +}([\mathfrak a], [\mathfrak b])$ be any component of the codimension-1 stratum.
	Then along $M'$ the moduli space $\check M^{\red +}([\mathfrak a],[\mathfrak b])$ is a $C^0$-manifold with boundary.
\end{thm}

\section{Floer Homologies}
\label{sec:floer} 
Let $Y$ be a compact connected 3-manifold and $ \upiota:Y \to Y$ be an involution.
Let $g$ be an $\upiota$-invariant Riemannian metric, and $(\mathfrak s, \uptau)$ be a real spin\textsuperscript{c} structure.
Let $\mathcal P = \mathcal P(Y,\mathfrak s, \uptau)$ be a large Banach space of tame perturbations in the sense of Definition~\ref{defn:largebanachspaceperturb}.

Choose $\mathfrak q \in \mathcal P$ so that all critical points of $(\grad \pertL)^{\sigma}$ in $\mathcal B^{\sigma}_k(Y,\mathfrak s, \uptau)$ are nondegenerate, and all moduli spaces $M([\mathfrak a],[\mathfrak b])$ are regular.
In addition, if $c_1(S)$ is not torsion, choose the perturbation so that there are no reducible critical points.

\begin{defn}
	A tame perturbation $\mathfrak q$ is \emph{admissible} for $(Y,g,\mathfrak s,\uptau)$ if the critical points are non-degenerate, the moduli spaces are regular, and there are no reducibles unless $c_1(S)$ is torsion.
\end{defn}

\subsection{Floer Homologies for Three-manifolds with Involutions}
\hfill \break
Let $\mathfrak C \subset \mathcal B^{\sigma}_k(Y;\mathfrak s,\uptau)$ be the set of critical points of the perturbed CSD functional.
We decompose $\mathfrak C$ into irreducible, boundary-stable, and boundary-unstable reducible critical points:
\begin{equation*}
	\mathfrak C = \mathfrak C^o \cup \mathfrak C^s \cup \mathfrak C^u.
\end{equation*}
The critical points generate the following groups over $\mathbb F_2$
\begin{equation*}
	C^o = \bigoplus_{[\mathfrak a] \in \mathfrak C^o} \mathbb F_2[\mathfrak a], \quad 
	C^s = \bigoplus_{[\mathfrak a] \in \mathfrak C^s} \mathbb F_2[\mathfrak a], \quad 
	C^u = \bigoplus_{[\mathfrak a] \in \mathfrak C^u} \mathbb F_2[\mathfrak a],
\end{equation*}
and chain complexes
\begin{equation*}
	\check C = C^o \oplus C^s, \quad
	\hat C = C^o \oplus C^u, \quad
	\bar C = C^s \oplus C^u.
\end{equation*}
We define the differential $\bar\del:\bar C \to \bar C$ by
\begin{equation*}
	\bar\del [\mathfrak a] = \sum_{[\mathfrak b],z} \# \check M_z^{\text{red}}([\mathfrak a],[\mathfrak b]) \cdot \mathfrak [\mathfrak b],
\end{equation*}
where $[\mathfrak a],[\mathfrak b] \in (\mathfrak C^s \cup \mathfrak C^u)$, and $\check M_z^{\text{red}}([\mathfrak a],[\mathfrak b])$ are unparametrized moduli spaces of \emph{reducible} trajectories with dimension zero.
In components $\bar C = C^s \oplus C^u$, the differential $\bar\del$ looks like
\begin{equation*}
	\bar\del = \begin{pmatrix}
		\bar\del^s_s && \bar\del^u_s\\
		\bar\del^s_u && \bar\del^u_u
	\end{pmatrix}.
\end{equation*}
Similarly, we define
\begin{align*}
	\del^o_o: C^o \to C^o \quad [\mathfrak a] \mapsto \sum_{[\mathfrak b] \in \mathfrak C^o} \#\check M_z([\mathfrak a],[\mathfrak b]) \cdot [\mathfrak b],\\
	\del^o_s: C^o \to C^s \quad [\mathfrak a] \mapsto \sum_{[\mathfrak b] \in \mathfrak C^s} \#\check M_z([\mathfrak a],[\mathfrak b]) \cdot [\mathfrak b],\\
	\del^u_o: C^u \to C^o \quad [\mathfrak a] \mapsto \sum_{[\mathfrak b] \in \mathfrak C^o} \#\check M_z([\mathfrak a],[\mathfrak b]) \cdot [\mathfrak b],\\
	\del^u_s: C^u \to C^s \quad [\mathfrak a] \mapsto \sum_{[\mathfrak b] \in \mathfrak C^s} \#\check M_z([\mathfrak a],[\mathfrak b]) \cdot [\mathfrak b],
\end{align*}
by counting points in zero-dimensional moduli spaces modulo two.
\begin{defn}
	On the ``from'' chain group $\check C = C^o \oplus C^s$ the differential $\check\del : \check C \to \check C$ is
	\begin{equation*}
		\check\del = 
		\begin{pmatrix}
			\del^o_o && -\del^u_o\bar\del^s_u\\
			\del^o_s && \bar\del^s_s - \del^u_s\bar\del^s_u
		\end{pmatrix}.
	\end{equation*}	
	On the ``to'' chain group $\hat C = C^o \oplus C^u$, the differential $\hat C : \hat C \to \hat C$ is
	\begin{equation*}
		\hat\del = \begin{pmatrix}
			\del^o_o && \del^u_o\\
			-\bar\del^s_u\del^o_s &&
			-\bar\del^u_u\del^u_s
		\end{pmatrix}.
	\end{equation*}
\end{defn}
\begin{prop}
	The squares	$(\bar\del)^2,(\check\del)^2$, and $(\hat\del)^2$ are zero.
\end{prop}
\begin{proof}
	The proposition is formally the same as \cite[Proposition~22.1.4]{KMbook2007}.
	The ingredients of proof have $\uptau$-invariant analogues in Section~\ref{sec:compact} and Section~\ref{sec:Gluing}.
\end{proof}
Now, apply homology to the chain complexes, and suppress the perturbation $\mathfrak q$ in the notation.
\begin{defn}
	We define the three flavours of Floer homology as the homology groups of the respective chain complexes:
\begin{equation*}
	\widecheck{\HMR}(Y;\mathfrak s, g,\uptau) = H_*(\check C,\check \del), \quad
	\widehat{\HMR}(Y;\mathfrak s, g,\uptau) = H_*(\hat C,\hat \del), \quad
	\overline{\HMR}(Y;\mathfrak s, g,\uptau) = H_*(\overline C,\bar \del).
\end{equation*}
\end{defn}
We denote $\HMR^{\circ}$ as one of the flavours of the real monopole Floer homologies, where $\circ \in \{\bigvee, \bigwedge, -\}$.
The following is a Floer-theoretic version of the homology long exact sequence of pairs for $(\mathcal B^{\sigma}, \del \mathcal B^{\sigma})$:
\begin{prop}
	For any $(Y;\mathfrak s, \uptau)$, there is an exact sequence
	\begin{equation*}
	\begin{tikzcd}
		\ar[r,"i_*"] 
		& \widecheck{\HMR}(Y;\mathfrak s,g,\uptau)	 \ar[r,"j_*"] 
		& \widehat{\HMR}(Y;\mathfrak s, g,\uptau) \ar[r,"p_*"]
		& \overline{\HMR}(Y;\mathfrak s, g,\uptau) \ar[r,"i_*"]
		& \widehat{\HMR}(Y;\mathfrak s, g,\uptau) \ar[r]
		&{}
	\end{tikzcd}
\end{equation*}
in which the maps $i_*, j_*$ and $p_*$ arise from the (anti-) chain maps
\begin{equation*}
	i: \bar C \to \check C,\quad
	j: \check C \to \hat C, \quad
	p: \check C \to \bar C,
\end{equation*}
given by
\begin{equation*}
	i = \begin{pmatrix}
		0 & \del^u_o \\
		1 & -\del^u_s
	\end{pmatrix},
	\quad 
	j = \begin{pmatrix}
		1 & 0 \\
		0 & -\delbar^u_s
	\end{pmatrix}, 
	\quad 
	p = \begin{pmatrix}
		\del^o_s & \del^u_s \\
		0 & 1
	\end{pmatrix}.
\end{equation*}
The first two maps are chain maps, and $p$ satisfies $p\check\del = - \delbar p$.
\end{prop}
\begin{proof}
	See \cite[Proposition~22.2.1]{KMbook2007}.
\end{proof}

\begin{defn}
	We define the reduced Floer homology group $\HMR_*(Y;\mathfrak s, \uptau)$ as the image of the map
	\begin{equation*}
		j_*: \widehat{\HMR}_*(Y;\mathfrak s,\uptau) \to 
		\widecheck{\HMR}_*(Y;\mathfrak s,\uptau)
	\end{equation*}	
	The reduced $\HMR$ has finite rank, by the same arguments as \cite[Proposition~22.2.3]{KMbook2007}.
\end{defn}

\subsection{Grading of $\HMR^{\circ}$}
\hfill \break
We grade the Floer homology groups by a group $\mathbb J(\mathfrak s, \uptau)$.
The definition of $\mathbb J(\mathfrak s, \uptau)$ involves an equivalence relation $\sim$ on $\mathcal B^{\sigma}_k(Y,\mathfrak s,\uptau) \times \mathcal P \times \mathbb Z$.
Let $([\mathfrak a],\mathfrak q_1,n)$ and $([\mathfrak b],\mathfrak q_2,n)$ be in $\mathcal B^{\sigma}_k(Y,\mathfrak s,\uptau) \times \mathcal P \times \mathbb Z$.
Let $\zeta$ be a path joining $[\mathfrak a]$ to $[\mathfrak b]$, and $\mathfrak p$ be an one-parameter family of perturbations joining $\mathfrak q_1$ to $\mathfrak q_2$.
There is a Fredholm operator $P_{\upgamma,\mathfrak p}$ associated to $\zeta$ and $\mathfrak p$, described in \cite[Section~20]{KMbook2007}, which roughly is
\begin{equation*}
	\mathcal Q_{\upgamma} \oplus -\Pi^+ \oplus \Pi^-,
\end{equation*}
over the cylinder, and where the $\mathcal Q_{\upgamma}$ is the linearization of perturbed Seiberg-Witten operator coupled with Coulomb gauge, and $\Pi^{\pm}$ are some spectral projections (such that the boundary value problem is Fredholm.)

We declare $([\mathfrak a],\mathfrak q_1,m) \sim ([\mathfrak b],\mathfrak q_2,n)$ if there exists paths $\zeta$ and $\mathfrak p$ such that
	\begin{equation*}
		\ind (P_{\upgamma,\mathfrak p}) = n - m.
	\end{equation*}
	When $[\mathfrak a]$ and $[\mathfrak b]$ are critical points of the same perturbation, the index of $P_{\upgamma}$ is equal to $\gr_z([\mathfrak a],[\mathfrak b])$.
The grading group is the quotient
	\begin{equation*}
		\mathbb J(\mathfrak s, \uptau) = (\mathcal B^{\sigma}_k(Y,\mathfrak s,\uptau)\times \mathcal P \times \mathbb Z)/\sim.
	\end{equation*}
The action of $\mathbb Z$ on the quotient is given by $([\mathfrak a],\mathfrak q,m) \mapsto ([\mathfrak a],\mathfrak q,m+n)$ for $n \in \mathbb Z$.
The grading of a critical point $[\mathfrak a]$ of the perturbed CSD function by $\mathfrak q$ is then by definition, the equivalence class
	\begin{equation*}
		\gr[\mathfrak a] = ([\mathfrak a],\mathfrak q,0)/\sim.
	\end{equation*}
We write $j+n$ for the resulting element of $n \in \mathbb Z$ acting on $j \in \mathbb J$.
This grading is additive in the sense that
	\begin{equation*}
		\gr[\mathfrak a]= \gr[\mathfrak b] + \gr_z([\mathfrak a],[\mathfrak b]) \in \mathbb J(\mathfrak s).
	\end{equation*}
In the reducible case, we introduce another grading
	\begin{equation*}
		\bar\gr[\mathfrak a] = \begin{cases}
			\gr[\mathfrak a], & [\mathfrak a]\in \mathfrak C^s\\
			\gr[\mathfrak a]-1, & [\mathfrak a] \in \mathfrak C^u
		\end{cases}.
	\end{equation*}
The chain complexes decomposes according to the grading $j \in \mathbb J(\mathfrak s)$:
	\begin{align*}
		\check{C}_j &= \mathbb F_2 \left\{[\mathfrak a] \in \mathfrak C^o \cup \mathfrak C^s: \gr[\mathfrak a]= j\right\},\\
		\hat{C}_j &= \mathbb F_2 \left\{[\mathfrak a] \in \mathfrak C^o \cup \mathfrak C^u: \gr[\mathfrak a]= j\right\},\\
		\bar{C}_j &= \mathbb F_2 \left\{[\mathfrak a] \in \mathfrak C^s \cup \mathfrak C^u: \bar\gr[\mathfrak a]= j\right\}.
	\end{align*}
	If we grade all chain groups $C^o, C^s, C^u$ by the grading $\gr$, then
	\begin{equation*}
		\check{C}_j = C^o_j \oplus C_j^s, \quad
		\hat{C}_j = C^o_j \oplus C^u_j, \quad
		\bar{C}_j = C^s_j \oplus C^u_j.
	\end{equation*}

\begin{lem}
\label{lem:ZactsonJ}
	The action of $\mathbb Z$ on $\mathbb J(\mathfrak s)$ is transive and the stablizer is the image of the map
	\begin{align*}
		H_2(Y;\mathbb Z)^{-\upiota^*} 
		&\to \mathbb Z\\
		[\sigma] &\mapsto
		\frac{1}{2}\langle c_1(\mathfrak s),[\sigma]\rangle.
	\end{align*}
	In particular, the action is free if and only if $c_1(\mathfrak s)$ is torsion.
\end{lem}
	\begin{proof}
		Transitivity is by definition and the fact that $\mathcal B^{\sigma}_k(Y,\mathfrak s,\uptau)$ is connected.
		The formula about the stablizer follows from Lemma~\ref{lem:loopindex1}.
		In particular, unlike the ordinary case, there is a $(1/2)$-factor.
		Compare this with \cite[Lemma~23.2.2]{KMbook2007}.
\end{proof}
\begin{rem}
	By the extra $(1/2)$-factor in Lemma~\ref{lem:ZactsonJ}, 
	there is no analogus absolute $\mathbb Z/2$-gradings in \cite[Section~22.4]{KMbook2007} on the real monopole Floer homology groups.
\end{rem}

\begin{lem}
		The free abelian groups $\check C_j$, $\hat C_j$, and $\bar C_j$ are all finitely generated.
\end{lem}
\begin{proof}
		The lemma follows from the compactness results downstairs and the fact that fibres above reducible all have different gradings.
		See \cite[Lemma~22.3.3]{KMbook2007}.
		
\end{proof}

We introduce the cochain complexes, graded by $\mathbb J(\mathfrak s, \uptau)$, by applying Hom to $\mathbb Z_2$:
\begin{equation*}
	\check C^j = \Hom(\check C_j,\mathbb Z_2), \quad
	\hat C^j = \Hom(\hat C_j, \mathbb Z_2), \quad 
	\bar C^j = \Hom(\bar C_j,\mathbb Z_2).
\end{equation*}
These give rise to the cohomology groups
\begin{equation*}
	\widehat{\HMR}^j(Y;\mathfrak s,\uptau), \quad 
	\widecheck{\HMR}^j(Y;\mathfrak s,\uptau), \quad 
	\overline{\HMR}^j(Y;\mathfrak s,\uptau),
\end{equation*}
related by a long exact sequence
\begin{prop}
	For any $(Y;\mathfrak s, \uptau)$, there is an exact sequence
\begin{equation*}
	\begin{tikzcd}
		{...}
		& \widecheck{\HMR}^k \ar[l,"i^*"] 
		& \widehat{\HMR}^k \ar[l,"j^*"]
		& \overline{\HMR}^{k-1} \ar[l,"p^*"]
		& \widehat{\HMR}^{k-1} \ar[l,"i^*"]
		&{...} \ar[l,"j^*"]
	\end{tikzcd}
\end{equation*}
\end{prop}
The \emph{reduced Floer cohomology} $\HMR^l(Y;\mathfrak s;\uptau)$ is the image of $j^*$.
\begin{defn}
	We write $\widehat{\HMR}^*(Y;\mathfrak s,\uptau), \widecheck{\HMR}^*(Y;\mathfrak s,\uptau),
	\overline{\HMR}^*(Y;\mathfrak s,\uptau), \HMR^*(Y;\mathfrak s,\uptau)$ for the total Floer homologies, as the direct sums, e.g.
	\begin{equation*}
		\widehat{\HMR}^*(Y;\mathfrak s,\uptau) = \bigoplus_{j \in \mathbb J(\mathfrak s, \uptau)} \widehat{\HMR}^j(Y;\mathfrak s,\uptau).
	\end{equation*}
	For homologies, we write, for instance
	\begin{equation*}
		\widehat{\HMR}_*(Y;\mathfrak s,\uptau) = \bigoplus_{j \in \mathbb J(\mathfrak s, \uptau)} \widehat{\HMR}_j(Y;\mathfrak s,\uptau).
	\end{equation*}
\end{defn}

\subsection{Three-sphere}
\hfill \break
Let $Y = S^3$ be the three-sphere, and let $\upiota$ be the involution on $S^3$ as the covering involution of the double branched cover along the unknot.
Concretely, $\upiota$ is the conjugation of coordinates on $\C^2$, where $S^3 \subset \C^2$ is the unit sphere.
In particular, the induced metric from the Euclidean metric on $\C^2$ is $\upiota$-invariant.
Let $\mathfrak s = (S,\rho)$ be the unique spin\textsuperscript{c} structure on $S^3$ and $\uptau: S \to S$ be a compatible real structure.
Since $Y$ has positive scalar curvature, there is a unique critical point $[B,0]$ of the unperturbed Chern-Simons-Dirac functional $\mathcal L$, where $B$ is a spin\textsuperscript{c} connection invariant under $\uptau$.
The spectrum of the Dirac operator
\begin{equation*}
	D_B: \Gamma(S)^{\uptau} \to \Gamma(S)^{\uptau}
\end{equation*} 
is not simple, so we pick a perturbation $\mathfrak q \in \mathcal P$ for which the perturbed Dirac operator $D_{\mathfrak q,B}$ has simple spectrum which contains no zero.

Label the eigenvalues $\{\lambda_i:i \in \mathbb Z\}$ of $D_{\mathfrak q, B}$ in increasing order, so that $\lambda_0$ is the smallest positive eigenvalue.
Let $\{[\mathfrak a_i]\} \subset \mathcal B^{\sigma}(Y,\mathfrak s, \uptau)$ be the reducible critical points of $(\grad \pertL)^{\sigma}$.
In particular, $\mathfrak C^s$ consists of $[\mathfrak a_i]$ with $i \ge 0$, and $\mathfrak C^u$ consists of $[\mathfrak a_i]$ with $i < 0$.
As for the grading, $\mathbb Z$ acts freely and transitively on $\mathbb J$, and so we identify $\mathbb J(\mathfrak s, \uptau)$ with $\mathbb Z$ so that $\gr[\mathfrak a_0] = 0$.
For any homotopy class $z$:
\begin{equation*}
	\gr_z([\mathfrak a_i],[\mathfrak a_{i-1}]) = 
	\begin{cases}
			1 & \text{if $\lambda_i$ and $\lambda_{i-1}$ have the same sign}\\
			0 & \text{if $i=0$}
	\end{cases}.
\end{equation*}
As a result,
\begin{equation*}
	\gr[\mathfrak a_i] = 
	\begin{cases}
		i & i \ge 0\\
		i+1 & i < 0
	\end{cases},
\end{equation*}
and $\bar\gr[\mathfrak a_i] = i$ for all $i$.
It follows that
\begin{itemize}
	\item $\check C_j = \mathbb F_2$ for all $j$ nonegative,
	\item $\hat C_j = \mathbb F_2$ for all $j$ nonpositive,
	\item $\overline{C}_j = \mathbb F_2$ for all $j$.
\end{itemize}
Unlike the usual monopole Floer homology, the adjacent gradings differ only by one instead of two.
But by Proposition~\ref{prop:modspaceRPi}, the unparameterized moduli space of trajectories between neighbouring critical points consists of two points, and the differentials are all zero.
Thus the Floer homology groups are isomorphic to their respective chain groups.

\subsection{Floer homology for links}
\hfill \break
Assume $(Y,\upiota)$ is the double branched cover of $S^3$ along a link $K$, equipped with the covering involution $\upiota$.
Let $g$ be an $\upiota$-invariant Riemannian metric and $\mathfrak s$ be any spin\textsuperscript{c} structure on $Y$.
By Lemma~\ref{lem:cohom_char_real_str}, any spin\textsuperscript{c} structure on the cover admits a real structure $\uptau$.
Moreover, all real structures on $\mathfrak s$ are equivalent.
For such choices of $(Y,\upiota,g,\mathfrak s, \uptau)$, we choose an admissible perturbation $\mathfrak q \in \mathcal P(K,\mathfrak s,\uptau)$.

\begin{defn}
	We define the Floer homology $\HMR(K, \mathfrak s,\uptau, \mathfrak q)$ of the link $K$, as the Floer homology of its double branched cover $Y$ with covering involution $\upiota$:
\begin{align*}
	\widecheck{\HMR}(K,\mathfrak s, \uptau, \mathfrak q) 
	&:= \widecheck{\HMR}(Y,\mathfrak s, g,\uptau, \mathfrak q), 
	\\
	\widehat{\HMR}(K,\mathfrak s,\uptau,\mathfrak q)
	&:= \widehat{\HMR}(Y,\mathfrak s, g,\uptau,\mathfrak q),
	\\
	\overline{\HMR}(K,\mathfrak s,\uptau,\mathfrak q) 
	&:= \overline{\HMR}(Y,\mathfrak s, g,\uptau,\mathfrak q). 
\end{align*}
\end{defn}

\subsection{The completion $\HMR^{\circ}_{\bullet}$}
\hfill \break
We shall see that the cobordism maps in $\HMR_*^{\circ}$ may have infinitely many entries along the negative direction.
There is a need to take the ``negative completion''.
(Dually, the cobordism maps in the Floer cohomology may have infinitely many entries in the positive degrees.)
For example, the negative completion of the ring of finite Laurent polynomials $\mathbb F_2[\upsilon^{-1},\upsilon]$ is the ring of Laurent series $\Ftwo[[\upsilon^{-1},\upsilon]$ which is infinite in the negative degree.
 
In general, let $G_*$ be an abelian group graded by a set $\mathbb J$.
Let $O_{\alpha}$, $\alpha \in A$ be the free $\mathbb Z$-orbits in $\mathbb J$, and choose an element $j_{\alpha} \in O_{\alpha}$ for each $\alpha$.
Let $G_*[n] \subset G_*$ be the subgroup
\begin{equation*}
	G_*[n] = \bigoplus_{\alpha}\bigoplus_{m \ge n} G_{j_{\alpha} - m}.
\end{equation*}
This defines a decreasing filtration for $G_*$.
We define the negative completion of $G_8$ to be the topological group $G_{\bullet} \supset G_*$ obtained by completing with respect to this filtration.
This is an analogous construction of positive completion where we replace $j_{\alpha} - m$ by $j_{\alpha} + m$ in the above definition.
Finally, we apply the positive/negative completion to the grading $\mathbb J = \mathbb J(Y,\mathfrak s,\uptau)$.
\begin{defn}
	The Floer homology group $\HMR^{\circ}_{\bullet}(Y,\mathfrak s,\uptau)$ is defined to be the completion of $\HMR^{\circ}_*$ in the negative completion of $\mathbb J(Y,\mathfrak s,\uptau)$.
	The Floer cohomology group $\HMR^{\circ,\bullet}(Y,\mathfrak s,\uptau)$ is defined to be the completion of $\HMR^{\circ, *}$ in the positive completion of $\mathbb J(Y,\mathfrak s,\uptau)$.
\end{defn}

\section{Moduli Spaces over Manifolds with Boundaries}
\label{sec:mod_mfld_boundary}
This section is modelled on \cite[Section~24]{KMbook2007}. 
We study various aspects of Seiberg-Witten moduli spaces over general real 4-manifolds with boundaries, from the Fredholm theory, perturbations, and transversality, to the structure of their compactifications.
We also discuss similar results for families of moduli spaces, which are needed for later applications.

\subsection{Perturbations and moduli spaces}
\subsubsection*{(\textbf{Compact with boundary})}
Let $X$ be a compact, connected, oriented 4-manifold with non-empty boundary $\del X = Y$.
Let $\upiota_X$ be an involution which preserves the components of the boundary.
Equip $X$ with an $\upiota_X$-invariant Riemannian metric.
For the restriction of the involution on the boundary $Y$, we write $\upiota: Y \to Y$.
Assume $X$ contains an equivariant isometric copy of $(I \times Y, 1_I \times \upiota)$ for some interval $I = (-C,0]$ with $\del X$ identified with $Y$.
We label the boundary components by $\alpha$:
\begin{equation*}
	Y = \bigsqcup_{\alpha} Y^{\alpha}.
\end{equation*}
Write the configuration space and Banach spaces of tame perturbation as
\begin{equation*}
	\mathcal B^{\sigma}_k(Y,\mathfrak s,\uptau) =
	\prod \mathcal B^{\sigma}_k(Y^{\alpha},\mathfrak s^{\alpha},\uptau^{\alpha}),
	\quad
	\mathcal P^{\sigma}_k(Y,\mathfrak s,\uptau) 
	=\prod 
	\mathcal P^{\sigma}_k(Y^{\alpha},\mathfrak s^{\alpha},\uptau^{\alpha}).
\end{equation*}
The Seiberg-Witten map on $X$ is a smooth section
\begin{equation*}
	\mathfrak F^{\sigma}:\mathcal C^{\sigma}_k(X,\mathfrak s_X, \uptau_X) \to \mathcal V^{\sigma}_{k-1}.
\end{equation*}

Let us perturb $\mathfrak F^{\sigma}$ on the cylindrical regions.
Let $\beta$ be a cut-off function, equal to $1$ near $t = 0$ and to $0$ near $t = -C$.
Let $\beta_0$ be a bump-function with compact support in $(-C,0)$.
Let $\mathfrak q$ and $\mathfrak p_0$ be two elements in $\mathcal P(Y,\mathfrak s,\uptau)$.
Define the section $\hat{\mathfrak p}:\mathcal C_k(X,\mathfrak s_X) \to \mathcal V_k$ by the formula
\begin{equation*}
	\hat{\mathfrak p} = \beta \hat{\mathfrak q} + \beta_0 \hat{\mathfrak p}_0.
\end{equation*} 
 Over the blown-up configuration space $\mathcal C_k^{\sigma}(X,\mathfrak s_X)$, we define the perturbation $\hat{\mathfrak p}^{\sigma}$ by
 \begin{align*}
 	\hat{\mathfrak p}^{0,\sigma}(A,s,\phi) &= \hat{\mathfrak p}^0(A,s\phi),\\
 	\hat{\mathfrak p}^{1,\sigma}(A,s,\phi) &=
 		(1/s)\hat{\mathfrak p}^1(A,s\phi),
 \end{align*}
 when $s \ne 0$ and where the $\{0,1\}$ denote the components of $\mathcal V_k = L^2_k(X;i\mathfrak{su}(S^+)\oplus S^-)^{-\uptau_X}$.
Denote the perturbed Seiberg-Witten operator by $\mathfrak F^{\sigma}_{\mathfrak p}$, and the moduli space of solutions of the perturbed Seiberg-Witten equations by
\begin{equation*}
	M(X,\mathfrak s_X,\uptau_X) = \{(A,s,\phi)|\mathfrak F^{\sigma}_{\mathfrak p} = 0\}/\mathcal G_{k+1}(X,\uptau_X).
\end{equation*}
Similarly, we have the larger space $\tilde M(X, \mathfrak s_X, \uptau_X) \subset \tilde B_k^{\sigma}(X,\mathfrak s_X,\uptau_X)$ obtained by dropping the condition $s \ge 0$.
\begin{prop}
\label{prop:mfld_boundary_Hilbert}
	The section $\mathfrak F^{\sigma}_{\mathfrak p}$ of $\mathcal V^{\sigma}_{k-1}$ is transverse to zero, and the subset $\tilde M(X,\mathfrak s_X) \subset \tilde B_k^{\sigma}(X,\mathfrak s_X)$ is a smooth Hilbert manifold.
	The moduli space $M(X,\mathfrak s_X)$ is therefore a Hilbert manifold with boundary, and can be idenitified with the quotient of $\tilde M(X,\mathfrak s_X)$ by the involution $s \mapsto -s$.	
\end{prop}
\begin{proof}
	The proof of \cite[Proposition 24.3.1]{KMbook2007} translates to our setup.
	Let us give a sketch.
	As before, it suffices to prove the subjectivity of an larger operator 
\[
\mathcal Q_{\gamma}^{\sigma} = \mathcal D_{\upgamma} \mathfrak F^{\sigma}_{\mathfrak p} \oplus 
	\mathbf{d}^{\sigma, \dag}_{\upgamma},
\]
	which is the linearization of the perturbed Seiberg-Witten operator coupled with a Coulomb gauge condition.
	The proof of surjectivity is based on \cite[Corollary 17.1.5]{KMbook2007}, which roughly states that  for perturbation of a elliptic operator acting on sections of vector bundles
	\begin{equation*}
		D = D_0 + K,
	\end{equation*}
	if $D^*$ has the property that every non-zero restriction of $D^*v= 0$ has non-zero restriction to the boundary $Y$, then $D$ is surjective.
	One then use the usual trick to cast the operator $\mathcal Q^{\sigma}_{\gamma}$ as an operator between fixed vector bundles.
	Finally, to verify the hypothesis of the corollary one appeals to the unique continuation properties.
\end{proof}
Moreover, for the restriction map 
\begin{equation*}
	R: \tilde M(X,\mathfrak s_X,\uptau_X) \to 
	\tilde{\mathcal B}^{\sigma}_{k-1/2} (Y,\mathfrak s, \uptau),
\end{equation*}
we have the analogue of Theorem~\ref{prop:cyinder_finite_fredholm}, where the proof is essentially the same as the cylindrical case:
\begin{prop}
\label{prop:mfld_boundary_res_fred}
	Let $X$ be a compact 4-manifold with boundary $Y$. 
	Let $[\upgamma] \in \tilde M(X,\mathfrak s_X)$ and let $[\mathfrak a] \in \mathcal B^{\sigma}_{k-1/2}(Y,\mathfrak s)$ be the restriction of $[\upgamma]$ to the boundary.
	Let $\pi$ be the projections of $\mathcal K^{\sigma}_{k-1/2,\mathfrak a}$ to $\mathcal K^{-}_{\mathfrak a}$, with kernel $\mathcal K^{+}_{\mathfrak a}$.
	Then the linear operators
	\begin{align*}
		\pi \circ \mathcal D R 
		&: T_{[\upgamma]}\tilde M(X,\mathfrak s_X) \to \mathcal K^-_{\mathfrak a},\\
		(1-\pi) \circ \mathcal D R 
		&: T_{[\upgamma]}\tilde M(X,\mathfrak s_X) \to \mathcal K^+_{\mathfrak a},
	\end{align*}
	are Fredholm and compact, respectively.
	\hfill \qedsymbol
\end{prop}
\subsubsection*{(\textbf{Cylindrical ends})}
Next, we attach (equivariant) cylindrical ends to $X$ at boundaries.
Let $Z = [0,\infty) \times Y$, equipped with the involution acting trivially on $[0,\infty)$ and $\upiota$ on the $Y$ factor.
The manifold with cylindrical end is 
\begin{equation*}
	X^* = X \cup_Y Z,
\end{equation*}
and we define the $L^2_{k,\loc}$-configuration space
\begin{equation*}
	\mathcal C_{k,\loc}(X^*,\mathfrak s_X) =
	\mathcal A_{k,\loc} \times L^2_{k,\loc}(X^*,S^+).
\end{equation*}
Due to the absence of a $\tau$-model and $L^2_{k,\loc}$ not being Banach, the blown-up configuration spaces must be treated as in \cite[Section~6.1]{KMbook2007}.
That is, instead taking the unit sphere of a Hilbert space, we define $\mathbb S$ to be the topological quotient $L^2_{k,\loc}(X^*;S^+)\setminus 0$ by the action of $\reals^+$.
Thus the real blow-up of $L^2_{k,\loc}(X^*;S^+)^{\uptau_X}$ is the set of pairs 
\begin{equation*}
	\left\{(\reals^+\phi, \Phi)| \Phi \in \reals^{\ge}\phi\right\}.
\end{equation*}
The blown-up configuration space is defined as
\begin{equation*}
	\mathcal C^{\sigma}_{k,\loc}(X^*,\mathfrak s_X) =
	\left\{(A,\reals^{\ge}\phi,\Phi): \Phi \in \reals^{\ge}\phi \right\}
	\subset \mathcal A_{k,\loc} \times \mathbb S \times
	L^2_{k,\loc}(X^*,S^+)^{\uptau_X}.
\end{equation*}
The blown-up version of the tangent bundle is defined as
\begin{equation*}
	\mathcal V^{\sigma}_{k-1}
	= \mathcal O(-1)^* \otimes \pi^*(\mathcal V_{k-1}) \to \mathcal C^{\sigma}_{k,\loc}(X^*,\mathfrak s_X,\uptau_X),
\end{equation*}
where $\mathcal O(-1)$ is the tautological real line bundle on $\mathbb S$, and $\pi$ is the blown-down map.
The Seiberg-Witten map is
\begin{align*}
	\mathfrak F^{\sigma}: \mathcal O(-1) 
	&\to \pi^*(\mathcal V_{k-1}),\\
	\mathfrak F^{\sigma}(A,\reals^+\phi, \Phi)(\psi)&=
	\left(\frac{1}{2}\rho(F^+_{A^t}) - (\Phi\Phi^*)_0, D^+_A\psi\right).
\end{align*}
To perturb $\mathfrak F^{\sigma}$,
we start with a perturbation $\hat{\mathfrak p}^{\sigma}$ supported on the collar $(-C,0] \times Y$ and extend it to be $\hat{\mathfrak q}_0^{\sigma}$ on the rest of the cylindrical end.
The perturbed Seiberg-Witten map $\mathfrak F_{\mathfrak p}= \mathfrak F^{\sigma} + \hat{\mathfrak p}^{\sigma}$, as a section 
\begin{equation*}
	\mathcal V^{\sigma}_j \to C^{\sigma}_{k,\loc}(X^*;\mathfrak s_X,\uptau_X),
\end{equation*}
can be defined using the recipe in Section~\ref{sec:perturb}.
The zeros of $\mathfrak F^{\sigma}_{\mathfrak p}$ in $\mathcal B^{\sigma}_{k,\loc}(X^*;\mathfrak s_X,\uptau_X)$ restricts to configuration on $Z$, as a map
\begin{equation*}
	\{[\upgamma] \in \mathcal B^{\sigma}_{k,\loc}(X^*,\mathfrak s_X)| \mathfrak F^{\sigma}_{\mathfrak p}(\upgamma) = 0\} \to 
	\mathcal B^{\tau}_{k,\loc}(Z;\mathfrak s,\uptau_X|_Z).
\end{equation*}
Assume the perturbation $\mathfrak q$ on the cylinder is regular.
\begin{defn}
	Let $[\mathfrak b]$ be a critical point in $\mathcal B^{\sigma}_k(Y,\mathfrak s)$. The moduli space
	\begin{equation*}
		M(X^*,\mathfrak s_X;[\mathfrak b]) \subset
		\mathcal B^{\sigma}_{k,\loc}(X^*,\mathfrak s_X)
\end{equation*}	
	is the set of all $[\upgamma]$ such that $\mathfrak F^{\sigma}_{\mathfrak p}(\upgamma) = 0$ and such that the restriction of $[\upgamma]$ is asymptotic to $[\mathfrak b]$ on the cylindrical end $Z$. 
	Decomposing the moduli space according to spin\textsuperscript{c} structures and lifts, we write
	\begin{equation*}
		M(X^*;[\mathfrak b]) = \bigsqcup_{\mathfrak s_X, \uptau_X} M(X^*;\mathfrak s_X, \uptau_X;[\mathfrak b]).
	\end{equation*}
\end{defn}
We have been using  $\tau$-model throuhgout our analysis of the Seiberg-Witten trajectories over the infinite cylinder.
Since only the $\sigma$-model is available for non-cylinders, it is convenient to adopt the fibre product description of $M(X^*;[\mathfrak b])$, where we use $\sigma$-model on the finite part, and $\tau$-model on the cylindrical part.
To this end, consider the following restriction maps:\begin{align*}
	R_+ : M(X,\mathfrak s_X, \uptau_X) 
	&\to \mathcal B^{\sigma}_{k-1/2}(Y,\mathfrak s, \uptau),\\
	R_- : M(Z,\mathfrak s_Z, \uptau_Z, [\mathfrak b]) 
	&\to \mathcal B^{\sigma}_{k-1/2}(Y,\mathfrak s,\uptau).
\end{align*}
Let $\Fib(R_+,R_-)$ be the fibre product, and consider the restriction map to the two regions:
\begin{equation*}
	\rho: M(X^*,\mathfrak s_X,\uptau_X) \to \Fib(R_+,R_-) \subset M(X,\mathfrak s_X, \uptau_X) \times
	M(Z, \mathfrak s_Z, \uptau_Z;[\mathfrak b]).
\end{equation*}
The fibre product can be shown to be a homeomorphism, by the proof of \cite[Lemma~24.2.2]{KMbook2007}.

\subsubsection*{(\textbf{Family version})}
In neck-stretching arguments or more generally definitions of chain maps in later sections, we must analyze moduli spaces parametrized by a family of Riemmannian metrics, and their degenerations at boundary.
To set the stage, let $P$ be a smooth finite dimensional manifold, possibly with boundary.
Assume each member $g^p$ contains an isometric copy of $(I \times Y, 1 \times \upiota)$.
In fact, there is a corresponding smooth family of real spin\textsuperscript{c} structures $(\mathfrak s_X^P,\uptau_X^P)$. 
We will identify the spinor bundles in the family with a fixed spinor bundle, and real structures with a fixed real structure.
Let $\mathfrak p_0^P \in \mathcal P(Y,\mathfrak s,\uptau)$ be a smooth family of perturbations, and define 
\begin{equation*}
	\mathfrak p^p = \beta(t)\mathfrak q_0 + \beta_0(t)\mathfrak p^p_0.
\end{equation*}
For each $p \in P$, we have a moduli space $M(X^*,\mathfrak s_X, \uptau_X, [\mathfrak b])$ and the total space
\begin{align*}
	M(X^*,\mathfrak s_X, \uptau_X, [\mathfrak b])_P
	&= 
	\bigcup_p \{p\}
 \times M(X^*,\mathfrak s_X, \uptau_X, [\mathfrak b])_p\\
	&\subset 
	P \times \mathcal B^{\sigma}_{k,\loc}(X^*,\mathfrak s_X, \uptau_X).
\end{align*}

\subsection{Transverality}
\hfill \break
We begin with definition of a single moduli space, based on the fibre product description.
Denote an element of the fibre product  by $([\upgamma_1], [\upgamma_2])$, where
\[ 
[\upgamma_1] \in M(X,\mathfrak s_X, \uptau_X), \quad
[\upgamma_2] \in M(Z,\mathfrak s_Z,\uptau_Z,[\mathfrak b]).
\]
By Proposition~\ref{prop:mfld_boundary_res_fred}, the sum of the derivatives 
\begin{equation}
\label{eqn:sum_der_restrictions}
\mathcal D_{[\gamma_1]} R_+ + \mathcal D_{[\gamma_2]} R_-: 
\T_{[\upgamma_1]} M(X,\mathfrak s_X,\uptau_X) \oplus
\T_{[\upgamma_2]} M(Z,\mathfrak s_Z,\uptau_Z,[\mathfrak b]) \to 
\T_{[\mathfrak b]} \mathcal B^{\sigma}_{k-1/2} (Y,\mathfrak s, \uptau)
\end{equation}
is Fredholm (cf. \cite[Lemma~24.4.1]{KMbook2007}.)
\begin{defn}
	Let $[\upgamma] \in M(X^*,\mathfrak s, \uptau;[\mathfrak b])$. If $[\mathfrak \upgamma]$ is irreducible, then the moduli space	$[\upgamma] \in M(X^*,\mathfrak s, \uptau;[\mathfrak b])$ is \emph{regular} at $[\upgamma]$ if $R_+$ and $R_-$ are transverse at $\rho[\upgamma]$. 
	If $[\upgamma]$ is reducible, then the moduli space is regular at $[\upgamma]$ if the restrictions
	\begin{align*}
		R_+: M^{\red}(X,\mathfrak s_X,\uptau_X) 
		&\to \prod_{\alpha} \mathcal B^{\sigma}_{k-1/2} (Y^{\alpha},\mathfrak s^{\alpha}, \uptau^{\alpha}), \\
		R_-: M^{\red}(Z,\mathfrak s_Z,\uptau_Z, [\mathfrak b]) 
		&\to \prod_{\alpha} \mathcal B^{\sigma}_{k-1/2} (Y^{\alpha},\mathfrak s^{\alpha}, \uptau^{\alpha}),
	\end{align*}
	are transverse at $\rho[\upgamma]$.
	The moduli space is \emph{regular} if it is regular at all points.
\end{defn}
For regular moduli spaces, there are three scenarios.
The following is the 4-manifold version of Proposition~\ref{prop:class_regular_mod_space} in the cylinder case.
\begin{prop}
	Let $[\mathfrak b]$ be a critical point, and let $[\mathfrak b^{\alpha}]$ be the restriction to the $\alpha$-component of $Y$.
	Suppose the moduli space $M(X^*,\mathfrak s, \uptau_X;[\mathfrak b])$ is non-empty and regular.
	Then the moduli space is:
	\begin{enumerate}[leftmargin=*]
		\item[(i)] a smooth manifold consisting only of irreducibles, if any $[\mathfrak b^{\alpha}]$ is irreducible;
		\item[(ii)] a smooth manifold consisting only of reducibles, if any $[\mathfrak b^{\alpha}]$ is reducible and boundary-unstable;
		\item[(iii)] a smooth manifold consisting only of reducibles, if any $[\mathfrak b^{\alpha}]$ is reducible and boundary-stable.
	\end{enumerate}
\end{prop}
Case~(ii) is the analogue of the boundary-obstructed case.
In particular, the Fredholm operator \eqref{eqn:sum_der_restrictions} above is not surjective.
The dimension of the cokernel depends on the number of boundary-unstable critical points on the cylindrical ends.

\begin{defn}
	If case~(ii) occurs and more than one of the $[\mathfrak b^{\alpha}]$ is boundary-unstable, then the solution $[\upgamma]$ is \emph{boundary-obstructed}.
	If $c+1$ of the $[\mathfrak b^{\alpha}]$ are boundary-unstable, then $[\upgamma]$ is boundary-obsructed with \emph{corank} c.	
\end{defn}
To analyze the equi-dimensional pieces of the total moduli spaces on cylindrical manifolds, we decompose $\mathcal B^{\sigma}(X, \upiota_X; [\mathfrak b])$ according to homotopy classes
\begin{equation*}
	z \in \pi_0(\mathcal B^{\sigma}(X, \upiota_X;[\mathfrak b])).
\end{equation*}
Since each $[\gamma] \in M(X^*,[\mathfrak b])$ is homotopic to configuration which is equal to the pull-back of $[\mathfrak b]$ on the cylindrical end, 
\begin{equation*}
	M(X^*;[\mathfrak b]) = \bigcup_{z} M_z(X^*;[\mathfrak b]).
\end{equation*}
The set $\pi_0(\mathcal B^{\sigma}(X;[\mathfrak b]))$ is a principal homogeneous space of 
\[ 
	H^2(X,Y;\mathbb Z)^{-\upiota_X^*} \times
	\frac{H^1(X,Y;\mathbb Z)^{-\upiota_X^*}}{2H^1(X,Y;\mathbb Z)^{-\upiota_X^*}}
\]
parametrizing the spin\textsuperscript{c} structures with real structures.
An element $z \in \pi_0(\mathcal B^{\sigma}(X;[\mathfrak b])$ belong to a given pair $(\mathfrak s_X, \uptau_X)$ lies in a homogeneous space of 
\begin{equation*}
	H^1(Y;\mathbb Z)^{-\upiota^*}/i^*_Y(H^1(X;\mathbb Z)^{-\upiota_X^*}).
\end{equation*}

Given $z \in \pi_0(\mathcal B^{\sigma}(X;[\mathfrak b]))$, let $[\upgamma] \in \mathcal B^{\sigma}(X;[\mathfrak b])$ and $\gamma$ be a gauge representative.
Let $[\upgamma_{\mathfrak b}]$ be the constant trajectory in $\mathcal B^{\tau}(Z,\mathfrak s,\uptau_Z)$ corresponding to $\mathfrak b$.
Recall the operator $\mathcal Q^{\sigma}_{\upgamma}$ on $X$ is defined in Proposition~\ref{prop:mfld_boundary_Hilbert}, and the translation-invariant operator $\mathcal Q_{\upgamma_{\mathfrak b}}$ on $Z$ is  defined in subsection~\ref{subsec:loc_str}.
We have the restriction maps 
\begin{align*}
	r_+ : \ker(\mathcal Q^{\sigma}_{\upgamma}) 
	&\to L^2_{k-1/2}(Y; i\T^*Y \oplus S \oplus i\reals)^{-\uptau},\\
	r_- : \ker(\mathcal Q_{\upgamma_{\mathfrak b}}) 
	&\to L^2_{k-1/2}(Y; i\T^*Y \oplus S \oplus i\reals)^{-\uptau}.
\end{align*}
\begin{defn}
\label{defn:Xgrade}
We define $\gr_z(W,[\mathfrak b])$ to be the index of 
\begin{equation*}
	r_+ - r_-:\ker(\mathcal Q^{\sigma}_{\upgamma}) \oplus \ker(\mathcal Q_{\upgamma_{\mathfrak b}})
	\to L^2_{k-1/2}(Y; i\T^*Y \oplus S \oplus i\reals)^{-\uptau}.
\end{equation*}
\end{defn}
This quantity is manifestly dependent only on $[\mathfrak b]$ and $z$, and makes sense even if the moduli space is empty.
When the moduli space $M_z(X^*, \upiota;[\mathfrak b])$ is nonempty, $\gr_z(W,[\mathfrak b])$ is equal to the index of the Fredholm opeator \eqref{eqn:sum_der_restrictions}.
It follows that nonempty regular boundary-unobsrtucted moduli spaces  has dimension equal to $\gr_z$.
Furthermore, if $M_z(X,\upiota_X;[\mathfrak b])$ is boundary-obstructed of corank $c$, then its dimension is $\gr_z(X;[\mathfrak b]) + c$.
We can also concatenate $z \in \pi_0(\mathcal B^{\sigma}(X, \upiota_X;[\mathfrak b]))$ with an homotopy class $z_1 \in \pi_1(\mathcal B^{\sigma}(Y);[\mathfrak a],[\mathfrak b])$ to obtain
\begin{equation*}
	z_1 \circ z \in 
\pi_0(\mathcal B^{\sigma}(X, \upiota_X;[\mathfrak b])) 
\end{equation*}
The index is additive under concatenation:
\begin{equation*}
	\gr_{z_1 \circ z}(X;[\mathfrak b]) =
	\gr_z(X;[\mathfrak a]) + \gr_z([\mathfrak a],[\mathfrak b]).
\end{equation*}

\begin{prop}
	Let $\mathfrak q$ be a fixed admissible perturbation for $Y$, let
	\begin{equation*}
		\hat{\mathfrak p} = \beta(t) \hat{\mathfrak q} + \beta_0(t) \hat{\mathfrak q}_0
	\end{equation*}	
	on the collar $I \times Y \subset X$, as before, 
	and let $[\mathfrak b]$ be a critical point on $Y$, then there is a residual subset of $\mathcal P(Y,\mathfrak s, \uptau)$ such that for all $\mathfrak p_0$ in this subset, the moduli space $M(X^*;[\mathfrak b])$ is regular.
\end{prop}
\begin{proof}
	See the proof of \cite[Proposition~24.4.7]{KMbook2007}.
\end{proof}

\subsubsection*{\textbf{(Family version)}} 
\begin{defn}
	Let $(p,[\upgamma]) \in M(X^*,\mathfrak s_X,\uptau_X;[\mathfrak b])_P$ and let $\rho[\upgamma] = ([\upgamma_0],[\upgamma_1])$.
	If $[\upgamma]$ is irreducible, then the moduli space $M(X^*,\mathfrak s_X,\uptau_X;[\mathfrak b])_P$ is \emph{regular} at $(p,[\upgamma])$ if the maps of Hilbert manifolds
	\begin{align*}
		R_+: M(X,\mathfrak s_X, \uptau_X)_P &\to \mathcal B^{\sigma}_{k-1/2}(Y,\mathfrak s,\uptau),\\
		R_-: M(Z,\mathfrak s_Z, \uptau_Z, [\mathfrak b])_P &\to \mathcal B^{\sigma}_{k-1/2}(Y,\mathfrak s,\uptau),
	\end{align*}
	are transverse at $((p,[\upgamma_0]),[\upgamma_1])$.
	If $[\upgamma]$ is reducible, then the moduli space is \emph{regular} if the restrictions
	\begin{align*}
		R_+: M^{\red}(X,\mathfrak s_X, \uptau_X)_P &\to \mathcal B^{\sigma}_{k-1/2}(Y,\mathfrak s,\uptau),\\
		R_-: M^{\red}(Z,\mathfrak s_Z, \uptau_Z, [\mathfrak b])_P &\to \mathcal B^{\sigma}_{k-1/2}(Y,\mathfrak s,\uptau),
	\end{align*}
	are transverse at $((p,[\upgamma_0]),[\upgamma_1])$.
	The moduli space is \emph{regular} if it is regular at all points.
\end{defn}
\begin{prop}
	Let $\mathfrak q$ be a fixed admissible perturbation for $Y$, let $g^P$ be a smooth family of Riemannian metric parametrized by $p \in P$ all containing an isometric copy of $(I \times Y, 1 \times \upiota)$, and let $\hat{\mathfrak p}^P$ be a family of perturbations defined as before.
	Let $P_0 \subset P$ be a closed subset and supposed that the parametrized moduli space $M(X^*,\upiota_X,[\mathfrak b])_P$ is regular at all points $(p_0,[\upgamma])$ with $p_0 \in P_0$.
	Then there is a new family of perturbations $\tilde{\mathfrak p}^P$, with
	\begin{equation*}
		\tilde{\mathfrak p}^p = \mathfrak p^p \quad
		\text{for all }p \in P_0 
	\end{equation*}
	such that the corresponding parametrized moduli space $M(X^*,\upiota_X,[\mathfrak b])_P$ is regular everywhere.
\end{prop}

\subsection{Compactness}
\begin{defn}
	Let $[\mathfrak a]$ be a critical point, and $[\mathfrak b^{\alpha}]$ be its restriction to the $\alpha$-th component.
	A \emph{broken $X$-trajectory asymptotic to $[\mathfrak b]$} consists of the following data
	\begin{itemize}
		\item an element $[\upgamma_0]$ in a moduli space $M_{z_0}(X^*,\upiota_X,[\mathfrak b_0])$;
		\item for each component $Y^{\alpha}$, an unparametrized broken trajectory $[\check{\upgamma}^{\alpha}]$ in a moduli space $\check M_{z^{\alpha}}^+([\mathfrak b_0^{\alpha}],[\mathfrak b^{\alpha}]$, where $[\mathfrak b_0^{\alpha}]$ is the restriction of $[\mathfrak b_0]$ to $Y^{\alpha}$.
	\end{itemize}
	Moreover, if $z_1$ is the homotopy class of paths from $[\mathfrak b_0]$ to $[\mathfrak b]$ whose $\alpha$-th component is $z^{\alpha}$, then the \emph{homotopy class} of the broken $X$-trajectory is the element
	\begin{equation*}
		z = z_1 \circ z_0 \in \pi_0(\mathcal B^{\sigma}(X,\upiota_X,[\mathfrak b])).
	\end{equation*} 
	We write $M^+_z(X^*,\upiota_X,[\mathfrak b])$ for the space of $X$-trajectories in the homotopy class $z$.
	This contains $M_z(X^*,[\mathfrak b])$ as the special case when each of the broken trajectories $[\check{\upgamma}^{\alpha}]$ has no components.
	We write a typical element of $M^+(X^*,[\mathfrak b])$ as $([\upgamma_0],[\check\upgamma])$, where $[\check\upgamma]$ is a possibly empty collection $[\check\upgamma^{\alpha}_i]$ of unparameterized trajectories on the components $Y^{\alpha}$, $1 \le i \le n^{\alpha}$.
\end{defn}
The compactification $M^+_z(X^*,[\mathfrak b])$ can be topologized as in \cite[Section 24.6]{KMbook2007}, in a way similar to the compactification of trajectories on cylinders.

\begin{thm}
	Let $\mathfrak p_0$ be chosen so that the moduli space of $X$-trajectories $M(X^*,[\mathfrak b])$ are regular for all critical points $[\mathfrak b]$ on $Y$.
	Then each moduli space of broken $X$-trajectories $M^+_z(X^*,[\mathfrak b])$ is compact.	
\end{thm}

\subsubsection*{(\textbf{Family version})}
The compactification of the parametrized moduli space is defined fibrewise:
\begin{equation*}
	M^+(X^*,\upiota_X,[\mathfrak b])_P
	= 
	\bigcup_{p} \{p\} \times M^+(X^*,\upiota_X,[\mathfrak b])_p.
\end{equation*}
\begin{thm}[Compactness: family Version]
	Suppose that the families $M(X^*,\upiota_X,[\mathfrak b])_P$ are regular for all $[\mathfrak b]$.
	Then for each $[\mathfrak b]$ the family of moduli spaces is proper over $P$;
	that is, the map 
\[
	M_z^+(X^*,\upiota_X,[\mathfrak b])_P \to P
\]
is proper.
Moreover, this parametrized moduli space is non-empty for only finitely many components $z \in \pi_0(\mathcal B^{\sigma}_k(X,\upiota_X,[\mathfrak b]))$.
\end{thm}
The space $M_z^+(X^*,\upiota_X,[\mathfrak b_0])_P$ is stratified by subspaces of the form
\begin{equation*}
	M_{z_0}(X^*,\upiota_X,[\mathfrak b_0])_P
	\times
	\prod_{\alpha}
	\check M^+_{z^{\alpha}}([\mathfrak b^{\alpha}_0],[\mathfrak b^{\alpha}]),
\end{equation*}
where each of the factors is stratified by manifolds.
If $M_z(X^*,\upiota_X,[\mathfrak b])_P$ contains irreducible trajectories, then the top stratum consists of the irreducible parts.


\subsection{Gluing}
\hfill \break
The  boundary-obstructed codimension-1 strata over general 4-manifold, 
unlike purely cylindrical case, 
can have corank higher than $1$.
We modify the definition $\delta$-structure in Definition~\ref{defn:delta1str} to incorporate this new phenomenon.
\begin{defn}
\label{defn:delta2str}
	Let $N^d$ be a $d$-dimensional space stratified by manifolds and $M^{d-1} \subset N$ a union of components of the $(d-1)$-dimensional stratum.
	Then $N$ has a \emph{codimension-$c$ $\delta$-structure} along $M^{d-1}$ if $M^{d-1}$ is smooth and we have the following additional data:
	\begin{itemize}
		\item an open set $W \subset N$ containing $M^{d-1}$,
		\item an embedding $j:W \to EW$, and
		\item a map
		\begin{equation*}
			\mathbf{S} = (S_1,S_2): EW \to (0,\infty]^{c+1}
		\end{equation*}
	\end{itemize}
	satisfying the following:
	\begin{enumerate}[leftmargin=*]
		\item[(i)] the map $\mathbf{S}$ is a topological submersion along the fibre over $(\infty,\infty)$;
		\item[(ii)] the fibre of $\mathbf S$ over $(\infty,\infty)$ is $j(M^{d-1})$;
		\item[(iii)] the subset $j(W) \subset EW$ is the zero set of a map $\delta:EW \to \Pi^c$, where $\Pi^c \subset \reals^{c+1}$ is the hyperplane $\Pi^c = \{\delta \in \reals^{c+1}| \sum \delta_i = 0\}$;
		\item[(iv)] if $e \in EW$ has $S_{i_0}(e) = \infty$ for some $i_0$, then $\delta_{i_0}(e) \le 0$, with equality only if $S_{i}(e) = \infty$ for all $i$;
		\item[(v)] on the subset of $EW$ where all $S_i$ are both finite, $\delta$ is smooth and transverse to zero.
	\end{enumerate}
\end{defn}
\begin{exmp}
	Let $EW \subset \reals^{c+1}$ be the set $\{x|x_i \ge 0\text{ for all }i\}$, let $S_i = 1/x_i$ as a function from $EW$ to $(0,\infty]$, and let 
	\begin{equation*}
		\delta_i = cx_i = \sum_{j \ne i}x_i.
	\end{equation*}
	The zero locus of $\delta$ is the half-line $W \subset EW$ where all $x_i$ are equal.
\end{exmp}
\begin{thm}
	Suppose the moduli space $M_z(X^*,\upiota_X,[\mathfrak b])_P$ is $d$-dimensional and contains irreducible trajectories, so that $M^+(X,\upiota_X,[\mathfrak b])$ is a $d$-dimensional space stratified by manifolds having $M_z(X^*,\upiota_X,[\mathfrak b])_P$ as its top stratum.
	Let $M' \subset M_z^+(X^*,\upiota_X,[\mathfrak b])_P$ be any component of the codimension-1 stratum. Then along $M'$, the moduli space $M^+(X^*,\upiota_X,[\mathfrak b])_P$ either is a $C^0$-manifold with boundary, or has a codimension-$c$ $\delta$-structure.
\end{thm}
\begin{proof}
	Adaption of the proof of \cite[Theorem~24.7.2]{KMbook2007} requires no essential changes.
\end{proof}

\section{Cobordism Maps, Invariance, and Module Structures}
\label{sec:cob_module}
Our treatment follows \cite[Section~23-25]{KMbook2007}.
The definition of a cobordism map involves all real spin\textsuperscript{c} structures at once, so let us set up the notations.
Let $(Y,\upiota)$ be a 3-manifold with involution.
Consider the product of large Banach spaces of tame perturbations, over all isomorphism classes of real spin\textsuperscript{c} structures
\begin{equation*}
	\mathcal P(Y,\upiota) = \prod_{\mathfrak s,\uptau} \mathcal P(Y,\upiota;\mathfrak s, \uptau).
\end{equation*}
In addition to the usual transversality assumptions in the definition of admissibility, we assume some uniformity in the estimates.
\begin{defn}
	An element $\mathfrak q = \{\mathfrak q_{\mathfrak s,\uptau}\}$ in $\mathcal P(Y,\upiota)$	if all components of $\mathfrak q$ are admissible, and the bound $m_2$ in Definition~\ref{defn:tame} is uniform across all $(\mathfrak s,\uptau)$.
\end{defn}
Given an admissible perturbation $\mathfrak q$, we write $\HMR^{\circ}(Y,\upiota, \mathfrak q)$ for the direct sum over all isomorphism classes of real spin\textsuperscript{c} structures
\begin{equation*}
	\HMR^{\circ}_*(Y,\upiota,g,\mathfrak q) =
	\bigoplus_{\mathfrak s,\uptau} \HMR^{\circ}_*(Y,\upiota,g,\mathfrak q; \mathfrak s, \uptau).
\end{equation*}
Same for the negative completion $\HMR^{\circ}_{\bullet}(Y,\upiota,g,\mathfrak q) $.
The grading set of $\HMR^{\circ}(Y,\upiota, \mathfrak q)$ is the union over all isomorphism classes of real spin\textsuperscript{c} structures
\begin{equation*}
	\mathbb J(Y,\upiota, g, \mathfrak q) = 
	\coprod_{\mathfrak s,\uptau} \mathbb J(Y,\upiota, g, \mathfrak q; \mathfrak s,\uptau).
\end{equation*}
We formulate $\HMR^{\circ}$ as functor to the category $\textsc{group}$ of abelian groups.
\begin{defn}
	Let $\textsc{cob}_{\mathbb Z/2}$ be the category whose objects are pairs $(Y, \upiota)$, where $Y$ is a compact, oriented, closed 3-manifold and $\upiota: Y \to Y$ is an involution with codimension-1 fixed points.
	A morphism in $\textsc{cob}_{\mathbb Z/2}$ is a connected cobordism $W$, equipped with an involution $\iota_W$ that induces the respective involutions on the boundary 3-manifolds.
	Two objects $(Y_1, \upiota_1)$ and $(Y_2, \upiota_2)$ in $\textsc{cob}_{\mathbb Z/2}$ are isomorphic if the underlying $3$-manifolds admit a diffeomorphism that intertwines the two involutions.
\end{defn}
\begin{defn}
	Let $\widetilde{\textsc{cob}}_{\mathbb Z/2}$ be the category whose objects are quadruples $(Y, \upiota; g, \mathfrak q)$, where $(Y,\upiota)$ is an object in $\textsc{cob}_{\mathbb Z/2}$, $g$ is an $\upiota$-invariant Riemannian metric, and $\mathfrak q$ is an admissible perturbation.
	A morphism in $\widetilde{\textsc{cob}}_{\mathbb Z/2}$ is a morphism in $\textsc{cob}_{\mathbb Z/2}$.
	Two objects $(Y_1, \upiota_1; g_1, \mathfrak q_1)$ and $(Y_2, \upiota_2; g_2, \mathfrak q_2)$ in $\widetilde{\textsc{cob}}_{\mathbb Z/2}$ are isomorphic if the underlying \emph{real} $3$-manifolds admit a diffeomorphic that intertwines the two involutions.
\end{defn}
\begin{thm}
\label{thm:HMRfunctor_tilde_COB}
	The real monopole Floer homologies define covariant functors
	\begin{equation*}
		\HMR^{\circ}_{\bullet}: \widetilde{\textsc{cob}}_{\mathbb Z/2} \to \textsc{group},
	\end{equation*}
	The real monopole Floer cohomology defines contravariant functors
	\begin{equation*}
		\HMR^{\circ,\bullet}: \widetilde{\textsc{cob}}_{\mathbb Z/2} \to \textsc{group}.
	\end{equation*}
\end{thm}
\begin{proof}
	The theorem will be divided into Proposition~\ref{prop:trivialcob} and Proposition~\ref{prop:compositionlaw}.
\end{proof}
Since the morphisms $\widetilde{\textsc{cob}}_{\mathbb Z/2}$ is not decorated with any auxiliary data, Theorem~\ref{thm:HMRfunctor_tilde_COB} contains the invariance of $\HMR^{\circ}$ with respect to metrics and perturbations.
In fact, the cylinder cobordism $[0,1] \times Y$ provides a canonical isomorphism.
\begin{cor}
	If two objects $(Y, \upiota; g_1, \mathfrak q_1)$ and $(Y, \upiota; g_2, \mathfrak q_2)$ in $\widetilde{\textsc{cob}}_{\mathbb Z/2}$ have the same underlying $\mathbb Z_2$ 3-manifold, then  $\HMR^{\circ}(Y, \upiota; g_1, \mathfrak q_1)$ and $\HMR^{\circ}(Y, \upiota; g_2, \mathfrak q_2)$ are canonically isomorphic.
\end{cor}
We can define the following decorated category $\widetilde{\textsc{link}}_{S^3}$ of links in $S^3$.
The objects are triples $(K, g, \mathfrak q)$ where $g$ is an orbifold metric on $S^3$ and $\mathfrak q$ is an adimssible perturbation on the double branched cover of $S^3$ along $K$.
The morphisms are unoriented cobordisms of links, i.e. properly embedded surfaces in $I \times S^3$ whose boundary are the the respective links.
The isomorphisms are ambient diffeomorphisms of $S^3$ preserving the links.
Then $\HMR^{\circ}$ of links can be thought of as a composite of functors
\begin{equation*}
	\begin{tikzcd}
			\widetilde{\textsc{link}}_{S^3} \ar[d] \ar[rd]\\
			\widetilde{\textsc{cob}}_{\mathbb Z/2} \ar[r] 
			&  \textsc{group}
	\end{tikzcd}
\end{equation*}
where the vertical map takes a link to its double branched cover, and link cobordisms to their double branched covers.

\subsection{Moduli space over a cobordism}
\hfill \break
We switch the point of view from manifolds with boundary to corbordisms and rephrase the results of the preceding sections.
Let $(W,\upiota_W)$ be a $\mathbb Z_2$-cobordism between a pair $(Y_{\pm},\upiota_{\pm})$ of non-empty connected $\mathbb Z_2$-3-manifolds.
Here, we use $\pm$-sign to suggest that $Y_-$ is the incoming manifold, and $Y_+$ is the outgoing manifold.
The boundary of $W$ is equipped with the orientation convention
\begin{equation*}
	\del W = \overline{Y_-} \sqcup Y_+,
\end{equation*}
where the overline indicate reversal of orientations.
Equip $W$ with a real spin\textsuperscript{c} structure $(\mathfrak s_W, \uptau_W)$, and denote $(\mathfrak s_{\pm}, \uptau_{\pm})$ the restriction of the real spin\textsuperscript{c} structure.
We denote the disjoint union of equivalence classes of configurations, over the isomorphism classes of real spin\textsuperscript{c} structures, as
\begin{equation*}
	\boldsymbol{\mathcal B}^{\sigma}(M, \upiota) 
	= \coprod_{\mathfrak s,\uptau}
	\mathcal B^{\sigma}(M, \upiota;\mathfrak s, \uptau),
\end{equation*}
where $M$ is either $Y_{\pm}$ or $W$, and $\upiota$ is either $\upiota_{\pm}$ or $\upiota_W$.

\begin{defn}
A $W$-path $\xi$ from $[\mathfrak a_0]$ to $[\mathfrak a_1]$ is an element $[\upgamma]$ in $\boldsymbol{\mathcal B}^{\sigma}(W)$ which under the  partially defined restriction map
\[
	r: \boldsymbol{\mathcal B}^{\sigma}(W,\upiota_W)
	\to 
	\boldsymbol{\mathcal B}^{\sigma}(Y_-,\upiota_-)
	\times
	\boldsymbol{\mathcal B}^{\sigma}(Y_+,\upiota_+)
\] 
is the given pair: $r([\upgamma]) = ([\mathfrak a_0],[\mathfrak a_1])$. Two $W$-paths are homotopic if they belong to the same path component of the fiber $r^{-1}([\mathfrak a_0],[\mathfrak a_1])$. We write
\begin{equation*}
	\boldsymbol{\pi}([\mathfrak a_0], W, \upiota_W, [\mathfrak a_1])
\end{equation*}
for the set of homotopy classes of $W$-paths. 
\end{defn}
Adjoining the cylindrical ends to the cobordism $W$, we obtain
\begin{equation*}
	W^* = (-\infty,0]\times Y_- \cup W \cup [0,\infty) \times Y_+.
\end{equation*}
Fix admissible perturbations $\mathfrak q_{\pm}$ on $Y_{\pm}$. 
Suppose $[\mathfrak a] \in \mathcal B^{\sigma}_{k-1/2}(Y_-)$ and $[\mathfrak b] \in \mathcal B^{\sigma}_{k-1/2}(Y_+)$ are critical points for perturbed  CSD functionals.
We denote moduli space by
\begin{equation*}
	M([\mathfrak a],(W^*, \upiota_W), [\mathfrak b]) \subset \boldsymbol{\mathcal B}^{\sigma}_{k,\loc}(W^*)
	= \bigcup_{\mathfrak s_W, \uptau_W} B^{\sigma}_{k,\loc}(W^*,\mathfrak s_W, \uptau_W).
\end{equation*}
Decomposing with respect to homotopy classes of $W$-paths, we write
\begin{equation*}
	\bigcup_{\mathfrak s_W, \uptau_W} M([\mathfrak a],W^*, \upiota_W ,\mathfrak s_W,[\mathfrak b]) =
	\bigcup_z M_z([\mathfrak a], (W^*,\upiota_W),[\mathfrak b]).
\end{equation*}

Similar to the cylinder case, the moduli space is boundary-obstructed (with corank-1) if $[\mathfrak a]$ is boundary-stable and $[\mathfrak b]$ is boundary-unstable.
Choose perturbation $\mathfrak p = (\mathfrak p_-,\mathfrak p_+)$ supported over the collars of $Y_{\pm}$, so that all moduli spaces are regular.

Let $\gr_z([\mathfrak a],(W,\upiota_W), [\mathfrak b])$ be as in Definition~\ref{defn:Xgrade}.
This grading agrees with the dimension of the moduli space, except in the boundary-obstructed case, when the dimension of the moduli spaces is larger by 1.
Let $M^+([\mathfrak a],(W^*,\upiota_W),[\mathfrak b])$ be the compactification consists of broken trajectories.
We write a typical element as $([\check\upgamma_-],[\upgamma_0],[\check\upgamma_+])$, where
\begin{equation*}
	[\check{\pmb{\upgamma}}_-] \in \check M^+([\mathfrak a],[\mathfrak a_0]), \quad
	[\check{\pmb{\upgamma}}_+] \in \check M^+([\mathfrak b_0],[\mathfrak b]), \quad
	[\upgamma_0] \in M([\mathfrak a_0], W^*,[\mathfrak b_0]).
\end{equation*}
In the case when $W$ is a cobordism, we describe the codimension-1 strata more explicitly.

\begin{prop}
\label{prop:W_cod1_strata}
	Suppose $M_z([\mathfrak a], (W^*,\upiota_W) , [\mathfrak b])$ contains irreducible solutions and has dimension $d$.
	Then $M_z([\mathfrak a], W^*,\upiota_W, [\mathfrak b])$ is a $d$-dimensional space stratified by manifolds, with top stratum the irreducible part of $M_z([\mathfrak a], W^*,\upiota_W, [\mathfrak b])$.
	The $(d-1)$-dimensional stratum in $M_z([\mathfrak a], W^*,\upiota_W, [\mathfrak b])$ consists of elements of the types:
\begin{align*}
	\check M_{-1} &\times M_0	\\
	M_0 &\times \check M_1\\
	\check M_{-2} \times &\check M_{-1} \times M_0\\
	\check M_{-1} \times &M_0 \times \check M_1\\
	M_0 \times &\check M_1 \times \check M_2,
\end{align*}
and finally
\begin{equation*}
	M_z^{\red}([\mathfrak a], W^*, \upiota_W, [\mathfrak b])
\end{equation*}
in the case that the moduli space contains both reducibles and irreducibles.
We use $M_0$ to denote a moduli space on $W^*$ and $\check M_{-n}$ and $\check M_n$ $(n > 0)$ to denote typical unparametrized moduli spaces on $Y_-$ and $Y_+$ (which changes from line to line).
In the strata with three factors, the middle factor is boundary-obstructed.
\end{prop}

\subsection{Cobordism maps for 3-manifolds with involutions}
\hfill \break
Our first goal is to define the cobordism map 
\begin{equation*}
	\HMR^{\circ}(W,\upiota_W): \HMR^{\circ}(Y_-, \upiota_-) \to \HMR^{\circ}(Y_+, \upiota_+).
\end{equation*}
In fact, we follow \cite{KMbook2007} by considering a more general situation where we simultaneously evaluate a cohomology class $u$ over $\mathcal B^{\sigma}(W, \upiota_W)$, as a map
\begin{equation*}
	\HMR^{\circ}(u|W, \upiota_W):
	\HMR^{\circ}(Y_-,\upiota_-) \to \HMR^{\circ}(Y_+,\upiota_+)
\end{equation*}
This construction allows to define both cobordism maps (if $u=1$) and module structures (if $W = I \times Y$.)
\newline

Let $d_0 > 0$ be an integer.
Consider all triples $(z,[\mathfrak a],[\mathfrak b])$ for which the moduli space $\bar M_z([\mathfrak a],W^*,\upiota_W,[\mathfrak b])$ or $\bar M_z^{\text{red}}([\mathfrak a],W^*,\upiota_W,[\mathfrak b])$ has dimension $d_0$ or less.
The compatifications 
\[\bar M_z([\mathfrak a],W^*,\upiota_W,[\mathfrak b])\text{ and }\bar M_z^{\text{red}}([\mathfrak a],W^*,\upiota_W,[\mathfrak b])
\] form a locally finite collection of closed subsets of $\mathcal B^{\sigma}_{k,\loc}(W^*)$.
By Lemma~\ref{lem:refinetransversestrata}, every open cover of $\mathcal B^{\sigma}_{k,\loc}$ has a refinement transverse to all strata in all compactified moduli spaces $\bar M$ and $\bar M^{\red}$ of dimension at most $d_0$.
Let $\mathcal U$ be an open cover transverse to these moduli spaces and  $u \in C^d(\mathcal U;\Ftwo)$ be a \v{C}ech cochain with $d \le d_0$.
If $M_z([\mathfrak a],W^*,[\mathfrak b])$ has dimension $d$, then there is a well-defined evaluation
\begin{equation*}
	\langle u, [M_z([\mathfrak a], W^*,[\mathfrak b])]\rangle \in \Ftwo,
\end{equation*}
where we set the evaluation to be zero if the dimension of the moduli space is not $d$.
From this evaluation, we define several maps, first using irreducible trajectories:
\begin{align*}
	m^o_o:C^d(\mathcal U;\Ftwo) \otimes C^o_\bullet(Y_-, \upiota_-) &\to C^o_\bullet(Y_+, \upiota_+), \quad
u \otimes [\mathfrak a]\mapsto \sum_{[\mathfrak b] \in \mathfrak C^o(Y_+, \upiota_+)} \sum_z 
	\langle u, [M_z([\mathfrak a], W^*,[\mathfrak b])]\rangle [\mathfrak b],\\
	m^o_s:C^d(\mathcal U;\Ftwo) \otimes C^o_\bullet(Y_-, \upiota_-) &\to C^s_\bullet(Y_+, \upiota_+), \quad
u \otimes [\mathfrak a]~\mapsto \sum_{[\mathfrak b] \in \mathfrak C^s(Y_+, \upiota_+)} \sum_z 
	\langle u, [M_z([\mathfrak a], W^*,[\mathfrak b])]\rangle [\mathfrak b],\\
	m^u_o:C^d(\mathcal U;\Ftwo) \otimes C^u_\bullet(Y_-, \upiota_-) &\to C^o_\bullet(Y_+, \upiota_+), \quad
u \otimes [\mathfrak a]\mapsto \sum_{[\mathfrak b] \in \mathfrak C^o(Y_+, \upiota_+)} \sum_z 
	\langle u, [M_z([\mathfrak a], W^*,[\mathfrak b])]\rangle [\mathfrak b],\\
	m^u_s:C^d(\mathcal U;\Ftwo) \otimes C^u_\bullet(Y_-, \upiota_-) &\to C^s_\bullet(Y_+, \upiota_+), \quad
u \otimes [\mathfrak a]\mapsto \sum_{[\mathfrak b] \in \mathfrak C^s(Y_+, \upiota_+)} \sum_z 
	\langle u, [M_z([\mathfrak a], W^*,[\mathfrak b])]\rangle [\mathfrak b],
\end{align*}
and then using reducible trajectories:
\begin{align*}
	\bar m^s_s :C^d(\mathcal U;\Ftwo) \otimes \bar C^s_\bullet(Y_-, \upiota_-) 
	&\to C^s_\bullet(Y_+, \upiota_+), \quad
	u \otimes [\mathfrak a] 
	\mapsto \sum_{[\mathfrak b] \in \mathfrak C^s(Y_+, \upiota_+)} \sum_z 
	\langle u, [M_z^{\red}([\mathfrak a], W^*,[\mathfrak b])]\rangle [\mathfrak b],
	\\
	\bar m^u_u : C^d(\mathcal U;\Ftwo) \otimes C^u_\bullet(Y_-, \upiota_-) 
	&\to C^u_\bullet(Y_+, \upiota_+), \quad
	u \otimes [\mathfrak a] 
	\mapsto \sum_{[\mathfrak b] \in \mathfrak C^u(Y_+, \upiota_+)} \sum_z 
	\langle u, [M_z^{\red}([\mathfrak a], W^*,[\mathfrak b])]\rangle [\mathfrak b],
	\\
	\bar m^s_u :C^d(\mathcal U;\Ftwo) \otimes C^s_\bullet(Y_-, \upiota_-) 
	&\to C^u_\bullet(Y_+, \upiota_+), \quad
	u \otimes [\mathfrak a] 
	\mapsto \sum_{[\mathfrak b] \in \mathfrak C^u(Y_+, \upiota_+)} \sum_z 
	\langle u, [M_z([\mathfrak a], W^*,[\mathfrak b])]\rangle [\mathfrak b],
	\\
	\bar m^u_s:C^d(\mathcal U;\Ftwo) \otimes C^u_\bullet(Y_-, \upiota_-) 
	&\to C^s_\bullet(Y_+, \upiota_+), \quad
	u \otimes [\mathfrak a] 
	\mapsto \sum_{[\mathfrak b] \in \mathfrak C^s(Y_+, \upiota_+)} \sum_z 
	\langle u, [M_z([\mathfrak a], W^*,[\mathfrak b])]\rangle [\mathfrak b].
\end{align*}
We combine the above maps as matrix entries, to obtain the chain maps in the following definitions.
\begin{defn}
\label{defn:cech_chain_map}
	The map $\bar m: C^d(\mathcal U;\Ftwo) \oplus \bar C_\bullet(Y_-, \upiota_-) \to \bar C_\bullet(Y_+, \upiota_+)$ with respect to the decomposition $\bar C_\bullet = C^s \oplus C^u$ is given by
\begin{equation*}
	\bar m = \begin{pmatrix}
		\bar m^s_s && \bar m^u_s\\
		\bar m^s_u && \bar m^u_u
	\end{pmatrix}.
\end{equation*}
Over $\check C_\bullet = C^o_\bullet \oplus C^s_\bullet$, we define
\begin{equation*}
	\check m: C^d(\mathcal U;\Ftwo) \otimes
	\check C_\bullet(Y_-, \upiota_-) \to
	\check C_\bullet(Y_+, \upiota_+),
\end{equation*}
for $d \le d_0$, by
\begin{equation*}
	\check m = \begin{pmatrix}
		m^o_o && -m^u_o \bar\del^s_u(Y_-) - \del^u_o(Y_+)\bar m^s_u\\
		m^o_s && \bar m^s_s - m^u_s\bar\del^s_u(Y_-,) -\del^u_s(Y_+)\bar m^s_u
	\end{pmatrix},
\end{equation*}
where $\del^u_o(Y_+)$ denotes the operator on $Y_+$.
Over $\hat C_\bullet = C^o_\bullet \oplus C^u_\bullet$, we define
\begin{equation*}
	\hat m:C^d(\mathcal U;\Ftwo) \otimes
	\hat C_\bullet(Y_-, \upiota_-) \to
	\hat C_\bullet(Y_+, \upiota_+),
\end{equation*}
by
\begin{equation*}
	\hat m = \begin{pmatrix}
		m^o_o && m^u_o\\
		\bar m^s_u \del^o_s(Y_-)\sigma  - \bar\del^s_u(Y_+)m^o_s 
		&& 
			\bar m^u_s \sigma + \bar m^s_u \del^u_s (Y_-) \sigma - \bar\del^s_u(Y_+m^u_s
	\end{pmatrix}.
\end{equation*}
\end{defn}
These maps define maps on the homology level, by the following proposition.
\begin{prop}
\label{prop:check_identities}
	The operators $\check m, \hat m,$ and $\bar m$ satisfy the following identities:
	\begin{align*}
		(-1)^d\check\del(Y_+, \upiota_+) \check m(u \otimes \check \xi) 
		&= -\check m(\delta u \otimes \check\xi) + \check m(u \otimes \check\del(Y_-, \upiota_-)\check\xi)\\
		(-1)^d\hat\del(Y_+, \upiota_+) \hat m(u \otimes \hat \xi) 
		&= -\hat m(\delta u \otimes \hat\xi) + \hat m(u \otimes \hat\del(Y_-, \upiota_-)\hat\xi)\\
		(-1)^d\bar\del(Y_+, \upiota_+) \bar m(u \otimes \bar \xi) 
		&= -\bar m(\delta u \otimes \bar\xi) + \bar m(u \otimes \bar\del(Y_-, \upiota_-)\bar\xi),
	\end{align*}
	for $u \in C^d(\mathcal U;\Ftwo)$, $\check\xi \in \check C_\bullet(Y_-)$, etc. and $d \le d-1$, which descends to maps
	\begin{align*}
		\check m : \check H^d(\mathcal U; \Ftwo) \otimes
		\widecheck{\HMR}_j(Y_-, \upiota_-) 
		&\to \widecheck{\HMR}_{k-d}(Y_+, \upiota_+)\\
		\hat m : \check H^d(\mathcal U; \Ftwo) \otimes
		\widehat{\HMR}_j(Y_-, \upiota_-) 
		&\to \widehat{\HMR}_{k-d}(Y_+, \upiota_+)\\
		\bar m : \check H^d(\mathcal U; \Ftwo) \otimes
		\overline{\HMR}_j(Y_-, \upiota_-) 
		&\to \overline{\HMR}_{k-d}(Y_+, \upiota_+)
	\end{align*}
	for any open cover $\mathcal U$ of $\mathcal B^{\tau}_{k,\loc}(W^*)$ transverse to all the moduli spaces of dimension less than or equal to $d_0$.
\end{prop}
\begin{proof}
	Similar to the $\del^2 = 0$ arguments, the identities follows from studying the codimension-1 strata of moduli spaces  of the form
	\[M_z([\mathfrak a], W^*, \upiota_W, [\mathfrak b]),\]
	described by Proposition \ref{prop:W_cod1_strata}.
	The description is identical in the ordinary case, and the precise formulae can be found in  \cite[Lemma~25.3.6]{KMbook2007}.
	For a complete proof of the proposition, see the proof of \cite[Proposition 25.3.4]{KMbook2007}. 
\end{proof}
\begin{defn}
	By taking the limit over all open covers of $\boldsymbol{\mathcal B}^{\sigma}_{k,\loc}(W^*,\upiota_W)$ transverse to the moduli space, and identifying the \v{C}ech cohomology $\check{H}^d(\boldsymbol{\mathcal B}^{\sigma}_{k,\loc}(W^*,\upiota_W); \Ftwo)$ with $H(\boldsymbol{\mathcal B}^{\sigma}_{k,\loc}(W^*,\upiota_W); \Ftwo)$
\begin{align*}
		\widecheck{\HMR}(u|W) : H^d(\mathcal B^{\sigma}_{k,\loc}(W^*,\upiota_W); \Ftwo) \otimes
		\widecheck{\HMR}_j(Y_-, \upiota_-) 
		&\to \widecheck{\HMR}_{k-d}(Y_+, \upiota_+)\\
		\widehat{\HMR}(u|W): H^d(\mathcal B^{\sigma}_{k,\loc}(W^*,\upiota_W); \Ftwo) \otimes
		\widehat{\HMR}_j(Y_-, \upiota_-) 
		&\to \widehat{\HMR}_{k-d}(Y_+, \upiota_+)\\
		\overline{\HMR}(u|W) :H^d(\mathcal B^{\sigma}_{k,\loc}(W^*,\upiota_W); \Ftwo) \otimes
		\overline{\HMR}_j(Y_-, \upiota_-) 
		&\to \overline{\HMR}_{k-d}(Y_+, \upiota_+)
\end{align*}
\end{defn}
So far, the notation $\HMR^{\circ}(u|W)$ is somewhat ambiguous as the definition involves a choice of an $\upiota_W$-invariant metric $g_W$, and perturbation $\mathfrak p$.
The following Proposition states that $\HMR^{\circ}(u|W)$ is independent of such choices.
By setting $u = 1$, we obtain Proposition~\ref{prop:trivialcob}, which is part of Theorem~\ref{thm:HMRfunctor_tilde_COB}.
\begin{prop}
	Let $g(0)$ and $g(1)$ be two $\upiota_W$-invariant metrics on $W$, isometric in a collar of the boundary to the same equivariant cylindrical metric.
	Let $\mathfrak p(0)$ and $\mathfrak p(1)$ be two perturbations on $W$, again constructed using the same perturbations on $Y_{\pm}$.
	Assume the corresponding moduli spaces are regular in both cases.
	Let $u$ be a \v{C}ech cocycle as above, and $\check m(0)$ and $\check m(1)$ be defined by the formulae in Definition~\ref{defn:cech_chain_map}., using the moduli space obtained from $(g(0),\mathfrak p(0))$ and $(g(1),\mathfrak p(1))$, respectively.
	Then there is an operator
	\begin{equation*}
		\check K: C^d(\mathcal U;\Ftwo) \otimes 
		\check C_{\bullet}(Y_-, \upiota_-) \to 
		\check C_{\bullet}(Y_+, \upiota_+)
	\end{equation*}
	for $d \le d_0$, satisfying the chain-homotopy identity
	\begin{equation*}
		(-1)^d\check\del \check K(u \otimes \check\xi) = 
		-\check K(\delta u \otimes \check\xi) + 
		\check K(u \otimes \check\del \check\xi)
		+(-1)^d\check m(0)(u \otimes \check\xi)
		-(-1)^d\check m(1)(u \otimes \check\xi).
	\end{equation*}
\end{prop}
\begin{proof}
	This follows from counting boundary points in a parametrized moduli space
	\begin{equation*}
		M_z([\mathfrak a], W^*, \upiota_W, [\mathfrak b]),
	\end{equation*}
	over $P = [0,1]$.
	There are two sources of contributions to the codimension-1 strata, depending on whether they sit above interior points of $P$ or boundary points.
	We define $\check K$ to be the contributions from the compactified moduli space above interior points, which are described in Proposition~\ref{prop:W_cod1_strata}
	The contributions above the boundary points constitute the rest of the terms.
	The reference for this proposition is \cite[Proposition~25.3.8]{KMbook2007}.
	In fact the proof is easier as we work over $\Ftwo$, and the signs in the formula above is immaterial.  
\end{proof}
\begin{prop}
\label{prop:trivialcob}
	If $W$ is the trivial cylindrical cobordism from $(Y,\upiota;g,\mathfrak q)$ to itself.
	Then $\HMR^{\circ}(W,\upiota_W)$	 induces the identity map.
\end{prop}
\begin{proof}
	The $0$-dimensional moduli space $M_z([\mathfrak a], W^*, \upiota_W, [\mathfrak b])$ must consist of translation invariant solutions on the infinite cylinder, for non-constant solutions belong to dimension at least 1.
	But this only happens if $z$ is trivial, and $[\mathfrak a] = [\mathfrak b]$, in which case the moduli space consists of a single point.
\end{proof}

\subsection{Composition of Cobordisms}
\hfill \break
Let $W = W_1 \circ W_2$.
The composition laws asserts that
\begin{equation*}
	\HMR^{\circ}(W_2 \circ W_1) 
	=
	\HMR^{\circ}(W_2) \circ \HMR^{\circ}(W_1).
\end{equation*}
We continue with the more general cobordism maps incorporating cohomology of $\boldsymbol{\mathcal B}^{\sigma}(W,\upiota_W)$.
To this end, we define the product $u = u_1u_2 \in \boldsymbol{\mathcal B}^{\sigma}(W_1,\upiota_{W_1})$ for $u_i \in \boldsymbol{\mathcal B}^{\sigma}(W_i,\upiota_{W_i})$ as follows.
Let $R_i$ be the partially defined restriction map
\begin{equation*}
	R_i: \boldsymbol{\mathcal B}^{\sigma}(W,\upiota_W) \to \boldsymbol{\mathcal B}^{\sigma}(W_i,\upiota_{W_i}),
\end{equation*}
which is well-defined over a weak homotopy equivalence subset, and defines a pullback map
\begin{equation*}
	R_i^*: H^*(\boldsymbol{\mathcal B}^{\sigma}(W_i,\upiota_{W_i})) \to \boldsymbol{\mathcal B}^{\sigma}(W,\upiota_W).
\end{equation*}
The product $u$ is by definition
\begin{equation*}
	u_1 u_2 = R_1^*(u_1) \cup R_2^*(u_2).
\end{equation*}
\begin{prop}
\label{prop:compositionlaw}
	Let $(Y_0,\upiota_0), (Y_1,\upiota_1)$, and $(Y_2,\upiota_2)$ be $\mathbb Z_2$-3-manifolds with metrics and admissible perturbations. Let $(W_1, \upiota_{W_1}): (Y_0,\upiota_0) \to (Y_1,\upiota_1)$ and  $(W_2, \upiota_{W_2}): (Y_1,\upiota_1) \to (Y_2,\upiota_2)$ be cobordisms.
	Denote the composite cobordism as $W$.
	For $i=1,2$, let $u_i \in H^{d_i}(\boldsymbol{\mathcal B}^{\sigma}(W_i,\upiota_{W_i});\Ftwo)$ be cohomology classes and $u = u_1u_2$ be the product $H^d(\boldsymbol{\mathcal B}^{\sigma}(W_1,\upiota_{W_1}))$.
	Then
	\begin{equation*}
		\HMR^{\circ}(u|W, \upiota_W) 
	=
	\HMR^{\circ}(u_2|W_2,\upiota_{W_2}) \circ \HMR^{\circ}(u_1|W_1,\upiota_{W_1}).
	\end{equation*}
\end{prop}
\begin{proof}
	The proof of the composition law for $\HM^{\circ}$ occupies \cite[Section~26]{KMbook2007}, which readily applies to the real monopole Floer case.
	Let sketch the proof.
	Let $W(S)$ be the composite cobordism with neck $Y_1$ of length $T$:
	\begin{equation*}
		W(T) = W_{1} \cup ([0,T] \times Y_1) \cup W_2.
	\end{equation*}
	The perturbation $\hat{\mathfrak p}$ is be $t$-dependent and supported on the four collars of the boundaries of $W_1$ and $W_2$, unlike the usual perturbations supported only on the collar of boundaries of $W(T)$.
	Still, we continue to denote the Seiberg-Witten moduli space over the cylindrical-end manifold $W(T)^*$ as 
	$\mathcal M_z([\mathfrak a],W(T)^*,[\mathfrak b])$. As $T$ varies in $[0,\infty)$, we obtained the parametrized moduli space
	\begin{equation*}
		\mathcal M_z([\mathfrak a],[\mathfrak b]) = \bigcup_{S \in [0,\infty)}\mathcal M_z([\mathfrak a],W(T)^*, \upiota_T,[\mathfrak b]).
	\end{equation*}
	This space can be compactified as $
		\mathcal M_z^+([\mathfrak a],[\mathfrak b])$, and there is an associated smaller compactification 
	\[
		\bar{\mathcal M}_z([\mathfrak a],[\mathfrak b]) 
		\subset
		[0,\infty] \times
		\boldsymbol{\mathcal B}^{\sigma}_k(W_1) \times
		\boldsymbol{\mathcal B}^{\sigma}_k(W_2).
	\]
	Since the parameter space $[0,\infty)$ is noncompact, one must add limits of sequences of trajectories that live on ever-longer neck.
	In particular, the codimension-1 strata above $T = \infty$,
	\begin{enumerate}
		\item[(i)] $M_{01} \times M_{12}$
		\item[(ii)] $\check M_0 \times M_{01} \times M_{12}$
		\item[(iii)] $M_{01} \times \check M_{1} \times M_{12}$
		\item[(iv)] $M_{01} \times M_{12} \times \check M_{2}$
	\end{enumerate}
	where $M_{(i-1)i}$ denotes a typical moduli space on $W_i$, and $\check M_i$ denotes a typical unparametrized moduli space on $\reals \times Y_i$.
	This contribution will be compared with the contribution of $T = 0$.
	
	Next, we choose open covers $\mathcal U_{01}$, $\mathcal U_{12}$, 
	$\mathcal V$, for $\boldsymbol{\mathcal B}^{\sigma}_k(W_1)$, $\boldsymbol{\mathcal B}^{\sigma}_k(W_2)$, $[0,\infty] \times \boldsymbol{\mathcal B}^{\sigma}_k(W_1) \times \boldsymbol{\mathcal B}^{\sigma}_k(W_2)$ respectively, so that they are transverse to all moduli spaces.
	The last cover is a refinement of the product of first two covers.
	Let $\mathcal U_{02}^{\circ}$ be the pullback of $\mathcal V$ of
	\begin{equation*}
		\boldsymbol{\mathcal B}^{\sigma}_k(W)^{\circ} \to 
		\{0\} \times \boldsymbol{\mathcal B}^{\sigma}_k(W_1) 
		\times \boldsymbol{\mathcal B}^{\sigma}_k(W_2).
	\end{equation*}
	Let $u_{01} \in C^{d_{01}}(\mathcal U_{01})$, $u_{12} \in C^{d_{12}}(\mathcal U_{12})$, and let $u_{02}^{\circ}$ be the pullback of $u_{01} \times u_{12}$ in the above map, which represents $u = u_{01}u_{12}$.
	In other words, we have an external product 
	\begin{align*}
		C^{d_{01}}(\mathcal U_{01}) \otimes
		C^{d_{12}}(\mathcal U_{12}) 
		&\to 
		C^{d_{01} + d_{12}}(\mathcal V)\\
		u_{01} \otimes u_{12}
		&\mapsto
		u_{01} \times u_{12},
	\end{align*}
	and an inner product
	\begin{align*}
		c:C^{d_{01}}(\mathcal U_{01}) \otimes
		C^{d_{12}}(\mathcal U_{12}) 
		&\to 
		C^{d_{01} + d_{12}}(\mathcal V)\\
		u_{01} \otimes u_{12}
		&\mapsto
		u_{02}^{\circ},
	\end{align*}
For the assertion of the proposition, it suffices to prove that these two products are related by a chain homotopy $K$.

The chain homotopy $\check K$, for instance, will be constructed using pairing of $u_{01} \times u_{12}$ with $\mathcal M_{z}([\mathfrak a],[\mathfrak b])$.
To give a taste of the formula, we present one of the entry:
\begin{equation*}
	K^o_o(u_{01} \otimes u_{12} \otimes -)
	=
	\sum_{[\mathfrak a] \in \mathfrak C^o(Y_0,\upiota_0)} 
	\sum_{[\mathfrak b] \in \mathfrak C^o(Y_2,\upiota_2)} 
	\sum_z
	\langle u_{01} \times u_{12}, \mathcal M_z([\mathfrak a], [\mathfrak b]\rangle.
\end{equation*}
There are seven other entries, including the count of reducible trajectories.
The homotopy formula \cite[Lemma~26.2.2]{KMbook2007} looks like
\begin{equation*}
	-\check K(\underline{\delta} \otimes 1)
	-\check K(1 \otimes 1 \otimes \check\del)
	-\check \del \check K(\sigma \otimes 1)
	-\check m_{02}(c \otimes 1) 
	+\check m_{12}(1 \otimes \check m_{01})(\tau \otimes 1) = 0,
\end{equation*}
where $\underline{\delta}:C^*(\mathcal U_{01} \otimes C^*(\mathcal U_{12}) \to C^*(\mathcal U_{01}) \otimes C^*(\mathcal U_{12})$ is the coboundary map on the tensor product locus,
and $\tau$ is the operator that interchanges the factors in the tensor product.

The proof of the lemma is a careful analysis of the boundary strata, and uses the version of Stokes theorem in Section~\ref{sec:Gluing}.
Once we plug in \v{C}ech cocycles $u_{01}$, $u_{12}$, $u_{02} = c(u_{01} \otimes u_{12})$ and $\check\xi$, the homotopy formula becomes
\begin{equation*}
	-\check m_{02}(u_{02} \otimes \check\xi) +
	\check m_{12}(u_{12} \otimes \check m_{01} (u_{01} \otimes \check\xi)) =
	(-1)^d\del\check\check K(u_{01} \otimes u_{12} \otimes \check\xi).
\end{equation*}
At the level of homology, this implies the statement of the Proposition.
We have kept $(-1)^{\bullet}$ from the formulae in \cite[Section~26]{KMbook2007}, although all coefficients are $\Ftwo$.
\end{proof}

\subsubsection*{\textbf{Double branched covers of link corbordism}}
\hfill \break
Let $K_{\pm} \subset S^3$ be two links and $I = [0,1]$.
Let $Y_{\pm}$ be the double branched cover $K_{\pm}$, and choose a real spin\textsuperscript{c} structure $(\mathfrak s, \uptau)$.
With admissible perturbations $\mathfrak q_{\pm}$ on each end, we consider the Floer homology $\HMR^{\circ}(K_{\pm},\mathfrak s, \uptau)$.
A \emph{corbodism} from $K_-$ to $K_+$ is a smooth embedded (not necessarily orientable) surfaces $S \subset I \times Y$ with boundary $\del S = \overline{K_-} \sqcup K_+$, where overline indicates the mirror.
The double branched cover of $I \times S^3$ along $S$ gives rise to a cobordism $(W, \upiota_W)$ of 3-manifolds and a covering involution $\upiota_W$.
\begin{defn}
	Let the cobordism maps $\HMR^{\circ}(S)$ of $S$ be the corresponding cobordism maps on the double branched covers of $S$:
\begin{align*}
		\widecheck{\HMR}(u|S) : H^d(\mathcal B^{\sigma}(S); \Ftwo) \otimes
		\widecheck{\HMR}_j(K_-)
		&\to \widecheck{\HMR}_{k-d}(K_+),\\
		\widehat{\HMR}(u|S): H^d(\mathcal B^{\sigma}(S); \Ftwo) \otimes
		\widehat{\HMR}_j(K_-)
		&\to \widehat{\HMR}_{k-d}(K_+),\\
		\overline{\HMR}(u|S) : H^d(\mathcal B^{\sigma}(S); \Ftwo) \otimes
		\overline{\HMR}_j(K_-) 
		&\to \overline{\HMR}_{k-d}(K_+).
\end{align*}
\end{defn}

We use the notation $\mathcal B^{\sigma}(S)$ for the spaces of equivalence classes of blown-up configuration spaces over $W$, and use $\mathcal B^{\sigma}(Y)$ to denote the configuration space on the double branched cover of $S^3$, and so on.

\subsection{Point classes and module structure}
\hfill \break
Let $W: Y_- \to Y_+$ be a cobordism, and $K_{\pm} \subset Y_{\pm}$ be the fixed point sets for $\upiota_{\pm}$.
The fixed point set $\Sigma \subset W$ is a 2-manifold with boundary $-K_- \cup K_+$.
There are types of point classes for $\HMR$, of cohomological degree $1$ and $2$, respectively.
Let $x \in W$ be a point.
\begin{itemize}[leftmargin=*]
	\item If $\upiota(x) = x$, that is $x \in \Sigma$, then any 4-dimensional gauge transformation $\mathcal G(\upiota_W)$ necessarily takes value $\pm 1$.
Evaluation at $x$ defines a homomorphism
\begin{equation*}
	\ev_x: \mathcal G(\upiota_W) \to \{\pm 1\}.
\end{equation*}
Since $\mathcal B^{\sigma}(W,\uptau_W)$ has the homotopy type of $\pt/\mathcal G(\upiota_W)$, 
we obtain a real line bundle
\begin{equation*}
	L_x \to \mathcal B^{\sigma}(W,\upiota_W),
\end{equation*}
and the associated first Stiefel-Whitney class $w_1(L_x) \in H^1(B^{\sigma}(W,\upiota_W);\Ftwo)$.

	\item If $\upiota(x) \ne x$, i.e. $x$ does not lie on the fixed point locus $\Sigma$, then the evaluation map is $S^1$-valued
\begin{equation*}
	\ev_x: \mathcal G(\upiota_W) \to S^1
\end{equation*}
which gives rise to a complex line bundle
\[
	L_x \to \mathcal B^{\sigma}(W,\upiota_W)
\]
and the second Stiefel-Whitney class $w_2(L_x) \in H^2(B^{\sigma}(W,\upiota_W);\Ftwo)$.
\end{itemize}
\begin{rem}
The map $\HMR^{\circ}(w_i(L_x)|W)$ can be more concretely defined by counting moduli spaces of the form
\begin{equation*}
	M_z([\mathfrak a],(W^*,\upiota_W),[\mathfrak b]) \cap V_x,
\end{equation*}
where $V_x \subset \mathcal B^{\sigma}([\mathfrak a], (W^*,\upiota_W,),[\mathfrak b] ) $ is the zero set of a smooth section $s_x$ of $L_x$.
\end{rem}

Assume now $(Y_{\pm}, \upiota_{\pm}) = (Y,\upiota)$ and $(W,\upiota_W) = ([0,1] \times Y, \text{id} \times \upiota)$.
The fixed point locus is simply $[0,1] \times K$, where $K \subset Y$ is the fixed point locus of $\upiota$.
Fix $y \in Y$.
We choose a point $x = (t, y) \in [0,1] \times Y$ for some $t \in (0,1)$.
Associated to the above two bullet points, we define the action of point classes on $\HMR(Y,\upiota)$ as follows.
\begin{itemize}[leftmargin=*]
\item If $x$ is a fixed point,
	, we define $\upsilon_x$ as $\HMR^{\circ}(w_1(L_x)|W)$:
\begin{equation*}
	\upsilon_x:\HMR_{*}^{\circ}(Y, \upiota) \to \HMR_{*-1}^{\circ}(Y, \upiota).
\end{equation*}

	\item If $x$ is not a fixed point, we define $U_x$ as $\HMR^{\circ}(w_2(L_x)|W)$:
\begin{equation*}
	U_x:\HMR_*^{\circ}(Y, \upiota) \to \HMR_{*-2}^{\circ}(Y, \upiota)  
\end{equation*}
\end{itemize}
If $x_1, x_2$ lie on the same component of fixed points, then $\upsilon_{x_1} = \upsilon_{x_2}$.
Moreover, all $U_x = U$ are cohomologous, and
\begin{equation*}
	U = \upsilon_x^2
\end{equation*}
for any fixed point $x$. 
In general, to define the Floer cap product we use the isomorphism
\begin{equation*}
	H^*(\mathcal B^{\sigma}([0,1] \times Y,1 \times \upiota_Y) \cong
	H^*(\mathcal B^{\sigma}(Y,\upiota_Y) 
\end{equation*}
to obtain an action of 
\begin{equation*}
	H^1(\mathbb T(Y,\upiota)) \cong H^1(Y;\mathbb Z)^{-\upiota^*} \cong H_1(Y;\mathbb Z)^{-\upiota_*}/\text{Tor}.
\end{equation*}
By composition law, this defines a action of the exterior algebra $\Lambda(H_1(Y;\mathbb Z)^{-\upiota_*}/\text{Tor})$ on $\HMR^{\circ}_*(Y,\upiota)$
equipped with negative grading.
Similarly, for Floer cohomology, this procedure defines Floer theoretic cup products.

\subsubsection*{\textbf{Module structures on} $\boldsymbol{\HMR^{\circ}(K)}$}
\hfill \break
In the case of double branched cover of 3-spheres along links $K$, evaluations of the point classes define actions:
\begin{equation*}
	\upsilon_x, U_x:\HMR^{\circ}(K) \to \HMR^{\circ}(K) ,
\end{equation*}
where $\upsilon_x$ depends on the component of $K$.
Let $n$ be the number of components of $K$.
Then $\HMR^{\circ}(K)$ is a $\mathcal R_n$-module, over the graded ring 
\begin{equation*}
	\mathcal R_n = \frac{\mathbb F_2[\upsilon_1, \dots, \upsilon_n]}{\upsilon_i^2 = \upsilon_j^2},
\end{equation*}
and each $\upsilon_j$ has degree $-1$.

\section{More Examples}
\label{sec:examples}
While this section is a continuation of Subsection~\ref{subsec:examples_involutions} on examples of 3-manifolds with involutions, the exposition will be self-contained.
\subsection{Involution-invariant positive scalar curvature}
\label{subsec:psc}
\hfill \break 
Let $Y$ be a 3-manifold, and $\upiota:Y \to Y$ be an involution.
Suppose the Riemannian metric is $\upiota$-invariant and has positive scalar curvature.
Let $(\mathfrak s, \uptau)$ be a real spin\textsuperscript{c} structure, such that $c_1(\mathfrak s)$ is torsion.
Let $\mathbb T$ be the invariant Clifford torus $H^1(Y;i\reals)^{-\upiota^*}/H^1(Y;i\mathbb Z)^{-\upiota^*}$,
parametrizing the reducible solutions to the unperturbed equations.
Let $f: \mathbb T \to \reals$ be a Morse function, and $f_1 = f \circ p$ be the corresponding function on $\mathcal B(Y,\upiota;\mathfrak s, \uptau)$, defined in the perturbation section.
The connections $[A]$ in $\mathbb T$ are those with $A^t$ flat, and it follows the Weitzenb\"{o}ck formula that the corresponding Dirac operator $D_A$ has no kernel.
It follows that for any path $[A(t)]$ in $\mathbb T$, the family of 3-dimensional Dirac operator has no spectral flow.
If $[\alpha]$ and $[\beta]$ are critical points in the blow-up as $[\mathfrak a_i]$ and $[\mathfrak b_j]$ in increasing order of the index, with $[\mathfrak a_0]$ and $[\mathfrak b_0]$ corresponding to the first positive eigenvalues of the Dirac operator at $\alpha$ and $\beta$ respectively.
The critical points are boundary stable if $i \ge 0$ and unstable if $i < 0$.
The relative grading on the complex $\bar C_*$ is given by
\begin{equation*}
	\bar{\gr}([\mathfrak a_i], [\mathfrak b_j]) =
	\ind_f[\mathfrak a] - \ind_f[\beta] + i - j,
\end{equation*}
where the first two terms are the ordinary Morse indices.
\begin{lem}
	The component $\bar\del^s_u$ of the boundary map $\bar\del$ in the complex $\bar C_*(Y,\upiota;\mathfrak s, \uptau)$ is zero.	
\end{lem}
\begin{proof}
	While the statement is the same as \cite[Lemma 36.1.1]{KMbook2007}, the arguments are slightly different.
	Let $[\mathfrak a_i], [\mathfrak b_j]$ be two critical points.
	Since the trajectories in $\mathcal B^{\sigma}(Y,\upiota;\mathfrak s, \uptau)$ counted by $\bar\del$ projects to ordinary Morse trajectories, the index difference $\ind_f[\alpha] - \ind_f[\beta]$ can be assumed to be positive.
	If without loss of generality $i \ge 0$ and $j < 0$, then
	\begin{equation*}
		\bar{\gr}([\mathfrak a_i],[\mathfrak b_j]) \ge 2
	\end{equation*}
	and there is no trajectory from boundary-stable to boundary unstable critical points.
\end{proof}
\begin{lem}
	If we replace the Morse function $f$ by a smaller positive multiple, $\epsilon f$, then for small enough $\epsilon$, the component $\bar\del^u_s$ of the boundary map $\bar\del$ in the complex $\bar C_*(Y,\upiota;\mathfrak s,\uptau)$ is zero.	
\end{lem}
\begin{proof}
	Same as \cite[Lemma~36.1.2]{KMbook2007} which follows from a Weitzenb\"{o}ck formula type argument.

\end{proof}
To analyze the complexes $\check C_*$ and $\check C_*$, observe that because of positive scalar curvature
\begin{equation*}
	\check C_* = C^s_*, \check\del = \bar\del^s_s, \quad 
	\hat C_* = C^u_*, \hat \del = -\bar\del^u_u.
\end{equation*}
Also, for small enough $\epsilon$, with $[-1]$ indicating a degree shift, we have
\begin{equation*}
	\bar C_* = \check C_* \oplus \hat C_*[-1], \quad
	\delbar = \begin{pmatrix}
		\check\del & 0\\
		0 & -\hat\del
	\end{pmatrix}.
\end{equation*}
We have the following structural result of 3-manifolds admitting positive scalar curvature and involution.
\begin{prop}
\label{prop:psc_Floer}
	Suppose $(Y,\upiota)$ has strictly positive scalar curvature, and $(\mathfrak s, \uptau)$ is a real spin\textsuperscript{c} structure on $Y$ with $c_1(\mathfrak s)$ torsion.	
	Then the map $j$ in the long exact sequence relating the three flavours of $\HMR$ is zero and
	\begin{equation*}
		\overline{\HMR}_*(Y,\upiota;\mathfrak s,\uptau) = 
		\widecheck{\HMR}_*(Y,\upiota;\mathfrak s,\uptau) 
		\oplus
		\widehat{\HMR}_*(Y,\upiota;\mathfrak s,\uptau)_{[-1]} 
	\end{equation*}
	as groups graded by $\mathbb J(Y,\mathfrak s) \cong \mathbb Z$.
	Furthermore, we can choose a base point on the fixed point locus of $\upiota$ and an identification of $\mathbb J(Y,\mathfrak s)$ with $\mathbb Z$, we have, as $\mathbb Z$-graded module over $\mathbb F_2[\upsilon]$, we have
	\begin{align*}
		\overline{\HMR}_*(Y,\upiota;\mathfrak s,\uptau) 
		&= H_*(\mathbb T;\mathbb F_2) \otimes \mathbb F_2[\upsilon^{-1},\upsilon]\\
		\widehat{\HMR}_*(Y,\upiota;\mathfrak s,\uptau) 
		&= H_*(\mathbb T;\mathbb F_2) \otimes \mathbb F_2[\upsilon]\\
		\widecheck{\HMR}_*(Y,\upiota;\mathfrak s,\uptau) 
		&= H_*(\mathbb T;\mathbb F_2) \otimes \left( \mathbb F_2[\upsilon^{-1},\upsilon]/\mathbb F_2[\upsilon]\right).
	\end{align*}
	If the spin\textsuperscript{c} structure non-torsion, then the Floer homology group is trivial.
\end{prop}

\subsection{Lens spaces and 2-bridge knots}
\hfill \break 
Let $L(p,q)$ be the lens space for $p > q$ coprime and $p > 2$. 
Concretely, $L(p,q)$ is the quotient space of the unit sphere $S^3 \subset \C^2$ under the isometric action
\begin{equation*}
	(z_1,z_2) \mapsto (e^{2\pi i/p}z_1,e^{2\pi i/q}z_2).
\end{equation*}
Under the induced metric from $S^3$, the conjugation involution $\upiota$
\begin{equation*}
	(z_1,z_2) \mapsto (\bar z_1,\bar z_2)
\end{equation*}
descends to an isometry, so we are in the framework of Proposition~\ref{prop:psc_Floer}.
Furthermore, $\upiota$ is the covering involution of  $L(p,q) \to (S^3, K(p,q))$ where $K(p,q)$ is a two-bridge knot.
There is a unique spin structure and $p$ spin\textsuperscript{c} structures.
Hence every spin\textsuperscript{c} structure admits a unique compatible real structure, so there are $p$ real spin\textsuperscript{c} structures.
The $\HMR^{\circ}$ of the lens space with $\upiota$ action, equipped with real spin\textsuperscript{c} structure $(\mathfrak s,\uptau)$, is isomorphic to $\HMR^{\circ}$ of $S^3$ with the unique real spin\textsuperscript{c} structure:
\begin{align*}
		\overline{\HMR}_*(K(p,q);\mathfrak s,\uptau) 
		&\cong \mathbb F_2[\upsilon^{-1},\upsilon],\\
		\widehat{\HMR}_*(K(p,q);\mathfrak s,\uptau) 
		&\cong \mathbb F_2[\upsilon],\\
		\widecheck{\HMR}_*(K(p,q);\mathfrak s,\uptau) 
		&\cong \mathbb F_2[\upsilon^{-1},\upsilon]/\mathbb F_2[\upsilon].
\end{align*}
When the spin\textsuperscript{c} is the one induced from the spin structure, we will assign absolute $\mathbb Q$-gradings to $\HMR^{\circ}$ in a later paper, which completely determines the Floer homologies of the two-bridge knot.

\subsection{$\boldsymbol{S^1\times S^2}$ and the 2-component unlink}
\hfill \break 
Let $Y = S^1 \times S^2$ and $\upiota$ be the covering involution of the covering map $Y \to (S^3, U_2)$.
Equip $Y$ with a product metric which is invariant under $\upiota$ and has positive scalar curvature.
Let $\mathfrak s_0$ be the torsion spin\textsuperscript{c} structure and $\uptau_0:S \to S$ be a real structure.
The involution $-\upiota^*$ acts trivially on first cohomology, so the invariant Clifford torus is a circle
\begin{equation*}
	\mathbb T = H^1(Y; i\reals)/H^1(Y; i\mathbb Z).
\end{equation*}
The space of equivalence classes of configuration $\mathcal B^{\sigma}(Y, \upiota;\mathfrak s, \uptau)$ has the homotopy type of
\begin{equation*}
	\mathbb T \times \mathbb{RP}^{\infty}.
\end{equation*}
Choose a Morse function $f: \mathbb T \to \reals$ with two critical points, so that
\begin{equation*}
	H_*(\mathbb T; \mathbb F_2) = \mathbb F_2 \alpha^0 \oplus \mathbb F_2 \alpha^1,
\end{equation*}
where $\alpha^i$ has Morse index $i$.
From Proposition~\ref{prop:psc_Floer}, we deduce that the Floer homologies of the 2-component unlink consists of two towers:
\begin{cor}
The real monopole Floer homologies of the 2-component unlink, with torsion spin\textsuperscript{c} structure $\mathfrak s_0$ and a compatible real structure $\uptau_0$ are
	\begin{align*}
		\overline{\HMR}_*(U_2;\mathfrak s_0,\uptau_0) 
		&\cong \mathbb F_2[\upsilon^{-1},\upsilon]  \oplus \mathbb F_2[\upsilon^{-1},\upsilon]_{[+1]}\\
		\widehat{\HMR}_*(U_2;\mathfrak s_0,\uptau_0) 
		&\cong \mathbb F_2[\upsilon] \oplus F_2[\upsilon]_{[+1]},\\
		\widecheck{\HMR}_*(U_2;\mathfrak s_0,\uptau_0) 
		&\cong \mathbb F_2[\upsilon^{-1},\upsilon]/\mathbb F_2[\upsilon] \oplus
		\left(\mathbb F_2[\upsilon^{-1},\upsilon]/\mathbb F_2[\upsilon]\right)_{[+1]},
\end{align*}
where the bracket $[+1]$ denotes shifting of degree by $+1$.
Moreover, for a non-torsion spin\textsuperscript{c} structure, all real monopole Floer homology groups are zero.
\end{cor}

\subsection{Another involution on $\boldsymbol{S^1 \times S^2}$}
\hfill \break 
Set $Y = S^1 \times S^2$, and let $\upiota: Y \to Y$ be of the form $\text{id}\times \varphi$ where $\varphi: S^2 \to S^2$ is an  order-2 rotation fixing the north and the south pole.
Choose a product Riemann metric on $Y$ of the form $\eta^2 + g_{S^2}$, invariant under $\upiota$ and where $\eta = dt$ for $t$ the coordinate on $S^1$, and $g_{S^2}$ has constant curvature.
The set of spin\textsuperscript{c} structures is a torsor over the set of line bundles, isomorphic to
\begin{equation*}
	H^2(Y) \cong  
	H^2(S^2;\mathbb Z)
	\cong \mathbb Z,
\end{equation*}
on which $\upiota^*$ acts trivially.
The condition $\upiota^*L\cong \bar{L}$ implies that only the trivial line bundle admits a real structure, so there is a unique spin\textsuperscript{c} structure admitting real structures.
Moreover, by Lemma~\ref{lem:realspinset} the choices of real structures are isomorphic to
\begin{equation*}
	H^1(S^1;\mathbb Z)/2H^1(S^1;\mathbb Z) \cong \mathbb Z/2\mathbb Z.
\end{equation*}
There exist two spin structures on $Y$, both of which are fixed by $\iota$.
Consider the unique spin structure on $S^2$, that is, a square root $E_0$ of the canonical bundle $K_{S^2}$.
The spin bundle on $S^2$ is given by 
	$$E_0 \oplus (E_0 \otimes K_{S^2}^{-1})$$
where the Clifford multiplication is given by the symbol of $\sqrt{2}(\delbar +\delbar^*)$.
Pulling back this bundle to $Y$ and letting $\eta$ act by $\pm i$ on $E_0 \oplus (E_0 \otimes K_{S^2}^{-1})$ we obtain a spin bundle on $Y$.
There is no irreducible solution to the 3-d Seiberg-Witten equations for positive scalar curvature reason. 
For each choice of real lift, there is a unique reducible critical point of CSD up to gauge.
Let $\uptau_0$ be the real structure coming from the spin lift of $\upiota$ as in Lemma~\ref{lem:spin_lift}, and $\uptau_1$ the other lift which differs by a nontrivial homomorphism $\pi_1(S^1 \times S^2) \to \{\pm 1\} \subset \U(1)$.
By similar arguments in Subsection~\ref{subsec:psc}, we expect
\begin{equation*}
	\widehat{\HMR}_*(Y,\upiota;\mathfrak s,\uptau_0) 
		\cong 
		\widehat{\HMR}_*(Y,\upiota;\mathfrak s,\uptau_1) 
		\cong
		 \mathbb F_2[\upsilon].
\end{equation*}
There are isomorphisms in the other two flavours.
This example highlights the role of choices real structures on a given spin\textsuperscript{c} structure.

\subsection{Brieskorn spheres and torus knots}
\label{subsec:Brieskorn-torus}
\hfill \break
Consider the Brieskorn integral homology sphere $Y = \Sigma(p,q,r)$ for $p = 2$, and $(p,q,r)$ coprime.
As a concrete model, $\Sigma(p,q,r)$ is
given by the link of singularity 
$$
	S^5 \cap \{z_1^p + z_2^q + z_3^r = 0\} \subset \C^3
$$
Since $p = 2$, the map $\upiota(z_1,z_2,z_3) = (-z_1,z_2,z_3)$ is an involution.
Alternatively, view $Y$ as a Seifert fibred space with three exceptional fibres of order $2,q,r$ respectively.
Let $\pi: Y \to C$ be the corresponding projection map where $C$ is the orbifold sphere $S^2(2,q,r)$.
Then $\upiota$ acts on $Y$ by rotation of the $S^1$-fibres, fixing the fibre above the order-2 orbifold point.
The image of this invariant fibre under the quotient of $Y$ to $Y/\upiota = S^3$ is the torus knot $T(q,r)$.

Let us describe the action of $\upiota$ on the Seiberg-Witten moduli spaces, via Mrowka-Ozsv\'{a}th-Yu~\cite{MOY}.
We choose a metric $g$ on $Y$ of the form $\eta^2 + \pi^*(g_C)$ where $i\eta$ is a connection form for a constant curvature connection on $C$.
Since the involution rotates the fibres, $g$ is $\upiota$-invariant.
In fact, \cite{MOY} is based a metric connection (instead of the Levi-Civita connection) that is compatible with a splitting $\T Y \cong \underline{\reals} \oplus \pi^*(\T C)$.
We can think of the use of this nonstandard metric connection as introducing a tame perturbation to the CSD functional.
There is a unique spin\textsuperscript{c} structure $\mathfrak s$ on $Y$ induced from the spin structure on $C$, so the spin bundle over $Y$ is of the form
\begin{equation*}
	S = \pi^*(E_0) \oplus \pi^*(E_0 \otimes K_{C}^{-1}).
\end{equation*}
By the vanishing spinors argument of \cite[Subsection~5.5]{MOY}, the solutions to the Seiberg-Witten equations are circle invariant, and naturally corresponding to (orbifold) K\"ahler vortices on $C$.
A \emph{K\" ahler vortex} is a triple $(B, \alpha, \beta)$, 
satisfying
\begin{align*}
	2F_{B} - F_{K_{C}} &= i(|\alpha|^2 + |\beta|^2) \\
	\delbar_{B}\alpha = 0
	&\text{ and } 
	\delbar_{B}^*\beta = 0\\
	\alpha = 0
	&\text{ or }
	\beta = 0,
\end{align*}
where $B$ is an orbifold connection on an orbifold line bundle $E$, and $\alpha,\beta$ are orbifold sections of $E$ and $E \otimes K_{C}$, respectively.
The moduli space of vortices $\mathcal M^*_{v}$, under taking the zero-set of the spinors, is isomorphic to the moduli space of effective orbifold divisors.
By \cite[Theorem~5.19]{MOY}, the space of the irreducible Seiberg-Witten solutions $\mathcal M_{sw}^*$ consits of the union of points
\begin{equation}
\label{eqn:sum_over_divisors}
    \coprod_E \mathcal \{\alpha_E, \beta_E\},
\end{equation}
where $\alpha_E$ is a vortex whose $\beta_E$-component is zero, and vice versa.
The union is taken over orbifold line bundles
$E \to C$ subject to condition that 
\begin{equation*}
    0 \le \deg(E) < \frac{\deg (K_{\Sigma})}{2}
    = \frac{1}{2}(1 - \frac{1}{2} - \frac{1}{q} - \frac{1}{r}) < 1
\end{equation*}
and $\pi^*(E) \cong \pi^*(E_0)$.
Let $\uptau$ be the real structure over $S \to Y$ induced from the spin lift $\hat{\upiota}$  of $\upiota$.
By circle invariance of solutions, $\hat{\upiota}$ acts trivially on the irreducible part of moduli space, whereas $\jmath$ acts freely by swapping $\alpha_E$ and $\beta_E$.
In particular, $\uptau$ has no irreducible fixed points.
On the other hand, the real structure a unique reducible solution $\theta$.
It follows that
\begin{equation*}
    \HMR^{\circ}(Y;\mathfrak s,\uptau) \cong
    \HMR^{\circ}(S^3;\mathfrak s_{S^3}, \mathfrak \uptau_{S^3}).
\end{equation*}
For a more detailed description of the action of $\uptau$ on the critical points and trajectory see \cite[Lemma~8.7]{montague2022seibergwitten}.

\subsection{Brieskorn spheres and Montesinos knots}
\hfill \break
Let $Y$ again be $\Sigma(p,q,r)$, and follow this same notations above.
Suppose $p,q,r$ are coprime integers.
Viewing $\Sigma(p,q,r)$ as a subset of $\C^3$, we consider the involution
\begin{equation*}
	\upiota(z_1,z_2,z_3) = 
	(\bar z_1, \bar z_2, \bar z_3).
\end{equation*}
This involution is the deck transformation (see e.g. \cite{Saveliev1999}) of the double branched cover $Y \to S^3$, along the Montesinos knot $k(p,q,r)$.
If we think of $Y$ as a Seifert-fibred space $\pi:Y \to C$, then $\upiota$ fixes two points on each fibre, and the image of the fixed point set in $Y$ under $\pi$ is a circle in $C = S^2(p,q,r)$.
And $\upiota$ descends to $C$ as a reflection along the circle.
In particular, the three orbifold points lie on this circle.

We will not carry out the computations of $\HMR^{\circ}(Y)$ in full, but let us convince the readers that there are examples where $\HMR$'s do contain irreducibles.
We appeal again to the Mrowka-Ozsv\'{a}th-Yu description of the critical points of the CSD functional on Seifert-fibred spaces.

First of all, 
the metric $\eta^2 + \pi^*(g_C)$ can be assumed to be $\upiota$-invariant.
Let $\mathfrak s$ be the unique spin structure on $Y$, and $S$ be the spin bundle.
Choose a spin lift $\hat{\upiota}$ and compose it with the right $j$-multiplication to obtain a real structure $\uptau: S \to S$.
To describe the  $\uptau$-real solutions to the Seiberg-Witten equations, modulo the real gauge group, it suffices to describe the action of $\uptau$ on the equivalence classes of ordinary Seiberg-Witten solutions.
Indeed, since $b^1(Y) = 0$, for any solution whose gauge equivalence class is preserved by $\uptau$, we choose an appropriate square root of a gauge transformation to find a representative fixed by $\uptau$.

The pullback 2-form $i\pi^*(\mu_{C})$ of the volume form $\mu_{C}$ on $C$ acts on $S$ by Clifford multiplication, which induces an eigen-decomposition $S \cong S^+ \oplus S^-$, where $i\pi^*(\mu_{C})$ acts by $\pm 1$ on $S^{\pm}$.
Each $S^{\pm}$ is isomorphic to the pullback of a line bundles on $C$, and \cite{MOY} tells us that a solution to the Seiberg-Witten equations are sections either of $S^{+}$ or $S^-$.
But for any spinor $\Psi \in \Gamma(S^+)$, we have
\begin{equation*}
   \rho(i\pi^*(\mu_{C}))\uptau(\Psi)
    = -\uptau(\rho(i\upiota^*\pi^*(\mu_{C}))\Psi)
    = \uptau(\rho(i\pi^*(\mu_{C}))\Psi)
    = \uptau(\Psi).
\end{equation*}
by the anti-linearity, compatibility, and the orientation-reversing property of $\upiota$ on the orbifold Riemann surface.
Since $\upiota$ fixes the divisors, it follows from the above observation that $\uptau$ acts trivially on $\mathcal M_{sw}^*$.
Hence the set of real critical points consists of a unique reducible $\theta$ and the set of irreducible critical points
\begin{equation*}
    \coprod_E \mathcal \{\alpha_E, \beta_E\}.
\end{equation*}

Unlike the ordinary monopole theory, it is no longer the case that the relative indices between the irreducible critical points are even integers.
Indeed, since there is no spectral flow in the connection part, the relative index in ordinary $\HM$ is the real index of a complex Dirac operator, and in $\HMR$ the relative index will be half of the ordinary index.
The relative index in turn was calculated in \cite[Corollary~1.4]{MOY}, using the holomorphic description of the trajectory spaces as divisors on ruled surfaces.
Since the real structure is compatible with the Seifert structure, the holomorphic description can be adapted to the real setting.
The corresponding spaces of real trajectories between the irreducible critical points of the CSD functional can be understood as spaces of divisors invariant under an anti-holomorphic involution on the corresponding ruled surface in \cite{MOY}.
There is a similar interpretation of trajectories from irreducibles to the reducible \cite[Section~10]{MOY} in the real setup.
The author hopes to return to the MOY description of real trajectories in a future paper, as it is crucial for computations of $\HMR^{\circ}$ and Fr\o yshov invariant.

In special cases a closer examination at the values of the CSD functional, compared with the relative indices, rules out the existence of trajectories between some critical points, allowing us to compute $\HMR^{\circ}$.
Indeed, let us consider some families of Brieskorn spheres, with deck transformations from Montesinos knots.

We begin with some general facts about the $\Sigma(p,q,r)$'s.
The chain complexes $\mathfrak C$ of $\HMR^{\circ}$'s are generated by the irreducible solutions of the Seiberg-Witten equations, along with a $\mathbb Z$-tower $\{\theta_i: i \in \mathbb Z\}$ of reducible critical points above the unique reducible Seiberg-Witten solution $\theta$.
Consecutive reducibles differ by relative index $\gr(\theta_i,\theta_{i-1}) = 1$, except that $\gr(\theta_0,\theta_{-1}) = 0$, where $\theta_0$ is boundary stable, and $\theta_{-1}$ is boundary unstable.
Since the value of CSD at the reducible $\theta$ is strictly less than the values at the irreducibles and the flow lines on the blow-up projects down to flow lines of CSD, there does not exist any flow line from the reducibles to irreducibles.
As a convention, we set an absolute grading by declaring the canonical irreducible from the empty orbifold divisor to have degree zero.
\begin{itemize}[leftmargin=*]
	\item $Y_k = \Sigma(2,3,6k+1)$. 
	There are $2\lfloor k/2 \rfloor$ irreducibles all with grading $0$.
	The reducible $\theta_{-1}$ is one degree lower than the irreducibles, so there is no differential between the reducibles and irreducibles, or between irreducibles.
	We conclude that
	\begin{equation*}
		\widehat{\HMR}(Y_k,\upiota;\mathfrak s,\uptau) \cong
		(\mathbb F_2)^{2\lfloor k/2 \rfloor} \oplus
		\mathbb F_2[\upsilon], \quad
		\widecheck{\HMR}(Y_k,\upiota;\mathfrak s,\uptau) \cong
		(\mathbb F_2)^{2\lfloor k/2 \rfloor} \oplus
		\mathbb F_2[\upsilon^{-1},\upsilon]/\mathbb F_2[\upsilon],
	\end{equation*}
	where $\upsilon$ has degree $(-1)$ and each $(\mathbb F_2)$-summand is generated by an irreducible.
	
	\item $Y_k = \Sigma(2,3,6k-1)$.
	There are again $2\lfloor k/2 \rfloor$ irreducibles supported at index zero, but now the reducible $\theta_{-1}$ has index $-1$.
	Since $\theta$ achieves minimum of $\CSD$, we can still deduce that
	\begin{equation*}
		\widecheck{\HMR}(Y_k,\upiota;\mathfrak s,\uptau) \cong
		(\mathbb F_2)^{2\lfloor k/2 \rfloor} \oplus
		\mathbb F_2[\upsilon^{-1},\upsilon]/\mathbb F_2[\upsilon],
	\end{equation*}
	but there could be trajectories from the boundary unstable critical point to an irreducible that may contribute to the differential $\hat{\del}$ of $\widehat{\HMR}$.
	Computation of this component of the differential requires better understanding of the real flow lines from irreducible to reducibles.
	Thus we are unable to compute $\widehat{\HMR}(Y_k,\upiota;\mathfrak s,\uptau)$ the same way as the previous family.
	\item $Y_k = \Sigma(2,5,10k-1)$. 
	For $0 \le i \le k-1$, there are two irreducibles with index $i$, and at index $k$ there are $2\lfloor k/2 \rfloor$ irreducibles.
	The indices are inverse-proportional with respect to the values of CSD, so there exists no trajectory between the irreducible critical points.
	Moreover, the reducible $\theta_{-1}$ has index $k+1$, and the same reasoning before implies that
	\begin{equation*}
		\widehat{\HMR}(Y_k,\upiota;\mathfrak s,\uptau) \cong
		(\mathbb F_2)^{2k+2\lfloor k/2 \rfloor} \oplus
		\mathbb F_2[\upsilon], \quad
		\widecheck{\HMR}(Y_k,\upiota;\mathfrak s,\uptau) \cong
		(\mathbb F_2)^{2k+2\lfloor k/2 \rfloor} \oplus
		\mathbb F_2[\upsilon^{-1},\upsilon]/\mathbb F_2[\upsilon].
	\end{equation*}
	\item $Y_k = \Sigma(2,5,10k+1)$. 
	The irreducibles are the same as in the $(10k-1)$-case, but $\theta_{-1}$ now has grading $k$.
	Hence we stil have
	\begin{equation*}
		\widehat{\HMR}(Y_k,\upiota;\mathfrak s,\uptau) \cong
		(\mathbb F_2)^{2k+2\lfloor k/2 \rfloor} \oplus
		\mathbb F_2[\upsilon], \quad
		\widecheck{\HMR}(Y_k,\upiota;\mathfrak s,\uptau) \cong
		(\mathbb F_2)^{2k+2\lfloor k/2 \rfloor} \oplus
		\mathbb F_2[\upsilon^{-1},\upsilon]/\mathbb F_2[\upsilon].
	\end{equation*}
	\item $Y = \Sigma(2,7,29)$. 
	There are $6$ pairs of irreducibles, and in the decreasing order of CSD-values, their gradings are
	\begin{equation*}
		0, 2, 4, 5, 6, 5.
	\end{equation*}
	The index-$1$ flow lines from an index-$6$ critical point to an index-$5$ critical point cannot be excluded for index reason alone.
	(The ordinary monopole relative index would be $2$.)
	We expect such components of the differential to vanish.
	Indeed, from the holomorphic description of trajectories \cite[Section~7,8,9]{MOY}, the non-empty space of index-$1$ trajectories is isomorphic to $\C^* \cong S^1 \times \reals$.
	The corresponding moduli space of real trajectories is the union of two $\reals$'s, which modulo reparametrization, contributes $0$ over $\mathbb F_2$ to the differential.	
	Finally, the boundary unstable reducible $\theta_{-1}$ has grading $7$ so there is no differential between the reducibles and irreducibles.
\end{itemize}
\bibliographystyle{alpha}
\bibliography{./HMR.bib}

\begin{thebibliography}{MOY97}

\bibitem[AB68]{AtiyahBott1968II}
M.~F. Atiyah and R.~Bott.
\newblock A {L}efschetz fixed point formula for elliptic complexes. {II}. {A}pplications.
\newblock {\em Ann. of Math. (2)}, 88:451--491, 1968.

\bibitem[APS73]{APS1973}
M.~F. Atiyah, V.~K. Patodi, and I.~M. Singer.
\newblock Spectral asymmetry and {R}iemannian geometry.
\newblock {\em Bull. London Math. Soc.}, 5:229--234, 1973.

\bibitem[Ati66]{AtiyahK1966}
M.~F. Atiyah.
\newblock {$K$}-theory and reality.
\newblock {\em Quart. J. Math. Oxford Ser. (2)}, 17:367--386, 1966.

\bibitem[BH21]{BaragliaHekmati2021}
David Baraglia and Pedram Hekmati.
\newblock Equivariant seiberg-witten-floer cohomology.
\newblock {\em arXiv 2108.06855}, 2021.

\bibitem[BH22]{BaragliaHekmati2022}
David Baraglia and Pedram Hekmati.
\newblock Brieskorn spheres, cyclic group actions and the milnor conjecture.
\newblock {\em arXiv 2208.05143}, 2022.

\bibitem[CO21]{CengizOzturk2021}
Merve Cengiz and Ferit Ozturk.
\newblock Every real 3-manifold is real contact.
\newblock {\em arXiv 2104.05265}, 2021.

\bibitem[Flo88]{FloerInstanton1988}
Andreas Floer.
\newblock An instanton-invariant for {$3$}-manifolds.
\newblock {\em Comm. Math. Phys.}, 118(2):215--240, 1988.

\bibitem[Fy10]{Froyshov2010}
Kim~A. Fr\o~yshov.
\newblock Monopole {F}loer homology for rational homology 3-spheres.
\newblock {\em Duke Math. J.}, 155(3):519--576, 2010.

\bibitem[Gay04]{GayetSWR2004}
Damien Gayet.
\newblock Seiberg-witten invariants and real curves.
\newblock {\em https://arxiv.org/abs/math/0404556}, 2004.

\bibitem[Kat22]{Kato2022}
Yuya Kato.
\newblock Nonsmoothable actions of {$\Bbb Z_2 \times \Bbb Z_2$} on spin four-manifolds.
\newblock {\em Topology Appl.}, 307:Paper No. 107868, 13, 2022.

\bibitem[KM07]{KMbook2007}
P.~B. Kronheimer and T.~S. Mrowka.
\newblock {\em Monopoles and three-manifolds}, volume~10 of {\em New Mathematical Monographs}.
\newblock Cambridge University Press, Cambridge, 2007.

\bibitem[KM11]{KMunknot2011}
P.~B. Kronheimer and T.~S. Mrowka.
\newblock Khovanov homology is an unknot-detector.
\newblock {\em Publ. Math. Inst. Hautes \'{E}tudes Sci.}, (113):97--208, 2011.

\bibitem[KMT21]{KMT2021}
Hokuto Konno, Jin Miyazawa, and Masaki Taniguchi.
\newblock Involutions, knots, and floer k-theory.
\newblock {\em arXiv 2110.09258}, 2021.

\bibitem[KMT22]{KMT2022}
Hokuto Konno, Jin Miyazawa, and Masaki Taniguchi.
\newblock Involutions, knots, and floer cohomologies.
\newblock {\em in Preparation}, 2022.

\bibitem[KT76]{KauffmanTaylor1976}
Louis~H. Kauffman and Laurence~R. Taylor.
\newblock Signature of links.
\newblock {\em Trans. Amer. Math. Soc.}, 216:351--365, 1976.

\bibitem[Kui65]{KUIPER196519}
Nicolaas~H. Kuiper.
\newblock The homotopy type of the unitary group of hilbert space.
\newblock {\em Topology}, 3(1):19--30, 1965.

\bibitem[LM18]{LidmanManolescu2018}
Tye Lidman and Ciprian Manolescu.
\newblock The equivalence of two {S}eiberg-{W}itten {F}loer homologies.
\newblock {\em Ast\'{e}risque}, (399):vii+220, 2018.

\bibitem[Man03]{Manolescu2003}
Ciprian Manolescu.
\newblock Seiberg-{W}itten-{F}loer stable homotopy type of three-manifolds with {$b_1=0$}.
\newblock {\em Geom. Topol.}, 7:889--932, 2003.

\bibitem[Mon22]{montague2022seibergwitten}
Ian Montague.
\newblock Seiberg-witten floer k-theory and cyclic group actions on spin four-manifolds with boundary.
\newblock {\em arXiv 2210.08565}, 2022.

\bibitem[MOY97]{MOY}
Tomasz Mrowka, Peter Ozsv\'{a}th, and Baozhen Yu.
\newblock Seiberg-{W}itten monopoles on {S}eifert fibered spaces.
\newblock {\em Comm. Anal. Geom.}, 5(4):685--791, 1997.

\bibitem[Nag70]{KNagami1970}
Kei\^{o} Nagami.
\newblock {\em Dimension theory}.
\newblock Pure and Applied Mathematics, Vol. 37. Academic Press, New York-London, 1970.
\newblock With an appendix by Yukihiro Kodama.

\bibitem[Nak15]{NNakamura2015}
Nobuhiro Nakamura.
\newblock {$\mathrm{Pin}^{-} (2)$-monopole invariants}.
\newblock {\em Journal of Differential Geometry}, 101(3):507 -- 549, 2015.

\bibitem[OS15]{OzturkSalepci2015}
Ferit Ozturk and Nermin Salepci.
\newblock Real open books and real contact structures.
\newblock {\em Advances in Geometry}, 15(4):415--431, 2015.

\bibitem[Sav99]{Saveliev1999}
Nikolai Saveliev.
\newblock {Floer homology of Brieskorn homology spheres}.
\newblock {\em Journal of Differential Geometry}, 53(1):15 -- 87, 1999.

\bibitem[Tau00]{TaubesSWGR2000}
Clifford~Henry Taubes.
\newblock {\em Seiberg {W}itten and {G}romov invariants for symplectic {$4$}-manifolds}, volume~2 of {\em First International Press Lecture Series}.
\newblock International Press, Somerville, MA, 2000.
\newblock Edited by Richard Wentworth.

\bibitem[TW09]{TianWang2009}
Gang Tian and Shuguang Wang.
\newblock Orientability and real {S}eiberg-{W}itten invariants.
\newblock {\em Internat. J. Math.}, 20(5):573--604, 2009.

\end{thebibliography}
\end{document}